\documentclass{uiucthesis2021}
\usepackage[english]{babel}
\usepackage{csquotes}
\usepackage{microtype}
\usepackage{amsmath,amsthm,amssymb}
\usepackage[usenames]{color}
\usepackage[bookmarksdepth=3,linktoc=all,colorlinks=true,urlcolor=blue,linkcolor=blue,citecolor=blue]{hyperref}
\usepackage[capitalize]{cleveref}
\usepackage[backend=biber, style=ieee]{biblatex}
\usepackage{tabularx,multirow} 
\usepackage{graphicx,float}

\AfterPreamble{\hypersetup{
		pdfauthor={Jin Won Kim},
		pdftitle={Duality for Nonlinear Filtering},
		pdfsubject={Mathematics; Engineering, Electronics and Electrical; Engineering, Mechanical},
		pdfkeywords={Stochastic filtering, Backward stochastic differential equations, Duality}
}}

\def\Re{\mathbb{R}}

\newcommand{\tr}{\mbox{tr}}
\def\IEEEQEDclosed{\mbox{\rule[0pt]{1.3ex}{1.3ex}}}
\def\qed{\nobreak\hfill\IEEEQEDclosed}
\newcommand{\ud}{\,\mathrm{d}}
\def\half{\frac{1}{2}}
\newcommand{\newP}[1]{\medskip\noindent{\bf #1:}}

\theoremstyle{plain}
\newtheorem{theorem}{Theorem}[chapter]
\newtheorem{proposition}{Proposition}[chapter]
\newtheorem{corollary}{Corollary}[chapter]
\newtheorem{lemma}{Lemma}[chapter]

\theoremstyle{definition} 
\newtheorem{example}{Example}[chapter]
\newtheorem{definition}{Definition}[chapter]
\newtheorem{remark}{Remark}[chapter]

\newtheorem{assumption}{Assumption}[chapter]

\def\clA{{\cal A}}
\def\clB{{\cal B}}
\def\clC{{\cal C}}

\def\clD{{\cal D}}
\def\clE{{\cal E}}
\def\clF{{\cal F}}

\def\clH{{\cal H}}
\def\clI{{\cal I}}

\def\clL{{\cal L}}

\def\clM{{\cal M}}
\def\clN{{\cal N}}
\def\clO{{\cal O}}
\def\clP{{\cal P}}

\def\clS{{\cal S}}

\def\clU{{\cal U}}
\def\clV{{\cal V}}

\def\clX{{\cal X}}
\def\clY{{\cal Y}}
\def\clZ{{\cal Z}}
\def\E{{\sf E}}
\def\bS{\mathbb{S}}
\def\sJ{{\sf J}}
\def\bsJ{{\sf J}}
\def\ones{{\sf 1}}
\def\sP{{\sf P}}
\def\tsP{{\tilde{\sf P}}}
\def\tE{{\tilde{\sf E}}}

\def\tp{{\hbox{\rm\tiny T}}}

\def\Nsp{{\sf N}} 
\def\Rsp{{\sf R}} 
\def\hE{\tilde{\sf E}}

\def\dv{\operatorname{diag}}
\def\sp{\operatorname{span}}

\def\bmu{{\bar\mu}}
\def\bSig{\bar{\Sigma}}
\def\barX{\bar{X}}

\def\den{\operatorname{enr}}
\def\dvar{\operatorname{var}}

\def\cov{{\cal V}}

\def\chisq{{\chi^2}}
\def\kl{{\sf D}}
\def\tv{{\mathrm{TV}}}
\def\sW{{\sf W}}

\def\sQ{{\sf Q}}

\def\opt{{\text{\rm (opt)}}}
\renewcommand{\limsup}{\mathop{\operatorname{limsup}}}
\renewcommand{\liminf}{\mathop{\operatorname{liminf}}}
\def\weakto{\stackrel{\ast}{\rightharpoonup}}
\def\divg{\nabla\cdot}

\begin{document}

\title{Duality for Nonlinear Filtering}
\author{Jin Won Kim}
\department{Mechanical Engineering}
\phdthesis
\degreeyear{2022}
\committee{
    Professor Prashant G.~Mehta, Chair and Director of Research\\
    Professor Bruce Hajek\\
    Associate Professor Maxim Raginsky\\
    Assistant Professor Partha S.~Dey}
\maketitle

\frontmatter

\begin{abstract}
	
	This thesis is concerned with the stochastic filtering problem for a hidden Markov model (HMM) with the white noise observation model. For this filtering problem, we make three types of original contributions: (1) dual controllability characterization of stochastic observability, (2) dual minimum variance optimal control formulation of the stochastic filtering problem, and (3) filter stability analysis using the dual optimal control formulation.
	
	For the first contribution of this thesis, a backward stochastic differential equation (BSDE) is proposed as the dual control system. The observability (detectability) of the HMM is shown to be equivalent to the controllability (stabilizability) of the dual control system. For the linear-Gaussian model, the dual relationship reduces to classical duality in linear systems theory. 
	
	The second contribution is to transform the minimum variance estimation problem into an optimal control problem. The constraint is given by the dual control system. The optimal solution is obtained via two approaches: (1) by an application of maximum principle and (2) by the martingale characterization of the optimal value. The optimal solution is used to derive the nonlinear filter.
		
	The third contribution is to carry out filter stability analysis by studying the dual optimal control problem. Two approaches are presented through  Chapters 7 and 8. In Chapter 7, conditional Poincar\'e inequality (PI) is introduced. Based on conditional PI, various convergence rates are obtained and related to literature. In Chapter 8, the stabilizability of the dual control system is shown to be a necessary and sufficient condition for filter stability on certain finite state space model.

\end{abstract}


\begin{acknowledgments}

First I appreciate my advisor, Professor Prashant Mehta, who led me to accomplish Ph.D successfully. 
I am grateful for his continuous financial / academic / mental support over 6 years.
He has not only given me academic advice, but also encouraged me when I was having hard time.
I remember the first day I just ask you to be my academic advisor after a class, and it turns out the day was the luckiest day in my grad life.

I was lucky to have my committee members, Professors Bruce Hajek, Maxim Raginsky, and Partha Dey. I appreciate for sharing their wisdom and giving me challenge during prelim and final exams.
I took courses with each of my committee members, and those courses were so great that I had to ask for being a committee member. I appreciate everyone for granting their time on examinations.

My sincere appreciation goes to collaborators of my works. I thank Professor Sean Meyn for giving me great feedback on my papers and sharing his positive energy and passion.
Professor Amirhossein Taghvaei (Amir) deserves a huge appreciation for being a good friend and a wonderful collaborator. His work and attitude have always been a great motivation.
I thank Yagiz Olmez, Shubham Aggrawal, Erik Miehling and his colleagues for allowing me to tag in to their project. 
Although I don't know their names, there was huge help from anonymous reviewers on my conference papers. 

I acknowledge all the helps from administration and staffs in Coordinated Science Laboratory and department of Mechanical Science and Engineering, including but not limited to: Angie Ellis, Stephanie McCullough, and Kathy Smith.
I also thank the IEEE committee who granted me the best student paper award in 2019 IEEE 58th Conference on Decision and Control in Nice, France.

I thank my colleagues in CSL348. A lot of thanks to Amir, Chi Zhang, Ram Sai Gorugantu and Mayank Baranwal for helping me settle down in CSL 348 and for our friendship.
I thank Heng-Sheng Chang, Tixian Wang, Udit Halder, and Anant Joshi for cheering me up and taking all my silly joke and pranks. I thank Yagiz again for making my last day in Champaign.
I thank Prabhat Mishra for great advice and helpful discussions.  
Although our time in CSL 348 does not overlap, I thank Adam Tilton and Shane Ghiotto for sharing industrial inspiration on my early work.
I am also thankful to my friends outside of CSL who shared enjoyable time in Urbana-Champaign.

Most of all, I would like to thank my family. My parents always believe me and give infinite love and support.
I also thank my brother Joowon and his wife Youngrang for having joyful time whenever we meet.

This thesis is dedicated to my loving memory of grandpa. I am sad that he cannot see me graduating, but I believe he would have been proud of me.

\end{acknowledgments}

{
    \hypersetup{linkcolor=black}  
	\setcounter{tocdepth}{1}  
    \tableofcontents
}

\chapter{List of Abbreviations}

\begin{abbrevlist}
	\item[ARE] Algebraic Riccati Equation
	\item[B.M.] Brownian motion
	\item[BSDE] Backward Stochastic Differential Equation
	\item[DRE] Dynamic Riccati Equation
	\item[HMM] Hidden Markov Model
	\item[KL] Kullback–Leibler (divergence)
	\item[LTI] Linear Time-Invaraiant
	\item[LQ] Linear Quadratic
	\item[MAP] Maximum A Posteriori
	\item[MMSE] Minimum Mean-Squared Error
	\item[ODE] Ordinary Differential Equation
	\item[PDE] Partial Differential Equation
	\item[RN] Radon-Nikodym (derivative)
	\item[SDE] Stochastic Differential Equation
	\item[SPDE] Stochastic Partial Differential Equation
\end{abbrevlist}

\chapter{List of Symbols}

\begin{symbollist}[0.7in]
\item[$\bS$] State space.
\item[$\Re^d$] Euclidean space of dimension $d$.
\item[$\clA$] Infinitesimal generator of the state process.
\item[$\clA^\dagger$] The adjoint operator of $\clA$.
\item[$\Gamma$] Carr\'e du champ operator.
\item[$C_b(\bS)$] Set of continuous and bounded functions on $\bS$.
\item[$\clM(\bS)$] Set of regular, bounded, finitely additive measures on $\bS$.
\item[$\clP(\bS)$] Set of probability measures on $\bS$.
\item[$\langle \cdot,\cdot\rangle$] Duality pairing of a vector space and its dual space, or the inner product of a Hilbert space.
\item[$\overline{A}$] Closure of $A$.
\item[$\dv(\cdot)$] Vector-to-matrix diagonal operator.
\item[$\dv^\dagger(\cdot)$] Matrix-to-vector diagonal operator.
\item[$\nabla f$] Gradient of $f$.
\item[$\divg f$] Divergence of $f$.
\item[$\Delta f$] Laplacian of $f$.
\item[$N(m_0,\Sigma_0)$] Gaussian density with mean $m_0$ and variance $\Sigma_0$.
\item[$\kl(\mu\mid\nu)$] Kullback–Leibler (KL) divergence of $\mu$ from $\nu$.
\item[$\chisq(\mu\mid\nu)$] Pearson's $\chi^2$ divergence of $\mu$ from $\nu$.
\item[$\clF_t$] The canonical filtration.
\item[$\clZ_t$] The filtration generated by the observation process. 
\item[$\pi_t^\mu$] Nonlinear filter at time $t$ from initial distribution $\mu$ (superscript is often omitted.)
\item[$\sigma_t^\mu$] Un-normalized filter at time $t$ from initial distribution $\mu$ (superscript is often omitted.)
\item[$\Psi_{t,\tau}$] Solution operator of the Zakai equation.
\item[$\ones$] Constant 1 function.
\item[$\ones_A$] Indicator function for the set $A$.
\item[$\clV_t(f)$] Conditional variance at time $t$ of the function $f$.
\item[$\cov_t(f,g)$] Conditional covariance at time $t$ of the functions $f$ and $g$.
\end{symbollist}

\mainmatter

\chapter{Introduction}

\nocite{kim2019duality,kim2019observability,kim2021ergodic,kim2021detectable,kim2020smoothing,kim2021dynamicprogramming}

\begin{quote}
	{\it Duality in mathematics is not a theorem, but a ``principle''.}\par\raggedleft---Sir Michael F. Atiyah~\cite{atiyah2007duality}
\end{quote}

\bigskip


The word {\em duality} means a problem can be viewed in two aspects. Duality appears in many different contexts in mathematics: in linear algebra, topology, geometry, analysis, number theory, quantum physics and more~\cite{atiyah2007duality}. Duality is a principle that given a mathematical object, there is a ``dual'' object that provides better understanding of the property of the original object~\cite[Section III.19]{prinston2008}.

\subsubsection{What is duality in this thesis?}

In control theory, estimation and control are viewed as dual problems.
The most basic of these relationships is the duality between controllability and observability of a linear system~\cite{kalman1960general}.
Duality is coeval with the origin of modern systems and control theory:~it appears
in the seminal 1961 paper of Kalman and Bucy~\cite{kalman1961}, where the problem of
optimal (minimum variance) estimation is shown to be dual to a linear
quadratic optimal control problem.  Notably, duality explains why
the Riccati equation is the fundamental 
equation for \textit{both} optimal estimation and optimal control (in
the linear Gaussian settings of the problem).

Sixty years have elapsed since the original Kalman-Bucy paper. One would
imagine that duality for the nonlinear stochastic systems (hidden Markov models) is well
understood by now. It is a foundational question at the heart of
modern systems and control theory, and its modern avatars such as reinforcement learning. However, this is not the case! In his
2008 paper~\cite{todorov2008general}, Todorov writes:
\begin{quote}
	``\it{Kalman's duality has been known for half a
		century and has attracted a lot of attention. If a straightforward
		generalization to non-LQG settings was possible it would have been
		discovered long ago. Indeed we will now show that Kalman's duality,
		although mathematically sound, is an artifact of the LQG setting.}''
\end{quote}
Is this to suggest that there is no previous work to extend duality to nonlinear systems? Au contraire! 
As we describe in Chapter~\ref{ch:duality-background}, almost every definition of nonlinear observability, and there have been many throughout the decades, appeals to duality in some manner.
Likewise, Mortensen and related minimum energy algorithms, originally invented in 1960s, are routinely re-discovered.  There have been seminal contributions on the subject from Bene\v{s}, Bensoussan, Fleming, Krener, Mitter, Mortensen, and many others.
In Todorov's paper, several reasons are noted on why the duality described in the prior works of Mitter, Fleming and others (e.g.,~\cite{mortensen1968, fleming1982optimal,mitter2003}) are {\em not} generalizations of the original Kalman-Bucy duality.

\medskip

\subsubsection{How is duality useful?}  
The classical duality between controllability and observability is
useful \emph{both} for analysis and the design of estimation algorithms.  For
example, most proofs of stability of the Kalman filter (see e.g.,~\cite[Ch.~9]{xiong2008introduction}) rely---in direct or indirect
fashion---on duality theory.  Specifically,~(1) Because of duality, asymptotic stability
of the Kalman filter is equivalent to 
asymptotic stability of the (dual) optimal control problem, 
(2) necessary and
sufficient conditions for the same are stabilizability for the 
control problem, and (because of duality) detectability for the
estimation problem,
and (3) analysis of the optimal control problem (e.g., convergence of the value function to its stationary limit) yields
useful  conclusions on asymptotic filter stability.  Even in the
deterministic settings of the estimation problem, the rich literature
on the design of observers and the minimum energy estimators (MEE) is based on duality~\cite{krener2003convergence,mayne2014model}.  The asymptotic analysis of these algorithms rely on input/output-to-state
stability (IOSS) concepts which again have a distinct
control-theoretic
flavor~\cite{sontag1997output,sontag2013mathematical}.

\section{Summary of Original Contributions and its relationship to literature}

In this thesis, we consider the stochastic filtering problem for a hidden Markov model (HMM) with a white noise observation model.  The mathematical model is introduced in Section~\ref{sec:problem-formulation}. For this filtering problem, we make three types of original contributions:
\begin{enumerate}
	\item Dual controllability characterization of stochastic observability.
	\item Dual minimum variance optimal control formulation of the stochastic filtering problem.
	\item Filter stability analysis using the dual optimal control formulation.
\end{enumerate}
Each of the three contribution has a well-established foundational counterpart in linear systems: (1) Classical duality between controllability and observability is reviewed in Section~\ref{ssec:duality-LTI};
(2) Minimum variance optimal control formulation of the linear Gaussian filtering problem is reviewed in Section~\ref{ssec:Kalman-filter}; and
(3) Filter stability analysis of the Kalman filter, including a discussion of the relevance of dual technique for the same, appears in Section~\ref{ssec:Kalman-filter-stability}.

For nonlinear systems, there has been decades of research on the three topics. We provide a quick summary here with pointers to sections where additional details appear.

\paragraph{Observability.}
Generalization of the observability definition to nonlinear deterministic and stochastic systems has been an area of historical and current research interest.
Classical definitions of Krener~\cite{hermann1977nonlinear} and Sontag~\cite{sontag1997output} are reviewed in Section~\ref{ssec:nonlinear-observability}.
Both these definitions are based on duality.
For HMMs, the fundamental definition for stochastic observability is due to van Handel~\cite{van2009observability,van2009uniform}. 
The definition is reviewed in Section~\ref{ssec:observability-hmm} together with a discussion of some recent extensions in Section~\ref{ssec:other-obs-notion}.
The stochastic observability definition is entirely probabilistic. Our contribution is to describe a dual control system such that the controllability of the dual control system is equivalent to the stochastic observability of the HMM. (The equivalence is expressed in terms of the closed range theorem.) This is the main topic of Chapter~\ref{ch:observability}.

\paragraph{Duality between optimal control and filtering.}
Duality between observability and controllability suggests that the problem of filter (estimator) design can be re-formulated as a variational problem of optimal control. In classical linear Gaussian settings, the dual formulations are well-understood. These are of two types: (1) minimum variance and (2) minimum energy estimator.
Minimum variance duality is related to the filtering problem while the minimum energy duality is related to the smoothing problem.
Minimum energy duality has several counterparts in nonlinear settings. One of the earliest is the Mortensen's maximum likelihood estimator~\cite{mortensen1968}. In the model predictive control (MPC) community, minimum energy estimation~\cite{hijab1980minimum} is widely studied for algorithm design~\cite{krener2003convergence,krener2015minimum,rawlings2017model,mayne2014model}. 
Historically, one of the reason to introduce observability definition is to prove stability of the minimum energy estimator.
For the stochastic filtering and smoothing problem, the most prominent name in duality theory is Sanjoy Mitter~\cite{fleming1982optimal,mitter2003}.
Fleming-Mitter~\cite{fleming1982optimal} is one of the first paper to note that negative log of the posterior is a value function for a certain optimal control problem. Such a relationship is referred to as log transformation~\cite{fleming1978exit}. Although the meaning of the optimal control problem was not clarified in the original paper~\cite{fleming1982optimal}, Mitter-Newton~\cite{mitter2003} introduced a dual optimal control problem based on a completely classical variational interpretation of the Bayes formula.
A chapter length review of Mitter and related work is included in Appendix~\ref{apdx:min-energy} of the thesis.
In Appendix~\ref{apdx:min-energy} and also in Section~\ref{sec:historical-remarks}, it is shown that Mitter-Newton optimal control problem reduces to the minimum energy estimator for the linear Gaussian model.

\paragraph{Filter stability.}
Viewed from a certain lens, the story of filter stability
is a story of two parts: (1) stability of the Kalman filter in the
linear Gaussian settings of the problem where dual
definitions and methods are paramount, and (2) stability of the
nonlinear filter where there is no hint of such methods.  The
disconnect is already seen in the earliest works---in the two parts
of the pioneering paper of Ocone and Pardoux~\cite{ocone1996asymptotic} on the topic of
filter stability, or in the two parts of Bensoussan's 
textbook on partially observed Markov decision processes~\cite{bensoussan1992stochastic}.  One notable exception (that really proves the rule) is found in the  PhD thesis of van
Handel~\cite{van2006filtering} where Mitter-Fleming duality is used to obtain results on filter stability.
However, these results are not especially strong, in part because the duality employed is for
smoothing (and not filtering) problem. In his later papers, van Handel abandons 
the approach of his PhD thesis in favor of the so called intrinsic (probabilistic) approach to filter stability. 
A review of filter stability literature appears in Chapter~\ref{ch:filter-stability-literature}.
The prior use of dual optimal control based technique for filter stability analysis is discussed in Section~\ref{sec:66} based on van Handel's PhD thesis~\cite{van2006filtering}.


\section{Summary of papers}


The results in thesis were first reported in the following four conference papers.

\medskip

\paragraph{1. Basic paper on the subject:} The dual optimal control problem
is introduced for the first time in our 2019 paper~\cite{kim2019duality}.
The dual control system is a backward stochastic differential equation (BSDE).
It is shown to be an exact extension of the original Kalman-Bucy duality, in the sense that the
dual optimal control problem has the same minimum variance structure
for \textit{both} linear and nonlinear filtering problems. 
This paper won the \textit{Best Student Paper Award} at the IEEE Conf.~on Decision and Control (CDC) 2019 from a competitive field of 65 nominations for this award.  From
one of the anonymous reviews of the paper:
\begin{quote}
	``{\it The paper is concerned with extending, to the nonlinear
		case,
		classic duality results between control and estimation. There has been previous work on this over a period of
		decades but the 
		particular version of that problem tackled here had been
		thought to be unsolvable.}''
\end{quote}

\paragraph{2. Dual definition for observability:} In a follow-up
paper~\cite{kim2019observability}, stochastic observability of an
HMM is expressed in dual terms:
as controllability of the dual control system.  It is shown that (1) the resulting characterization is equivalent to the stochastic observability definition of van Handel, and
(2) the BSDE is a dual to the Zakai equation of nonlinear filtering.  

\paragraph{3. Filter stability of ergodic signals:} The
paper~\cite{kim2021ergodic} is the first of the two papers on the subject of
stochastic filter stability (asymptotic forgetting of the initial condition).  A key 
contribution of the paper is the notion of conditional Poincar\'e inequality
(PI) which is shown to yield filter stability.   Using the dual 
methods, we are able to derive \textit{all} the prior results where
explicit convergence rates are obtained.  From
one of the 
anonymous reviews of the paper:
\begin{quote}
	``{\it The paper is a new take on the stability problem of the Wonham
		filter [...], I find this work highly original and definitely deserving publication. Even
		though the obtained stability conditions were essentially known before, I have a feeling
		that the new perspective on the problem will bear much more fruit in the nearest future.}''
\end{quote}

\paragraph{4. Filter stability of non-ergodic signals:} The
paper~\cite{kim2021detectable} is the second of the two papers on the subject of
stochastic filter stability.  The contribution
of this paper is to introduce the definition for stabilizability of
the BSDE~\eqref{eq:dual-bsde}, and establish that it is necessary
and sufficient for filter stability (for the case when $\bS$ is finite).  
This theory is entirely parallel to the Kalman-Bucy filter stability theory in
the linear Gaussian settings of the problem.

\section{Outline of this thesis}

The thesis contains nine chapters and two appendices.  Chapters 2, 3 and 6 largely contain the background
information and literature survey.  The remaining chapters 4, 5, 7, 8, and 9 contain
original results.  A short summary of each of the chapters is as
follows:

\begin{itemize}
	\item {\bf Chapter~\ref{ch:background}} introduces the mathematical problem of stochastic
	filtering and provides a summary of the prominent solution approaches to 
	derive the basic equations of stochastic filtering.
	
	\item {\bf Chapter~\ref{ch:duality-background}} is a review of basic duality theory in systems and control.  Two types of dualities are discussed: (1) duality
	between controllability and observability for linear systems, and
	(2) duality between linear Gaussian (Kalman) filter and linear
	quadratic optimal control.
	The chapter also includes a review of prior work on extending these to nonlinear deterministic and stochastic systems.
	
	\item {\bf Chapter~\ref{ch:observability}} presents our original work on extending the
	classical duality between controllability and observability to the
	stochastic filtering model.  Specifically, a BSDE model for the dual
	control system is introduced.  For this system, controllability
	and stabilizability are
	defined and shown to be dual to the stochastic observability and detectability, respectively.    
	
	\item {\bf Chapter~\ref{ch:duality-principle}} contains the main contribution of this thesis, namely, the dual optimal control problem. Its solution is described using two approaches: (1) via an application of the maximum principle; and (2) through a martingale characterization. Each of these approaches is shown to yield an explicit feedback form of the optimal control law. The feedback form is used to obtain a novel derivation of the equation of stochastic filtering.
	
	\item {\bf Chapter~\ref{ch:filter-stability-literature}} is a review of the filter stability results.
	The chapter begins with a discussion of the stability theory of the Kalman
	filter, drawing mainly on Ocone and Pardoux' classical paper on the
	subject.  The remainder of this chapter is devoted to a discussion
	of the main techniques and results for analysis of the nonlinear
	filter.  
	
	\item {\bf Chapter~\ref{ch:filter-stability}} contains the first set of results on filter stability
	analysis using the dual optimal control problem.  The definition of
	conditional Poincar\'e inequality (PI) is introduced and shown to be the
	simplest sufficient condition to obtain filter stability.
	Based on conditional PI, convergence rates are obtained for several examples. These are related to literature.
	
	\item {\bf Chapter~\ref{ch:filter-stability-2}} is also on the subject of filter stability but
	focussed on the finite state-space case.  For this case,
	stabilizability of the dual control system is shown to be necessary
	and sufficient to detect the correct ergodic class. 
	
	\item {\bf Chapter~\ref{ch:future}} contains a discussion of some open problems.   
	
	\item {\bf Appendix~\ref{apdx:bsde}} contains backgound results on existence uniqueness and optimal control theory for BSDEs.
	
	\item {\bf Appendix~\ref{apdx:min-energy}} provides a self-contained exposition of minimum energy dual optimal control formulation and its connection to nonlinear smoothing equation.  
\end{itemize}  

%

\newpage


\chapter{Nonlinear filtering}\label{ch:background}


In this chapter, we introduce the mathematical model for the nonlinear filtering problem in Section~\ref{sec:problem-formulation} and describe the main solution approaches in Section~\ref{sec:Girsanov}.


\subsubsection{Notation}

For a locally compact Polish space $S$, the following notation is adopted:
\begin{itemize}
	\item $\clB(S)$ is the Borel $\sigma$-algebra on $S$.
	\item $\clM(S)$ is the space of regular, bounded and finitely additive 
	signed measures (rba measures) on $\clB(S)$.
	\item $\clP(S)$ is the subset of $\clM(S)$ comprising of probability 
	measures.
	\item $C_b(S)$ is the space of continuous and bounded real-valued functions on $S$.
	\item For measure space $(S;\clB(S);\lambda)$, $L^2(\lambda)= L^2(S;\clB(S);\lambda)$ is the Hilbert space of real-valued functions on $S$ equipped with the inner product
	\[
	\langle f, g\rangle_{L^2(\lambda)} = \int_S f(x)g(x)\ud \lambda(x)
	\]
\end{itemize}

For functions $f:S\to \Re$ and $g:S\to \Re$, the notation $fg$ is used to denote element-wise 
product of $f$ and $g$, namely,
\[
(fg)(x) := f(x)g(x),\quad x\in S
\]
In particular, $f^2 = ff$. The constant function is denoted by $\ones$ ($\ones(x) = 1$ for all $x\in S$). 

For $\mu\in \clM(S)$ and $f\in C_b(S)$,
\[
\mu(f) := \int_S f(x) \ud \mu(x)
\]
and for $\mu,\nu\in \clM(S)$ such that $\mu$ is absolutely continuous with 
respect to $\nu$ (denoted $\mu\ll\nu$), the Radon-Nikodym (RN) derivative is 
denoted by $\dfrac{\ud \mu}{\ud \nu}$. 

\medskip

For a vector $a\in \Re^d$, $A = \dv(a)$ is a $d\times d$ matrix with $A(i,i) = a_i$ and $A(i,j) = 0$ for $i\neq j$. For a $d\times d$ matrix $A$, $a=\dv^\dagger(A)$ is $d$-dimentional vector with $a_i = A(i,i)$.

\section{Nonlinear filtering problem}\label{sec:problem-formulation}


Throughout the thesis, we consider continuous time processes on a finite time horizon $[0,T]$ with $T<\infty$.
Fix the probability space $(\Omega, \clF_T, \sP)$ along with the filtration $\{\clF_t:0\le t \le T\}$ with respect to which all the stochastic processes are adapted. Of special interest are a pair of continuous-time stochastic processes $(X,Z)$ defined as follows:
\begin{itemize}
	\item The \emph{state process} $X = \{X_t\in\bS:0\le t \le T\}$ is a Feller-Markov 
	process on the state-space $\bS$. Its initial measure (prior) is denoted by 
	$\mu \in \clP(\bS)$ and $X_0\sim \mu$. The infinitesimal generator of the 
	Markov process is denoted by $\clA$. In terms of $\clA$, the \emph{carr\'e 
		du champ} operator $\Gamma$ is defined as follows:
	\[
	(\Gamma f)(x) = (\clA f^2)(x) - 2f(x)(\clA f)(x),\quad x\in \bS
	\]
	for a suitable subset of test functions $f \in C_b(\bS)$.
	A sample path $t \mapsto X_t(\omega)$ is a $\bS$-valued c\'adl\'ag function (that is right continuous with left limits).  The space of such functions is denoted by $D\big([0,T];\bS\big)$. In particular, $X(\omega)\in D\big([0,T];\bS\big)$ for $\omega \in \Omega$.
	
	\item  The \emph{observation process} $Z = \{Z_t \in \Re^m:0\le t \le T\}$ satisfies the following stochastic differential equation (SDE):
	\begin{equation}\label{eq:obs-model}
		Z_t = \int_0^t h(X_s) \ud s + W_t,\quad t \ge 0
	\end{equation}
	where $h:\bS\to \Re^m$ is a continuous function and $W = \{W_t:0\le t \le T\}$ is an $m$-dimensional Brownian motion (B.M.). We say $W$ is $\sP$-B.M. 
	It is assumed that $W$ is independent of $X$.
	A sample path $t\mapsto Z_t(\omega)$ is a $\Re^m$-valued continuous function. The space is denoted by $C\big([0,T];\Re^m\big)$.
\end{itemize}
The above is referred to as the \emph{white noise observation model} of 
nonlinear filtering. In the remainder of this thesis, the model is denoted by 
$(\clA,h)$.  In the case where $\bS$ is not finite, additional assumptions are 
typically necessary to ensure that the model is well-posed.


The canonical filtration $\clF_t = \sigma\big(\{(X_s,W_s):0\le s \le t\}\big)$. The filtration generated by the observation is denoted by $\clZ :=\{\clZ_t:0\le t\le T\}$  where $\clZ_t = \sigma\big(\{Z_s:0\le s\le t\}\big)$. 
The \emph{filtering problem} is to compute the conditional expectation for a given function $f\in C_b(\bS)$:
\[
\pi_t(f) := \E\big(f(X_t)\mid \clZ_t\big),\quad 0\le t \le T
\]
The measure-valued process $\pi = \{\pi_t\in \clP(\bS):0\le t \le T\}$ is referred to as the nonlinear filter.

Clearly, $\clZ_T\subset \clF_T$. We denote the restriction of $\sP$ to $\clZ_T$ by $\sP|_{\clZ_T}$. It is obtained using the defining relation
\[
\sP|_{\clZ_T} (A) := \sP(A),\quad A\in \clZ_T
\]

In problems concerned with observability of the model $(\clA,h)$ or filter stability, there are reasons to consider more than one prior $\mu$. We reserve the notation $\mu$ to denote the true but possibly unknown prior and the notation $\nu$ to denote the prior that is used to compute the filter. If $\mu$ is exactly known then $\mu = \nu$. In all other cases, it is assumed that $\mu\ll\nu$.

To stress the dependence on the initial measure $\mu$, we use the superscript 
notation $\sP^\mu$ to denote the probability measure $\sP$ when $X_0\sim \mu$. 
The expectation operator is denoted by $\E^\mu(\cdot)$ and the nonlinear filter 
$\pi_t^\mu(f) = \E^\mu\big(f(X_t)\mid \clZ_t\big)$. On the common measurable 
space $(\Omega, \clF_T)$, $\sP^\nu$ is used to denote another probability measure 
such that the transition law of $(X,Z)$ are identical but $X_0\sim \nu$. The 
associated expectation operator is denoted by $\E^\nu(\cdot)$ and 
$\pi_t^\nu(f) = \E^\nu\big(f(X_t)\mid \clZ_t\big)$.
The precise definition of $\sP^\mu$ and $\sP^\nu$ appears 
in~\cite[Section 2.2]{clark1999relative} where the following relationship between the two is 
also established:

\medskip

\begin{lemma}[Lemma 2.1 in \cite{clark1999relative}] \label{lm:change-of-Pmu-Pnu}
	Suppose $\mu\ll \nu$. Then 
	\begin{itemize}
		\item $\sP^\mu\ll\sP^\nu$, and the change of measure is given by
		\begin{equation*}\label{eq:P-mu-P-nu}
			\frac{\ud \sP^\mu}{\ud \sP^\nu}(\omega) = \frac{\ud \mu}{\ud \nu}\big(X_0(\omega)\big)\quad \sP^\nu\text{-a.s.}
		\end{equation*}
		\item For all $t > 0$, $\pi_t^\mu \ll \pi_t^\nu$, $\sP^\mu|_{\clZ_t}$-almost surely.
	\end{itemize} 
	
\end{lemma}

%
%
%

\subsection{Guiding examples for the Markov processes}\label{ssec:guiding}

The most important examples are (1) the state space $\bS$ is finite, and (2) the state space $\bS$ is Euclidean. In the Euclidean case, the linear Gaussian Markov process is of historical interest.
In the following, we introduce notation and additional assumptions for these examples. An important objective is to describe the explicit form of the carr\'e du champ operator for these various examples.

\subsubsection{Finite state space} The state-space $\bS$ is finite, namely, $\bS= \{1,2,\ldots,d\}$. 
In this case, the space $C_b(\bS)$ and $\clM(\bS)$ are both isomorphic to $\Re^d$: a real-valued function $f$ (or a finite measure $\mu$) is identified with a vector in $\Re^d$, where the $i^{\text{th}}$ element of the vector represents $f(i)$ (or $\mu(i)$). In this manner, the observation function $h$ is also identified with a matrix $H\in\Re^{d\times m}$.

The generator $\clA$ of the Markov process is identified with a row-stochastic rate matrix $A\in\Re^{d\times d}$ (the non-diagonal elements of $A$ are non-negative and the row sum is zero). $A$ acts on
a function $f \in \Re^d$ through right-multiplication:
\[
\clA: f\mapsto A f
\]
Its adjoint, denoted $\clA^\dagger$, acts on measures $\clA^\dagger:\mu \mapsto A^\tp \mu$.
The carr\'e du champ operator $\Gamma:\Re^d\to \Re^d$ is as follows:
\begin{equation}\label{eq:Gamma-finite}
	(\Gamma f)(i) = \sum_{j \in \mathbb{S}} A(i,j) (f(i) - f(j))^2,\quad i\in\bS
\end{equation}
For notational ease, we define a matrix-valued function $Q:\bS \to \Re^{d\times d}$:
\begin{equation}\label{eq:Q-matrix}
	Q(i):= \sum_{j \in \mathbb{S}} A(i,j) (e_i - e_j)(e_i-e_j)^\tp, \quad i\in\bS
\end{equation}
where $\{e_1,e_2,\ldots,e_d\}$ is the standard basis in $\Re^d$. With this definition, $(\Gamma f)(x) = f^\tp Q(x) f$ for $x\in\bS$.


\begin{remark}
	The notation for the finite state space case is readily extended to the countable state space $\bS = \{1,2,\ldots\}$. In this case, $A = \{A(i,j): i,j\in \bS\}$ and the carr\'e du champ is also given by the equation~\eqref{eq:Gamma-finite}.
	Typically, additional assumptions are needed to ensure that the Markov process $X$ is well-defined over the time horizon $[0,T]$. The simplest condition is that $A$ has bounded rates, i.e., $\sup_{i\in \bS} \sum_{j\neq i} A(i,j) < \infty$. Additional conditions are noted as needed to obtain various results in the thesis. 
\end{remark}

\subsubsection{Euclidean state space} The state space $\mathbb{S}=\Re^d$. 
We restrict $\clM(\bS)$ to measures which are absolutely continuous with respect to the Lebesgue measure. With a slight abuse of notation, we use the same notation $\rho$ to denote the measure and its density, writing
\[
\rho(f) = \int_{\Re^d} f(x)\rho(x)\ud x
\]
The Markov process $X$ is an It\^o diffusion modeled using a stochastic differential equation (SDE):
\begin{equation*}\label{eq:dyn_sde}
	\ud X_t = a (X_t) \ud t + \sigma(X_t) \ud B_t, \quad X_0\sim \mu
\end{equation*}
where $\mu$ is now a probability density on $\Re^d$, $a\in C^1(\Re^d; \Re^d)$ and $\sigma\in C^2(\Re^d; \Re^{d\times p})$ satisfy appropriate technical conditions such that a strong solution exists for $[0,T]$, and $B=\{B_t:0\le t \le T\}$ is a standard B.M.~assumed to be independent of $X_0$ and $W$.
The observation function $h\in C^1(\Re^d;\Re^m)$.

The infinitesimal generator $\clA$ acts on $C^2(\Re^d;\Re)$ functions in its 
domain according to~\cite[Thm. 7.3.3]{oksendal2003stochastic}
\[
(\clA f)(x):= a^\tp(x) \nabla f(x) + \half  \tr\big(\sigma\sigma^\tp(x)(D^2f)(x)\big),\quad x\in\Re^d
\]
where $\nabla f$ is the gradient vector and $D^2 f$ is
the Hessian matrix. The adjoint operator acts on density $\rho \in C^2(\Re^d;\Re)$ according to
\[
({\cal A}^\dagger \rho)(x) = -\divg(a\rho)(x) + \half  \sum_{i,j=1}^d \frac{\partial^2}{\partial x_i \partial x_j}\big([\sigma\sigma^\tp]_{ij} \rho\big)(x),\quad x\in \Re^d
\]
where $\divg(\cdot)$ is the divergence operator.
For $f\in C^1(\Re^d;\Re)$, the carr\'e du champ operator is given by 
\begin{equation}\label{eq:Gamma-Euclidean}
	(\Gamma f) (x) = \big|\sigma^\tp(x) \nabla f(x) \big|^2,\quad x\in\Re^d
\end{equation}


\subsubsection{Linear-Gaussian model}
The linear-Gaussian model is a historically important example on the Euclidean state-space. 
In linear Gaussian settings, the functions $a(\cdot)$ and $h(\cdot)$ are 
linear, $\sigma$ is a constant matrix, and $\mu$ is Gaussian. Explicitly,
\begin{subequations}\label{eq:linear-Gaussian-model}
	\begin{align}
		\ud X_t &= A^\tp X_t \ud t + \sigma \ud B_t,\quad X_0\sim N(m_0,\Sigma_0) \label{eq:linear-Gaussian-model-a}\\
		\ud Z_t &= H^\tp X_t \ud t + \ud W_t \label{eq:linear-Gaussian-model-b}
	\end{align}
\end{subequations}
where $N(m_0,\Sigma_0)$ denotes the Gaussian density with mean $m_0\in \Re^d$ 
and variance $\Sigma_0 \succ 0$.
The model parameters $A\in\Re^{d\times d}$, $H\in\Re^{d\times m}$, and $\sigma\in\Re^{d\times p}$.
%
With a slight abuse of notation, we express a linear function as \[f(x) = f^\tp 
x,\quad x\in\Re^d\] 
where on the right-hand side $f\in \Re^d$. Then $\clA f$ is a linear function 
given by
\[
\big(\clA f\big)(x) = (Af)^\tp x,\quad x\in\Re^d
\]
and $\Gamma f$ is a constant function given by the following quadratic form:
\[
\big(\Gamma f\big)(x) = f^\tp \big(\sigma \sigma^\tp\big) f,\quad x\in\Re^d 
\]
In the remainder of this thesis, the model~\eqref{eq:linear-Gaussian-model} is referred to as the linear-Gaussian filtering problem. Its solution is given by the celebrated Kalman-Bucy filter.

\section{Equations of nonlinear filtering}\label{sec:Girsanov}

The solution to the nonlinear filtering problem is obtained by first deriving 
the equation for the nonlinear filter $\pi$.
There are two classical solution approaches to derive this equation: (1) Based on Girsanov change of measure; and (2) based on the innovation method. These are briefly reviewed in the following two sections.
The first section is based on~\cite[Chapter 5]{xiong2008introduction} 
and~\cite[Chapter 1]{van2006filtering}, and the second section follows~\cite[Chapter VI.8]{rogers2000diffusions}.

\subsection{Girsanov change of measure} \label{ssec:Zakai}

To motivate this approach, first consider the trivial case when $h = 0$. In this 
case, $Z = W$ is a Brownian motion. Therefore, $X$ and $\clZ$ are 
independent and
the conditional law is simply the marginal. In 
particular, 
for any given bounded functional $\phi:D\big([0,T];\bS\big)\times C\big([0,T];\Re^m\big) \to \Re$, 
define
\[
f_\phi(z) := \E\big(\phi(X,z)\big)
\]
Then the conditional expectation is obtained as 
\begin{equation}\label{eq:trivial_case}
	\E\big(\phi(X,Z)\mid \clZ_T\big) = f_\phi(Z)
\end{equation}
The idea is extended to the general case in the following steps:
\begin{itemize}
	\item Find a new measure $\tsP$ on $(\Omega,\clF_T)$ such that the 
	probability law for $X$ is unchanged but $Z$ is a $\tsP$-B.M.~that is 
	independent of $X$.
	\item Evaluate the conditional expectation with respect to $\tsP$ as 
	in~\eqref{eq:trivial_case}.
	\item Compute the conditional expectation with respect to the original 
	measure $\sP$ by using the change of measure (Bayes) formula for conditional 
	expectation. 
\end{itemize}

For the first step, the new measure $\tsP \ll \sP$ is obtained by setting
\[
\frac{\ud \tsP}{\ud \sP} = \exp\Big(-\int_0^T 
h^\tp(X_t) \ud W_t - \half \int_0^T |h(X_t)|^2\ud t\Big) 
\]
The expectation with respect to $\tsP$ is denoted by $\tE(\cdot)$. 
The following proposition is a consequence of Girsanv theorem~\cite[Theorem 5.22]{le2016brownian} and it is the key result in nonlinear filtering. The proof appears in Section~\ref{ssec:pf-girs}. 


\medskip

\begin{proposition}[Girsanov, Lemma 1.1.5 in~\cite{van2006filtering}] \label{prop:Girsanov}
	Assume the Novikov's condition: 
	\begin{equation}\label{eq:Novikov}
		\E\Big[\exp\Big(\half \int_0^T |h(X_t)|^2\ud t\Big)\Big] < \infty
	\end{equation}
	Then the following holds:
	\begin{enumerate}
		\item $Z$ is a $\tsP$-B.M.
		\item The probability law for $X$ is identical under $\sP$ and $\tsP$.
		\item $X$ and $Z$ are independent under $\tsP$.
		\item $\sP\ll\tsP$ with
		\[
		\frac{\ud \sP}{\ud \tsP} = \exp\Big(\int_0^T 
		h^\tp(X_t) \ud Z_t - \half \int_0^T |h(X_t)|^2\ud t\Big) =:D_T
		\]
	\end{enumerate}
\end{proposition}

\medskip

We define a process $D=\{D_t: 0\le t\le T\}$ as follows:
\begin{equation}\label{eq:D-t}
	D_t := 
	\exp\Big(\int_0^t 
	h^\tp(X_s) \ud Z_s - \half \int_0^t |h(X_s)|^2\ud s\Big),\quad 0\le t \le T
\end{equation}
A simple application of It\^o formula shows that
\[
\ud D_t = D_t h^\tp(X_t)\ud Z_t
\]
and therefore $D$ is a $\tsP$-martingale whereby $D_t = \tE(D_T\mid \clF_t)$ for $0\le t \le T$. The change of measure formula for conditional expectation is given in the following proposition whose proof appears in Section~\ref{ssec:pf-KS}.

\medskip

\begin{proposition}[Bayes formula, Theorem 3.22 in~\cite{xiong2008introduction}] \label{prop:Kallianpur-Striebel}
	For any $f\in C_b(\bS)$,
	\begin{equation}\label{eq:Bayes-formula}
		\E\big(f(X_t)|\clZ_t\big) =  
		\frac{\tE\big(D_tf(X_t)|\clZ_t\big)}{\tE\big(D_t|\clZ_t\big)}
	\end{equation}	
\end{proposition}

\medskip

In order to express the formula on the right-hand side of~\eqref{eq:Bayes-formula} succinctly, we define a 
measure-valued process $\sigma = \{\sigma_t\in\clM(\bS):0\le t \le T\}$ by
\[
\sigma_t(f) := \tE\big(D_tf(X_t)|\clZ_t\big),\quad 0\le t \le T, \; f\in 
C_b(\bS)
\]
The process $\sigma$ is referred to as the \emph{un-normalized filter}. Expressing~\eqref{eq:Bayes-formula} using this notation, we obtain the following result:

\medskip

\begin{corollary}[Kallianpur-Striebel formula, Theorem 5.3 
	in~\cite{xiong2008introduction}]
	For any $f\in C_b(\bS)$,
	\begin{equation}\label{eq:normalize-Zakai}
		\pi_t(f) = \frac{\sigma_t(f)}{\sigma_t(\ones)} 
	\end{equation}
\end{corollary}

\medskip
%

Since $Z$ is an independent B.M., the un-normalized filter is computed simply by applying~\eqref{eq:trivial_case}.The derivation for the same appears in Section~\ref{ssec:pf-Zakai}.
The derivation is novel and utilizes techniques that will be expanded upon and revisited latter chapters of this thesis.

\medskip

\begin{proposition}[Zakai equation, Theorem 5.5 
	in~\cite{xiong2008introduction}] \label{prop:Zakai} The un-normalized filter 
	for $f\in C_b(\bS)$ satisfies the following stochastic differential equation 
	(SDE): 
	\begin{equation}\label{eq:Zakai}
		\sigma_t(f) = \mu(f) + \int_0^t \sigma_s(hf)^\tp \ud Z_t + 
		\int_0^t \sigma_s(\clA f)\ud s
	\end{equation}
\end{proposition}

Once the equation for un-normalized filter is known, the nonlinear filter is obtained by using It\^o formula for the ratio $\frac{\sigma_t(f)}{\sigma_t(\ones)}$.

\medskip

\begin{proposition}[Kushner-Stratonovich equation, Theorem 5.7 
	in~\cite{xiong2008introduction}] 
	\label{prop:Kushner}
	The nonlinear filter satisfies the following SDE:
	\begin{equation} \label{eq:nonlinear-filter}
		\ud \pi_t(f) = \pi_t(\clA f) \ud t + 
		\big(\pi_t(hf)-\pi_t(h)\pi_t(f)\big)^\tp \big(\ud Z_t - \pi_t(h)\ud t\big)
	\end{equation}
	with $\pi_0(f) = \mu(f)$.
\end{proposition}

\medskip
\subsubsection{Solution operator of the Zakai equation}

The Zakai equation~\eqref{eq:Zakai} is a linear time-invariant stochastic partial 
differential equation (SPDE). Although the un-normalized filter is meaningful for $\mu\in\clP(\bS)$,~\eqref{eq:Zakai} is well defined for $\mu\in\clM(\bS)$. Upon such an extension,
the resulting linear solution operator $\Psi_t:\clM(\bS)\times \Omega\to\clM(\bS)$
defined by
\begin{equation}\label{eq:Zakai-soln-operator}
	\Psi_t (\mu) = \sigma_t,\quad \mu\in\clM(\bS)
\end{equation}

\medskip

\begin{example}\label{ex:finite-Zakai-operator}
	For finite state-space case, $\Psi_t$ is a $d \times d$ random matrix which is given by the solution of the following SDE:
	\[
	\ud \Psi_t = A^\tp \Psi_t \ud t + \sum_{j=1}^m\dv(H^j)\Psi_t\ud Z_t^j,\quad \Psi_0 = I
	\]
	where $H^j$ denotes the $j^{\text{th}}$ column of $H$, and $Z_t^j$ is the  $j^{\text{th}}$ element of $Z_t$. 
\end{example}

\medskip

\begin{remark}
	By suitably defining $\Omega$ as the space of sample paths of $(X,Z)$, Atar and Zeitouni~\cite[Eq.~13]{atar1997exponential} define a shift operator $\theta_\tau$ on $\Omega$. Using this shift operator, they define $\Psi_{t,\tau} := \Psi_{t-\tau}\circ \theta_\tau$ where $0\le \tau \le t$. Thus, $\Psi_{t,\tau}:\clM(\bS)\times \Omega \to \clM(\bS)$ and one obtains a semigroup like property
	\[
	\Psi_{t,0}\mu = \Psi_{t,\tau}\Psi_{\tau,0}\,\mu,\quad 0\le \tau \le t
	\]
	The early work on filter stability is based on analysis of the contraction properties of this map~\cite{atar1997exponential} (see Section~\ref{sec:Lyapunov-analysis}).
	
\end{remark}

\subsection{Innovation method}

The \emph{innovation process} $I:=\{I_t\in \Re^m:0\le t \le T\}$ is defined as follows:
\begin{equation}\label{eq:innovation-def}
	I_t = Z_t - \int_0^t \pi_s(h)\ud s,\quad 0\le t \le T
\end{equation}
The innovation process plays an important role in the theory of nonlinear 
filtering.
For example, it appears as the driving term in the Kushner-Stratonovich 
equation~\eqref{eq:nonlinear-filter}. The innovation is understood as the fresh
information brought by the observation process~\cite{kailath1970innovations}
because of the following result:

\medskip

\begin{proposition}[Theorem VI.8.4(i) in \cite{rogers2000diffusions}]
	The innovation process $I$ is a $\sP$-B.M.
\end{proposition}

\medskip

The filtration generated by $I$ is denoted by $\clI:=\{\clI_t:0\le t \le T\}$ where 
$\clI_t = \sigma\big(\{I_s:0\le s\le t\}\big)$. Because $I$ is a $\clZ$-adapted 
process, it is true that $\clI_t\subset \clZ_t$. It was a famous conjecture of 
Kailath~\cite[Remark IV.1]{kailath1970innovations} that
\[
\clZ_t = \clI_t\quad \forall \, 0\le t \le T
\]
up to $\sP$-null sets. For the white noise observation model considered in this thesis, the conjecture was proved by Allinger and Mitter~\cite{allinger1981new}.
Consequently, the martingale representation theorem~\cite[Theorem 
4.3.4]{oksendal2003stochastic} yields the following representation result. (For 
an alternative proof that does not requires the Kailath's conjecture, 
see~\cite[Chapter VI.8]{rogers2000diffusions}.)

\medskip

\begin{proposition}[Theorem VI.8.4(ii) in \cite{rogers2000diffusions}] \label{prop:innov-representation}
	Let $M = \{M_t\in \Re: 0\le t \le T\}$ be a $\sP$-martingale such that
	\[
	\E\Big(\int_0^T |M_t|^2\ud t\Big) < \infty
	\]
	Then there exists a $\clZ$-adapted process $\varphi = \{\varphi_t\in \Re^m: 0\le t \le T\}$ such that
	\[
	M_t = M_0 + \int_0^t \varphi_s^\tp \ud I_s,\quad \sP\text{-a.s.}\;\;0\le t\le T
	\]
\end{proposition}

\medskip

The representation result provides an alternative proof of the Proposition~\ref{prop:Kushner}. The proof appears in Section~\ref{ssec:pf-kushner}.


\section{Proofs of the statements}

\subsection{Proof of Proposition~\ref{prop:Girsanov}}\label{ssec:pf-girs}

Item 4 is direct from the fact that the change of measure is strictly positive. For the rest of claims, we follow proof of~\cite[Lemma 1.1.5]{van2006filtering}.

Consider non-negative bounded functions $f:D\big([0,T];\bS\big)\to \Re$ and $g: C\big([0,T];\Re^m\big)\to \Re$.
Denote $\sP|_{X}$ be the restriction of $\sP$ on the state process and $\lambda_w$ be the Wiener measure on $C\big([0,T];\Re^m\big)$. We want to show that $\ud \tsP(x,z) = \ud \sP|_X(x) \ud \lambda_w(z)$.
It suffices to show that
\[
\E\big(f(X)g(W)\big) = \int_{D([0,T];\bS)} \int_{C([0,T];\Re^m)}f(x) g(w)D_T^{-1}(x,z) \ud \lambda_w(z) \ud \sP|_{X}(x)
\]
where $w = \{w_t:0\le t \le T\}$ denotes the sample path of the measurement noise $W$ and $z:=\{z_t:0\le t \le T\}$ denotes the sample path of the observation. $w$ and $z$ are related by the model~\eqref{eq:obs-model}:
\[
w_t = z_t - \int_0^t h(x_s)\ud s
\]
Since $X$ and $W$ are independent under $\sP$, the left-hand side becomes
\[
\E\big(f(X)g(W)\big) = \int_{D([0,T];\bS)} f(x)\ud \sP|_X(x) \int_{C([0,T];\Re^m)} g(w)\ud \lambda_w(w) 
\]
Meanwhile the right-hand side is expressed by
\begin{align*}
	\int_{D([0,T];\bS)} \int_{C([0,T];\Re^m)}&f(x) g(w)D_T^{-1}(x,z) \ud \lambda_w(z) \ud \sP|_{X}(x)\\
	&=\int_{D([0,T];\bS)} f(x) \bigg[\int_{C([0,T];\Re^m)}g(w(x,z))D_T^{-1}(x,z) \ud \lambda_w(z)\bigg] \ud \sP|_{X}(x)
\end{align*}
Note that the change of measure $D_T^{-1}$ is the Dol\'eans exponential of $\int_0^T h(X_t)\ud W_t$. Therefore by the Girsanov theorem~\cite[Theorem 5.22]{le2016brownian},
\[
\int_{C([0,T];\Re^m)}g(w(x,z))D_T^{-1}(x,z) \ud \lambda_w(z) = \int_{C([0,T];\Re^m)}g(w) \ud \lambda_w(w) 
\]
Thus the proof is complete. \qed

%
%

\subsection{Proof of Proposition~\ref{prop:Kallianpur-Striebel}}\label{ssec:pf-KS}

The conditional expectation $\E\big(f(X_t)|\clZ_t\big)$ is defined by a 
$\clZ_t$-measurable random variable $S_T$ such that
\[
\E\big(\ones_A S_t\big) = \E\big(\ones_A f(X_t)\big),\quad \forall\,A\in\clZ_t
\]
Equivalently under $\tsP$, 
\[
\tE \big(D_t\ones_A S_t\big) = \tE \big(D_t \ones_Af(X_t)\big),\;\; 
\forall\, A \in \clZ_t\quad 
\Longleftrightarrow\ \quad \tE\big(D_tS_t|\clZ_t\big) = \tE\big(D_tf(X_t)|\clZ_t\big)
\]
Since $S_t$ is $\clZ_t$-measurable, $S_t$ is pulled out and we obtain
\begin{equation*}
	S_t = \frac{\tE\big(D_tf(X_t)|\clZ_T\big)}{\tE\big(D_t|\clZ_t\big)}
\end{equation*}
\qed

\subsection{Proof of Proposition~\ref{prop:Zakai}}\label{ssec:pf-Zakai}

In this proof, we will consider two function spaces:
\begin{itemize}
	\item $L^2_\clZ\big([0,T];\Re^m\big)$ is the Hilbert space of $\Re^m$-valued $\clZ$-adapted stochastic processes.
	\item $L^2_{\clZ_T}(\Omega;\Re)$ is the Hilbert space of $\clZ_T$-measurable random variables.
\end{itemize}
These function spaces are formally introduced in Section~\ref{sec:function-spaces}.

We use the following properties of conditional expectation:
\begin{itemize}
	\item The conditional mean $\tE\big(D_Tf(X_T)\mid \clZ_T\big)$ is the unique solution of the 
	optimization problem
	\begin{equation}\label{eq:optim-Zakai}
		\min\big\{\tE\big(|D_Tf(X_T)-S_T|^2\big): S_T\in L_{\clZ_T}^2(\Omega;\Re)\big\}
	\end{equation}
	\item Since $Z$ is a $\tsP$-B.M., It\^o representation theorem~\cite[Theorem 
	4.3.3]{oksendal2003stochastic} shows that for any $S_T\in L_{\clZ_T}^2(\Omega;\Re)$, there exists $\{U_t\in \Re^m:0\le t \le T\} \in L^2_{\clZ}\big([0,T];\Re^m\big)$ such that
	\begin{equation}\label{eq:estim-Zakai}
		S_T = \tE(S_T) +\int_0^T U_t^\tp \ud Z_t,\quad \tsP\text{-a.s.}
	\end{equation}
\end{itemize}

The strategy is to consider the estimator of the form~\eqref{eq:estim-Zakai} 
and obtain the solution by solving the optimization 
problem~\eqref{eq:optim-Zakai}.
For this purpose, let $\{N_t(g):0\le t\le T\}$ be the martingale associated with the 
infinitesimal generator $\clA$ defined by:
\begin{equation}\label{eq:martingale-generator}
	N_t(g) := g(X_t) - \int_0^t \clA g(X_s)\ud s
\end{equation}
Also consider a deterministic backward PDE:
\begin{equation}\label{eq:BKE-pf}
	-\frac{\partial y_t}{\partial t}(x) = (\clA y_t)(x),\quad y_T(x) = f(x)
\end{equation}
By applying It\^o rule on $D_t y_t(X_t)$,
%
\[
D_Tf(X_T) = y_0(X_0) + \int_0^T D_t(hy_t)(X_t)\ud Z_t + \int_0^T D_t\ud N_t(y_t)
\]
Consequently, $\tE(D_Tf(X_T)) = \tE(D_0y_0(X_0)) = \mu(y_0)$, and therefore set
\[
S_T = \mu(y_0)+\int_0^T U_t \ud Z_t
\]
Subtract $S_T$ on both sides, and take square and expectation to have
\[
\tE\big(|D_Tf(X_T) - S_T|^2\big) = \tE\Big(|y_0(X_0)-\mu(y_0)|^2 + \int_0^T |D_t(hy_t)(X_t) - U_t|^2 + D_t^2 (\Gamma y_t)(X_t) \ud t\Big)
\]
It is straightforward that the right-hand side is minimized at
\[
U_t = \tE\big(D_t(hy_t)(X_t) \mid \clZ_t\big) = \sigma_t(hy_t)\quad \forall\, t \in [0,T]
\]
It is hence concluded that
\[
\sigma_T(f) = \mu(y_0) + \int_0^T \sigma_t(hy_t) \ud Z_t
\]
Let $\Phi(T,s)$ be the transition operator of the system~\eqref{eq:BKE-pf} from time $T$ to $s$, then 
\[
\sigma_T(f) = \mu(\Phi(T,0)f) + \int_0^T \sigma_t(h\Phi(T,t)f) \ud Z_t
\]
By differentiating with respect to $T$, we obtain the Zakai equation
\[
\ud \sigma_t(f) = \sigma_t(\clA f) \ud t + \sigma_t(hf)\ud Z_t
\]
where we write $t$ instead of $T$.
%
\qed

%

\subsection{Proof of Proposition~\ref{prop:Kushner} using innovation method}\label{ssec:pf-kushner}

By the Proposition~\ref{prop:innov-representation}, there exists $\{\varphi_t: t \ge 0\}$ such that
\begin{equation}\label{eq:representation}
	\pi_t(f) - \int_0^t \pi_s(\clA_s f) \ud s = \int_0^t\varphi_s\ud I_s
\end{equation}
In order to compute $\varphi_t$, apply It\^{o} product rule on $f(X_t)Z_t$:
\begin{align*}
	\ud \big(f(X_t) Z_t\big) &= f(X_t)\ud Z_t + Z_t \ud \big(f(X_t)\big)\\
	&=f(X_t)\big(h(X_t)\ud t + \ud W_t\big) + Z_t\big(\ud N_t(f) + \clA f(X_t)\big)\ud t\\
	&=\big(hf(X_t) + Z_t\clA f(X_t)\big)\ud t + f(X_t)\ud W_t + Z_t \ud N_t(f)
\end{align*}
Take conditional expectation on both sides to show the following:
\[
\ud \big(\pi_t(f)Z_t\big) = \big(\pi_t(hf)+Z_t\pi_t(\clA f)\big)\ud t + (\text{martingale})
\]
Meanwhile, the It\^{o} product rule on $\pi_t(f) Z_t$ using~\eqref{eq:representation} yields
\begin{align*}
	\ud \big(\pi_t(f)Z_t\big) &= \pi_t(f)\ud Z_t + Z_t \ud \pi_t(f) + \varphi_t \ud t\\
	&= \pi_t(f)\big(\ud I_t + \pi_t(h)\ud t\big) + Z_t\big(\varphi_t \ud I_t + \pi_t(\clA f)\ud t\big)\\
	&=\big(\pi_t(f)\pi_t(h) + Z_t\pi_t(\clA f) + \varphi_t\big) \ud t + (\text{martingale})
\end{align*}
We subtract the two, and then we have a martingale with finite variation, which must be 0 almost surely. Thus we conclude
\[
\varphi_t = \pi_t(hf) - \pi_t(f)\pi_t(h)
\]
Putting this back to~\eqref{eq:representation} yields the filter~\eqref{eq:nonlinear-filter}.

\newpage


\chapter{Duality in control literature}\label{ch:duality-background}


There is a fundamental dual relationship between estimation and control. The dual relationship is expressed in two inter-related manners:
\begin{itemize}
	\item Duality between controllability and observability.
	\item Duality between optimal control and optimal filtering. This means expressing one type of problem as another type of problem. Of particular interest is to express a filtering problem as an optimal control problem.
\end{itemize}
Section~\ref{ssec:duality-LTI} is a survey of the first and Section~\ref{sec:duality-filtering-oc} of the second. Much of the survey is focused on the linear systems where duality is best understood.

Concerning the duality between controllability and observability, the
survey includes a discussion of nonlinear deterministic and stochastic
observability.  For the deterministic case, both the classical work of
Hermann and Krener~\cite{hermann1977nonlinear} and the output to state stability (OSS) definitions of Wang and Sontag~\cite{sontag1997output} are reviewed in Section~\ref{ssec:nonlinear-observability}. 
For the stochastic case, the original definition is due to van Handel~\cite{van2009observability} which is introduced briefly in Section~\ref{ssec:observability-hmm}.  In the latter chapters, several refinements of the basic definition are described to help relate it to our work. 

Concerning the dual optimal control formulation, the focus in Section~\ref{sec:duality-filtering-oc}
is on the linear Gaussian case.  For this case, the two types of dual
constructions, namely, the minimum variance and the minimum energy
optimal control problems, are described. 

The final Section~\ref{sec:historical-remarks} of this chapter includes a historical survey of the optimal control formulation of the nonlinear filtering and
smoothing problems.  Although the section contains a self-contained
summary of the main aspects, a more complete discussion appears in
Appendix~\ref{apdx:min-energy} which includes details for both log transformation and
Mitter-Newton duality.

\medskip

\section{Observability and controllability}

\subsection{Dual vector spaces}\label{ssec:dual-space}

In this section, we briefly review dual vector spaces.
The discussion closely follows~\cite[Chapter 5, 6]{luenberger1997optimization}.

\begin{definition}
	Let $\clX$ be a Banach space equipped with norm $\|\cdot\|_\clX$. The {\em dual space}, denoted by $\clX^\dagger$, is the space of bounded linear functionals on $\clX$.  For $x^\dagger\in\clX^\dagger$, the notation 
	\[
	\langle x,x^\dagger\rangle = x^\dagger(x)
	\]
	is used to denote the evaluation at $x\in\clX$. The bilinear map $\langle \cdot ,\cdot\rangle:\clX\times \clX^\dagger \to \Re$ is called the {\em duality pairing}.
	
\end{definition}

The dual space $\clX^\dagger$ is also a Banach space with norm~(see~\cite[Theorem 5.3.1]{luenberger1997optimization})
\[
\|x^\dagger\|_{\clX^\dagger} = \sup_{\|x\|_{\clX}\le 1} \big|\langle x,x^\dagger\rangle\big|
\]
The following subspace of the dual space is of particular interest:

\begin{definition}
	Let $S\subset\clX$. The \emph{annihilator} of $S$, denoted by $S^\bot$, is the subspace
	\[
	S^\bot = \{x^\dagger \in \clX^\dagger: \langle x,x^\dagger\rangle = 0 \text{ for all }x\in S \}
	\]
\end{definition}

\medskip

\begin{example}\label{ex:dual-function-measure}
	The dual of $C_b(\bS)$ is $\clM(\bS)$~\cite[Theorem IV.6.2]{dunford1958linear}. For $f\in C_b(\bS)$, the norm is
	\[
	\|f\|_\infty = \sup_{x\in\bS} |f(x)|
	\]
	The dual norm for $\mu\in\clM(\bS)$ is the \emph{total variation} norm:
	\[
	\|\mu\|_\tv = \sup\big\{\mu(f): f\in C_b(\bS),\; \|f\|_\infty \le 1\big\}
	\]	
\end{example}

\medskip

\begin{example}[Riesz representation theorem, Theorem 5.3.2 in~\cite{luenberger1997optimization}]\label{ex:dual-Hilbert}
	Let $\clX$ be a Hilbert space equipped with inner product $\langle \cdot,\cdot\rangle$. $\clX$ is self-dual in the sense that for any linear functional $x^\dagger \in\clX^\dagger$, there exists $y\in \clX$ such that
	\[
	x^\dagger (x) = \langle x, y\rangle_\clX,\quad \|x^\dagger \|_{\clX^\dagger} = \|y\|_{\clX}
	\]
	In this chapter, two important examples of Hilbert spaces are as follows:
	
	\begin{itemize}
		\item Euclidean space $\Re^d$ equipped with the inner-product 
		\[
		\langle \eta,\xi\rangle_{\Re^d} = \eta^\tp \xi,\quad \eta,\xi \in \Re^d
		\]
		\item The function space of square integrable $\Re^m$-valued signals
		\[
		\clU := L^2\big([0,T];\Re^m\big)
		\]
		equipped with the inner product
		\[
		\langle u,v\rangle_{\clU} = \int_0^T u_t^\tp v_t \ud t,\quad u,v\in \clU
		\]
		
	\end{itemize}
	
	Another important example of Hilbert space is the space of $\clZ$-adapted stochastic processes  $L^2_\clZ\big(\Omega\times[0,T];\Re^m\big)$. The notation for the same will be introduced in the next chapter.
\end{example}

%

\medskip

Consider another Banach space $\clY$ and let $\clL:\clX\to\clY$ be a bounded linear operator. The adjoint operator $\clL^\dagger:\clY^\dagger\to \clX^\dagger$ is obtained from the defining relation:
\[
\langle x, \clL^\dagger y^\dagger \rangle = \langle \clL x, y^\dagger\rangle
\]
A detailed explanation showing that $\clL^\dagger$ is well-defined, linear and bounded appears in~\cite[Chapter 6.5]{luenberger1997optimization}.

In finite-dimensional settings, it is an elementary fact that the range space of a matrix is orthogonal to the null space of its transpose~\cite{strang1993fundamental}. The following theorem provides a generalization of this fact for an operator $\clL$ and its adjoint $\clL^\dagger$. The proof appears in Section~\ref{ss:pf-thm31}.

\medskip

\begin{theorem}[Theorem 6.6.1 in~\cite{luenberger1997optimization}]\label{thm:dual-adjoint}
	Let $\clX$ and $\clY$ be Banach spaces and $\clL:\clX\to\clY$ is a bounded linear operator. Then
	\[
	\Rsp(\clL)^\bot = \Nsp(\clL^\dagger)
	\]
\end{theorem}

\medskip

In particular, $\Nsp(\clL^\dagger)$ is trivial if and only if $\Rsp(\clL)$ is dense in $\clY$. This property is important to discuss duality between the controllability and observability.

\subsection{Controllability and observability of linear systems}\label{ssec:duality-LTI}


In this section, we review the classical duality between controllability and observability by utilizing the tools from the previous section.

In the study of deterministic linear time-invariant (LTI) systems, the function spaces are the Hilbert spaces: $\clX = \clU = L^2\big([0,T];\Re^m\big)$ and $\clY = \Re^d$. There are two properies of interest:
\begin{enumerate}
	\item Controllability is a property of a linear operator $\clL: \clU \to \Re^d$.
	\item Observability is a property of its adjoint $\clL^\dagger: \Re^d\to \clU$. 
\end{enumerate}
In the following, we describe each of these properties. 

\subsubsection{Controllability of an LTI system}

For given matrices $A\in\Re^{d\times d}$ and
$H\in\Re^{d\times m}$ consider a linear state-input system:
\begin{equation}\label{eq:LTI-ctrl}
	-\frac{\ud y_t}{\ud t} = A y_t + H u_t ,\quad y_T = 0
\end{equation}
where $u = \{u_t\in\Re^m:0\le t \le T\}$ is referred to as the control input.

The basic control problem is to design a control input $u$ that steers the system from 
a given initial condition $y_0 = \eta$ to $y_T = 0$. Regarding the solution of 
this problem, the following definition naturally arises.

%

\medskip

\begin{definition}
	For the linear system~\eqref{eq:LTI-ctrl}, the \emph{controllable subspace} is defined by:
	\[
	\clC = \big\{ \eta \in \Re^d: \exists u \in \clU \text{ such that the solution to~\eqref{eq:LTI-ctrl} satisfies } y_0 = \eta \text{ and } y_T= 0\big\}
	\]
	The linear system~\eqref{eq:LTI-ctrl} is \emph{controllable} if $\clC = \Re^d$
\end{definition}

In words, the controllable subspace $\clC$ is the set of initial conditions that can be driven to 0. 
Therefore, the system~\eqref{eq:LTI-ctrl} is controllable if every initial condition can be driven to zero in a given finite time $T$. 

\subsubsection{Observability of an LTI system}

Consider the following linear state-output system:
\begin{subequations}\label{eq:LTI-obs}
	\begin{align}
		\frac{\ud x_t}{\ud t} &= A^\tp x_t,\quad x_0 = \xi \label{eq:LTI-obs-a}\\
		z_t &= H^\tp x_t \label{eq:LTI-obs-b}
	\end{align}
\end{subequations}
Over a fixed time interval $[0,T]$, the output is denoted by $z = \{z_t\in\Re^m:0\le t\le T\}$. 
Clearly, the output depends upon the initial condition $\xi$. This dependence is indicated by using the superscript, whereby we write $z = z^{\xi}$.

The basic problem is to determine the initial condition $\xi$ from the output 
$z^{\xi}$. Regarding the solution of this problem, the following definition 
naturally arises.

\medskip

\begin{definition}\label{def:LTI-obs-def}
	The linear system~\eqref{eq:LTI-obs} is \emph{observable} if:
	\[
	z^{\xi_1} = z^{\xi_2} \quad\Longrightarrow\quad \xi_1 = \xi_2,\quad \forall\,\xi_1,\xi_2\in\Re^d
	\]
\end{definition}


\subsubsection{Dual relationship}

For the state-output system~\eqref{eq:LTI-ctrl} , the solution map $u\mapsto y_0$ is used to define a linear operator $\clL:\clU \to \Re^d$ as follows:
\begin{equation*}\label{eq:LTI-adjoint}
	\clL u := y_0 = \int_0^T e^{At}Hu_t \ud t
\end{equation*}
Note the controllable subspace $\clC = \Rsp(\clL)$. Its adjoint is given by
\[
(\clL^\dagger \xi)(t) = H^\tp e^{A^\tp t}\xi,\quad 0\le t \le T
\]
and represents the solution map from initial condition $\xi\mapsto z^\xi$ for the state-output system~\eqref{eq:LTI-obs}.
%
%
%
The duality relationship is expressed as
\[
\langle \xi,\clL u\rangle_{\Re^d} = \langle \clL^\dagger \xi, u \rangle_\clU,\quad \forall\, \xi \in \Re^d,\; u\in \clU
\]
We say the state-input system~\eqref{eq:LTI-ctrl} is \emph{dual} to the 
state-output system~\eqref{eq:LTI-obs}.
The following proposition is a simple consequence of the 
Theorem~\ref{thm:dual-adjoint} (see also Fig.~\ref{fig:LTI-duality}).

\medskip

\begin{proposition}
	T.F.A.E.
	\begin{enumerate}
		\item The system~\eqref{eq:LTI-obs} is observable.
		\item $\Nsp(\clL^\dagger) = \{0\}$.
		\item $\Rsp(\clL) = \Re^d$.
		\item The system~\eqref{eq:LTI-ctrl} is controllable.
	\end{enumerate}
\end{proposition}

\begin{figure}
	\centering
	\includegraphics[width=0.6\linewidth]{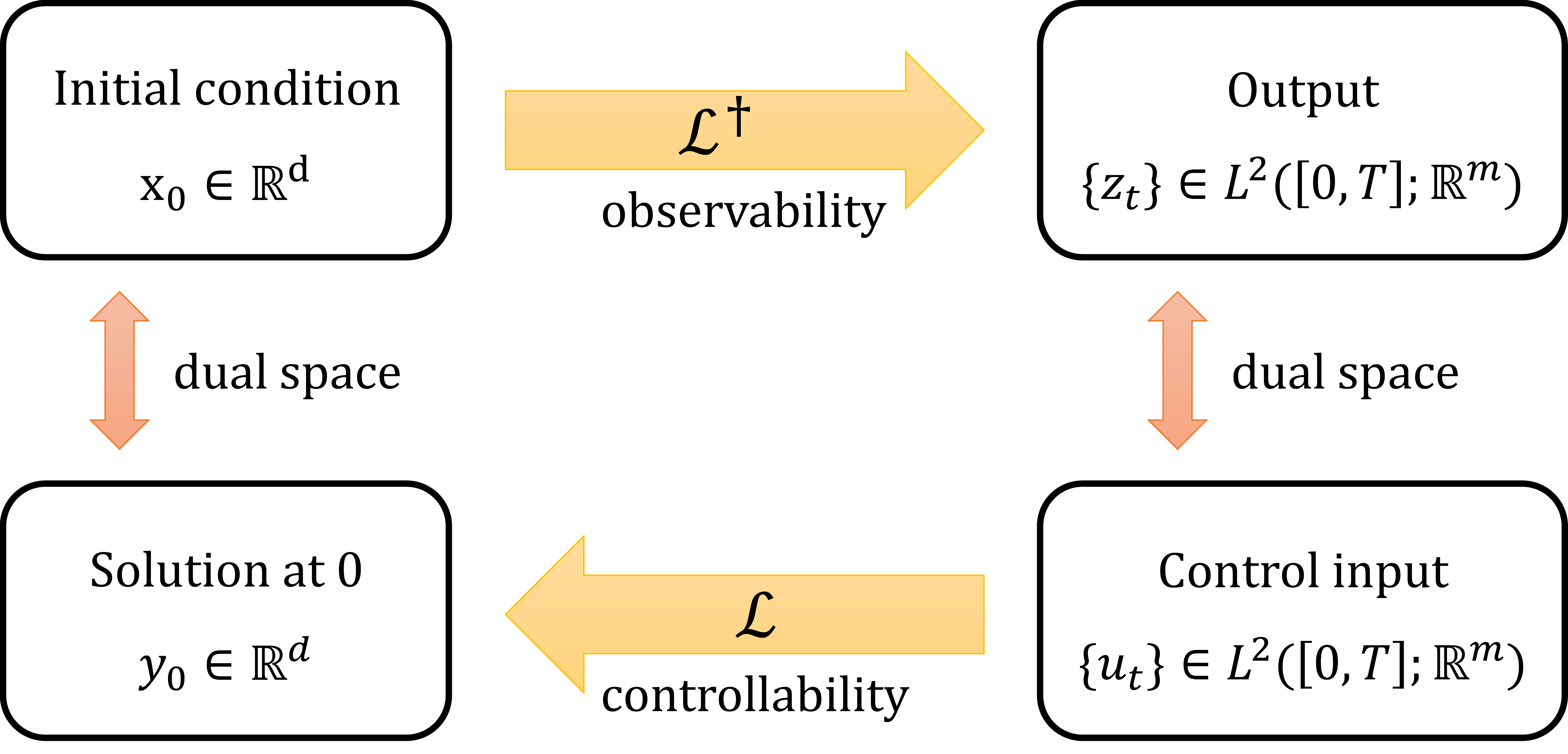}
	\caption{An illustration of the dual relationship between observability and 
		controllability}
	\label{fig:LTI-duality}
\end{figure}


\subsubsection{Controllable subspace and controllability gramian}


Using the Cayley-Hamilton theorem~\cite[Theorem 2.4.2]{horn_johnson_1985}, explicit formula for the controllable subspace $\Rsp(\clL)$ is obtained as 
\begin{equation}\label{eq:ctrl-subspace-LTI}
	\Rsp(\clL) = \sp\big\{H, A H , \ldots, A^{d-1} H \big\}
\end{equation}
The \emph{controllability gramian} is a $d\times d$ matrix defined as follows:
\[
\sW := \clL \clL^\dagger = \int_0^T e^{At}HH^\tp e^{A^\tp t}\ud t
\]
It is readily shown that the system is controllable if and only if $\Rsp(\sW)\succ 0$~\cite[Appendix C.3]{kailath2000linear}. The gramian is useful to obtain an explicit formula for the control that achieves the transfer $y_0=\eta \mapsto y_T=0$. This is described in the following proposition whose proof appears in Section~\ref{ss:pf-prop32}.

\medskip

\begin{proposition}\label{prop:LTI-gramian}
	Suppose $\eta \in \Rsp(\sW)$, so there exists $\xi \in \Re^d$ such that $\eta = \sW\xi$. Then the control
	\[
	u_t = H^\tp e^{A^\tp t} \xi,\quad 0\le t \le T
	\]
	transfers the system~\eqref{eq:LTI-ctrl} from $y_0 = \eta$ to $y_T = 0$.
	Suppose $v$ is another control input which also achieves the same transfer, then
	\[
	\int_0^T |v_t|^2 \ud t \ge \int_0^T 
	|u_t|^2 \ud t
	\]
\end{proposition}


\subsubsection{Stabilizability and detectability}

The stable subspace of the system~\eqref{eq:LTI-obs-a} is defined as follows:
\[
S_{s} := \big\{x_0\in \Re^d: |x_T|\to 0 \text{ as } T\to \infty\text{ where } 
x\text{ is the solution to~\eqref{eq:LTI-obs-a}}\big\}
\]
It is the span of the left generalized eigenvectors of $A$ whose eigenvalues have strictly negative real part. 
The unstable subspace of the system is its orthogonal complement $S_s^\bot$.
These definitions are useful for the study of asymptotic convergence as $T\to 
\infty$.
The stabilizability and detectability are defined as follows:

\medskip

\begin{definition}
	The linear system~\eqref{eq:LTI-ctrl} is \emph{stabilizable} if $S_s^\bot \subset \Rsp(\clL)$. 
\end{definition}

\medskip

\begin{definition}
	The linear system~\eqref{eq:LTI-obs} is \emph{detectable} if $\Nsp(\clL^\dagger) \subset S_s$. 
\end{definition}

\medskip

\begin{corollary}
	The system~\eqref{eq:LTI-ctrl} is stabilizable if and only if the 
	system~\eqref{eq:LTI-obs} is detectable.
\end{corollary}

\subsection{Observability for deterministic nonlinear 
	systems}\label{ssec:nonlinear-observability}

In continuous-time settings, a standard model of a state-output nonlinear system 
is the nonlinear ordinary differential equation (ODE):
\begin{subequations}\label{eq:NL-determin-obs}
	\begin{align}
		\frac{\ud x_t}{\ud t} &= a(x_t),\quad x_0 = \xi \label{eq:NL-determin-obs-a}\\
		z_t &= h(x_t) \label{eq:NL-determin-obs-b}
	\end{align}
\end{subequations}
The generator
\[
(\clA f)(x) = a^\tp(x) \nabla f(x),\quad x\in \Re^d
\]
Without loss of generality, it is assumed that $x= 0$ is an equilibrium point, 
i.e.,~$a(0) = 0$ and $h(0) = 0$.
Broadly speaking, there are two conceptual frameworks for defining observability and detectability for the model~\eqref{eq:NL-determin-obs}:
\begin{itemize}
	\item Local observability about $x = x_0\in\Re^d$. The original 
	paper is by Hermann and Krener~\cite{hermann1977nonlinear}.
	\item Output-to-state stability (OSS). The original paper is by 
	Sontag and Wang~\cite{sontag1996detectability,sontag1997output}
\end{itemize}
Both of these frameworks are based on duality and as such admit dual counterparts for controllability. 

\subsubsection{Local observability}

The following quote is from~\cite{hermann1977nonlinear} where basic definitions of controllability and observability are given:
\begin{quote}
	``{\it duality between ``controllability'' and ``observability'' [...] is, mathematically, just the duality between vector fields and differential forms}''
\end{quote}

\medskip

\begin{definition}[Definition 6.1.4 in~\cite{sontag2013mathematical}]
	The system~\eqref{eq:NL-determin-obs} is \emph{observable} if 
	\begin{equation}\label{eq:distinguishable}
		z^{\xi_1} = z^{\xi_2} \quad\Longrightarrow\quad \xi_1 = \xi_2,\quad \forall \,\xi_1,\xi_2\in \Re^d
	\end{equation}
\end{definition}

\medskip

While the nonlinear observability is a direct generalization of the 
Def.~\ref{def:LTI-obs-def}, it is untractable to verify for general class of nonlinear system~\eqref{eq:NL-determin-obs} (see~\cite[p.~733]{hermann1977nonlinear}). A local property of the system around a point $x_0\in\Re^d$ is as follows:

\medskip

\begin{definition}[Definition 6.4.1 in~\cite{sontag2013mathematical}]
	The system~\eqref{eq:NL-determin-obs} is \emph{locally observable} 
	at $x_0\in\Re^d$ if there exists a neighborhood $\clN \ni x_0$ such that~\eqref{eq:distinguishable} holds with $\xi_1 = x_0$ for all $\xi_2\in \clN$.
\end{definition}

\medskip

The local observability admits a rank condition test. 

\medskip

\begin{proposition}[Theorem 3.1 in~\cite{hermann1977nonlinear}]  Consider the following  subspace:
	\[
	\clO:= \sp\big\{\nabla (\clA^k h) 
	(x_0)\,:\,k\ge 0\big\}
	\]
	If the dimension of $\clO$ is $d$, then the system~\eqref{eq:NL-determin-obs} is locally observable at $x_0$.
\end{proposition}

\medskip

It is also shown that the condition above reduces 
to~\eqref{eq:ctrl-subspace-LTI} for the linear systems~\cite[Example 3.11]{hermann1977nonlinear}.

%

\subsubsection{Output-to-state stability}

The output-to-state stability (OSS) is dual to the input-to-state stability 
(ISS) concept which is central to the stability theory of nonlinear systems 
with input. In the original paper on the subject~\cite{sontag1997output}, Sontag 
and Wang write:
\begin{quote}
	''{\it Given the central role often played in control theory by the duality between input/state and state/output behavior, one may reasonably ask what concept obtains if outputs are used instead of inputs in the [input-to-state stability (ISS)] definition. This corresponds roughly to asking that ``no matter the initial state, if the observed outputs are small, then the state must be eventually small''. For linear systems, the notion that arises is that of \emph{detectability}. Thus, it would appear that this dual property, which we will call \emph{output-to-state stability} (OSS), is a natural candidate as a concept of nonlinear (zero-)detectability.}''
\end{quote}

\medskip

Before stating the OSS definition, we need to define some classes of functions. We 
denote the nonnegative real line by $\Re^+$. ${\cal K}$ is the family of 
monotonically increasing continuous functions $\alpha:\Re^+\to\Re^+$ with 
$\alpha(0) = 0$. ${\cal K}_\infty$ is a subset of ${\cal K}$ comprising of 
functions such that $\alpha(r)\to \infty$ as $r\to \infty$. The class ${\cal 
	KL}$ is a family of functions $\beta:\Re^+\times \Re^+ \to \Re^+$ such that 
$\beta(\cdot,t)\in{\cal K}$ for each $t\in \Re^+$ and $\beta(r,t) \to 0$ as 
$t\to \infty$ for each $r\in \Re^+$. 

\medskip

\begin{definition}[Definition 1 in~\cite{sontag1997output}]
	The system~\eqref{eq:NL-determin-obs} is output-to-state stable if $\exists \, \beta \in {\cal KL}$ and $\gamma \in {\cal K}$ such that
	\[
	|x_t| \le \max \Big\{\beta\big(|\xi|,t\big),\;\gamma\big(\sup_{s\le 
		t}|z_s|\big)\Big\},\quad t \ge 0
	\]
\end{definition}

\medskip

Since the first term decreases over time, eventually the second term dominates. 
According to Sontag~\cite{sontag1997output}, the OSS can be considered as a 
notion of detectability, in the sense that if the output is identically zero, then $x_t \to 0$.
Moreover, if the system is OSS then any initial state $\xi$ is asymptotically distinguishable and from the zero state~\cite[Section 5]{sontag1997output}.
For linear systems, this is equivalent to the detectability~\cite[Excercise 7.3.12]{sontag2013mathematical}.
Several variations of observability definition and their relationship are discussed in~\cite{hespanha2002nonlinear}.

The ISS and OSS definition enjoy a central place in nonlinear control theory in 
part because these are amenable to certain dissipative characterizations. 

\medskip

\begin{definition}
	A function $V:\Re^d \to \Re^+$ is a OSS-Lyapunov function if $\exists\,\underline{\alpha},\, \overline{\alpha}\in {\cal K}_\infty$ such that:
	\[
	\underline{\alpha}\big(|x|\big) \le V(x) \le \overline{\alpha}\big(|x|\big),\quad \forall\, x\in \Re^d
	\]
	and $\exists\,\alpha_1,\, \alpha_2\in {\cal K}_\infty$ such that
	\begin{equation}\label{eq:NL-Lyapunov}
		\clA V(x) \le -\alpha_1\big(|x|\big) + \alpha_2\big(|h(x)|\big),\quad \forall x \in \Re^d
	\end{equation}
	
\end{definition}

\medskip

The following proposition is from~\cite{sontag1997output}:

\begin{proposition}[Theorem 3 in~\cite{sontag1997output}]
	Consider the system~\eqref{eq:NL-determin-obs}. T.F.A.E.:
	\begin{enumerate}
		\item The system is OSS.
		\item The system admits an OSS-Lyapunov function.
	\end{enumerate}
\end{proposition}

\medskip

\begin{remark}
	Note that $\dfrac{\ud}{\ud t}V(x_t) = (\clA V)(x_t)$, and therefore the 
	integral form of~\eqref{eq:NL-Lyapunov} is as follows:
	\[
	V(x_t) \le V(x_0) + \int_0^t -\alpha_1\big(|x_s|\big) + \alpha_2\big(|h(x_s)|\big)\ud s,\quad \forall\,t\ge 0
	\]
	This is an example of a \emph{dissipation inequality} where the 
	OSS-Lyapunov function serves as a storage function~\cite[Chapter 
	6]{khalil2002nonlinear}.
\end{remark}



\subsection{Observability for hidden Markov model}\label{ssec:observability-hmm}

In deterministic settings, observability is defined by the property that every distinct initial condition 
produces distinct output. In~\cite{van2009observability}, the 
idea is extended to define the observability of an HMM with compact state-space. The 
following definitions are introduced in~\cite{van2009observability}.
Although we state these for the model $(\clA,h)$, the definitions are for a general class of HMMs.

\begin{definition}[Definition 2 in~\cite{van2009observability}]\label{def:observability}
	The model $(\clA,h)$ is \emph{observable} if 
	\[
	\quad \sP^\mu|_{\clZ_T} = 
	\sP^\nu|_{\clZ_T} \; \Longrightarrow\; \mu = \nu,\quad \forall\,\mu,\nu\in\clP(\bS)
	\]
\end{definition} 

The condition is used to define an equivalence relation in $\clP(\bS)$ as follows: 
\[
\mu \simeq \nu\quad \text{if} \quad \sP^\mu|_{\clZ_T} = 
\sP^\nu|_{\clZ_T}
\]
The equivalence relation is the counterpart of~\eqref{eq:distinguishable} for the deterministic definition of observability. 
Using this notation, the following definition naturally arises:

\medskip

\begin{definition}[Definition 3 
	in~\cite{van2009observability}]\label{def:un-observable-measures-observable-functions}
	The space of \emph{observable functions} 
	\[
	\clO = \{f\in C_b(\bS): \mu(f) = \nu(f) \; \forall\, \mu\simeq \nu\}
	\]
	The space of \emph{unobservable measures} 
	\[
	\clN = \{\alpha\mu - \alpha\nu \in \clM(\bS): \alpha \in \Re,\; \mu,\nu 
	\in \clP(\bS) \text{ such that } \mu\simeq\nu\}
	\]
\end{definition} 

\begin{remark}
	While the observability for deterministic nonlinear system considers a map from initial condition to output trajectory, the stochastic observability considers a map from the initial measure to the probability law of the output process.
\end{remark}

\begin{remark}
	By definition, the model $(\clA, h)$ is observable if and only if $\clN = \{0\}$. Next, $\clO^\bot = \clN$ and therefore the HMM is observable if and only if $\clO$ is dense in $C_b(\bS)$~\cite[p.~42]{van2009observability}. Note that constant functions are trivially in $\clO$ and therefore $\clO$ is non-trivial. 
\end{remark}



\subsection{Other contributions on stochastic observability in literature}\label{ssec:other-obs-notion}

In contrast to the fundamental definition (Def.~\ref{def:observability}) of observability, there are a large umber of ``functional'' definitions of observability that have been described in literature.
The functional definition is typically a sufficient condition on the model to obtain a desired conclusion for the estimation and/or control problem. Examples of such definition can be found in~\cite{liu2011stochastic} to investigate asymptotic properties of the minimum energy estimator,~\cite{kawamura2020nonlinear} to investigate model reduction,~\cite{mcdonald2018stability,mcdonald2019cdc} for filter stability, etc.
It must be said that even though it is a fundamental concept in linear systems theory, duality between controllability \& observability for general class of stochastic system has not been widely studied.

In the following, we briefly describe two works that have followed up on van Handel's definition. Both of these work are in discrete time settings.




An information theoretic notion of stochastic observability is presented in~\cite{liu2011stochastic}. 

\medskip

\begin{definition}[Definition 9 in~\cite{liu2011stochastic}]	
	A discrete-time dynamical system is LB-observable if for any measurable 
	$g:\bS\to \Re$ such that $g(X_0)$ is not deterministic, there exists some $N \in \mathbb{N}$ such the mutual information between $g(X_0)$ and output time sequence $\{Z_t: t=0,1,\ldots,N\}$ is strictly positive. 
\end{definition}

\medskip

It is also shown that if the system is LB-observable if and only if it is observable in the sense of Def.~\ref{def:observability}~\cite[Theorem 11]{liu2011thesis}.

\medskip

An extension of the observability definition appear in recent papers by 
McDonald and Y\"uksel~\cite{mcdonald2018stability, mcdonald2019cdc}. This approach considers an ability to reconstruct the prior in weak sense.

\medskip

\begin{definition}[Definition 3.1(ii) in~\cite{mcdonald2019cdc}]	
	A discrete-time partially observed Markov process is MY-observable if for every $f\in C_b(\bS)$ and $\epsilon >0$, there exists $N$ and a bounded function $g: \clY^N \to \Re$ such that 
	\[
	\Big\|f(\cdot) - \int g(z) \sP\big(\ud z\,|\,X_0 = \,\cdot\,\big)\Big\|_\infty < \epsilon
	\]
\end{definition}

\medskip

Note that $g$ is provides a inverse representation of $X_0$ from the output time sequence, and therefore this definition is somewhat dual to  Def.~\ref{def:observability}.
In~\cite{mcdonald2018stability}, the definition is used to investigate the finite memory property of the nonlinear filter.

\section{Duality between stochastic filtering and optimal control}\label{sec:duality-filtering-oc}

In stochastic filtering theory, duality commonly refers to the derivation and analysis of the optimal filter as a solution of an optimal control problem.
In classical linear-Gaussian settings, there are two types of optimal control constructions~\cite[Chapter 7.3]{bensoussan2018estimation}.
These constructions are referred to as minimum variance and minimum energy dualities. We illustrate the two constructions with a simple example before describing the linear Gaussian case later in this section.

\subsection{Simple example}

We begin with a simple example
(adapted from~\cite[Section 3.5]{kailath2000linear}) to illustrate the main ideas.
Consider a linear estimation problem defined by the model:
\begin{equation*}\label{eq:basic-problem}
	Z = H^\tp X + W
\end{equation*}
where $X\sim N(m,\Sigma)$, $W\sim N(0,Q)$ are independent Gaussian random variables of dimension $d$ and $p$, respectively.
The goal is to compute the conditional mean $\E(X\mid Z)$.

\subsubsection{Minimum variance construction}

Fix $f\in \Re^d$. The estimation objective is to compute $\E(f^\tp X\mid Z)$. 
Since all random variables are Gaussian, it suffices to consider an estimator 
$S$ of the form
\begin{equation}\label{eq:simple-estimator}
	S= b - u^\tp Z 
\end{equation}
where $b \in \Re$ and $u \in \Re^p$ are deterministic. The minimum variance optimization problem is~\cite[Corollary 1.10]{le2016brownian}
\begin{equation}\label{eq:basic-stochastic-problem}
	\min_{\substack{b\in\Re, u\in\Re^p}}\E\big(|f^\tp X - S|^2\big)
\end{equation}
With the estimator~\eqref{eq:simple-estimator}, the optimization objective 
becomes
\[
\E\big(|f^\tp X - S|^2\big) =  (f + H u)^\tp \Sigma (f + H u) + u^\tp Q u + 
\big((f + H u)^\tp m - b\big)^2
\]
Set $y = f+H u$ and then it follows that
$b = y^\tp m$ is the optimal choice, and the minimum error 
variance problem becomes a quadratic programming problem: 
\begin{align*}
	\min_{u\in\Re^p}\quad & y^\tp \Sigma y + u^\tp Q u\\
	\text{s.t. } \quad&y = f+H u
\end{align*}
%
Its solution is given by
\[
u = (H^\tp \Sigma H +Q)^{-1} H^\tp \Sigma f
\]
and the corresponding optimal estimator is
\[
S= f^\tp \big(m + \Sigma H(H^\tp \Sigma H +Q)^{-1}(Z-H^\tp m)\big) 
\]
Since $f$ is arbitrary, 
\begin{equation}\label{eq:solution-simple-example}
	\E(X\mid Z) = m + \Sigma H(H^\tp \Sigma H +Q)^{-1}(Z-H^\tp m)
\end{equation}

\subsubsection{Minimum energy / maximum likelihood construction}

While the previous problem considers the (minimum variance) property of the conditional expectation, the minimum energy problem begins with the Bayes' formula for conditional density:
\[
\rho_{X\mid Z}(x\mid z) = \frac{\rho_{X,Z}(x,z)}{\rho_Z(z)},\quad x\in\Re^d, z\in \Re^m
\]
where $\rho_{X,Z}$ denotes the joint probability density function, $\rho_Z$ is the marginal and $\rho_{X\mid Z}$ denotes the conditional density.
The objective is to compute the maximum-likelihood estimate of $X$.
Since the event $[X=x,Z=z]$ is the same as $[X=x,W=z-Hx]$, we have
\[
-2\log\big(\rho_{X\mid Z}(x\mid z)\big) = (x-m)^\tp \Sigma^{-1} (x-m) + 
(z-H^\tp x)^\tp 
Q^{-1} (z-H^\tp x) + c(z)
\]
where the constant $c(z)$ only depends on $z$. Therefore, the maximum 
likelihood problem is given by 
%
\begin{equation}\label{eq:basic-deterministic-problem}
	\min_{x\in \Re^d}\; (x-m)^\tp \Sigma^{-1} (x-m) + (z-H^\tp x)^\tp Q^{-1} 
	(z-H^\tp x)
\end{equation}
Its optimal solution is obtained as
\[
x = (\Sigma^{-1} + H Q^{-1} H^\tp )^{-1}(\Sigma^{-1} m + H Q^{-1} z)
\]
By an application of the matrix inversion lemma~\cite[Appdx. A.1]{kailath2000linear}, this formula is identical to~\eqref{eq:solution-simple-example} with $Z=z$. 

\medskip

	Kailath~\cite{kailath2000linear} refers to
	the~\eqref{eq:basic-stochastic-problem} 
	and~\eqref{eq:basic-deterministic-problem}
	as the stochastic problem and the deterministic problem, respectively.
	It is noted by Kailath that~\cite[p.~100]{kailath2000linear} that the minimum costs~\eqref{eq:basic-stochastic-problem} and~\eqref{eq:basic-deterministic-problem} are not directly related to each other even though they share the same solution. 

\subsection{Minimum variance duality for Kalman-Bucy filter}\label{ssec:Kalman-filter}

In the remainder of this section, we consider the linear the linear Gaussian filtering problem~\eqref{eq:linear-Gaussian-model} introduced in Chapter~\ref{ch:background}.
The goal is to compute
\[
\hat{X}_T := \E(X_T\mid \clZ_T)
\]

For the minimum variance duality, 
we follow the treatment in~\cite[Section 7.3.1]{bensoussan2018estimation} and~\cite[Chapter 7.6]{astrom1970}.
The formulation is a direct extension of the simple example. Again, 
fix $f\in \Re^d$ and consider a scalar random variable $f^\tp \hat{X}_T$.
Because all random variables are Gaussian, the conditional expectation $\hat{X}_T$ is also a Gaussian random variable~\cite[Lemma 6.12]{bain2009fundamentals}. Threrefore, it suffices to consider estimator $S_T$ of the form (cf.~\eqref{eq:simple-estimator}):
\begin{equation*}
	S_T := b - \int_0^T u_t^\tp \ud Z_t
\end{equation*}
where $b\in\Re$ and $u \in \clU = L^2\big([0,T];\Re^m\big)$ are both deterministic.
The minimum variance optimization problem is 
\begin{equation}\label{eq:LG-optimal-estimation}
	\min_{\substack{b\in\Re, u\in\clU}}\E\big(|f^\tp X_T - S_T|^2\big) 
\end{equation}
By introducing a suitable dual process, the problem is converted into a linear quadratic (LQ) optimal control problem.


\subsubsection{Minimum variance optimal control problem}

\begin{subequations}\label{eq:LG-optimal-control}
	\begin{align}
		\text{Minimize:}\qquad \bsJ_T(u) &= y_0^\tp \Sigma_0 y_0 + \int_0^T |u_t|^2 + y_t^\tp Q y_t \ud t \label{eq:LG-optimal-control-a}\\
		\text{Subject to:}\quad\;\; -\frac{\ud y_t}{\ud t} &= A y_t + H u_t,\quad y_T = f \label{eq:LG-optimal-control-b}
	\end{align}
\end{subequations}
where $Q = \sigma\sigma^\tp$. The relationship between the dual optimal control problem and the minimum variance problem~\eqref{eq:LG-optimal-estimation} is described in the following proposition whose proof appears in Section~\ref{ss:pf-prop35}.

\begin{proposition}[Duality principle, linear-Gaussian case] \label{prop:duality-KalmanBucy}
	For any admissible $u \in \clU$, consider an estimator
	\begin{equation}\label{eq:LG-estimator}
		S_T:= y_0^\tp m_0 - \int_0^T u_t^\tp \ud Z_t
	\end{equation}
	Then
	\begin{equation}\label{eq:LG-duality-principle}
		\bsJ_T(u) = \E\big(|f^\tp X_T - S_T|^2\big) 
	\end{equation}
\end{proposition}

The \emph{duality principle}~\eqref{eq:LG-duality-principle} transforms the optimal estimation problem~\eqref{eq:LG-optimal-estimation} into the optimal control objective $\bsJ_T(u)$. It is important to note that the constraint is the dual control system~\eqref{eq:LTI-ctrl}.

\subsubsection{Derivation of Kalman-Bucy filter}
The optimal solution to a linear-quadratic (LQ) problem is given in a linear feedback form~\cite[Theorem 3.1]{bensoussan2018estimation}: 
\[
u_t = -H^\tp \Sigma_ty_t
\]
where $\Sigma_t$ is the solution to the (forward-in-time) dynamic Riccati equation (DRE):
\begin{equation}\label{eq:Ricc-LG}
	\frac{\ud }{\ud t}\Sigma_t = A^\tp \Sigma_t + \Sigma_t A + Q - \Sigma_t H H^\tp \Sigma_t,\quad \Sigma_0\text{ given}
\end{equation}
Let $\Phi(T,t)$ be the transition matrix from time $T$ to $t$ of the closed loop system
\[
-\frac{\ud }{\ud t}\Phi(T,t) = (A-HH^\tp \Sigma_t )\Phi(T,t),\quad \Phi(T,T) = I
\]
Substituting the optimal control into~\eqref{eq:LG-estimator},
\begin{align*}
	f^\tp \hat{X}_T &= y_0^\tp m_0 + \int_0^T y_t^\tp \Sigma_t H\ud Z_t\\
	&= f^\tp\Phi^\tp(T,0) m_0 + \int_0^T f^\tp \Phi^\tp(T,t) \Sigma_t H \ud Z_t
\end{align*}
Since $f$ is arbitrary,
\[
\hat{X}_T = \Phi^\tp(T,0)m_0 + \int_0^T \Phi^\tp(T,t)\Sigma_t H \ud Z_t
\]
Because $T$ is arbitrary, we denote it as $t$:
\[
\hat{X}_t = \Phi^\tp(t,0)m_0 + \int_0^t \Phi^\tp(t,s)\Sigma_s H \ud Z_s
\]
Differentiating both sides with respect to $t$ yields the equation of the Kalman-Bucy filter:
\begin{align}
	\ud \hat{X}_t &= (A-HH^\tp\Sigma_t)^\tp\Big( \Phi^\tp(t,0)m_0 + \int_0^t 
	\Phi^\tp(t,s)\Sigma_s H \ud Z_s\Big)\ud t + \Phi^\tp(t,t)\Sigma_t H\ud 
	Z_t \nonumber\\
	&= A^\tp \hat{X}_t \ud t + \Sigma_t H\big(\ud Z_t - H^\tp \hat{X}_t\ud t\big) 
	\label{eq:KF-equation}
\end{align}

\subsection{Minimum energy duality for linear-Gaussian smoothing}\label{ssec:minimum-energy-LG}

This section follows the treatment in~\cite[Section 
7.3.2]{bensoussan2018estimation}.
As with the minimum variance duality, the minimum energy duality is also a direct extension of the calculation described for the simple example~\eqref{eq:basic-deterministic-problem}. The object of interest is 
\[
-2\log \rho_{X\mid Z}(x\mid z)
\]
where $x = \{x_t:0\le t \le T\}$ and $z = \{z_t:0\le t \le T\}$ are state and output trajectories, respectively.
In Section~\ref{ssec:derivation-min-energy}, it is explicitly evaluated based on similar calculations in literature. In carrying out the calculation, we use the model of Mortensen~\cite{mortensen1968} which is somewhat more general than the linear Gaussian model. The following dual optimal control problem is written for the linear Gaussian model.

\subsubsection{Minimum energy optimal control problem}
\begin{subequations}\label{eq:min-energy-optimal-control}
	\begin{align}
		\mathop{\text{Minimize:}}_{\substack{x_0\in\Re^d,u\in L^2([0,T];\Re^p)}}\quad \bsJ_T(u,x_0;\dot{z}) &= (x_0-m_0)^\tp \Sigma_0^{-1}(x_0-m_0) + \int_0^T |u_t|^2 + |\dot{z}_t-H^\tp x_t|^2 \ud t \label{eq:LG-min-energy-cost}\\
		\text{Subject to:}\qquad\qquad\qquad	\frac{\ud x_t}{\ud t} &= A^\tp x_t + \sigma u_t\label{eq:modified-linear-model}
	\end{align}
\end{subequations}

\begin{remark}
	Concerning the dual optimal control problem~\eqref{eq:min-energy-optimal-control}, Bensoussan writes in~\cite[p.~180]{bensoussan2018estimation}:
	\begin{quote}
		''{\it The notation is reminiscent of the probabilistic origin. The function $z$ is a given $L^2\big([0,T];\Re^m\big)$ function. It is reminiscent of the observation process, in fact rather the
			derivative of the observation process (which, as we know, does not exist). Similarly,
			$u$ is reminiscent of the noise that perturbs the system (again its derivative),
			and $x_0$ is the value of the initial condition, which we do not know. The cost
			functional~\eqref{eq:LG-min-energy-cost} contains weights related to the covariance matrices that were part of the initial probabilistic model.
		}''
	\end{quote}
\end{remark}

The optimal solution is given by a pair of forward and backward ODE given in the following proposition. The derivation appears in Section~\ref{ss:pf-prop36}.

\begin{proposition}\label{prop:min-energy-KalmanBucy}
	Consider the optimal control problem~\eqref{eq:min-energy-optimal-control}.
	Then for any choice of $u$ and $x_0$, 
	\[
	\bsJ_T(u,x_0;\dot{z}) \ge \int_0^T |\dot{z}_t-H^\tp \hat{x}_t|^2 \ud t
	\]
	where the process $\hat{x} = \{\hat{x}_t:0\le t\le T\}$ is the solution to:
	\begin{equation}\label{eq:LG-smoother-forward}
		\frac{\ud \hat{x}_t}{\ud t} = A^\tp \hat{x}_t + \Sigma_tH(\dot{z}_t - H^\tp \hat{x}_t),\quad \hat{x}_0 = m_0
	\end{equation}
	The equality holds with the optimal trajectory given by the backward equation:
	\begin{equation}\label{eq:LG-smoother-backward}
		\frac{\ud x_t}{\ud t} = A^\tp x_t + Q \Sigma_t^{-1}(x_t-\hat{x}_t),\quad x_T = \hat{x}_T
	\end{equation}
	
\end{proposition}

\begin{remark}
	One can note that the dynamics of $\hat{x}$ is similar to the Kalman-Bucy 
	filter where we formally write $\dot{z}_t\ud t = \ud z_t$.
	Since the optimal trajectory agrees with $\hat{x}_T$ at time $T$, the 
	forward equation~\eqref{eq:LG-smoother-forward}	leads to Kalman-Bucy filter.
	In fact, the optimal trajectory~\eqref{eq:LG-smoother-forward}--\eqref{eq:LG-smoother-backward} is identical to the forward-backward optimum smoother by Fraser and Potter~\cite[Eq.~(16)-(17)]{fraser1969optimum}. 
\end{remark}

\section{Historical remarks on duality for nonlinear filtering}\label{sec:historical-remarks}


For the problems of nonlinear filtering and smoothing, solution approaches in literature based on duality include the following:
\begin{itemize}
	\item Mortensen's maximum likelihood nonlinear filter~\cite{mortensen1968}.
	\item Minimum energy estimator (MEE) such as the full information estimator (FIE) and the moving horizon estimator (MHE)~\cite[Chapter 4]{rawlings2017model}.
	\item Fleming-Mitter duality, relating Zakai equation and Hamilton-Jacobi-Bellman (HJB) equation of an optimal control problem~\cite{fleming1982optimal}.
	\item Mitter-Newton's variational formulation of nonlinear estimation~\cite{mitter2003}.
\end{itemize}

A common theme connecting all of these prior works is that they are all variation/generalization of the minimum energy estimator~\eqref{eq:min-energy-optimal-control} for the linear Gaussian problem. While there are minor differences in specification of the optimal control objective, the constraint in all these cases is a modified copy of the signal model. Additional details on each of the four approaches appears in the following four subsections.

\subsection{Mortensen's maximum likelihood nonlinear filter}


In his pioneering paper, Mortensen~\cite{mortensen1968} considered the maximum likelihood smoothing problem for the following model:
\begin{align*}
	\ud X_t &= a(X_t)\ud t + \sigma\ud B_t,\quad X_0 \sim N(m_0,\Sigma)\\
	\ud Z_t &= h(X_t)\ud t + \ud W_t,\quad Z_0 = 0
\end{align*}
where $B$ and $W$ are mutually independent B.M. As for the linear Gaussian problem, the objective is to compute the maximum likelihood trajectory $x = \{x_t\in\Re^d:0\le t \le T\}$ that maximizes
\[
\rho_{X\mid Z}(x\mid z)
\]
given the output $z = \{z_t\in\Re^m: 0\le t \le T\}$. The calculation for the same appears in Section~\ref{ssec:derivation-min-energy} to obtain the following optimal control problem:

\subsubsection{Maximum likelihood estimation (MLE) problem}
\begin{subequations}\label{eq:MLE-problem}
	\begin{align}
		\mathop{\text{Minimize:}}_{x_0\in\Re^d, u\in L^2([0,T];\Re^p)}\bsJ_T(u,x_0;\dot{z}) &= (x_0-m_0)^\tp \Sigma_0^{-1}(x_0-m_0) + \int_0^T |u_t|^2 + |\dot{z}_t-h(x_t)|^2 \ud t \label{eq:MLE-problem-a}\\
		\text{Subject to:}\qquad\qquad \;\;	\frac{\ud x_t}{\ud t} &= a(x_t) + \sigma u_t \label{eq:MLE-problem-b}
	\end{align}
\end{subequations}
In Mortensen's paper, an algorithm to solve the MLE problem~\eqref{eq:MLE-problem} is proposed based on an application of the maximum principle.
As in the linear Gaussian case, the algorithm requires a forward and backward recursion to obtain the maximum likelihood trajectory. 

Since Mortensen's early work, related optimization-type problem formulation and forward-backward solution approach have appeared for a plethora of filtering and smoothing problems. In different communities, these are referred by different names, e.g., maximum likelihood estimation (MLE), maximum a posteriori (MAP) estimation and minimum energy estimation (MEE).

\subsection{Minimum energy estimator (MEE)}

Given the enormous success of model predictive control (MPC), related algorithms have been developed to solve the state estimation problems~\cite[Chapter 4]{rawlings2017model}.
In continuous-time setting, the optimal control problem is precisely the problem~\eqref{eq:MLE-problem}. In the MPC community, it is referred to as the minimum energy estimation problem. Broadly, there are two classes of MEE algorithms: 
\begin{itemize}
	\item Full information estimator (FIE) where the entire history of observation is used.
	\item Moving horizon estimator (MHE) where only a most recent fixed window of observation is used.
\end{itemize}
We refer the reader to Section 4.7 of~\cite{rawlings2017model} where a discussion on history of these approaches is provided. In this section, the authors note that dual constructions are useful for stability analysis. The authors describe certain results, e.g.~\cite[Theorem 4.10]{rawlings2017model},
originally reported in~\cite{hu2015optimization}, based on certain i-IOSS (incremental input/output to state stability) properties of the model.


\subsection{Fleming-Mitter-Newton duality}\label{ssec:MN-duality}

For the non-Gaussian problem, one of the criticisms of the MLE and MEE is that these do \emph{not} provide the conditional expectation (as the filter does).
For the white noise observation model, the first hint that the filtering equations are also related to an optimal control problem appears in the 1982 paper of Fleming and Mitter~\cite{fleming1982optimal}. In this paper, it is shown that the Zakai equation can be transformed into the Hamilton-Jacobi-Bellman (HJB) equation of an optimal control problem.
The particular transformation is an example of the log transformation whereby the negative log of the posterior density is the value function for a certain optimal control problem.
The interpretation of the optimal control problem itself appears in the 2003 paper of Mitter and Newton~\cite{mitter2003}. 
In this paper, the authors consider a control-modified version of the Markov process $X$ denoted by $\tilde{X} :=
\{\tilde{X}_t:0\le t\le T\}$. 
The control problem is to pick (1) the initial distribution
$\pi_0$ and (2) the state transition, such that
the distribution of $\tilde{X}$ equals the conditional distribution.

The optimization problem is formulated on the space of probability laws. 
Let $\sP_X$ denote the law for $X$, $\sQ$ denote the law for
$\tilde{X}$, and $\sP_{X\mid z}$ denote the law for $X$ given an observation path
$z=\{z_t:0\le t\le T\}$.  Assuming
$\sQ\ll\sP_X$, the objective function is the relative entropy between $\sQ $ and $\sP_{X\mid z}$:
\begin{equation*}
	\min_{\sQ} \quad \E_{\sQ}\Big(\log \frac{\ud \sQ}{\ud \sP_X}\Big) - \E_{\sQ}\Big(\log\frac{\ud \sP_{X\mid z}}{\ud \sP_X}\Big).
\end{equation*}

\medskip

For the Euclidean state-space, this procedure yields the following
stochastic optimal control problem (see Appendix~\ref{sec:control-problem}):
\begin{subequations}\label{eq:opt-cont-sde-hjb-intro}
	\begin{align}
		\mathop{\text{Min }}_{\pi_0, \; U}: \quad \sJ(\pi_0,U\,;z) 
		& = \E\Big(\log \frac{\ud \pi_0}{\ud \nu_0}(\tilde{X}_0) - z_T h(\tilde{X}_T) + \int_0^T \ell(\tilde{X}_t,U_t\,;z_t)\ud t\Big)\\
		\text{Subj.} : \;\;\quad\qquad \ud \tilde{X}_t &= a(\tilde{X}_t)\ud t +
		\sigma(\tilde{X}_t)(U_t\ud t +
		\ud \tilde{B}_t), \quad
		\tilde{X}_0 \sim \pi_0
	\end{align}
\end{subequations}
where $
l(x,u\,;z_t) := \half |u|^2 + \half h^2(x) + z_t(\clA^u h)(x)$ where
$\clA^u$ is the generator of the controlled Markov process $\tilde{X}$.
It is shown in Appendix~\ref{apss:linear-Gaussian} that the problem reduces to the minimum energy duality for linear-Gaussian case.

The solution of the optimal control problem~\eqref{eq:opt-cont-sde-hjb-intro} is given in the following proposition which reveals the connection to the log transformation. 

\begin{proposition}
	Consider the optimal control problem~\eqref{eq:opt-cont-sde-hjb-intro}.
	For this problem, the HJB equation for the value function $V$ is
	as follows:
	\begin{align*}
		-\frac{\partial V_t}{\partial t}(x) &= \big(A(V_t+z_th)\big)(x) + \half h^2(x) -\half|\sigma^\tp\nabla (V_t+z_th)(x)|^2\\
		V_T(x) &= - z_Th(x),\quad x\in \Re^d
	\end{align*}	
	The optimal control is of the state feedback form given by $U_t = -\sigma^\tp \nabla(V_t + z_th)(\tilde{X}_t)$.  
\end{proposition}

\medskip

Expressing $V_t(x) = -\log \big(\eta_t(x)e^{z_th(x)}\big)$ it is readily verified
that $\{\eta_t:0\leq t\leq T\}$ solves the backward Zakai equation.
The result also coincide with the result of Bene\v{s}~\cite{benevs1983relation} who considered the adjoint equation of pathwise Zakai PDE (see Remark~\ref{rm:backward-Zakai}).

\medskip

A tutorial style review of the log transformation, its link to the Zakai equation, specifically its path-wise robust representation, formulations of the optimal control problem and its link to the smoothing problem appears in Appendix~\ref{apdx:min-energy}. In addition to the Euclidean case, explicit formulae are also described for the finite state-space and linear Gaussian problem. The latter is used to recover the Mortensen's MLE problem~\eqref{eq:min-energy-optimal-control}.

\medskip

A recent focus on utilizing the optimal control formulation has been to develop numerical techniques, e.g., particle filters, to empirically approximate the conditional distribution;
cf.,~\cite{reich2019data,ruiz_kappen2017, kappen2016adaptive,chetrite2015,sutter2016variational,pathiraja2020mckean}.  

\subsection{Generalization of the minimum variance duality}

In spite of decades of work in this area, there is no satisfactory counterparts of the minimum variance (Kalman-Bucy) duality to nonlinear stochastic systems.
Two notable contributions on this line are:~\cite{simon1970duality} where duality in mathematical programming is used to provide rigorous explanation of the Kalman's duality and where certain extension to linear estimation problem with singular measurement noise is described; and~\cite{goodwin2005} where the Lagrangian dual of an estimation problem for truncated measurement noise process is considered.

In~\cite{todorov2008general}, Todorov writes:
\begin{quote}
	``{\it Kalman's duality has been known for half a century
		and has attracted a lot of attention. If a straightforward
		generalization to non-LQG settings was possible it would
		have been discovered long ago. Indeed we will now show
		that Kalman's duality, although mathematically sound, is an
		artifact of the LQG setting and needs to be revised before
		generalizations become possible.}''
\end{quote}
It is noted by Todorov that: (1) the dual relationship between the DRE of the
LQ optimal control and the covariance update equation of the Kalman
filter is {\em not} consistent with the interpretation of the negative
log-posterior as a value function; and (2) some of the linear
algebraic operations, e.g., the use of matrix transpose to define the
dual system, are not applicable to nonlinear
systems~\cite{todorov2008general}.

%

\section{Proofs of the statements}

\subsection{Proof of Theorem~\ref{thm:dual-adjoint}}\label{ss:pf-thm31}

Let $y^\dagger\in \Nsp(\clL^\dagger)$, then for any $x\in \clX$,
\[
\langle \clL x, y^\dagger\rangle = \langle x, \clL^\dagger y^\dagger\rangle = 0
\]
and therefore $y^\dagger \in \Rsp(\clL)^\bot$. Therefore $\Nsp(\clL^\dagger)\subset \Rsp(\clL)^\bot$.

For the other direction, if $y^\dagger \in \Rsp(\clL)^bot$ then  
$\langle \clL x,y^\dagger \rangle = 0$, and therefore
\[
\langle x, \clL^\dagger y^\dagger \rangle = 0
\]
Since this is true for all $x\in \clX$, it follows $\clL^\dagger y^\dagger = 0$, so $y\dagger \in \Nsp(\clL^\dagger)$. This implies $\Rsp(\clL)^\bot \subset \Nsp(\clL^\dagger)$.
\qed

%

\subsection{Proof of Proposition~\ref{prop:LTI-gramian}}\label{ss:pf-prop32}

The first property is because $\sW= \clL \clL^\dagger$, and therefore
\[
f = \sW \eta = \clL(\clL^\dagger \eta) = \clL u
\]
If another $v$ satisfies $\clL v = f$, then $\clL (u-v) = 0$, and therefore
\[
0 = \langle \clL (u-v), \eta \rangle = \langle u-v, \clL^\dagger \eta\rangle = \langle u-v,u\rangle
\]
The second claim follows because
\[
\|v\|^2 = \|u\|^2 + \|u-v\|^2 \ge \|u\|^2
\]
\qed

\subsection{Proof of Proposition~\ref{prop:duality-KalmanBucy}}\label{ss:pf-prop35}

Applying It\^o product formula on $y_t^\tp X_t$ yields:
\begin{align*}
	\ud (y_t^\tp X_t) &= -(y_t^\tp A^\tp + u_t^\tp H^\tp )X_t \ud t + y_t^\tp(A^\tp X_t \ud t +  \sigma\ud B_t)\\
	&=-u_t^\tp \ud Z_t + u_t^\tp \ud W_t + y_t^\tp  \sigma\ud B_t
\end{align*}
Integrating both sides from $0$ to $T$, 
\[
f^\tp X_T - \Big(\underbrace{y_0^\tp m_0 - \int_0^T u_t^\tp \ud Z_t}_{S_T}\Big) = \big(y_0^\tp X_0 - y_0^\tp m_0\big) + \int_0^T u_t^\tp \ud W_t + y_t^\tp \sigma\ud B_t
\]
Each of the terms on the right-hand side are mutually independent and have zero mean. Therefore, upon squaring and taking expectation,
\[
\E\big(|f^\tp X_T - S_T|^2\big) = y_0^\tp \Sigma_0 y_0 + \int_0^T |u_t|^2 + y_t^\tp Q y_t \ud t
\]
Therefore, the mean-squared error of the estimator~\eqref{eq:LG-estimator} becomes linear-quadratic optimal control objective on the dual system~\eqref{eq:LG-optimal-control}.
\qed

\subsection{Derivation of the minimum energy cost functional}\label{ssec:derivation-min-energy}

We consider the following nonlinear model considered by Mortensen~\cite{mortensen1968}:
\begin{align*}
	\ud X_t &= a(X_t)\ud t + \sigma \ud B_t,\quad X_0\sim N(m_0,\Sigma_0)\\
	\ud Z_t &= h(X_t)\ud t + \ud W_t
\end{align*}
In the linear Gaussian special case, $a(x) = A^\tp x$ and $h(x) = H^\tp x$. 
Consider a time discretization $0 = t_0 < t_1 < \ldots < t_N < t_{N+1} = T$ and denote $\Delta t_i := t_{i+1}-t_i$, $\Delta B_i := B_{t_{i+1}}-B_{t_i}$ and $\Delta W_i := W_{t_{i+1}}-W_{t_i}$. For a sample path $x = \{x_{t_i}\in\Re^d:i=0,1,\ldots,N+1\}$, observe that
\[
[X_{t_i}=x_{t_i}, X_{t_{i+1}} = x_{t_{i+1}}] = [X_{t_i}=x_{t_i}, \sigma \Delta B_i = x_{t_{i+1}} - x_{t_i} - a(x_{t_i})\Delta t_i] 
\]
The process $x$ is parameterized by~\eqref{eq:MLE-problem-b} using control input such that
\[
\sigma u_{t_i}\Delta t_i =  x_{t_{i+1}} - x_{t_i} - a(x_{t_i})\Delta t_i
\] 
Therefore we have
\[
[X_{t_i}=x_{t_i}, X_{t_{i+1}} = x_{t_{i+1}}] =  [X_{t_i}=x_{t_i}, \Delta B_i = u_{t_i}\Delta t_i] 
\]
Similarly, the event $[Z=z]$ is decomposed by 
\[
[Z_{t_i}=z_{t_i}, Z_{t_{i+1}} = z_{t_{i+1}}] = [Z_{t_i}=z_{t_i}, \Delta W_i = z_{t_{i+1}} - z_{t_i} - h(x_{t_i})\Delta t_i] 
\]
Since $X_0$, $W$ and $B$ are mutually independent, one obtains
\[
\rho_{X,Z}(x,z) = \rho_{X_0}(x_0)\prod_{i=1}^N \rho_{\Delta B_i}(u_{t_i}\Delta t_i)\prod_{i=1}^N \rho_{\Delta W_i}((\dot{z}_{t_i}-h(x_{t_i}))\Delta t_i) 
\]
where $\dot{z}_{t_i}\Delta t_i = z_{t_{i+1}}-z_{t_i}$. Take log to convert the product into the sum:
\begin{align*}
	\log \rho_{X,Z}(x,z) &= \log\big(\rho_{X_0}(x_0)\big) + \sum_{i=1}^N \log\big(\rho_{\Delta B_i}(u_{t_i}\Delta t_i)\big) + \log\big( \rho_{\Delta W_i}((\dot{z}_{t_i}-h(x_{t_i}))\Delta t_i) \big)\\
	&=-\half\Big[(x_0-m_0)^\tp \Sigma_0(x_0-m_0) + \sum_{i=1}^N |u_{t_i}|^2 \Delta t_i + |\dot{z}_{t_i}-h(x_{t_i})|^2\Delta t_i + o(\Delta t_i) + \text{(const.)}\Big]
\end{align*}
Letting $\Delta t_i\to 0$, the sum converges to the integral and therefore we obtain the cost functional~\eqref{eq:MLE-problem-a}.

\subsection{Proof of Proposition~\ref{prop:min-energy-KalmanBucy}}\label{ss:pf-prop36}

Let us parameterize the control by
\[
u_t = \sigma^\tp P_t(x_t-\hat{x}_t) + \tilde{u}_t
\]
where a symmetric matrix $P_t \in \Re^{d\times d}$ and $\hat{x}_t \in \Re^d$ are to be chosen. By using integration by parts formula, one obtains
\begin{align*}
	\bsJ_T(u,x_0;\dot{z}) &= (x_0-m_0)^\tp \Sigma_0^{-1}(x_0-m_0) + \int_0^T |\tilde{u}_t|^2 + |\dot{z}_t-H^\tp \hat{x}_t|^2 \ud t\\
	&\quad +\int_0^T (x_t-\hat{x}_t)^\tp \Big(-\frac{\ud}{\ud t}P_t - P_tA^\tp - AP_t - P_tQP_t + HH^\tp\Big)(x_t-\hat{x}_t) \ud t\\
	&\quad +2\int_0^T (x_t-\hat{x}_t)^\tp \Big(P_t\big(\frac{\ud \hat{x}_t}{\ud t} -A^\tp \hat{x}_t\big)-H(\dot{z}_t-H^\tp\hat{x}_t)\Big)\ud t\\
	&\quad + (x_T-\hat{x}_T)P_T(x_T-\hat{x}_T) - (x_0-\hat{x}_0)P_0(x_0-\hat{x}_0)
\end{align*}
Hence we set
\begin{align*}
	\frac{\ud}{\ud t}P_t &= -P_tA^\tp -AP_t -P_tQP_t + HH^\tp,\quad P_0 = \Sigma^{-1}\\
	\frac{\ud \hat{x}_t}{\ud t} &= A^\tp \hat{x}_t + P_t^{-1}H(\dot{z}_t - H^\tp \hat{x}_t),\quad \hat{x}_0 = m_0
\end{align*}
Note that the dynamics of $P_t$ is indeed the dynamics of inverse of $\Sigma_t$ defined by~\eqref{eq:Ricc-LG}. 
Under these choices of $P_t$ and $\hat{x}_t$, the cost functional becomes
\[
\bsJ_T(u,x_0;\dot{z}) = (x_T-\hat{x}_T)^\tp \Sigma_T^{-1}(x_T-\hat{x}_T) + \int_0^T |\tilde{u}_t|^2 + |\dot{z}_t-H^\tp \hat{x}_t|^2 \ud t
\]
Hence the claim follows by choosing $x_T = \hat{x}_T$ and $\tilde{u}_t = 0$ for all $t$.
\qed

\newpage


\chapter{Duality for stochastic observability}\label{ch:observability}


In this chapter, the first original contribution of this thesis is presented, namely, 
the dual control system for the model $(\clA,h)$. 
The dual control system is a linear \emph{backward stochastic differential equations} (BSDE). In the linear-Gaussian setting of the model, the BSDE reduces to the backward ODE~\eqref{eq:LTI-ctrl}.

The solution operator of the dual control system is used to define a
linear operator whose range space is the controllable subspace.  The
system is controllable if the range space is dense in $C_b(\bS)$.  The
controllability of the dual system is shown to be equivalent to
stochastic observability of the HMM: The controllable subspace is the
space of observable functions described in van Handel's work.  
Several properties of the
controllable subspace are noted along with its explicit
characterization in the finite state-space case.  A formula for the
controllability gramian is also described.   
The upshot of our work is that we can establish parallels between
linear and nonlinear models (see Table~\ref{tb:comparison}).

The outline of the remainder of this chapter is as follows: In Section~\ref{sec:observability-Zakai-relation},
stochastic observability is related to the Zakai equation.  The dual
control system is described in Section~\ref{sec:dual-control-system} together with the definition of the controllability and related concepts.  The explicit formulae for the
finite state space case appear in Section~\ref{sec:finite-observability}.

\section{Function spaces induced by $Z$ and $I$}\label{sec:function-spaces}

%

It is noted that $Z$ is a $\tsP$-B.M.~and $I$ is a $\sP$-B.M.~on a common measurable space $(\Omega,\clF_T)$. For a $\clZ_T$-measurable random variables, the following definition of Hilbert space is standard (see e.g.~\cite[Chapter 5.1.1]{le2016brownian})
\[
L^2_{\clZ_T}(\Omega;\Re^m) := L^2(\Omega;\clZ_T;\ud \tsP) =  \Big\{F:\Omega \to \Re^m: F\text{ is } \clZ_T\text{-measurable},\, \tE\big( |F|^2\big) < \infty\Big\}
\]
For a $\clZ$-adapted stochastic processes, the Hilbert space is
\begin{align*}
	L^2_{\clZ}\big(\Omega\times[0,T];\Re^m\big):=& L^2\big(\Omega\times[0,T];\clZ\otimes \clB([0,T]);\ud \tsP\ud t\big)\\
	=& \Big\{U:\Omega\times [0,T] \to \Re^m: U\text{ is }\clZ\text{-adapted},\,
	\tE\Big(\int_0^T |U_t|^2\ud t\Big) < \infty\Big\}
\end{align*}
where $\clB([0,T])$ is the Borel sigma algebra on $[0,T]$, $\clZ\otimes \clB([0,T])$ is the product sigma algebra and $\ud \tsP\ud t$ denotes the product measure on it.
The inner product for these spaces are
\[
\langle F,G\rangle_{L^2_{\clZ_T}} = \tE\big(F^\tp G\big),\quad \langle U,V\rangle_{L^2_\clZ} = \tE\Big(\int_0^T U_t^\tp V_t \ud t\Big)
\]
Suppose the state space admits a reference measure $\lambda$ on $(\bS,\clB(\bS))$. In this case, the space of random function
\[
L^2_{\clZ_T}\big(\Omega;L^2(\lambda)\big) = \Big\{F:\Omega \to L^2(\lambda):  F\text{ is }\clZ_T\text{-measurable},\, \tE\big( \|F\|_{L^2(\lambda)}^2\big) < \infty\Big\}
\]
is also a Hilbert space. For stochastic processes,
\[
L^2_{\clZ}\big(\Omega\times[0,T];L^2(\lambda)\big):=
\Big\{U:\Omega\times [0,T] \to L^2(\lambda):  U\text{ is }\clZ\text{-adapted},\,
\tE\Big(\int_0^T \|U_t\|_{L^2(\lambda)}^2\ud t\Big) < \infty\Big\}
\]
The inner product for these spaces are
\[
\langle F,G\rangle_{L^2_{\clZ_T}} = \tE\big(\langle F, G\rangle_{L^2(\lambda)}\big),\quad \langle U,V\rangle_{L^2_\clZ} = \tE\Big(\int_0^T\langle U_t, V_t\rangle_{L^2(\lambda)} \ud t\Big)
\]
The above Hilbert spaces suffice for finite or Euclidean case (where $\lambda$ is the Lebesgue measure). In general setting, the function space is $C_b(\bS)$ equipped with $\|\cdot\|_\infty$ norm. In these settings, we consider the Banach space:
\begin{align*}
	L^2_{\clZ_T}(\Omega;C_b(\bS)) &= \big\{F:\Omega\to C_b(\bS): F\text{ is }\clZ_T\text{-measurable},\; \tE\big(\|F\|_\infty^2\big) < \infty\big\}\\
	L^2_{\clZ}(\Omega\times[0,T];C_b(\bS)) &= \Big\{Y:\Omega\times [0,T]\to C_b(\bS)\;:\; Y\text{ is }\clZ\text{-adapted},\; \tE\Big(\int_0^T\|Y_t\|_\infty^2 \ud t\Big) < \infty\Big\}
\end{align*}

\medskip

Similar definitions are also obtained for innovation process $I$. For example,
\[
L^2_{\clI_T}(\Omega;\Re^m) := L^2(\Omega;\clZ_T;\ud \sP)
\]
\[
L^2_{\clI}\big(\Omega\times[0,T];\Re^m\big):= L^2\big(\Omega\times[0,T];\clZ\otimes \clB([0,T]);\ud \sP\ud t\big)
\]
Note the different choice of probability measure---$\tsP$ for $\clZ$ and $\sP$ for $\clI$. 

\section{Stochastic observability and its relationship to Zakai equation}\label{sec:observability-Zakai-relation}

We begin by recalling van Handel's definition for stochastic observability (Def.~\ref{def:observability}). An HMM is observable if
\[
\sP^\mu|_{\clZ_T} = \sP^\nu|_{\clZ_T} \quad \Longrightarrow \quad \mu = \nu
\]
In words, an HMM is observable if the map from prior to the probability measure on $(\Omega,\clZ_T)$ is injective. Note however that this map is not linear. Also, the domain and co-domain of the map are the spaces of probability measures which are not vector spaces.

For the white noise observation model, a quantitative analysis is possible based on the Kullback–Leibler (KL) divergence as described in the following proposition. The calculation for the same appears in Section~\ref{ssec:pf-clark-result}.

\begin{proposition}[Theorem 3.1 in \cite{clark1999relative}] \label{prop:clark-result}
	Consider the nonlinear model $(\clA,h)$. Then 
	\[
	\kl\big(\sP^\mu|_{\clZ_T} \mid \sP^\nu|_{\clZ_T}\big) = \half\E^\mu\Big(\int_0^T |\pi_t^\mu(h)-\pi_t^\nu(h)|^2\ud t\Big)
	\]
\end{proposition}

\medskip

Based on the Proposition~\ref{prop:clark-result}, the proof of the following theorem appears in Section~\ref{pf-obs-Zakai}.

\medskip

\begin{theorem}\label{thm:observability-Zakai-relation}
	T.F.A.E.:
	\begin{enumerate}
		\item The model $(\clA, h)$ is observable.
		\item For $\mu,\nu\in\clP(\bS)$,
		\[
		\pi_t^\mu(h) = \pi_t^\nu(h),\quad t\text{-a.e.},\;\sP^\mu|_{\clZ_T}\text{-a.s.} \quad \Longrightarrow \quad \mu = \nu
		\]
		\item For $\mu,\nu\in\clP(\bS)$,
		\[
		\sigma_t^\mu(h) = \sigma_t^\nu(h),\quad t\text{-a.e.},\; \sP^\mu|_{\clZ_T}\text{-a.s.} \quad \Longrightarrow \quad \mu = \nu
		\]
	\end{enumerate}
\end{theorem}

\medskip

A utility of Theorem~\ref{thm:observability-Zakai-relation} is that the un-normalized filter is the solution to the Zakai equation which is linear.
A linear operator $\clL^\dagger: \clM(\bS)\to L^2_\clZ\big(\Omega\times[0,T];\Re^m\big)\times \Re$ is defined as follows:
\begin{equation}\label{eq:observability-operator}
	\clL^\dagger \mu = \big(\{\sigma_t^\mu(h):0\le t\le T\}, \mu(\ones)\big)
\end{equation}
The notation is suggestive: In this chapter, we will define a linear operator $\clL$ such that the operator defined by~\eqref{eq:observability-operator} is its adjoint.

\medskip

\begin{corollary}\label{cor:linear-operator-observability}
	The nonlinear model $(\clA,h)$ is observable if and only if
	\[
	\Nsp(\clL^\dagger) = \{0\}
	\]
\end{corollary}

\medskip

\begin{remark}
	$\tilde{\mu}\in\Nsp(\clL^\dagger)$ has an interpretation of being the space of unobservable measures. Suppose $\Nsp(\clL^\dagger)$ is non-trivial. Then for $\mu\in\clP(\bS)$, choose $\epsilon \neq 0$ such that $\nu = \mu+\epsilon \tilde{\mu} \in \clP(\bS)$. Then owing to the linearity of~\eqref{eq:Zakai},
	\[
	\sigma_t^\mu(h) = \sigma_t^\nu(h)
	\] 
	From Theorem~\ref{thm:observability-Zakai-relation}, then $\sP^\mu|_{\clZ_T} = \sP^\nu|_{\clZ_T}$. Van Handel refers to $\tilde{\mu}$ as unobservable measure (see Def.~\ref{def:un-observable-measures-observable-functions}).
\end{remark}

\begin{remark}[Dual of the Zakai equation]  \label{rm:backward-Zakai}
	Because the Zakai equation is linear, its adjoint has previously been considered in literature.
	There are two types of equivalent constructions:
	\begin{enumerate}
		\item The most direct route is through a pathwise representation of the un-normalized filter obtained by using the log transformation $\sigma_t(x) = e^{\mu_t(x)+h(x)Z_t}$. As shown in Appendix~\ref{ssec:path-wise-Zakai}, $\{\mu_t: 0\le t\le T\}$ satisfies a deterministic linear PDE whose adjoint appears in~\cite[Eq.~4.17-4.18]{benevs1983relation}.
		\item The other type of adjoint is the \emph{backward Zakai equation}
		\begin{equation}\label{eq:backward-Zakai-eqn}
			-\ud \eta_t(x) = \big(\clA \eta\big)(x) \ud t + \big(h(x) \eta_t(x)\big) \cdot \overleftarrow{\ud Z_t},\quad \eta_T(x) = f(x),\quad x\in \bS
		\end{equation}
		where $\overleftarrow{\ud Z_t}$ denotes the backward It\^o integral, that is, the right-endpoints are chosen in the partial sum approximation of the stochastic integral (see~\cite[Remark 3.3]{pardoux1981non}).
		The forward and backward Zakai equation were first obtained by Pardoux~\cite{pardoux1979backward}. The two equation together yields the solution of the smoothing problem~\cite[Theorem 3.8]{pardoux1981non}.
	\end{enumerate}
	
	The two types of construction are equivalent because using the log transformation the backward Zakai equation is transformed to the pathwise adjoint.  These calculations are described in Appendix~\ref{ssec:path-wise-Zakai}.
	The backward and forward Zakai equation are adjoint because of the following:
	

	
	\medskip
	
	\begin{proposition}[Theorem 4.7.5 in~\cite{bensoussan1992stochastic}] \label{prop:dual-Zakai} Consider the forward Zakai equation~\eqref{eq:Zakai} and backward Zakai equation~\eqref{eq:backward-Zakai-eqn}. Then
		\[
		\sigma_T(f) = \mu(\eta_0)
		\]
	\end{proposition}
	
	\medskip
	
	In~\cite[Section 6.5]{xiong2008introduction}, Prop.~\ref{prop:dual-Zakai} is used to prove the uniqueness of the solution to the Zakai equation.
	
	Despite of the utility of the backward Zakai equation, it is distinct from the controllability--observability duality for linear systems theory in the following aspects:
	\begin{itemize}
		\item Equation~\eqref{eq:backward-Zakai-eqn} does not have a control input term.
		\item $\eta$ is not adapted to the forward-in-time filtration. In particular, $\eta_0$ is a $\clZ_T$-measurable random variable.
	\end{itemize}
	The dual control system described in the following section is original and distinct from these prior adjoint formulation.
	
\end{remark}

\section{Dual control system for an HMM}\label{sec:dual-control-system}

The objective is to define a linear operator $\clL$ whose adjoint is $\clL^\dagger$. Because of duality pairing between $C_b(\bS)$ and $\clM(\bS)$, the operator is defined for the function spaces as follows (see Figure~\ref{fig:nonlinear-duality}):
\[
\clL:L_\clZ^2\big(\Omega\times [0,T];\Re^m\big)\times \Re \to C_b(\bS)
\]


The main result (Theorem~\ref{thm:observability-definition} below) is to show that the operator $\clL$ is defined by the solution operator of the linear backward stochastic differential equation (BSDE):
%
\begin{subequations}\label{eq:dual-bsde}
	\begin{align}
		-\ud Y_t(x) &= \big(\clA Y_t(x) + h^\tp(x)(U_t+V_t(x))\big)\ud t - V_t^\tp (x) \ud Z_t\label{eq:dual-bsde-a}\\
		Y_T(x) &= c,\quad\forall\,x\in\bS\label{eq:dual-bsde-b}
	\end{align}
\end{subequations}
where $U \in L_\clZ^2(\Omega\times [0,T];\Re^m)$ is referred to as the control input and $c\in \Re$ is a deterministic constant.
The solution of the BSDE $(Y,V):=\{(Y_t,V_t) \in C_b(\bS)\times C_b(\bS)^m\,:\, 0\le t \le T\} \in L^2_{\clZ}\big(\Omega\times[0,T];C_b(\bS)\times C_b(\bS)^m\big)$ is (forward) adapted to the filtration $\clZ$.
The BSDE is the nonlinear counterpart of the backward ODE~\eqref{eq:LTI-ctrl} in the LTI setting.

Additional details on existence uniqueness and regularity theory for BSDEs appears in the Appendix~\ref{apdx:bsde} (see also~\cite{el1997backward,pardoux2014stochastic}).
Throughout the thesis, we assume that the solution of BSPDE $(Y,V)$ is uniquely determined in $L^2_{\clZ}\big(\Omega\times[0,T];C_b(\bS)\times C_b(\bS)^m\big)$ for each given $Y_T\in L^2_{\clZ_T}(\Omega;C_b(\bS))$ and $U\in\clU$.
For finite state space, it is proved in the seminal paper~\cite{pardoux1990adapted}. For the Euclidean case, the existence and uniqueness results were first obtained in~\cite{ma1999linear}.

\begin{figure}
	\centering
	\includegraphics[width=0.68\linewidth]{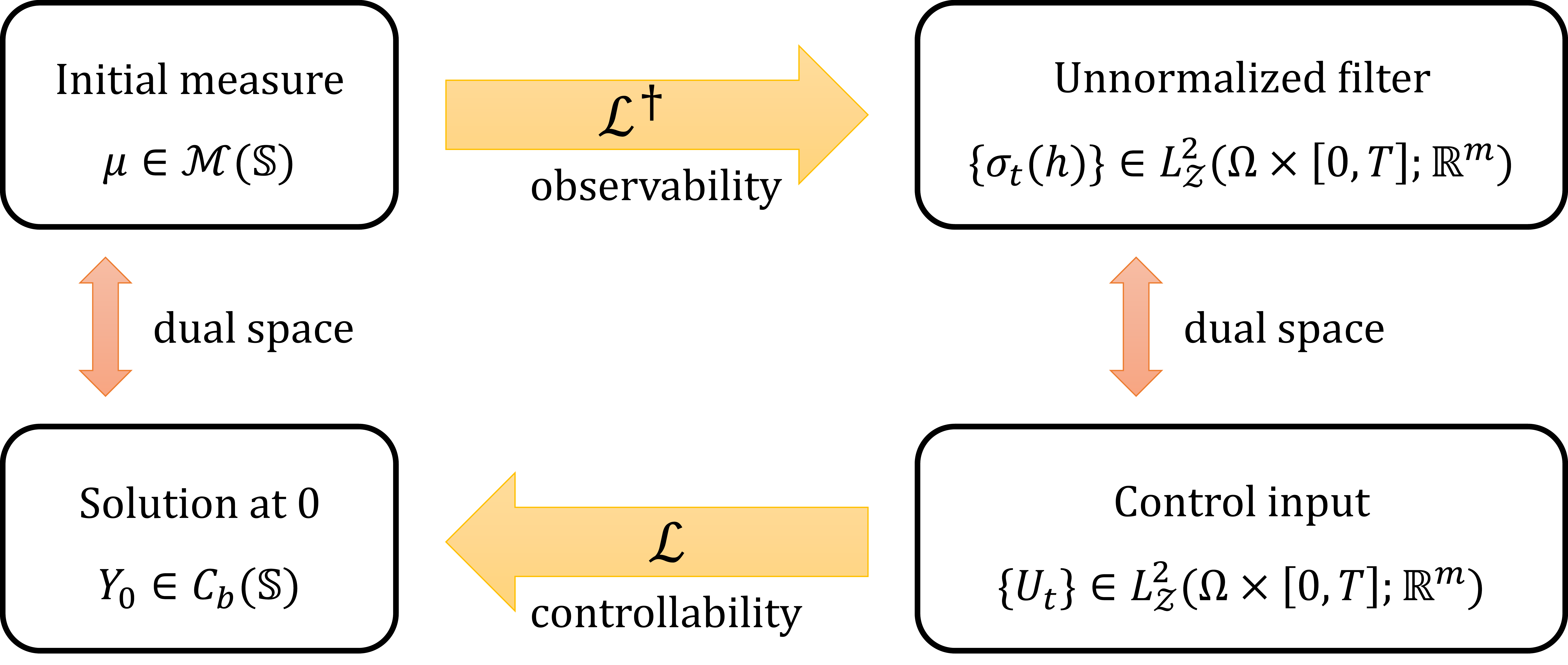}
	\caption{An illustration of the dual relationship for nonlinear filtering.}
	\label{fig:nonlinear-duality}
\end{figure}

The linear operator $\clL:L_\clZ^2\big(\Omega\times [0,T];\Re^m\big)\times \Re \to C_b(\bS)$ is defined as follows:
\begin{equation}\label{eq:ctrl-operator}
	\clL(U,c) = Y_0
\end{equation}
where $Y_0\in C_b(\bS)$ is the solution at time 0 to the BSDE~\eqref{eq:dual-bsde}.

The controllability is defined in the same way as linear systems theory. Note however that the target set (at time $T$) now is the space of constant functions (see also Remark~\ref{rm:terminal-condition}).

\begin{definition}
	For the BSDE~\eqref{eq:dual-bsde}, the \emph{controllable subspace}
	\begin{equation}\label{eq:ctrl-subspace}
		\clC := \Rsp(\clL) = \big\{y_0 \in C_b(\bS): \exists\, c\in \Re\text{ and } U \in L^2_\clZ\big(\Omega\times [0,T];\Re^m\big),  \text{ s.t. } Y_0 = y_0\text{ and } Y_T = c\ones\big\}
	\end{equation}
	The BSDE~\eqref{eq:dual-bsde} is said to be \emph{controllable} if $\clC$ is dense in $C_b(\bS)$.
\end{definition}
\medskip

The duality between observability of the model $(\clA, h)$ and the controllability of the BSDE~\eqref{eq:dual-bsde} is presented in the following theorem whose proof appears in Section~\ref{ssec:pf-observability-definition}:

\medskip

\begin{theorem}\label{thm:observability-definition}
	$\clL^\dagger$ is the adjoint operator of $\clL$. Consequently, the nonlinear model $(\clA, h)$ is observable if and only if the BSDE~\eqref{eq:dual-bsde} is controllable. 
\end{theorem}

\begin{table*}
	\centering
	\renewcommand{\arraystretch}{1.7}
	\small
	\begin{tabular}{m{0.14\textwidth}m{0.37\textwidth}m{0.4\textwidth}}
		& {\bf Linear deterministic case} & {\bf Nonlinear stochastic case} \\ \hline \hline
		Signal space & $\clU =  L^2([0,T];\Re^m)$\vspace{3pt} \newline  $\langle u,v\rangle = \displaystyle\int_0^T u_t^\tp v_t\ud t$ \vspace{2pt} & $\clU =  L^2_\clZ(\Omega\times [0,T];\Re^m)
		$\vspace{3pt} \newline $\langle U,V\rangle = \displaystyle\hE\Big(\int_0^T U_t^\tp V_t\ud t\Big)$\vspace{2pt} \\ \hline
		Function space & $\clY =\Re^d$ \vspace{2pt} \newline $\langle x,y \rangle = x^\tp y$ & $\clY = C_b(\bS)$, $\clY^\dagger = {\cal M}(\bS)$ \vspace{2pt} \newline $\langle \mu,y\rangle = \mu(y)$ \\ \hline
		Controllability & $\clL : \clU \to \clY$ \vspace{2pt} \newline
		\phantom{$\clL$} \hspace{0.2em} $u\mapsto y_0$ by ODE~\eqref{eq:LTI-ctrl} & $\clL:\clU\times\Re\to\clY$ \vspace{2pt} \newline
		\phantom{$\clL$} \hspace{0.2em}$(U,c)\mapsto Y_0$ by BSDE~\eqref{eq:dual-bsde} \\ \hline
		Observability & $\clL^\dagger : \clY \to \clU$ \newline
		\phantom{$\clL^\dagger$} \hspace{0.15em} $x_0\mapsto z_t$ by ODE~\eqref{eq:LTI-obs} & $\clL^\dagger:\clY^\dagger\to \clU\times\Re$, \vspace{2pt} \newline \phantom{$\clL^\dagger$} \hspace{0.15em} $\mu\mapsto (\sigma_t(h), \mu(\ones))$ by Zakai equation~\eqref{eq:Zakai} \\ \hline
		Duality & $\Rsp(\clL) = \clY \quad  \iff \quad \Nsp(\clL^\dagger) = \{0\}$ 
		& $\overline{\Rsp(\clL)} = \clY \quad  \iff \quad \Nsp(\clL^\dagger) = \{0\}$ 
		\\ \hline
	\end{tabular}
	\caption{Comparison of the controllability--observability duality for linear and nonlinear systems} \label{tb:comparison}
\end{table*}

\medskip

The BSDE~\eqref{eq:dual-bsde} is referred to as the \emph{dual control system} for the model $(\clA, h)$.
The correspondence between the linear and nonlinear cases appears as part of Table~\ref{tb:comparison}.
Before moving on, we make some remarks on function and measure spaces.

\medskip

\begin{remark}
	Note that the dual control system takes values in the infinite dimensional system $C_b(\bS)$. Generally speaking, the controllable subspace in infinite dimensional setting hardly satisfies $\clC = C_b(\bS)$ (see discussion on deterministic setting in~\cite[Chapter 4]{curtain2012introduction}). Rather, we defined the controllability by stating the closure $\overline{\Rsp(\clL)} = C_b(\bS)$. 
\end{remark}

\medskip

\begin{remark}\label{rm:terminal-condition}
	In the definition of $\clL$, the co-domain space is $L^2_{\clZ}\big(\Omega\times[0,T];\Re^m\big)\times \Re$. For the analysis of the filtering problem, it suffices to consider restriction of $\clL$ on the subspace $\clM_0(\bS) =\{\mu\in\clM(\bS): \mu(\ones) = 0\}$. (For example, the null-space of $\clL$ is a subspace of $\clM_0(\bS)$.) The advantage of considering the restriction is that the co-domain space for $\clL$, and therefore the domain of its adjoint, now is $L^2_{\clZ}\big(\Omega\times[0,T];\Re^m\big)$. The dual space of $\clM_0(\bS)$ is the quotient space $C_b(\bS)/\{c\ones:c\in\Re\}$ and therefore
	$\clL^\dagger:L^2_{\clZ}\big(\Omega\times[0,T];\Re^m\big)\to C_b(\bS)/\{c\ones:c\in\Re\}$. Although such a change will make duality between controllability and observability somewhat terser, we prefer to keep the function space as $\clM(\bS)$ and $C_b(\bS)$. 
	This has the advantage of not having to deal with the quotient space.
\end{remark}

\medskip

\begin{remark}
	The choice of function space $C_b(\bS)$ is guided by duality pairing between $C_b(\bS)$ and measure space $\clM(\bS)$ (see Example~\ref{ex:dual-function-measure} in Section~\ref{ssec:dual-space}). An important reason to consider this choice is to relate with the work of van Handel~\cite{van2009observability} who defines observable functions as a subspace of $C_b(\bS)$. 
	
	Alternatively, one may consider linear operator entirely on Hilbert spaces.
	A general setup is as follows:
	\begin{itemize}
		\item The state space $\bS$ admits a positive reference measure $\lambda$ (e.g., Lebesgue measure in Euclidean case or counting measure for finite / countable state space case).
		\item The space of functions is $L^2(\lambda)$
		\item The space of measures is the space of measures $\nu$ such that $\nu\ll\lambda$ and  $\dfrac{\ud \nu}{\ud \lambda} \in L^2(\lambda)$. 
	\end{itemize}
	In this case, one defines the linear operator as follows:
	\[
	\clL\;:\;L^2_\clZ\big(\Omega\times [0,T];\Re^m\big)\times \Re \to L^2(\lambda)
	\]
	Since $L^2(\lambda)$ is a Hilbert space, its adjoint
	\[
	\clL^\dagger\;:\; L^2(\lambda) \to L^2_\clZ\big(\Omega\times [0,T];\Re^m\big)\times \Re
	\]
	is again given by the solution of the Zakai equation.

\end{remark}

\medskip

\begin{remark}[Linear-Gaussian case] \label{rm:duality-reduces-LG}
	
	Consider the linear-Gaussian model~\eqref{eq:linear-Gaussian-model}. 
	We impose the following restrictions:
	\begin{itemize}
		\item The control input $U=u$ is restricted to be a deterministic
		function of time. In particular, it does not depend upon the
		observations (See Section~\ref{ssec:Kalman-filter}).  Such a control is trivially $\clZ$-adapted. 
		For such a control input, the solution $Y=y$ of the BSDE is a deterministic function of time, and $V=0$.  The BSDE becomes a PDE:
		\begin{equation}\label{eq:det_pde}
			-\frac{\partial y_t}{\partial t}(x) = (\clA y_t)(x) + h^\tp(x)
			u_t,\quad y_T = c\ones 
		\end{equation}
		where the lower-case notation is used to stress the fact that $u$
		and $y$ are now deterministic functions of time.
		
		\item 
		Instead of $C_b(\bS)$, it suffices to consider a finite ($d$-)dimensional space of linear functions:
		\[
		{\sf L}:=\{f \;:\;  f(x)=f^\tp x,\;\text{where } f\in\Re^d\}
		\]  
		Then ${\sf L}$ is an invariant subspace for the
		dynamics~\eqref{eq:det_pde}.  On ${\sf L}$, the PDE reduces to an ODE:  
		\[
		-\frac{\ud y_t}{\ud t}
		= A^\tp y_t + H^\tp u_t,\quad y_T = 0
		\]
		where the terminal condition 0 is the only constant function which is linear.
	\end{itemize}
	Therefore, the dual control system~\eqref{eq:dual-bsde} reduces to the LTI system~\eqref{eq:LTI-ctrl}. It is as yet unclear why it suffices to consider only deterministic control inputs. An explanation for this is provided in Chapter~\ref{ch:duality-principle}.
\end{remark}

\subsubsection{Explicit characterization of the controllable subspace}
The following proposition provides explicit characterization of the controllable subspace. Its proof appears in Section~\ref{ss:pf-prop43}.

\begin{proposition}\label{thm:controllable-subspace}
	Consider the linear operator~\eqref{eq:ctrl-operator}. For any finite $T > 0$,  the range space
	${\Rsp(\clL)}$ is the smallest such subspace $\clC\subset C_b(\bS)$ that satisfies the
	following two properties:
	\begin{enumerate}
		\item The constant function $\ones\in \clC$;
		\item If $g\in\clC$ then $\clA g \in \clC$ and $g h
		\in\clC$. 
	\end{enumerate}
\end{proposition}

\subsection{Controllability gramian} \label{ssec:gramian}
The \emph{controllability gramian} $\sW :\clM(\bS)\to C_b(\bS)$ is a deterministic linear operator defined as follows:
\[
\sW := \clL \clL^\dagger
\]	
Explicitly, for $\mu \in \clM(\bS)$, 
\[
\sW \mu = Y_0
\]
where $Y_0$ is obtained for solving the BSDE
\[
-\ud Y_t(x) = \big(\clA Y_t(x) + h^\tp(x)(\sigma_t(h)+V_t(x))\big)\ud t - V_t^\tp(x)\ud Z_t\quad Y_T(x) = \mu(\ones),\quad x\in \bS
\]
As in the deterministic settings, the gramian yields an explicit control input to transfer initial condition $Y_0 = f$ to $Y_T = c\ones$. The following proposition is proved in Section~\ref{ss:pf-prop44}

\begin{proposition}\label{prop:gramian-observability}
	Suppose $f\in \Rsp(\sW)$, i.e., there exists $\mu\in\clM(\bS)$ such that $f = \sW \mu$. Then the control
	\[
	U_t = \sigma_t^\mu(h),\quad 0\le t \le T
	\]
	transfers the system~\eqref{eq:dual-bsde} from $Y_T = \mu(\ones)\ones$ to $Y_0 = f$.
	Suppose $\tilde{U}$ is another control which also transfers $Y_T = c\ones$ to $Y_0 = f$ for some $c\in \Re$. Then
	\[
	\tE\Big(\int_0^T|\tilde{U}_t|^2 \ud t \Big) + c^2 \ge \tE\Big(\int_0^T |U_t|^2\ud t \Big) + \big(\mu(\ones)\big)^2
	\] 
\end{proposition}


\subsection{Stabilizability and detectability}
%

Analogous to the LTI case, the definitions for stabilizability and detectability begin with the definition of stable subspace. Consider the solution $\{\mu_t\in\clM(\bS): t\ge 0\}$ to the 
\emph{Forward Kolmogorov equation}:
\begin{equation}\label{eq:FKE}
	\mu_t(f) = \mu_0(f) + \int_0^t \mu_s\big(\clA f\big)\ud s,\quad t\ge 0
\end{equation}
The stable subspace of $\clA$ is defined by using the notation of weak convergence:
\[
S_s := \big\{\mu_0\in \clM(\bS): \mu_T(f)\,\to\,0 \text{ as }T\to \infty, \forall\, f\in C_b(\bS)\big\}
\]
Observe that a constant function is $\clA$-invariant and therefore 
$\mu_T(\ones) = \mu_0(\ones)$. Consequently, $S_s\subset \clM_0(\bS)$. 
The stabilizability and detectability are defined as follows:
%

\medskip

\begin{definition}\label{def:stabilizability}
	The BSDE~\eqref{eq:dual-bsde} is \emph{stabilizable} if $\Rsp(\clL)^\bot \subset S_s$. 
\end{definition}

\medskip

\begin{definition}\label{def:detectability}
	The nonlinear model $(\clA, h)$ is \emph{detectable} if $\Nsp(\clL^\dagger) \subset S_s$. 
\end{definition}

\medskip

\begin{corollary}\label{cor:detect-stab}
	The nonlinear model $(\clA, h)$ is detectable if and only if the 
	BSDE~\eqref{eq:dual-bsde} is stabilizable.
\end{corollary}

\medskip

\begin{remark}\label{rm:vh-detectabilty}
	Detectability of an HMM is also considered by van Handel~\cite[Definition V.1]{van2010nonlinear}.
	His statement is as follows: An HMM is detectable if for any $\mu,\nu\in\clP(\bS)$, either:
	\[
	\sP^\mu|_{\clZ_T} = \sP^\nu|_{\clZ_T}\quad \text{or} \quad \|\mu_T - \nu_T\|_\tv\; \longrightarrow \;0\quad \text{as }T\to \infty
	\]
	By Theorem~\ref{thm:observability-Zakai-relation}, this statement is identical to the Def.~\ref{def:detectability}.
\end{remark}

\medskip

\begin{remark}\label{rm:ergodicity}
	We say the state process is \emph{ergodic} if the it admits a unique 
	invariant measure $\bmu$ such that for all $\mu_0 \in \clP(\bS)$
	\[
	\|\mu_T - \bmu\|_\tv \; \longrightarrow \;0\quad \text{as }T\to 
	\infty
	\]
	Now, for any $\tilde{\mu}_0 \in 
	\clM_0(\bS)$, there exists $\mu_0^{(1)},\mu_0^{(2)}\in\clP(\bS)$ and 
	$c\in\Re$ such that $\tilde{\mu}_0 = c(\mu_0^{(1)}-\mu_0^{(2)})$. Therefore,
	\[
	\|\tilde{\mu}_T\|_\tv = c\|\mu_T^{(1)}-\mu_T^{(2)}\|_\tv 
	\le c\|\mu_T^{(1)}-\bmu\|_\tv +c\|\mu_T^{(2)}-\bmu\|_\tv 
	\;\longrightarrow\;0\quad \text{as }T\to \infty
	\]
	Therefore, if the state process is ergodic then $S_s = \clM_0(\bS)$, and the model $(\clA,h)$ is stabilizable irrespective of $h$.
\end{remark}


%
%
%

\medskip

\section{Explicit formulae for the finite state space case} \label{sec:finite-observability}

For finite state space, both $C_b(\bS)$ and $\clM(\bS)$ are isomorphic to $\Re^d$ (equipped with suitable norms). Therefore, the dual control system~\eqref{eq:dual-bsde} is expressed as follows:
\begin{equation}\label{eq:dual-ctrl-finite}
	-\ud Y_t = \Big(AY_t +HU_t+ \sum_{j=1}^mH^j\cdot V_t^j\Big)\ud t - V_t \ud Z_t,\quad Y_T = c\ones
\end{equation}
where $H^j$, $V_t^j$ denote the $j^{\text{th}}$ column of $H$ and $V_t$, respectively, and the dot notation denotes the element-wise product.
The solution pair is $(Y,V) \in L^2_\clZ([0,T];\Re^d)\times L^2_\clZ([0,T];\Re^{d\times m})$.

The controllable space $\clC$ is also a subspace of $\Re^d$. Directly by applying Prop.~\ref{thm:controllable-subspace}, it is computed as follows:
\begin{align}
	\clC = \sp\big\{\ones, &\,  H, \,  AH, \,  A^2H, \,  A^3H, \, \ldots, \label{eq:obs_gram_nl}\\
	&H\cdot H, \,  A(H \cdot H), \,  H\cdot (AH), \,  A^2(H\cdot H),\ldots, \nonumber\\
	&H\cdot (H\cdot H), \,  (AH)\cdot (H\cdot H), \,  H\cdot A(H \cdot H), \, \ldots \big\} \nonumber
\end{align}


One notes that the first line of~\eqref{eq:obs_gram_nl} is identical to the controllability matrix for the LTI system~\eqref{eq:LTI-ctrl}. Therefore, if the linear model~\eqref{eq:LTI-obs} is observable then the nonlinear model is also observable. However, the latter property is in general much weaker. 
For instance, the following proposition provides a sufficient condition for stochastic observability regardless of $A$. Its proof appears in Section~\ref{ss:pf-prop45}.

\begin{proposition}\label{prop:sufficient}
	Consider the nonlinear model $(\clA,h)$ for the finite state-space.
	The system is observable if $h(i)=H_i$ is an injective map from
	$\bS$ into $\Re^m$. (The map is injective if and only if $H_i \neq H_j$ for all $i\neq
	j$ where $H_i$ is the $i^\text{th}$ row of the $d\times m$ matrix $H$).    
	If $A=0$ then the injective property of the function $h$ is
	also necessary for observability.    
\end{proposition}

\begin{remark}\label{rm:finite-case-equivalency}
	In~\cite{van2009observability}, test for observability is provided 
	by defining the space of observable functions (see Def.~\ref{def:un-observable-measures-observable-functions}). For finite case, van Handel defines 
	$\{h_1,\ldots,h_r\}:=h(\bS)$ be the set of possible observations. Obviously 
	$r \le d$ with equality holds when every element in the state-space yields 
	a distinct outcome. He also defines projection matrices $P_{h_k}\in\Re^{d\times d}$ is 
	defined by $[P_{h_k}]_{ij} = 1$ if $i=j$ and $h(i)=H_i=h_k$, and zero 
	otherwise. The space of observable functions is then given by~\cite[Lemma 9]{van2009observability}
	\[
	\clO = \sp\big\{P_{n_0}AP_{n_0}AP_{n_2}\cdots AP_{n_k}\ones:k\ge 0,n_i\in h(\bS)\big\}
	\]
	It is shown in Section~\ref{ssec:justifiction-vanHandel-and-me-finite} that $\clO= \clC$ (formula in~\eqref{eq:obs_gram_nl}).
	%
	
\end{remark}

\subsubsection{Controllability gramian}

The controllability gramian $\sW$ is a $d\times d$ deterministic matrix.
Recall that the solution operator of the Zakai equation $\Psi_t$ is now $d\times d$ matrix in finite case (see Example~\ref{ex:finite-Zakai-operator}).
It is shown in Section~\ref{ssec:derivation-ctrl-gramian} that
\[
\sW = \ones\ones^\tp + \tE\Big(\int_0^T \Psi_t^\tp HH^\tp \Psi_t \ud t\Big)
\]	

Since $\sW$ is a deterministic matrix in $\Re^{d\times d}$, the Prop.~\ref{prop:gramian-observability} becomes a simple rank condition:
\[
\sW \text{ is full rank}\quad \Longleftrightarrow \quad \text{Dual system~\eqref{eq:dual-ctrl-finite} is controllable}
\]

\subsubsection{Stabilizability of the dual system}

The stabilizability Def.~\ref{def:stabilizability} reduces to simple inclusion property.
By Ger\v{s}gorin circle theorem~\cite[Theorem 6.1.1]{horn_johnson_1985},
%
all eigenvalues of $A$ are in one of closed discs centered at $A(i,i)$ and radius $|A(i,i)|$ for $i\in\bS$. In consequence, all eigenvalues of $A$ are either in the open left half-plane or at zero. Therefore, the unstable mode of $A$ is 
\[
S_s^\bot = S_0 := \{f \in \Re^d \mid \; Af = 0\} 
\]
This yields a simple characterization of stabilizability of the BSDE~\eqref{eq:dual-ctrl-finite} (and therefore also the detectability of model $(\clA,h)$ using duality).

\medskip

\begin{corollary}\label{prop:detectability-nullspace}
	The dual BSDE~\eqref{eq:dual-ctrl-finite} is stabilizable if and only if $S_0\subset\clC$.
\end{corollary}

\medskip

In Chapter~\ref{ch:filter-stability-2}, we will discuss the relationship of stabilizability to filter stability.

\section{Proofs of the statements}

\subsection{Proof of Proposition~\ref{prop:clark-result}}\label{ssec:pf-clark-result}

%

By applying Girsanov theorem on the innovation process, one 
obtains~\cite[Corollary 
1.1.15]{van2006filtering}:
\begin{align*}
	\frac{\ud \sP^\mu|_{\clZ_T}}{\ud \tsP^\mu|_{\clZ_T}} &= \exp\Big(\int_0^T 
	\pi_t^\mu(h) \ud Z_t - \half \int_0^T |\pi_t^\mu(h)|^2 \ud t\Big)\\
	\frac{\ud \sP^\nu|_{\clZ_T}}{\ud \tsP^\nu|_{\clZ_T}} &= \exp\Big(\int_0^T 
	\pi_t^\nu(h) \ud Z_t - \half \int_0^T |\pi_t^\nu(h)|^2 \ud t\Big)
\end{align*}
Note that $\tsP^\mu|_{\clZ_T} = \tsP^\nu|_{\clZ_T}$ because $Z$ has the same probability law under either probability measure. Therefore we have
\begin{align*}
	\frac{\ud \sP^\mu|_{\clZ_T}}{\ud \sP^\nu|_{\clZ_T}} &= \exp\Big(\int_0^T \pi_t^\mu(h) \ud Z_t - \half \int_0^T |\pi_t^\mu(h)|^2 \ud t-\int_0^T \pi_t^\nu(h) \ud Z_t + \half \int_0^T |\pi_t^\nu(h)|^2 \ud t\Big)\\
	&=\exp\Big(\int_0^T\pi_t^\mu(h)-\pi_t^\nu(h)\ud Z_t + \half \int_0^T-|\pi_t^\mu(h)|^2+|\pi_t^\nu(h)|^2\ud t\Big)\\
	&=\exp\Big(\int_0^T\pi_t^\mu(h)-\pi_t^\nu(h)\ud I_t^\mu + \half \int_0^T|\pi_t^\mu(h)-\pi_t^\nu(h)|^2\ud t\Big)
\end{align*}
Therefore,
\[
\kl\big( \sP^\mu|_{\clZ_T}\mid  \sP^\nu|_{\clZ_T}\big) = \E^\mu\Big(\log\frac{\ud \sP^\mu|_{\clZ_T}}{\ud \sP^\nu|_{\clZ_T}}\Big) = \half \E^\mu\Big(\int_0^T|\pi_t^\mu(h)-\pi_t^\nu(h)|^2\ud t\Big)
\]
because $I^\mu$ is $\sP^\mu$-martingale.
\qed

\subsection{Proof of Theorem~\ref{thm:observability-Zakai-relation}}\label{pf-obs-Zakai}

\noindent(1 $\Longleftrightarrow$ 2) It is directly deduced from Prop.~\ref{prop:clark-result}.

\noindent(2 $\Longrightarrow$ 3) By~\eqref{eq:normalize-Zakai}, $\sigma_t(\ones)$ is expressed by 
\[
\sigma_t(\ones) = 1 + \int_0^t \sigma_s(\ones)\pi_s(h)\ud Z_s
\]
and therefore, $\pi_s^\mu(h) = \pi_s^\nu(h)$ for all $0\le s\le t$ implies $\sigma_t^\mu(\ones) = \sigma_t^\nu(\ones)$. Due to~\eqref{eq:normalize-Zakai}, it implies $\sigma^\mu(h) = \sigma^\nu(h)$.

\noindent(3 $\Longrightarrow$ 2) From the Zakai equation~\eqref{eq:Zakai},
\[
\sigma_t(\ones) = 1 + \int_0^t \sigma_s(h)\ud Z_s
\]
and therefore, $\sigma_t^\mu(h) = \sigma_t^\nu(h)$ implies $\sigma_t^\mu(\ones) = \sigma_t^\nu(\ones)$. Due to~\eqref{eq:normalize-Zakai}, it implies $\pi_t^\mu(h) = \pi_t^\nu(h)$. Therefore, 3 implies 2.
\qed

\subsection{Proof of Theorem~\ref{thm:observability-definition}}\label{ssec:pf-observability-definition}

Note that $L^2_\clZ\big(\Omega\times[0,T];\Re^m\big)\times \Re$ is a Hilbert space equipped with inner product
\[
\langle (U,c),(V,d)\rangle = \tE\Big(\int_0^T U_t^\tp V_t \ud t\Big) + cd
\]

By linearity, $\clL(U,c) = \clL(U,0)+c\ones$ for $U\in L_\clZ^2\big([0,T],\Re^m\big)$ and $c\in
\Re$.  Therefore, for any $\mu \in \clM(\bS)$,
\[
\langle \mu, \clL(U,c) \rangle =  \langle \mu,\clL(U,0) \rangle + c \mu(\ones)
\]
Thus, the main calculation is to transform $\langle \mu,\clL(U,0)\rangle = \mu(Y_0)$.
Now use the It\^o-Wentzell formula for measures~\cite[Theorem 1.1]{krylov2011ito} on $\sigma_t(Y_t)$,
\begin{align*}
	\ud \big(\sigma_t(Y_t)\big) &= \big(\sigma_t(\clA Y_t) \ud t + \sigma_t(h^\tp Y_t)\ud Z_t\big) + \sigma_t(h^\tp V_t)\ud t \\
	&\quad+ \big(\sigma_t(-\clA Y_t - h^\tp U_t - h^\tp V_t) \ud t + \sigma_t(V_t)\ud Z_t\big)\\
	&= -U_t^\tp \sigma_t(h)\ud t + \sigma_t(h^\tp Y_t + V_t^\tp) \ud Z_t
\end{align*}
Integrating both sides,
\[
\sigma_t(Y_T) - \mu(Y_0) = -\int_0^T U_t^\tp\sigma_t(h)\ud t + \int_0^T \sigma_t(h^\tp Y_t+V_t^\tp) \ud Z_t
\]
Since $Z$ is a $\tsP$-B.M.,
\begin{equation*}\label{eq:pf_thm1_1}
	\mu(Y_0) = \tE\Big(\int_0^T U_t^\tp \sigma_t(h)\ud t \Big) = \langle \sigma(h),U\rangle
\end{equation*}
Therefore, 
\[
\langle \mu,\clL(U,c) \rangle = \langle \sigma(h), U \rangle + c\mu(\ones)
\] 
\qed

\subsection{Proof of Proposition~\ref{thm:controllable-subspace}}\label{ss:pf-prop43}

For notational ease, we assume $m=1$.  The idea of the proof is adapted from~\cite[Theorem 3.2]{peng1994backward}.
The definition of $\Nsp(\clL^\dagger)$ is:
\[
\mu \in \Nsp(\clL^\dagger) \Leftrightarrow \mu(\ones) = 0 \text{
	and } \sigma_t(h) \equiv 0 \quad \forall\;t\in[0,T]
\]
Since $\Nsp(\clL^\dagger)$ is the annihilator of $\Rsp(\clL)$, we have
$\ones,h \in \Rsp(\clL)$. Consider next the 
Zakai equation~\eqref{eq:Zakai} with the initial condition
$\mu\in\Nsp(\clL^\dagger)$ and $f=h$:
\[
\sigma_t(h) = \mu(h) + \int_0^t \sigma_s(\clA h) \ud s + \int_0^t \sigma_s(h^2) \ud Z_s
\]
Since $t$ is arbitrary, the left-hand side is identically zero for all $t\in[0,T]$ if and only if
\[
\mu(h) = 0,\quad \sigma_t(\clA h) \equiv 0,\quad \sigma_t(h^2)\equiv 0 \quad \forall\;t\in[0,T]
\]
and in particular, this implies $\clA h, h^2\in\Rsp(\clL)$. 

The subspace $\clC$ is obtained by continuing to repeat the steps
ad infinitum: If at the conclusion of the  $k^\text{th}$ step, we find
a function $g\in \clC$ such that $\sigma_t(g)\equiv 0$ for all
$t\in[0,T]$.  Then through the use of the Zakai equation,
\[
\mu(g) = 0,\quad \sigma_t(\clA g) \equiv 0,\quad \sigma_t(hg)\equiv 0
\quad \forall\;t\in[0,T]
\] 
so $\clA g, hg \in \clC$.  By construction, because $
\mu \in \Nsp(\clL^\dagger)$, $\clC
= \Rsp(\clL)$. 
\qed

\subsection{Proof of Proposition~\ref{prop:gramian-observability}}\label{ss:pf-prop44}

Note that $L^2_\clZ\big(\Omega\times[0,T];\Re^m\big)\times \Re$ is a Hilbert space equipped with inner product
\[
\langle (U,c),(V,d)\rangle = \tE\Big(\int_0^T U_t^\tp V_t \ud t\Big) + cd
\]

Suppose $f \in \operatorname{span}(\sW)$. Then there exists $\mu\in\clM(\bS)$ such that $\sW\mu = f$. Let $\big(U,\mu(\ones)\big) = \clL^\dagger\mu$, and apply the control $U$ to the BSDE with terminal condition $Y_T = c1$. Then 
\[
Y_0 = \clL(U,\mu(\ones)) = \clL \clL^\dagger \mu = \sW\mu = f
\]
If another $(\tilde{U},c)$ satisfies $\clL(\tilde{U},c) = f$. Then $\clL\big(U-\tilde{U},\mu(\ones)-c\big) = 0$, and therefore
\[
0 = \big\langle \clL\big(U-\tilde{U},\mu(\ones)-c\big) , \mu \big\rangle = \big\langle \big(U-\tilde{U},\mu(\ones)-c\big) , \clL^\dagger\mu \big\rangle = \big\langle \big(U-\tilde{U},\mu(\ones)-c\big),\big(U,\mu(\ones)\big) \big\rangle
\]
Therefore,
\[
\big\|(\tilde{U},c)\big\|^2 = \big\|(U,\mu(\ones))\big\|^2 + \big\|\big(U-\tilde{U},\mu(\ones)-c\big)\big\|^2 \ge \big\|(U,\mu(\ones))\big\|^2
\]
\qed

\subsection{Proof of Proposition~\ref{prop:sufficient}}\label{ss:pf-prop45}

\newP{Step 1} We first provide the proof for the case when $m=1$. In
this case, $H$ is a column vector and $H_i$ denotes its
$i^{\text{th}}$ element. 
We claim that if $H_i\neq H_j$ for all $i\neq j$, then 
\begin{equation}\label{eq:subset}
	\sp\{\ones, \, H, \, H\cdot H, \, \ldots,\, \underbrace{H\cdot H \cdots H}_{(d-1)\text{ times}}\} = \Re^d
\end{equation}
where (as before) the dot denotes the element-wise product.
Assuming that the claim is true, the result easily follows because the
vectors on left-hand side are contained in $\Rsp(\clL)$
(see~\eqref{eq:obs_gram_nl}). It remains to prove the claim. 
For this purpose, express the left-hand side of~\eqref{eq:subset} as
the column space of the following matrix:
\begin{align*}
	\begin{pmatrix}
		1 & H_1 & H_1^2 & \cdots & H_1^{d-1}\\
		1 & H_2 & H_2^2 &\cdots & H_2^{d-1}\\
		\vdots &\vdots  & \vdots &\cdots & \vdots\\
		1 & H_d & H_d^2 & \cdots & H_d^{d-1}
	\end{pmatrix}
\end{align*}
This matrix is easily seen to be full rank by using the Gaussian elimination:
\[
\begin{pmatrix}
	1 & H_1 & H_1^2 & \cdots & H_1^{d-1}\\
	0 & H_2-H_1 & H_2^2-H_1^2 &\cdots & H_2^{d-1}-H_1^{d-1}\\
	\vdots &\vdots  & \vdots &\cdots & \vdots\\
	0 & 0 & 0 & \cdots & \prod_{i=1}^{d-1} (H_d-H_i)\\
\end{pmatrix}
\]
The diagonal elements are non-zero because $H_i\neq H_j$.

\newP{Step 2} In the general case, $H$ is a $d \times m$ matrix and
$H_i$ denotes its $i^\text{th}$ row.  We claim that if $H_i\neq H_j$
for all $i\neq j$ then there exists a vector $\tilde{H}$ in the column
span of $H$ such that $\tilde{H}_i\neq \tilde{H}_j$ for all $i\neq
j$. Assuming that the claim is true, the result follows from the $m=1$
case by considering~\eqref{eq:subset} with $\tilde{H}$.
It remains to prove the claim.  Let $\{e_1,\ldots,e_d\}$ denote the
canonical basis in $\Re^d$.  The assumption means $(e_i-e_j)^\tp H$
is a non-zero row-vector in $\Re^m$ for all $i\neq j$. 
Therefore, the null-space of $(e_i-e_j)^\tp H$ is a
$(m-1)$-dimensional hyperplane in $\Re^m$. Since there are only  
finite such hyperplanes, there must exist a vector $a\in\Re^m$
such that $(e_i-e_j)^\tp Ha \neq 0$ for all $i\neq j$. Pick such an
$a$ and define $\tilde{H} := Ha$.

\newP{Step 3} To show the necessity of the injective property when
$A=0$, assume $H_i=H_j$ for some $i\neq j$. Then the corresponding row
is identical, so it cannot be rank $d$.

\qed

\subsection{Justification of the claim in Remark~\ref{rm:finite-case-equivalency}} \label{ssec:justifiction-vanHandel-and-me-finite}

We start from $A=0$ case:
\[
\clC = \sp\{\ones, H, \dv(H)H, \dv(H)^2H,\ldots\}
\]
Since $\dv(H)^nH = [h^{n+1}(1),\ldots,h^{n+1}(d)]$, an element of $f\in \clC$ can be expressed by
\[
f=\sum_{j=0}^\infty a_j[h^j(1),\ldots,h^j(d)] =  \sum_{j=0}^\infty a_j h_1^j P_{h_1}\ones + \ldots + \sum_{j=0}^\infty a_j h_r^j P_{h_r}\ones \in \clO
\]
Therefore, it follows that $\clC \subset \clO$. To show $\clO\subset\clC$, let
\[
f=\sum_{k=1}^r b_kP_{h_k}\ones
\]
It suffices to show that there exists $\{a_j:j=0,1,\ldots\}$ such that $b_k = \sum_{j=0}^\infty a_j h_k^j$ for all $k=1,\ldots,r$. In fact, such $a_j$ can be found, by setting $a_j = 0$ for $j\geq r$ and invert the following matrix:
\[
\begin{pmatrix}
	1 & h_1 & h_1^2 &\cdots & h_1^{r-1}\\
	1 & h_2 & h_2^2 &\cdots & h_2^{r-1}\\
	\vdots &\vdots & \vdots &\cdots & \vdots\\
	1 & h_r & h_r^2 &\cdots & h_r^{r-1}\\
\end{pmatrix}
\]
It is invertible, since it transforms via Gaussian elimination:
\[
\begin{pmatrix}
	1 & h_1 & h_1^2 & \cdots & h_1^{r-1}\\
	0 & h_2-h_1 & h_2^2-h_1^2 &\cdots & h_2^{r-1}-h_1^{r-1}\\
	\vdots &\vdots  & \vdots &\cdots & \vdots\\
	0 & 0 & 0 & \cdots & \prod_{i=1}^{r-1} (h_r-h_i)\\
\end{pmatrix}
\]
which is full-rank by the fact that $h_i$ are distinct.
For general $A\neq 0$ case, we repeat the same procedure as above for arbitrary matrices $M_1$ and $M_2$ which are multiples of $A$ and $\dv(H)$, to claim that
\begin{align*}
	\sp\{M_1M_2H, M_1\dv(H)M_2H, M_1\dv(H)^2M_2H,\ldots\} = \sp\{M_1P_{h_k}M_2H:k=1,\ldots,r\}
\end{align*}
The proposition is proved by repeating this for countable times.
\qed

\subsection{Derivation of the controllability gramian in finite state space case} \label{ssec:derivation-ctrl-gramian}

For a given input $\mu\in\Re^d$, $\sW\mu$ is the solution $Y_0$ via the dual BSDE:
\[
-\ud Y_t = \big(AY_t + HH^\tp \sigma_t + \sum_{j=1}^m H^j\cdot V_t^j\big) \ud t - V_t \ud Z_t,\quad Y_T = \ones\ones^\tp \mu
\]
Recall the solution operator $\Psi_t$ of the Zakai equation from Example~\ref{ex:finite-Zakai-operator}.
Consider the process
\[
\Theta_t:= \Psi_t^\tp Y_t + \int_0^t \Psi_s^\tp HH^\tp \sigma_s \ud s,\quad 0\le t \le T
\]
Then by It\^o product formula,
\[
\ud \Theta_t = \Psi_t^\tp \big(\dv(Y_t)H + V_t\big)\ud Z_t
\]
Therefore, $\Theta_t$ is a $\tsP$-martingale. In particular, 
\[
Y_0 = \tE\Big(\Psi_T^\tp 11^\tp \mu + \int_0^T \Psi_t^\tp HH^\tp\sigma_t\ud t\Big) 
\]
Since the un-normalized filter is given by $\sigma_t = \Psi_t\mu$, 
\[
\sW \mu = \tE\Big(\Psi_T^\tp 11^\tp + \int_0^T \Psi_t^\tp H H^\tp\Psi_t\ud t \Big) \,\mu
\]
Finally, $\tE(\Psi_T^\tp \ones\ones^\tp) = \ones\ones^\tp$ because $\ud \tE(\Psi_t^\tp \ones) = 0$. 

\newpage


\chapter{Duality for nonlinear filtering}\label{ch:duality-principle}


In this chapter, the second original contribution of this thesis, namely the dual optimal control formulation for the stochastic filtering problem, is presented.  
The mathematical statement of the dual relationship between optimal
filtering and optimal control is expressed in the form of a duality
principle (Theorem~\ref{thm:duality-principle}).  The principle relates the optimal value
function for the optimal control problem to the minimum variance of
the optimal filtering problem.  The proposed formulation is shown to be a generalization of the Kalman-Bucy duality principle. It is an exact extension
in the sense that the dual optimal control problem has the same
minimum variance structure for the linear and the nonlinear filtering
problems. In particular, Kalman and Bucy's linear-Gaussian result is shown to be a
special case.

The solution of the optimal control problem is obtained using the
stochastic maximum principle which is used to derive the Hamilton's equations.
Explicit form of the Hamiton's equations are obtained for the finite
state-space and the Euclidean cases.  In the usual manner, by relating
the co-state to the state through a linear transformation, a feedback
form of the optimal control input is also derived.  An alternative
approach to obtain the optimal control is through a martingale
characterization.  The formula for the optimal control is used to
obtain a novel derivation of the Kushner-Stratonovich equation~\eqref{eq:nonlinear-filter} of nonlinear filtering.  

The final section of this chapter includes an alternate derivation of these results using the innovation process instead of the observations.  

The outline of the remainder of this chapter is as follows: The dual
optimal control problem along the duality principle for the nonlinear
filter, and its relation to the linear-Gaussian case is described in
Section~\ref{sec:duality-principle}.
Its solution using the maximum principle and the martingale characterization appears in
Section~\ref{sec:standard-form} and Section~\ref{ssec:martingale},
respectively. 
A derivation of the equation of the nonlinear filter appears in
Section~\ref{sec:derivation-of-the-filter}.  The innovation based
approach appears in Section~\ref{sec:innovation-method}.

\section{The duality principle}\label{sec:duality-principle}

\subsection{Nonlinear filtering and its minimum variance interpretation}

Consider the nonlinear model $(\clA,h)$ over a fixed time horizon $[0,T]$ where $T<\infty$. For a function $F\in L^2_{\clZ_T}\big(\Omega;C_b(\bS)\big)$, the conditional mean 
$\pi_T(F)$ is the minimum variance estimate of $F(X_T)$~\cite[Section~6.1.2]{bensoussan2018estimation}:
\[
\pi_T(F) = \mathop{\operatorname{argmin}}_{S_T\in L^2_{\clZ_T}(\Omega;\Re)} \E\big(|F(X_T)-S_T|^2\big)
\]
Our goal in this chapter is to express this minimum variance optimization problem as a dual optimal control problem.

\paragraph{Notation.}
For $F,G \in L^2_{\clZ_T}\big(\Omega;C_b(\bS)\big)$, the conditional variance and covariance are denoted by follows:
\begin{align*}
	\text{(cond.~variance):}\qquad&\clV_T(F) := \E\big(|F(X_T)-\pi_T(F)|^2\mid \clZ_T\big) = \pi_T(F^2) - \big(\pi_T(F)\big)^2\\
	\text{(cond.~covariance):}\qquad&\cov_T(F,G) := \E\big((F(X_T)-\pi_T(F))(G(X_T)-\pi_T(G))\mid \clZ_T\big) = \pi_T(FG)-\pi_T(F)\pi_T(G)
\end{align*}
In the remainder of the thesis, we often refer to $\E\big(\clV_T(F)\big)$ as ``variance'' instead of the more verbose  ``expectation of the conditional variance.''

\subsection{Dual optimal control problem} 

The function space of admissible control is denoted by $\clU := L^2_{\clZ}\big(\Omega\times[0,T];\Re^m\big)$.
An element of $\clU$ is denoted $U = \{U_t \in \Re^m: 0\le t\le T\}$.
It is referred to as the control input.

\subsubsection{Dual optimal control problem}
\begin{subequations}\label{eq:dual-optimal-control}
	\begin{align}
		\mathop{\text{Minimize}}_{U\in\clU}\text{:}\quad\quad\sJ_T(U) &= \E\Big(|Y_0(X_0)-\mu(Y_0)|^2 + \int_0^T  (\Gamma Y_t)(X_t) + |U_t + V_t(X_t)|^2 \ud t
		\Big) \label{eq:dual-optimal-control-a}\\
		\text{Subject to:}\; -\ud Y_t(x) &= \big((\clA Y_t)(x) + h^\tp(x)(U_t+V_t(x))\big)\ud t - V_t^\tp(x)\ud Z_t,\quad Y_T(x) = F(x),\;x\in\bS \label{eq:dual-optimal-control-b}
	\end{align}
\end{subequations}
The constraint~\eqref{eq:dual-optimal-control-b} is the same as the dual control system~\eqref{eq:dual-bsde}, now with the terminal condition $Y_T = F$ where $F \in L^2_{\clZ_T}\big(\Omega;C_b(\bS)\big)$. 
The relationship to the minimum variance objective is expressed through the following theorem whose proof appears in Section~\ref{ss:pf-thm51}. 

\begin{theorem}[Duality principle]\label{thm:duality-principle}
	For any admissible control $U\in \clU$, define the estimator
	\begin{equation}\label{eq:estimator}
		S_T = \mu(Y_0) - \int_0^T U_t^\tp \ud Z_t
	\end{equation}
	Then 
	\begin{equation}\label{eq:duality-principle}
		\sJ_T(U) = \E\big(|F(X_T)-S_T|^2\big)
	\end{equation}
	
\end{theorem}

\medskip

Thus, formally, the problem of obtaining the minimum variance estimate $S_T$
of $F(X_T)$ (minimizer of the right-hand side of the identity~\eqref{eq:duality-principle}) is 
converted into the problem of finding the optimal control $U$
(minimizer of the left-hand side of the identity~\eqref{eq:duality-principle}).
In order to conclude that the conditional mean is obtained from solving the dual optimal control problem~\eqref{eq:dual-optimal-control}, it is both necessary and sufficient to show that there exists a $U\in \clU$ such that $S_T = \pi_T(F)$.
Since $Z$ is a $\tsP$-B.M., the following lemma is a consequence of the It\^o representation theorem~\cite[Theorem 4.3.3]{oksendal2003stochastic}:

\medskip

\begin{lemma}\label{lem:representation}
	For any $F \in L_{\clZ_T}^2(\Omega;C_b(\bS))$, there exists a unique $U\in \clU$ such that
	\[
	\pi_T(F) = \tE\big(\pi_T(F)\big) - \int_0^T U_t^\tp \ud Z_t,\quad \tsP\text{-a.s.}
	\]
\end{lemma}

\medskip

\begin{remark}\label{rm:mean and the variance}
	Combined with the duality principle, Lemma~\ref{lem:representation} has two implications:
	\begin{itemize}
		\item The optimal control $U^\opt = \{U_t^\opt:0\le t \le T\}$ obtained from solving the dual optimal control problem yields the conditional mean:
		\[
		\pi_T(F) = \mu(Y_0) - \int_0^T \big(U_t^\opt\big)^\tp \ud Z_t,\quad \sP\text{-a.s.}
		\]
		\item The optimal value is the expected value of conditional variance
		\[
		\E\big(\clV_T(F)\big) = \E\Big(|Y_0(X_0)-\mu(Y_0)|^2 + \int_0^T (\Gamma Y_t)(X_t)+|U_t+V_t(X_t)|^2 \ud t\Big)
		\]
	\end{itemize}
	where $(Y,V)$ is the optimal trajectory obtained using $U=U^\opt$ in~\eqref{eq:dual-optimal-control-b}.
\end{remark}

\medskip

In fact, these two implications carry over to the entire trajectory. The proof of the following proposition is based on a dynamic programming argument given in Section~\ref{ssec:pf-optimal-solution}.

\medskip

\begin{proposition}[Dynamic programming]\label{prop:optimal-solution}
	Consider the dual optimal control problem. Suppose
	$U^\opt=\{U_t^\opt:0\le t \le T\}$ is the optimal control input and that $(Y,V)$ is the associated optimal trajectory
	obtained as a solution of the BSDE.  Then for almost every $0\le t \le T$,
	\begin{align}
		\pi_t(Y_t) &= \mu(Y_0) - \int_0^t \big(U_s^\opt \big)^\tp\ud Z_s,\quad \sP\text{-a.s.} \label{eq:estimator-t} \\
		\E\big(\clV_t(Y_t)\big) &= \E\Big(\clV_0(Y_0) + \int_0^t \Gamma Y_s(X_s) + |U_s^\opt + V_s(X_s)|^2 \ud s\Big) \label{eq:estimator-t-variance}
	\end{align}
\end{proposition}

\medskip

Although DP reveals that the expected value of the conditional variance has an interpretation of the value function, we do not yet have a formula for the optimal control.
The difficulty arises because there is no HJB equation for BSDE-constrained optimal control problem. The literature on such problem utilizes the stochastic maximum principle for BSDE, which is the subject of the next section.

Before investigating the solution of the dual optimal control problem, we make several remarks.

\medskip

\begin{remark}
	The duality principle Theorem~\ref{thm:duality-principle} implies that the duality gap
	\[
	\sJ_T(U^\opt) - \E\big(\clV_T(F)\big) \ge 0
	\]
	That the duality gap is zero is on account of the It\^o representation formula which holds because $Z$ is a $\tsP$-B.M.
	It is important to note that the natural condition for the Lemma~\ref{lem:representation} to hold is
	\begin{equation}\label{eq:L2-condition-ito}
		\tE\big(|\pi_T(F)|^2\big) < \infty
	\end{equation}
	Clearly,~\eqref{eq:L2-condition-ito} holds if $F\in L^2_{\clZ_T}(\Omega;C_b(\bS))$. However,~\eqref{eq:L2-condition-ito} is more general, and provided that it holds and a unique solution $(Y,V)$ exists for the BSDE~\eqref{eq:dual-optimal-control-b}, the duality principle also applies.
	
	It is expected that similar dual optimal control construction may also apply to other type of filtering model where representation formula are available, e.g., measurement noise as a jump process~\cite{boel1975martingales}.
\end{remark}

\medskip

\begin{remark}\label{rm:general-form-of-estimator}
	The proof of the Theorem~\ref{thm:duality-principle} is presented in a slightly more general form where the estimator~\eqref{eq:estimator} is expressed as
	\[
	S_T = b-\int_0^TU_t^\tp \ud Z_t
	\]
	where $b\in\Re$ is an arbitrary deterministic constant. Then it is shown that
	\[
	\E\big(|F(X_T)-S_T|^2\big) = \sJ_T(U) + (\mu(Y_0)-b)^2
	\]
	This general form is useful if the measure $\mu$ is not known. This will be useful in Chapter~\ref{ch:filter-stability-2} for filter stability analysis.
\end{remark}

\medskip

\subsection{Linear-Gaussian case}

Recall the linear-Gaussian filtering problem~\eqref{eq:linear-Gaussian-model} introduced in Section~\ref{sec:problem-formulation}.
As discussed in Remark~\ref{rm:duality-reduces-LG}, with a deterministic control $u\in L^2\big([0,T];\Re^m\big)$ the dual BSDE reduces to the deterministic LTI system~\eqref{eq:LTI-ctrl}.
On the space of linear functions, the solution $Y_t(x) = y_t^\tp x$ where $y_t\in\Re^d$ and the carr\'e du champ operator is
\[
(\Gamma Y_t)(x) = y_t^\tp Q y_t
\]
where $Q = \sigma\sigma^\tp$. Because $V_t = 0$, the control cost $|u_t+V_t(X_T)|^2 = |u_t|^2$.
In summary, the  optimal control problem~\eqref{eq:dual-optimal-control}
reduces to the deterministic LQ problem: 
\begin{align*}
	\mathop{\text{Minimize}}_{u\in L^2([0,T];\Re^m)}\!:\quad \sJ(u) &= y_0^\tp \Sigma_0 y_0 + \int_0^{T} y_t^\tp Q y_t + |u_t|^2 \ud t \\
	\text{Subject to}\;\;:\; -\frac{\ud y_t}{\ud t}
	&= A y_t + H u_t,\quad y_T = f 
\end{align*}
The problem was first described in a seminal paper of Kalman and Bucy~\cite{kalman1961}. A review of the same also appears in Section~\ref{ssec:Kalman-filter} of this thesis.
The solution of the optimal control problem yields the optimal control input $u^\opt$, along with the vector $y_0$ that determines the minimum-variance
estimator:
\begin{align*}
	S_T &= \mu(y_0^\tp x) - \int_0^T \big(u_t^\opt \big)^\tp \ud Z_t
	= y_0^\tp m_0 - \int_0^T \big(u_t^\opt \big)^\tp \ud Z_t
\end{align*}
The Kalman filter is obtained by expressing $\{S_t(f) : t\ge 0,\ f\in\Re^d\}$ as the solution to a linear SDE as described in Section~\ref{ssec:Kalman-filter}.

\section{Solution of the dual optimal control problem}\label{sec:standard-form}

The optimal control problem~\eqref{eq:dual-optimal-control} is not a standard form of optimal control problem with BSDE constraints~\cite[Eq. 5.10]{pardoux2014stochastic}. There are two issues:
\begin{itemize}
	\item {\bf The probability space:} The driving martingale of the BSDE~\eqref{eq:dual-optimal-control-b} is $Z$, which is a $\tsP$-B.M. However, the expectation in defining the optimal control objective~\eqref{eq:dual-optimal-control-a} is with respect to the measure $\sP$. 
	\item {\bf The filtration:} The `state' of the optimal control problem $(Y,V)$ is adapted to the filtration $\clZ$. However, the cost function~\eqref{eq:dual-optimal-control-a} also depends upon the non-adapted exogenous process $X$.  
\end{itemize}

The second problem is easily fixed by using the tower property of conditional expectation.
To resolve the first problem, we have two choices:
\begin{enumerate}
	\item Use the change of measure to evaluate $\sJ_T(U)$ with respect to $\tsP$ measure, or
	\item Express the BSDE using a driving martingale that is a $\sP$-B.M. A convenient such process is the innovation process $I$. 
\end{enumerate}
In this section, the standard form of the dual optimal control problem is presented based on the first choice. A discussion on the second choice is described in Section~\ref{sec:innovation-method}.

\medskip

In order to express the expectation for the control objective~\eqref{eq:dual-optimal-control-a} with respect to $\tsP$, we use the change of measure~\eqref{eq:D-t}:
\begin{align*}
	\sJ_T(U) &= \E\Big(|Y_0(X_0)-\mu(Y_0)|^2 + \int_0^T (\Gamma Y_t)(X_t) + |U_t + V_t(X_t)|^2 \ud t\big)\\
	&=\tE\Big( D_0|Y_0(X_0)-\mu(Y_0)|^2 + \int_0^T D_t\big((\Gamma Y_t)(X_t) + |U_t + V_t(X_t)|^2 \big)\ud t\Big)\\
	&=\tE\Big(|Y_0(X_0)-\mu(Y_0)|^2 + \int_0^T \tE\big(D_t(\Gamma Y_t)(X_t) \mid \clZ_t\big)+ \tE\big(D_t|U_t + V_t(X_t)|^2\mid\clZ_t \big)\ud t\Big)\\
	&=\tE\Big(|Y_0(X_0)-\mu(Y_0)|^2 + \int_0^T \sigma_t\big(\Gamma Y_t\big)+ \sigma_t\big(|U_t + V_t|^2 \big)\ud t\Big)\\
	&=\tE\Big(|Y_0(X_0)-\mu(Y_0)|^2 + \int_0^T \ell(Y_t,V_t,U_t;\sigma_t)\ud t\Big)
\end{align*}
where the \emph{Lagrangian} $\ell:C_b(\bS)\times C_b(\bS)^m\times \Re^m \times \clM(\bS) \to \Re$ is defined by
\begin{equation*}
	\ell(y,v,u;\rho) = \rho\big(\Gamma y \big) + \rho\big(|u+v|^2\big)
\end{equation*}

\subsubsection{Dual optimal control problem (standard form)}	
\begin{subequations}\label{eq:dual-optimal-control-std}
	\begin{align}
		\mathop{\text{Minimize}}_{U\in\clU}\text{:}\quad\quad\sJ_T(U) &= \tE\Big(|Y_0(X_0)-\mu(Y_0)|^2 + \int_0^T \ell(Y_t,V_t,U_t;\sigma_t)\ud t
		\Big) \label{eq:dual-optimal-control-std-a}\\
		\text{\rm Subject to:}\; -\ud Y_t(x) &= \big((\clA Y_t)(x) + h^\tp(x)(U_t+V_t(x))\big)\ud t - V_t^\tp(x)\ud Z_t,\quad Y_T(x) = F(x),\;x\in\bS \label{eq:dual-optimal-control-std-b}
	\end{align}
\end{subequations}


%
%

\subsubsection{Solution using the maximum principle}

Define the \emph{Hamiltonian} $\clH: C_b(\bS)\times C_b(\bS)^m\times 
\Re^m\times \clM(\bS)\times \clM(\bS) \to \Re$ as follows:
\[
\clH(y,v,u,p;\rho) = -p\big(\clA y + h^\tp (u + v) \big)- \ell(y,v,u;\rho)
\]
In the following, Hamilton's equations for the optimal trajectory are derived by 
an application of the maximum principle for BSDEs~\cite[Theorem 
4.4]{peng1993backward}. The Hamilton's equations are obtained in terms of the 
derivatives of the Hamiltonian.
In order to take derivatives with respect to functions and measures, we 
adopt the notion of G\^ateaux differentiability. For instance, given a 
functional $F:\clY\to \Re$, the G\^ateaux derivative $F_y(y)\in \clY^\dagger$ is obtained from the defining relation~\cite[Section 
10.1.3]{bensoussan2018estimation}:
\[
\frac{\ud}{\ud \varepsilon} F(y + \varepsilon\tilde{y})\Big|_{\varepsilon=0} = 
\big\langle \tilde{y}, F_y(y)\big\rangle,\quad \forall\,\tilde{y} \in \clY
\]
The partial derivatives of the Hamiltonian are as follows:
\begin{align*}
	\clH_y(y,v,u,p;\rho) &= -\clA^\dagger p - \frac{\partial}{\partial 
		y}\rho\big(\Gamma y\big)\\
	\clH_v(y,v,u,p;\rho) &= -ph - 2(u+v)\rho\\
	\clH_u(y,v,u,p;\rho) &= -p(h) - 2\rho(\ones)u - 2\rho(v)\\
	\clH_p(y,v,u,p;\rho) &= -\clA y - h^\tp(u+v)
\end{align*}
Using this notation for the functional derivatives, the following theorem 
describes the Hamilton's equations. The proof appears in Section~\ref{ss:pf-thm52}.

\medskip

\begin{theorem}\label{thm:optimal-solution}
	Consider the optimal control problem~\eqref{eq:dual-optimal-control-std}.  Suppose
	$U^\opt$ is the optimal control input and the $(Y,V)$ is the associated optimal solution
	obtained by solving BSDE~\eqref{eq:dual-optimal-control-std-b}.
	Then there exists a $\clZ$-adapted measure-valued process $P=\{P_t \in\clM(\bS):0\le t \le T\}$ such
	that
	\begin{subequations}\label{eq:Hamilton_eqns}
		\begin{flalign}
			&\text{(forward)}  & \ud P_t &= -\clH_y(Y_t, V_t,U_t^\opt,P_t;\sigma_t)\ud t  - \clH_v^\tp(Y_t, V_t,U_t^\opt,P_t;\sigma_t) \ud Z_t
			& \label{eq:Hamilton_eqns2-a}\\
			&\text{(backward)}  &\ud Y_t &= \clH_p(Y_t,V_t,U_t^\opt,P_t;\sigma_t) \ud t + V_t\ud Z_t & \label{eq:Hamilton_eqns2-b}\\
			&\text{(boundary)} & \frac{\ud P_0}{\ud \mu}(x) &= 2\big(Y_0(x)-\mu(Y_0)\big),\quad Y_T(x) = F(x)\quad  x\in\bS  \label{eq:Hamilton_eqns2-c}
		\end{flalign}
	\end{subequations}
	where the optimal control is given by
	\begin{equation}\label{eq:opt-cont-soln}
		U_t^\opt = -\half \frac{P_t(h)}{\sigma_t(\ones)} - \pi_t(V_t),\quad \tsP\text{-a.s.},\; 0\le t \le T
	\end{equation}
\end{theorem}

\medskip

\begin{remark}
	From linear optimal control theory, it is known that $P_t$ is related to $Y_t$ by a ($\clZ_t$-measurable) linear transformation~\cite[Section 6.6]{yong1999stochastic}. The boundary condition $\frac{\ud P_0}{\ud \mu}(x) = 2\big(Y_0(x)-\mu(Y_0)\big)$ suggests that
	\begin{equation}\label{eq:P-t-ansatz}
		\frac{\ud P_t}{\ud \sigma_t}(x) = 2\big(Y_t(x) - \pi_t(Y_t)\big),\quad 0\le t \le T
	\end{equation}
	This is indeed the case as we formally verify in Section~\ref{ssec:pf-ansatz} that $P_t$ thus defined solves the Hamilton's equation~\eqref{eq:Hamilton_eqns2-a}. 
	Combining this formula with~\eqref{eq:opt-cont-soln}, we have a formula for optimal control input as a feedback control law:
	\begin{equation*}
		U_t^\opt = -\big(\pi_t(hY_t) - \pi_t(h)\pi_t(Y_t)\big) - \pi_t(V_t),\quad 0\le t \le T
	\end{equation*}
\end{remark}
%
\medskip

\subsubsection{Explicit formula for the finite case}

The dual system is~\eqref{eq:dual-ctrl-finite}. The Lagrangian $\ell:\Re^d\times \Re^{d\times m} \times \Re^m \times \Re^d \to \Re$ is given by:
\[
\ell(y,v,u;\rho) = y^\tp \rho(Q)y + \rho(\ones)|u|^2 + 2u^\tp v\rho +  \rho^\tp\dv^\dagger (vv^\tp)
\]
where 
\begin{equation*}
	Q(i) = \sum_{j \in \bS} A(i,j)(e_i-e_j)(e_i-e_j)^\tp,\quad	\rho(Q) = \sum_{i \in \bS}\rho(i)Q(i)
\end{equation*}
The Hamiltonian $\clH:\Re^d\times \Re^{d\times m} \times \Re^m \times \Re^d\times \Re^d \to \Re$ is given by:
\[
\clH(y,v,u,p;\rho) = -p^\tp Ay - p^\tp Hu - p^\tp \dv^\dagger(Hv^\tp) - l(y,v,u;\rho) 
\]
The partial derivatives of the Hamiltonian are as follows:
\begin{align*}
	\clH_y(y,v,u,p;\rho) &= -A^\tp p - 2\rho(Q)y\\
	\clH_v(y,v,u,p;\rho) &= -\dv(p)H - 2\rho u^\tp - 2\dv(\rho) v\\
	\clH_u(y,v,u,p;\rho) &= -H^\tp p - 2\rho(\ones)u -2v^\tp \rho\\
	\clH_p(y,v,u,p;\rho) &= -A y - Hu - \dv^\dagger(Hv^\tp)
\end{align*}
Therefore, the Hamilton's equations are given by	
\begin{flalign*}
	&\text{(forward)}  & \ud P_t &= \big(A^\tp P_t + 2\sigma_t(Q)Y_t\big) \ud t + \big(\dv(P_t)H + 2\sigma_tU_t^\tp +2\dv(\sigma_t)V_t\big) \ud Z_t
	& \\
	&\text{(backward)}  &\ud Y_t &= -\big(AY_t+HU_t+\dv^\dagger(HV_t^\tp)\big) \ud t + V_t \ud Z_t &\\
	&\text{(boundary)} & P_0 &= 2\Sigma_0 Y_0,\quad Y_T = F \in \Re^d
\end{flalign*}

\subsubsection{Explicit formula for the Euclidean case}

In the Euclidean case, $\rho$ is a probability density (with respect to Lebesgue measure) and the Lagrangian 
\[
\ell(y,v,u;\rho) = \int_{\Re^d} \rho(x)\big(|\sigma^\tp(x) \nabla y(x)|^2 + |u+v(x)|^2\big)\ud x
\]
The Hamiltonian
\[
\clH(y,v,u,p;\rho) = -\int_{\Re^d} p(x)\big(\clA y(x) + h^\tp(x)(u+v(x))\big)\ud x -\ell(y,v,u;\rho)
\]
where the momentum $p\in L^2(\lambda)$ is also a density.
The partial derivatives of the Hamiltonian are evaluated as follows:
\begin{align*}
	\clH_y(y,v,u,p;\rho) &= -\clA^\dagger p + 2\divg\big(\sigma\sigma^\tp(\nabla y) \rho\big)\\
	\clH_v(y,v,u,p;\rho) &= -ph-2(u+v)\rho\\
	\clH_u(y,v,u,p;\rho) &= -p(h) - 2\rho(\ones)u - 2\rho(v)\\
	\clH_p(y,v,u,p;\rho) &= -\clA y - h^\tp(u+v)
\end{align*}
Therefore, the Hamilton's equations are given by	
\begin{flalign*}
	&\text{(forward)}  & \ud P_t(x) &= \big(\clA^\dagger P_t - 2\divg \big(\sigma\sigma^\tp(x) (\nabla Y_t)(x) \sigma_t(x)\big)\big) \ud t + \big(P_t(x)h(x) + 2(U_t+V_t(x))\sigma_t(x)\big) \ud Z_t
	& \\
	&\text{(backward)}  &\ud Y_t(x) &= -\big(\clA Y_t+h^\tp(x)(U_t+V_t(x))\big) \ud t + V_t^\tp(x) \ud Z_t &\\
	&\text{(boundary)} & P_0(x) &= 2\mu(x)\big(Y_0(x)-\mu(Y_0)\big),\quad Y_T(x) = F(x),\quad x\in \Re^d
\end{flalign*}

\section{Martingale characterization of the optimal solution}\label{ssec:martingale}

Although we do not have an HJB equation, a martingale type characterization is possible as described in the following theorem whose proof appears in Section~\ref{ss:pf-thm53}.


\begin{theorem}\label{thm:martingale}
	Fix $U\in L^2_\clZ\big([0,T];\Re^m\big)$. Consider a $\clZ$-adapted process $M = \{M_t\in\Re:0\le t \le T\}$
	\begin{equation*}\label{eq:martingale-characterization}
		M_t := \clV_t(Y_t) - \int_0^t \ell(Y_s,V_s,U_s;\pi_s)\ud s,\quad 0\le t\le T
	\end{equation*}
	where $(Y,V)$ is the solution to the BSDE~\eqref{eq:dual-optimal-control-b}.
	Then $M$ is a $\sP$-supermartingale, and $M$ is a $\sP$-martingale if and only if
	\begin{equation}\label{eq:optimal-solution}
		U_t = -\big(\pi_t(hY_t) - \pi_t(h)\pi_t(Y_t)\big) - \pi_t(V_t),\quad 0\le t \le T
	\end{equation}
\end{theorem} 

\medskip

A direct consequence of the Theorem~\ref{thm:martingale} is the optimality of the control~\eqref{eq:optimal-solution}, because
\[
\E(M_T) \le \E(M_0)
\]
which means
\[
\E\big(\clV_T(F)\big) \le \E\Big(\clV_0(Y_0) + \int_0^T\ell(Y_t,V_t,U_t;\pi_t)\Big) = \sJ_T(U)
\]
with equality if and only if $U$ is given by~\eqref{eq:optimal-solution}.
%
%

\section{Derivation of the nonlinear filter}\label{sec:derivation-of-the-filter}

Using the formula~\eqref{eq:optimal-solution} for the optimal control in~\eqref{eq:estimator-t},
\begin{equation}\label{eq:estimator-with-optimal-ctrl}
	\pi_t(Y_t) = \mu(Y_0) + \int_0^t \big(\pi_t(hY_s) - \pi_s(h)\pi_s(Y_s)+\pi_s(V_s)\big)^\tp\ud Z_s,\quad 0\le t \le T,\;\; \sP\text{-a.s.}
\end{equation}

Using the fact that this equation holds for arbitrary choice of $F$ and $T$, the nonlinear filter is derived. The proof appears in Section~\ref{ssec:pf-nonlinear-filter-derivation}.

\medskip
\begin{theorem}\label{thm:derivation-nonlinear-filter}
	Consider the optimal estimator~\eqref{eq:estimator-with-optimal-ctrl} 
	where $(Y,V)$ is the optimal trajectory. Then for any $f\in C_b(\bS)$,
	\[
	\ud \pi_t(f) = \pi_t(\clA f) \ud t + \big(\pi_t(hf)-\pi_t(h)\pi_t(f)\big)\big(\ud Z_t- \pi_t(h) \ud t\big),\quad \pi_0(f) = \mu(f)
	\]
	
\end{theorem}

\section{Innovation based approach}\label{sec:innovation-method}

In this section, we consider the innovation method to formulate the dual optimal control problem in a standard form.

The innovation process $I$ defined in~\eqref{eq:innovation-def} is a $\sP$-B.M., and therefore we can use $I$ instead of $Z$ as a driving martingale. The function space of admissible control now is
\[
\clU = L^2_\clI\big(\Omega \times [0,T];\Re^m\big)
\]
The BSDE is modified as
\begin{equation}\label{eq:dual-bsde-innov}
	-\ud Y_t(x) = \big(\clA Y_t(x) + (h(x) - \pi_t(h))^\tp(U_t+V_t(x))\big) \ud t - V_t^\tp(x)\ud I_t,\quad Y_T(x) = F(x)
\end{equation}
where $F\in L^2_{\clI_T}\big(\Omega;C_b(\bS)\big)$ and the solution $(Y,V)\in L^2_{\clI}\big(\Omega\times[0,T];C_b(\bS)\times C_b(\bS)^m\big)$.

\medskip

\begin{remark}
	One can transform~\eqref{eq:dual-bsde} to use $I$ as a driving martingale in more direct way, by substituting
	\[
	\ud Z_t = \ud I_t + \pi_t(h)\ud t
	\]
	The BSDE~\eqref{eq:dual-bsde-innov} is slightly different because of the presence of an extra $\pi_t(h)^\tp U_t \ud t$ term.
	This extra term is necessary to ensure the duality principle holds. 
\end{remark}

\medskip

The dual optimal control problem is as follows:

\subsubsection{Dual optimal control problem (innovation method)}
\begin{subequations}\label{eq:dual-optimal-control-innov}
	\begin{align}
		\text{\rm Minimize:}\quad\quad\sJ_T(U) &= \E\Big(|Y_0(X_0)-\mu(Y_0)|^2 + \int_0^T \ell(Y_t,V_t,U_t;\pi_t)\ud t
		\Big) \label{eq:dual-optimal-control-innov-a}\\
		\text{\rm Subject to:}\; 		-\ud Y_t(x) &= \big(\clA Y_t(x) + (h(x) - \pi_t(h))^\tp(U_t+V_t(x))\big) \ud t - V_t^\tp(x)\ud I_t,\quad Y_T(x) = F(x) \label{eq:dual-optimal-control-innov-b}
	\end{align}
\end{subequations}

Note that~\eqref{eq:dual-optimal-control-innov-a} is the identical to~\eqref{eq:dual-optimal-control-a} by tower property of the conditional expectation.  The following theorem describes the duality principle using innovation process as a driving martingale.

\begin{theorem}[Duality principle using innovation]\label{thm:duality-principle-innov}
	For any admissible control $U\in \clU$, define the estimator
	\begin{equation*}\label{eq:estimator-innov}
		S_T = \mu(Y_0) - \int_0^T U_t^\tp \ud I_t
	\end{equation*}
	Then 
	\begin{equation*}\label{eq:duality-principle-innov}
		\sJ_T(U) = \E\big(|F(X_T)-S_T|^2\big)
	\end{equation*}
	
\end{theorem}

\subsubsection{Maximum principle}
Define the \emph{Hamiltonian} $\clH: C_b(\bS)\times C_b(\bS)^m\times 
\Re^m\times \clM(\bS)\times \clM(\bS) \to \Re$
\[
\clH(y,v,u,p;\rho) = -p\big(\clA y +(h-\rho(h))^\tp (u + v)\big) - \ell(y,v,u,\rho)
\]

\begin{theorem}\label{thm:optimal-solution-innov}
	Consider the optimal control problem~\eqref{eq:dual-optimal-control-innov}.  Suppose
	$U=\{U_t:0\le t \le T\}$ is the optimal control input and the $(Y,V) =
	\{(Y_t,V_t): 0\le t \le T\}$ is the associated optimal solution
	obtained by solving BSDE~\eqref{eq:dual-optimal-control-innov-b}.
	Then there exists a $\clZ$-adapted measure-valued process $P=\{P_t:0\le t \le T\}$ such
	that
		\begin{flalign*}
			&\text{(forward)}  & \ud P_t &= -\clH_y(Y_t, V_t,U_t,P_t;\pi_t)\ud t  - \clH_v^\tp(Y_t, V_t,U_t,P_t;\pi_t) \ud I_t
			& \\
			&\text{(backward)}  &\ud Y_t &= \clH_p(Y_t,V_t,U_t,P_t;\pi_t) \ud t + V_t\ud I_t & \\
			&\text{(boundary)} & \frac{\ud P_0}{\ud \mu} &= 2\big(Y_0(x)-\mu(Y_0)\big),\quad Y_T(x) = f(x)\quad \forall \, x\in\Re^d 
		\end{flalign*}
	where the optimal control is given by
	\begin{equation*}\label{eq:opt-cont-soln-innov}
		U_t^\opt = -\half P_t\big(h-\pi_t(h)\big)  - \pi_t(V_t)
	\end{equation*}
\end{theorem}

%
	%
\medskip

\begin{remark}
	The co-state process $P_t$ now becomes (cf.~\eqref{eq:P-t-ansatz})
	\[
	\frac{\ud P_t}{\ud \pi_t}(x) = 2\big(Y_t(x)-\pi_t(Y_t)\big),\quad 0\le t \le T
	\]
	and therefore the optimal control is the same as~\eqref{eq:optimal-solution}.
	
\end{remark}

\section{Proofs of the statements}

\subsection{Proof of Theorem~\ref{thm:duality-principle}}\label{ss:pf-thm51}

We provide the proof for a slightly more general estimator of the form
\begin{equation}\label{eq:estimator-general}
	S_T = b- \int_0^T U_t^\tp \ud Z_t
\end{equation}
where $b\in \Re$ is a deterministic constant.

Recall the martingale~\eqref{eq:martingale-generator} associated with the infinitesimal generator $\clA$. It is defined by
\begin{equation*}
	N_t(g) = g(X_t) - \int_0^t \clA g(X_s)\ud s
\end{equation*}
Apply It\^o-Wentzell theorem~\cite[Theorem 1.17]{rozovskiui2018stochastic} on $Y_t(X_t)$ and we obtain
\begin{align*}
	\ud Y_t(X_t) &= -U_t^\tp h(X_t) \ud t + V_t^\tp(X_t) \big(\ud Z_t-h(X_t)\ud t\big) + \ud N_t(Y_t)\\
	&= - U_t^\tp \ud Z_t + \big(U_t + V_t(X_t)\big) \ud W_t + \ud N_t(Y_t)
\end{align*}
Integrating both sides from $0$ to $T$,
\[
F(X_T) = Y_0(X_0) - \int_0^T U_t^\tp h(X_t)\ud t + \int_0^T V_t^\tp(X_t) \ud W_t + \int_0^T \ud N_t(Y_t)
\]
Then
\[
F(X_T) - \Big(b- \int_0^T U_t^\tp \ud Z_t\Big) = \big(Y_0(X_0)-b\big) + \int_0^T (U_t+V_t(X_t))^\tp \ud W_t + \int_0^T \ud N_t(Y_t)
\]
The left-hand side is the error of the estimator~\eqref{eq:estimator-general}. The three terms on the right-hand side are mutually independent. Therefore, squaring and taking expectation:
\[
\E\big(|F(X_T)-S_T|^2\big) = \E\big(|Y_0(X_0)-\mu(Y_0)|^2\big) + (\mu(Y_0)-b)^2+ \E\Big(\int_0^T |U_t+V_t(X_t)|^2 + (\Gamma Y_t)(X_t) \ud t \Big)
\]
The proof closes by setting $b = \mu(Y_0)$.
\qed

\subsection{Proof of Lemma~\ref{lem:representation}}\label{ss:pf-lem51}

Note that
\[
|\pi_T(F)|^2 \le \|F\|_\infty^2,\quad \tsP\text{-a.s.}
\]
If $F\in L^2_{\clZ_T}(\Omega;C_b(\bS))$ then $\pi_T(F)\in L^2_{\clZ_T}(\Omega;\Re)$. Therefore the conclusion follows from the Brownian motion representation theorem~\cite[Theorem 5.18]{le2016brownian}.\qed

\subsection{Proof of Proposition~\ref{prop:optimal-solution}}\label{ssec:pf-optimal-solution}

Fix $t \in [0,T]$ and let
\[
S_t = \mu(Y_0) - \int_0^t \big(U_s^\opt\big)^\tp \ud Z_s
\]
Then use the same procedure as the proof of Theorem~\ref{thm:duality-principle},
\[
\E\big(|Y_t(X_t)-S_t|^2\big) = \E\Big(|Y_0(X_0)-\mu(Y_0)|^2 + \int_0^t (\Gamma Y_s)(X_s) + |U_s^\opt+V_s(X_s)|^2 \ud s\Big)
\]
and therefore,
\[
\sJ_T(U^\opt) = \E\big(|Y_t(X_t)-S_t|^2\big) + \E\Big(\int_t^T 
(\Gamma Y_s)(X_s) + |U_s^\opt+V_s(X_s)|^2 \ud s\Big)
\]
Suppose~\eqref{eq:estimator-t-variance} is not true, namely,
\[
\E \big(|Y_t(X_t)-\pi_t(Y_t)|^2\big)< \E \big(|Y_t(X_t)-S_t|^2\big) 
\]
By the Lemma~\ref{lem:representation}, there exists $\hat{U}\in L_{\clZ}^2(\Omega\times [0,t];\Re^m)$ such that
\[
\pi_t(Y_t) = \tE(\pi_t(Y_t)) - \int_0^t \hat{U}_s^\tp \ud Z_s,\quad \tsP\text{-a.s.}
\]
Consider an admissible control $\tilde{U}$ defined by
\[
\tilde{U}_s = \begin{cases}
	\hat{U}_s \quad s \le t\\
	U_s^\opt \quad s > t
\end{cases}
\]
and let $(\tilde{Y},\tilde{V})$ be the solution to the control $\tilde{U}$.  By the uniqueness of the solution to BSDE, $\tilde{Y}_s = Y_s$ for all $s \ge t$ and therefore
\begin{align*}
	\sJ_T(\tilde{U}) &= \E \big(|Y_t(X_t)-\pi_t(Y_t)|^2\big) + \E\Big(\int_t^T \Gamma Y_s(X_s) + |U_s^\opt+V_s(X_s)|^2 \ud s\Big)\\
	&< \E \big(|Y_t(X_t)-S_t|^2\big)  + \E\Big(\int_t^T \Gamma Y_s(X_s) + |U_s^\opt+V_s(X_s)|^2 \ud s\Big) = \sJ_T(U^\opt)
\end{align*}
This violates the optimality of $U^\opt$, and therefore we have
\[
\E \big(|Y_t(X_t)-\pi_t(Y_t)|^2\big) = \E \big(|Y_t(X_t)-S_t|^2\big) 
\]
and~\eqref{eq:estimator-t} follows because the projection is unique.
\qed

\subsection{Proof of Theorem~\ref{thm:optimal-solution}}\label{ss:pf-thm52}

Equation~\eqref{eq:Hamilton_eqns} is 
the Hamilton's equation for optimal control of a BSDE~\cite[Theorem 4.4]{peng1993backward} (see also Appendix~\ref{sec:bsde-maximum-principle}).

The optimal control is obtained from the maximum principle:
\[
U_t = \mathop{\operatorname{argmax}}_{u\in\Re^m} \; \clH(Y_t,V_t,u,P_t;\sigma_t)
\]
Since $\clH$ is quadratic in the control input, the explicit
formula~\eqref{eq:opt-cont-soln} is obtained by evaluating the derivative and setting it to zero:
\[
\clH_u(Y_t,V_t,u,P_t;\sigma_t) = 2\sigma_t(\ones)u + 2\sigma_t(V_t) + P_t(h) = 0
\]

\subsection{Justification of the ansatz for $P_t$}\label{ssec:pf-ansatz}

For an arbitrary test function $f$, the ansatz~\eqref{eq:P-t-ansatz} is represented as
\[
\langle f, P_t\rangle = \big\langle 2f(Y_t-\pi_t(Y_t)), \sigma_t \big\rangle
\]
where $\langle \cdot,\cdot\rangle$ is the duality paring between functions and measures. 
From Prop.~\ref{prop:optimal-solution}, $\ud \big(\pi_t(Y_t)\big) = - U_t \ud Z_t$ with $U_t = U_t^\opt$, and therefore using the It\^o product formula on the right-hand side:
\begin{align*}
	\ud \langle f, P_t\rangle &= -2\big\langle f\big(\clA Y_t + h^\tp(U_t+V_t)\big),\sigma_t \big\rangle \ud t + 2\langle f(U_t+V_t),\sigma_t\rangle \ud Z_t\\
	&\quad + 2\langle f(Y_t-\pi_t(Y_t)),\clA^\dagger \sigma_t \rangle \ud t +  2\langle fh(Y_t-\pi_t(Y_t)),\sigma_t \rangle \ud Z_t\\
	&\quad +2\langle fh^\tp(U_t+V_t),\sigma_t \rangle \ud t\\
	&=2\big\langle \clA(fY_t)-f(\clA Y_t) - \pi_t(Y_t)(\clA f),\sigma_t\big\rangle \ud t + \Big[ \langle 2f(U_t+V_t),\sigma_t \rangle+ \langle fh,P_t\rangle\Big] \ud Z_t
\end{align*}
This agrees with Hamilton's equations because
\[
\langle \clA f, P_t \rangle = 2\langle Y_t\clA f - \pi_t(Y_t)\clA f,\sigma_t \rangle 
\]
and
\[
\frac{\ud}{\ud \epsilon}\Gamma (Y_t+\epsilon f)\Big|_{\epsilon = 0} = 2\big(\clA(Y_tf)-Y_t(\clA f) - f(\clA Y_t)\big)
\]
Therefore,
\begin{equation*}
	\ud \langle f, P_t\rangle = \Big[\langle \clA f, P_t \rangle + \frac{\ud}{\ud \epsilon}\sigma_t\big(\Gamma (Y_t+\epsilon f)\big)\Big|_{\epsilon = 0} \Big]\ud t + \Big[\langle fh, P_t \rangle + \langle 2f(U_t + V_t), \sigma_t \rangle\Big]\ud Z_t
\end{equation*}
\qed
%
%

\subsection{Proof of Theorem~\ref{thm:martingale}}\label{ss:pf-thm53}

We evaluate the derivative of $\clV_t(Y_t) = \pi_t(Y_t^2) - \big(\pi_t(Y_t)\big)^2$.
\begin{align*}
	\ud \pi_t(Y_t^2) =& \pi_t(\clA Y_t^2) \ud t + \big(\pi_t(hY_t^2)-\pi_t(h)\pi_t(Y_t^2)\big)\ud I_t +\pi_t\big(-2Y_t\big(\clA Y_t + h (U_t+V_t)\big) + |V_t|^2\big)\ud t \\
	&+ 2\pi_t\big(Y_tV_t\big) \ud Z_t +2\big(\pi_t(h Y_tV_t)-\pi_t(h)\pi_t(Y_tV_t)\big)\ud t\\
	=&\pi_t\big(\Gamma Y_t\big) \ud t + \pi_t(|V_t|^2)\ud t - 2\pi_t(hY_t)U_t \ud t +\big(\pi_t(hY_t^2)-\pi_t(h)\pi_t(Y_t^2) + 2\pi_t(Y_tV_t)\big)\ud I_t
\end{align*}
Similarly, 
\begin{align}
	\ud \pi_t(Y_t) =& \pi_t(\clA Y_t) \ud t + \big(\pi_t(hY_t)-\pi_t(h)\pi_t(Y_t)\big)\big(\ud Z_t - \pi_t(h)\ud t\big) \nonumber\\
	&-\pi_t\big(\clA Y_t + h(U_t+V_t)\big)\ud t + \pi_t\big(V_t\big) \ud Z_t \nonumber\\
	&+\big(\pi_t(h V_t)-\pi_t(h) \pi_t(V_t)\big)\ud t \nonumber\\
	=& \big(\pi_t(hY_t)-\pi_t(h)\pi_t(Y_t)+ \pi_t(V_t)\big)\ud Z_t \nonumber\\
	&- \big(U_t+\pi_t(hY_t)-\pi_t(h)\pi_t(Y_t)+ \pi_t(V_t)\big) \pi_t(h)\ud t \nonumber\\
	=& U_t^\opt \ud Z_t - (U_t-U_t^\opt)\pi_t(h)\ud t \label{eq:pi-t-y-t}
\end{align}
where $U_t^\opt := -\pi_t(hY_t)+\pi_t(h)\pi_t(Y_t) - \pi_t(V_t)$. 
Therefore,
\begin{align*}
	\ud \big(\pi_t(Y_t)\big)^2 =& 2\pi_t(Y_t)U_t^\opt \ud Z_t - 2 \pi_t(Y_t)(U_t-U_t^\opt)\pi_t(h)\ud t + |U_t^\opt|^2 \ud t
\end{align*}
Collecting terms, we have
\begin{align*}
	\ud M_t =& \pi_t\big(\Gamma Y_t\big) \ud t + \pi_t(|V_t|^2)\ud t - 2\pi_t(hY_t)U_t \ud t +\big(\pi_t(hY_t^2)-\pi_t(h)\pi_t(Y_t^2) + 2\pi_t(Y_tV_t)\big)\ud I_t\\
	&-2\pi_t(Y_t)U_t^\opt \ud Z_t + 2 \pi_t(Y_t)\big(U_t-U_t^\opt\big)\pi_t(h)\ud t - |U_t^\opt|^2 \ud t - \ell(Y_t,V_t,U_t;\pi_t)\\
	&=-|U_t-U_t^\opt|^2 \ud t + \big(\pi_t(hY_t^2)-\pi_t(h)\pi_t(Y_t^2) + 2\pi_t(Y_tV_t)\big)\ud I_t
\end{align*}
Since $|U_t - U_t^\opt|^2 \ge 0$ and $I$ is a $\sP$-martingale, $M$ is a $\sP$-supermartingale. It is a martingale if and only if $U_t = U_t^\opt$ for all $t$. 
\qed

\subsection{Proof of Theorem~\ref{thm:derivation-nonlinear-filter}}\label{ssec:pf-nonlinear-filter-derivation}

Substituting the optimal solution~\eqref{eq:optimal-solution} into the 
Prop.~\ref{prop:optimal-solution} yields:
\begin{equation}\label{eq:nonlinear-filter-using-control}
	\pi_t(Y_t) = \mu(Y_0) + \int_0^t \pi_s(hY_s)-\pi_s(h)\pi_s(Y_s)+\pi_s(V_s)\ud Z_s 
\end{equation}
Use an ansatz $\ud \pi_t(f) = \alpha(f) \ud t + \beta(f) \ud Z_t$ and 
differentiate both sides of~\eqref{eq:nonlinear-filter-using-control} to obtain
\begin{align*}
	\ud \big(\pi_t (Y_t)\big) &=
	\alpha(Y_t) \ud t + \beta(Y_t)\ud Z_t - \pi_t\big(\clA Y_t + h^\tp (U_t+V_t)\big) \ud t + \pi_t(V_t)\ud Z_t + \beta(V_t)\ud t \\
	&= \big(\pi_t(hY_t)-\pi_t(h)\pi_t(Y_t)+\pi_t(V_t)\big)\ud Z_t
\end{align*}
Collecting martingale terms, we have
\[
\beta(Y_t)\ud Z_t = \big(\pi_t(hY_t)-\pi_t(h)\pi_t(Y_t)\big)\ud Z_t
\]
Since $Y$ is arbitrary, we set
\[
\beta(f) = \pi_t(hf) - \pi_t(h)\pi_t(f)
\]
Now collect the finite variation terms
\begin{align*}
	\alpha(Y_t) &= \pi_t(\clA Y_t) - \pi_t(h)\big(\pi_t(hY_t)-\pi_t(h)\pi_t(Y_t)+\pi_t(V_t)\big) + \pi_t(hV_t) -  \pi_t(hV_t)+\pi_t(h)\pi_t(V_t)\\
	&=\pi_t(\clA Y_t) - \pi_t(h)\big(\pi_t(hY_t)-\pi_t(h)\pi_t(Y_t)\big)
\end{align*}
Therefore,
\[
\alpha(f) = \pi_t(\clA f) - \beta(f)\pi_t(h)
\]
and
\begin{align*}
	\ud \pi_t(f) 
	&= \big(\pi_t(\clA f) - \beta(f)\pi_t(h)\big)\ud t + \beta(f)\ud Z_t\\
	&=\pi_t(\clA f) \ud t + \big(\pi_t(hf) - \pi_t(h)\pi_t(f)\big)(\ud Z_t - \pi_t(h)\ud t)
\end{align*}
This is the Kushner-Stratonovich equation of the nonlinear filter.
\qed

\newpage


\chapter{Filter stability in literature}\label{ch:filter-stability-literature}


In this chapter, we present a review of filter stability (asymptotic forgetting of initial condition). Let $\pi^\mu = \{\pi_t^\mu:0\le t \le T\}$ and $\pi^\nu = \{\pi_t^\nu:0\le t \le T\}$ denote the nonlinear filter (solution of~\eqref{eq:nonlinear-filter}) initialized from prior $\mu$ and $\nu$, respectively.
The filter is said to be \emph{asymptotically stable} if~\cite[Definition 10.1]{xiong2008introduction}
\[
d\big(\pi_T^\mu,\pi_T^\nu\big) \; \longrightarrow \; 0
\]
for a suitable metric $d:\clP(\bS)\times \clP(\bS)\to \Re^+$ and a suitable notion of convergence (e.g., almost sure, $L^2$, etc.).
In this chapter and the next, we will provide additional details on the choice of metric and the relationship between various choices.

The stability analysis of the Kalman-Bucy filter is classical and appears in the original paper by Kalman and Bucy~\cite[Theorem 4]{kalman1961}.
A pioneering early contribution is the paper by Ocone and Pardoux~\cite{ocone1996asymptotic} which relied on certain earlier results of~\cite{kunita1971asymptotic} (which were later found
to contain a gap, as discussed in some detail in Section~\ref{sec:counter-example}).

There are two main cases in stability of the nonlinear filter: (1) The case where the state process forgets the initial measure and therefore the filter ``inherits'' the same property; (2) The case where the observation provides sufficient information about the hidden state, allowing the filter to correct its erroneous initialization. These two cases are referred to as the ergodic and non-ergodic signal cases, respectively.
Early work on the ergodic signal case is based on contraction analysis of the random matrix products arising from recursive application of the Bayes' formula~\cite{atar1997lyapunov} (see also~\cite[Ch. 4.3]{Moulines2006inference}).
For the model $(\clA,h)$, the counterpart in the analysis of the Zakai equation leads to useful formulae for the Lyapunov exponents~\cite{atar1997exponential,Atar1999}.
For non-ergodic signal case, a notable early contribution is~\cite{clark1999relative} where formulae for the relative entropy are derived and it is shown that the relative entropy is a Lyapunov function for the filter. 
Fundamental definitions of observability and detectability leading to useful filter stability conclusions first appears in~\cite{van2009observability,van2009uniform}.  
A comprehensive survey on filter stability appears in review papers~\cite{chigansky2006stability,chigansky2009intrinsic}. 

The outline of the remainder of this chapter is as follows: Ocone and Pardoux's paper is reviewed in Section~\ref{sec:61}.
In Section~\ref{sec:counter-example}, the famous counter-example of filtering theory is described. 
In Section~\ref{sec:Lyapunov-analysis}, Lyapunov exponent analysis of the Zakai equation is reviewed. In Section~\ref{sec:literature-clark}, formulas for relative entropy are described.
The so called intrinsic approach to filter stability analysis is described in Section~\ref{sec:65}.
Finally, duality-based analysis of filter stability, utilizing Mitter-Newton duality, is described in Section~\ref{sec:66}.

\section{Basic paper on the subject (Ocone-Pardoux 1996)}\label{sec:61}

Although there had been prior work from 1970s on ergodicity of the nonlinear filter~\cite{kunita1971asymptotic}, Ocone and Pardoux's 1996 paper~\cite{ocone1996asymptotic} is widely cited as a pioneering contribution on asymptotic stability of the nonlinear filter.
The paper is in two parts: the first part is on the stability of the Kalman-Bucy filter and the second part is on the stability of the nonlinear filter. For the second part of the paper, the authors assume that the Markov process $X$ is ergodic. (This is referred to as the ergodic signal case). In the following two subsections, we provide a summary of each of these two parts.


\subsection{Stability of the Kalman-Bucy filter}\label{ssec:Kalman-filter-stability}

Consider the linear Gaussian filtering model~\eqref{eq:linear-Gaussian-model} in Section~\ref{sec:problem-formulation}. For this model, the filter $\pi_t$ is Gaussian whose conditional mean and variance are denoted by $m_t$ and $\Sigma_t$, respectively. These are obtained as the solution to the Kalman-Bucy filter:
\begin{subequations}\label{eq:kalman-filter-ch6}
	\begin{align}
		\ud m_t &= A^\tp m_t + \Sigma_t H\big(\ud Z_t - H^\tp m_t \ud t\big) \label{eq:kalman-filter-ch6-a}\\
		\frac{\ud \Sigma_t }{\ud t} &= A^\tp \Sigma_t + \Sigma_t A + Q - \Sigma_t HH^\tp \Sigma_t \label{eq:kalman-filter-ch6-b}
	\end{align}
\end{subequations}
where these equations are initialized from the mean and variance $(m_0,\Sigma_0)$ of the (correct) Gaussian prior. With a different (incorrect) initialization $(\tilde{m}_0,\tilde{\Sigma}_0)$, the solution of~\eqref{eq:kalman-filter-ch6} is denoted by $\big\{(\tilde{m}_t,\tilde{\Sigma}_t): t\ge 0\big\}$.

It is noted that $\{\Sigma_t: t\ge 0\}$ and $\{\tilde{\Sigma}_t: t\ge 0\}$ are deterministic processes. These are solution of the dynamic Riccati equation (DRE).
The following assumption is crucial to the stability of the Kalman filter.

\begin{assumption}\label{as:R}
	There exists a positive semi-definite matrix $\Sigma_\infty$ such that: 
	\begin{enumerate}
		\item $\Sigma_\infty$ solves the algebraic Riccati equation (ARE)
		\[
		A^\tp \Sigma_\infty + \Sigma_\infty A + Q - \Sigma_\infty HH^\tp 
		\Sigma_\infty = 0
		\]
		\item $A^\tp-\Sigma_\infty HH^\tp$ is Hurwitz.
		\item The solution to the DRE~\eqref{eq:kalman-filter-ch6-b} $\Sigma_t\to \Sigma_\infty$ exponentially fast for any initial condition $\Sigma_0$.
	\end{enumerate}
\end{assumption}

\medskip

\begin{lemma}[Lemma 2.2 in~\cite{ocone1996asymptotic}] \label{lm:KF-sufficiency} 
	Suppose $(A,H)$ is stabilizable and $(A,\sigma^\tp)$ is detectable. Then Assumption~\ref{as:R} holds.	
\end{lemma}

\medskip

The minimum variance optimal control problem (see Section~\ref{ssec:Kalman-filter}) is useful to prove Lemma~\ref{lm:KF-sufficiency}. We provide a sketch of a proof below under a stronger condition that $(A, H)$ is controllable and $(A,\sigma^\tp)$ is observable. For a complete proof, see~\cite[Section 9.4]{xiong2008introduction} or~\cite[Theorem 4.11]{kwakernaak1969linear}.

\begin{proof}
	As described in Section~\ref{ssec:Kalman-filter}, the DRE~\eqref{eq:kalman-filter-ch6-b} is the optimality equation for the minimum variance optimal control problem~\eqref{eq:LG-optimal-control} repeated below:
	\begin{align*}
		\min_{u\in\clU}\qquad \bsJ_T(u) &= y_0^\tp \Sigma_0 y_0 + \int_0^T |u_t|^2 + y_t^\tp Q y_t \ud t\\
		\text{s.t.}\quad\;\; -\frac{\ud y_t}{\ud t} &= A y_t + H u_t,\quad y_T = f
	\end{align*}
	In particular, the value function $f^\tp \Sigma_T f = \min_{u\in\clU}\bsJ_T(u)$.
	The proof is obtained in the following steps:
	
	\medskip
	
	{\bf Step 1. Controllability of $(A,H)$} is used to show existence of $\Sigma_\infty$. For this purpose, first set $\Sigma_0 = 0$. Then $\{f^\tp \Sigma_T f: T\ge 0\}$ is a non-decreasing sequence and moreover bounded from above because $(A,H)$ is controllable.
	Therefore, because $f\in\Re^d$ is arbitrary, there exists $\Sigma_\infty$ s.t.~$\Sigma_T \to \Sigma_\infty$ as $T \to \infty$.
	Since $\{\Sigma_t:t\ge 0\}$ is the solution of the DRE~\eqref{eq:kalman-filter-ch6-b}, the limit $\Sigma_\infty$ solves the ARE.
	
	\medskip
	
	{\bf Step 2. Observaility of $(A,\sigma^\tp)$} is used to show $\Sigma_\infty \succ 0$. Suppose $f^\tp \Sigma_\infty f = 0$ then
	\[
	0 = \int_0^T |\sigma^\tp y_t|^2 + |u_t|^2 \ud t 
	\]
	which imples $f = 0$ if $(A,\sigma^\tp)$ is observable. Therefore, $\Sigma_\infty \succ 0$.
	
	\medskip
	
	{\bf Step 3. Asymptotic stability.} 
	To show that $(A-HH^\tp \Sigma_\infty)$ is Hurwitz, we follow the proof in~\cite[Theorem 23.2]{brockett2015finite}. First set $\Sigma_0 = \Sigma_\infty$. Then $\Sigma_t = \Sigma_\infty$ for all $t \in [0,T]$ and the optimal control is given by $u_t = -H^\tp \Sigma_\infty y_t$. Let $\tilde{y}_t = y_{T-t}$ then the closed-loop system
	\[
	\frac{\ud \tilde{y}_t}{\ud t} = \big(A-HH^\tp\Sigma_\infty\big) \tilde{y}_t,\quad \tilde{y}_0 = f
	\]
	Since $f^\tp \Sigma_\infty f$ is the optimal value,
	\[
	\int_0^T \tilde{y}_t^\tp Q \tilde{y}_t +|H^\tp \Sigma_\infty \tilde{y}_t|^2\ud t + \tilde{y}_T^\tp \Sigma_\infty \tilde{y}_T = f^\tp \Sigma_\infty f,\quad \forall\, T>0
	\]
	Therefore,
	\[
	\tilde{y}_T^\tp\Sigma_\infty \tilde{y}_T - \tilde{y}_{T+1}^\tp \Sigma_\infty\tilde{y}_{T+1} = \int_T^{T+1} \tilde{y}_t^\tp Q \tilde{y}_t +|H^\tp \Sigma_\infty \tilde{y}_t|^2\ud t
	\]
	Minimize this quantity over $\tilde{y}_T =f \in \Re^d$ with $|f| = 1$. Since the set is compact in $\Re^d$, the minimum value is attained. Suppose the minimum value is zero. Then by choosing $\tilde{y}_T=f^*$, the minimizer, the right-hand side
	\[
	\int_T^{T+1} \tilde{y}_t^\tp Q \tilde{y}_t +|H^\tp \Sigma_\infty \tilde{y}_t|^2\ud t =0
	\]
	which implies $H^\tp \Sigma_\infty \tilde{y}_t = 0$ for all $t \in [T,T+1]$.
	Therefore, the dynamics becomes $\frac{\ud \tilde{y}_t}{\ud t} = A\tilde{y}_t$ for $T\le t\le T+1$. However, then by observability assumption, we have $f=0$ which contradicts $|f|=1$. Therefore, the minimum value must be strictly positive. Namely, there exists some constant $c>0$,
	\[
	\tilde{y}_T^\tp\Sigma_\infty \tilde{y}_T - \tilde{y}_{T+1}^\tp \Sigma_\infty\tilde{y}_{T+1} \ge c\, |\tilde{y}_T|^2
	\] 
	Hence
	\[
	\sum_{n=0}^\infty c\, |\tilde{y}_n|^2 \le \int_0^\infty \tilde{y}_t^\tp Q \tilde{y}_t +|H^\tp \Sigma_\infty \tilde{y}_t|^2\ud t \le f^\tp \Sigma_\infty f < \infty 
	\]
	This implies that $|\tilde{y}_n| \to 0$. Consequently, $(A-HH^\tp\Sigma_\infty)$ is Hurwitz.
	
	\medskip
	
	{\bf Step 4. Convergence of $\Sigma_t$.} The exponential convergence $\Sigma_t\to\Sigma_\infty$ for arbitrary $\Sigma_0$ can be shown through either a direct argument (e.g.,~\cite{park1997Riccati}) or by showing that the stationary control law is asymptotically optimal (see~\cite[Theorem 9.33]{xiong2008introduction}).
	
\end{proof}

\medskip

It remains to study convergence of $m_t$ and $\tilde{m}_t$. In contrast to the covariance matrix, these are stochastic processes with a common forcing term $Z$.
Define the error process $e_t:= m_t-\tilde{m}_t$ for $t\ge 0$. Then
\[
\ud e_t = (A^\tp - \Sigma_\infty H H^\tp)e_t\ud t + (\Sigma_\infty - \Sigma_t)H H^\tp m_t \ud t + (\tilde{\Sigma}_t-\Sigma_\infty)H H^\tp\tilde{m}_t\ud t + \big(\Sigma_t-\tilde{\Sigma}_t\big)H\ud Z_t
\]
with initial condition $e_0 = m_0-\tilde{m}_0$.
Suppose $A^\tp - \Sigma_\infty H H^\tp$ is Hurwitz and define the constant
\[
\overline{\lambda} = \min\{-\operatorname{Re}(\lambda):\lambda\text{ is an eigenvalue of } A^\tp-\Sigma_\infty HH^\tp\}
\]
Then the following result can be easily shown by using the Burkholder-Davis-Gundy inequality~\cite[Theorem 3.12]{xiong2008introduction}.

\medskip 

\begin{proposition}[Theorem 2.3 in~\cite{ocone1996asymptotic}] 
	Suppose Assumption~\ref{as:R} holds. Then for any $0< \lambda < \overline{\lambda}$,
	\[
	\lim_{T\to \infty} (m_t-\tilde{m}_t)e^{\lambda T} = 0,\quad \sP^\mu\text{-a.s.}
	\]
\end{proposition}

\medskip

\begin{remark}\label{rm:stabilizability-and-kalman}
	It turns out that stabilizability of $(A^\tp,\sigma)$ is not necessary for asymptotic stability of the Kalman filter. In particular,
	it is proved in~\cite[Proposition 8]{van2009observability} that the Kalman filter is stable if $(A^\tp,H^\tp)$ is detectable, $\Sigma_0 \succ 0$, and $\tilde{\Sigma}_0\succ 0$,
	(without the assuming stabilizability of $(A^\tp, \sigma)$).  In this case, $(A^\tp -\Sigma_\infty HH^\tp)$ is not guaranteed to be Hurwitz, but the filter stability is still valid for a slightly more restricted class of initial conditions. 
\end{remark}

\subsection{Stability of the nonlinear filter (ergodic signal case)}

We adopt the notation of~\cite{ocone1996asymptotic}: For the signal, $\clS := \{\clS_t:t\ge 0\}$ is the signal semigroup on $C_b(\bS)$. For each $t \ge 0$, $\clS_t:C_b(\bS)\to C_b(\bS)$ is defined by
\[
(\clS_t f)(x) = \E\big(f(X_t)\mid X_0 = x\big),\quad x\in\bS
\]
For a measure $\mu \in \clP(\bS)$, we denote $\mu_t = \mu\clS_t$---which means $\mu_t(f) = \mu(\clS_t f)$, $\forall f\in C_b(\bS)$. Note $\{\mu_t:t\ge 0\}$ is the solution of the forward Kolmogorov equation~\eqref{eq:FKE}.

\begin{definition}
	The semigroup $\clS$ is \emph{ergodic} if there exists an invariant measure $\bmu \in \clP(\bS)$ such that
	\begin{equation}\label{eq:ergodicity-f}
		\limsup_{t\to \infty}\int_\bS \big|(\clS_t f)(x)-\bmu(f)\big|\ud \bmu(x) = 0,\quad \forall\,f\in C_b(\bS)
	\end{equation}
	We write $\mu_t\weakto \bmu$ if $\mu_t(f) = \mu(\clS_t f)\to \bmu(f)$ for all $f\in C_b(\bS)$.
\end{definition}

For the nonlinear filter, $\Phi = \{\Phi_{t,\tau}:0\le \tau \le t\}$ is the semigroup on $C_b(\clP(\bS))$.
For each $0\le \tau \le t$, $\Phi_{t,\tau}: C_b(\clP(\bS))\to  C_b(\clP(\bS))$ is defined by
\[
(\Phi_{t\tau} F)(\rho) = \E\big(F(\pi_t)\mid \pi_\tau = \rho\big),\quad \rho \in \clP(\bS)
\]
If $\tau = 0$, the semigroup is denoted by $\Phi_t$.

In his early papers, Kunita~\cite{kunita1971asymptotic,kunita1991ergodic} claimed that if the signal itself is ergodic then the nonlinear filter inherits the ergodic property.

\medskip

\begin{lemma}[Lemma 3.1 in~\cite{ocone1996asymptotic}, also Theorem 3.1-3.3 in~\cite{kunita1971asymptotic}] \label{prop:filter-is-ergodic}
	Suppose the $\clS$ satisfies~\eqref{eq:ergodicity-f}.
	Then there exists a unique measure $M$ on $\clP(\bS)$ such that $M$ is $\Phi_t$-invariant and
	\[
	\int_{\clP(\bS)} \rho(f) \ud M(\rho) = \bmu(f),\quad \forall\,f\in C_b(\bS)
	\]
	If, additionally $\mu\weakto \bmu$ and $\nu\weakto \bmu$ then
	\[
	\lim_{t\to \infty}(\Phi_t F)(\mu) = \lim_{t\to \infty}(\Phi_t F)(\nu) = \int_{\clP(\bS)}F(\rho)\ud M(\rho) =: M(F),\quad \forall\,F\in C_b(\clP(\bS))
	\]
\end{lemma}

\medskip

Let $\pi_{t,\tau}^\mu$ be the solution to the nonlinear filter~\eqref{eq:nonlinear-filter} from initial condition $\pi_\tau^\mu = \mu_\tau$. For a test function $f\in C_b(\bS)$, define $F:\clP(\bS)\to \Re$ by
\[
F(\rho) = (\rho(f))^2
\]
Observe that
\begin{align*}
	\E^\mu\big(|\pi_t^\mu(f)-\pi_{t,\tau}^\mu(f)|^2\big) &= \E^\mu\big(|\pi_t^\mu(f)|^2\big)-2\E^\mu\big(\pi_t^\mu(f)\pi_{t,\tau}^\mu(f)\big) + \E^\mu\big(|\pi_{t,\tau}^\mu(f)|^2\big)\\ &=\E^\mu\big(|\pi_t^\mu(f)|^2\big)- \E^\mu\big(|\pi_{t,\tau}^\mu(f)|^2\big)\\
	&=(\Phi_tF)(\mu) - (\Phi_{t,\tau}F)(\mu_\tau)
\end{align*}
By the ergodic property of $\clS_t$ and $\Phi_t$, the following finite-memory property of the filter is established.

\medskip

\begin{lemma}[Lemma 3.4 in~\cite{ocone1996asymptotic}] \label{lm:63}
	Suppose $\clS$ satisfies~\eqref{eq:ergodicity-f}, $\mu_t\weakto\bmu$, and $\nu_t\weakto \bmu$. Then for every $\epsilon > 0$, there exists a $T_\epsilon$ and $t_\epsilon$ such that
	\begin{align*}
		\E^\mu\big(|\pi_t^\mu(f)-\pi_{t,t-T_\epsilon}^\mu(f)|^2\big) &<\epsilon,\quad \forall\, t\ge t_\epsilon \\
		\E^\mu\big(|\pi_t^\nu(f)-\pi_{t,t-T_\epsilon}^\nu(f)|^2\big) &<\epsilon,\quad \forall\, t\ge t_\epsilon
	\end{align*}
\end{lemma}

\medskip

The finite memory property is used to deduce the filter stability. Observe that
\[
\E^\mu\big(|\pi_t^\mu(f)-\pi_t^\nu(f)|^2\big) = 3\Big(\E^\mu\big(|\pi_t^\mu(f)-\pi_{t,t-T_\epsilon}^\mu(f)|^2\big)+ \E^\mu\big(|\pi_t^\nu(f)-\pi_{t,t-T_\epsilon}^\nu(f)|^2\big)+ \E^\mu\big(|\pi_{t,t-T_\epsilon}^\mu(f)-\pi_{t,t-T_\epsilon}^\nu(f)|^2\big)\Big)
\]
The first two terms can be made arbitrarily small using Lemma~\ref{lm:63}. It remains to show for any fixed $T_\epsilon$,
\[
\lim_{t\to\infty} \E^\mu\big(|\pi_{t,t-T_\epsilon}^\mu(f)-\pi_{t,t-T_\epsilon}^\nu(f)|^2\big) = 0
\]
The idea of the proof is because $\mu_{t-T_\epsilon}\weakto \bmu$ and $\nu_{t-T_\epsilon}\weakto\bmu$, $\pi_{t,t-T_\epsilon}^\mu(f)$ and $\pi_{t,t-T_\epsilon}^\nu(f)$ converge to the same limit. For complete proof, see~\cite[Lemma 3.6]{ocone1996asymptotic}.
This leads to the following conclusion:

\begin{proposition}[Theorem 3.2 in~\cite{ocone1996asymptotic}]
	Consider for $\mu, \nu \in \clP(\bS)$ such that $\mu\ll\nu$.
	If $\clS$ satisfies~\eqref{eq:ergodicity-f}, $\mu\weakto\bmu$ and $\nu\weakto\bmu$ then 
	\[
	\lim_{T\to \infty} \E^\mu\big(|\pi_T^\mu(f)-\pi_T^\nu(f)|^2\big) = 0,\quad \forall\, f\in C_b(\bS)
	\]	
\end{proposition}

\medskip

\begin{remark}
	The proof utilizes the Kunita's result that the nonlinear filter admits a unique invariant measure.
	However, it is noted in~\cite[Section 3]{baxendale2004asymptotic} that the proof of Lemma~\ref{prop:filter-is-ergodic} contains an error. 
	Since the arguments in~\cite{ocone1996asymptotic} crucially relied on Kunita's work, their proof of filter stability inherited the same error. 
\end{remark}


\section{The counter-example of filtering theory}\label{sec:counter-example}

The following counter-example first appeared in~\cite{delyon1988lyapunov}. It describes a model whose signal is ergodic but the filter is not so. 

\begin{example}\label{ex:counter-example}
	Consider the state-space
	$\bS=\{1,2,3,4\}$ and the rate matrix 
	\[
	A = \begin{pmatrix}
		-1 & 1 & 0 & 0\\
		0 & -1 & 1 & 0\\
		0 & 0 & -1 & 1\\
		1 & 0 & 0 & -1
	\end{pmatrix}
	\]
	whose unique invariant measure $\bmu =
	[\frac{1}{4},\frac{1}{4},\frac{1}{4},\frac{1}{4}]$.
	The observation model is as follows:
	\[
	Z_t = \ones_{[X_t \in\{1,3\}]}
	\]
	The observation provides exact time instant $\{t_k \ge 0:k = 1,2,\ldots \}$ when the state jumps to another state. For $t_k \le t < t_{k+1}$, $\pi_t$ is constant. The table below illustrates a sample observation path and the corresponding optimal estimates.
	In this table, $p = \frac{\mu(1)}{\mu(1)+\mu(3)}$.
	
	\begin{center}
		\begin{tabular}{|c||c|c|c|c|c|c|}
			\hline
			$t$ & $[0,t_1)$ & $[t_1,t_2)$ & $[t_2,t_3)$ & $[t_3,t_4)$ & $[t_4,t_5)$ & $ \cdots$ \\
			\hline
			$Z_{t}$ & 1 & 0 & 1 & 0 & 1  & $ \cdots$\\
			\hline
			$\pi_{t}^\mu(1)$ & $p$ & 0 & $1-p$ & 0 & $p$  & $ \cdots$\\
			$\pi_{t}^\mu(2)$ & 0 & $p$ & 0 & $1-p$ & 0  & $ \cdots$\\
			$\pi_{t}^\mu(3)$ & $1-p$ & 0 & $p$ & 0 & $1-p$  & $ \cdots$\\
			$\pi_{t}^\mu(4)$ & 0 & $1-p$ & 0 & $p$ & 0  & $ \cdots$\\
			\hline
		\end{tabular}
	\end{center}
	For a different initialization $\nu$ of the prior, let $p' = \frac{\nu(1)}{\nu(1)+\nu(3)}$ and then 
	\[
	\|\pi_t^\mu-\pi_t^\nu\|_\tv = 2(p-p'),\quad \forall\,t > 0
	\]
	Therefore, the filter is not stable.
	
\end{example}

\section{Lyapunov exponent analysis of the Zakai equation}\label{sec:Lyapunov-analysis}

We begin by recalling the solution operator $\Psi=\{\Psi_t:t\ge 0\}$ of the Zakai equation defined by~\eqref{eq:Zakai-soln-operator}. 
The Lyapunov exponent analysis is based on the analysis of the contraction property of $\Psi$. One of earlier contribution on this theme appears in~\cite{delyon1988lyapunov}, which was expanded by Atar and Zeitouni in ~\cite{atar1997exponential,atar1997lyapunov,atar1998exponential}.
In these papers, the stability index is defined by
\[
\overline{\gamma}:=\limsup_{T\to\infty}\frac{1}{T}\log\|\pi_T^{\nu}-\pi_T^\mu\|_\tv, \quad \sP^\mu{-a.s.}
\]
If this value is negative, then the filter is asymptotically stable in total variation norm. Moreover, $\overline{\gamma}$ gives a bound on the rate of convergence.
For the analysis of $\overline{\gamma}$, the Hilbert projection metric is useful.


\begin{definition}[Eq.~9 in~\cite{atar1997exponential}]
	For non-negative finite equivalent measures $\mu$ and $\nu$  (i.e., $\mu\ll\nu$ and 
	$\nu\ll\mu$), the \emph{Hilbert projection metric} is
	\[
	H(\mu,\nu) := 
	\log\frac{\sup_{\{A:\nu(A)> 0\}}\mu(A)/\nu(A)}{\inf_{\{A:\nu(A)>0\}}\mu(A)/\nu(A)}
	\]
\end{definition}

This metric is useful to analyze the stability index because of the following lemma:

\begin{lemma}[Section 5.1 in~\cite{baxendale2004asymptotic}]
	The Hilbert projection metric satisfies the following:
	\begin{itemize}
		\item $H(c_1\mu,c_2\nu) = H(\mu,\nu)$ for any positive constants $c_1$ and $c_2$.
		\item For any $\mu,\nu\in\clP(\bS)$, $\|\mu-\nu\|_\tv \le \frac{2}{\log 3} H(\mu,\nu)$.
	\end{itemize}
\end{lemma} 
For the filter,
\[
\|\pi_t^\mu-\pi_t^\nu\|_\tv \le \frac{2}{\log 3}H(\pi_t^\mu,\pi_t^\nu) = \frac{2}{\log 3} H(\sigma_t^\mu,\sigma_t^\nu) = \frac{2}{\log 3} H(\Psi_t \mu, \Psi_t \nu)
\]
The right-most term is bounded by Birkhoff's contraction coefficient~\cite[Ch. 3]{seneta2006non} of $\Psi_t$, which is defined by
\[
\tau(\Psi_t) := \sup_{0<H(\mu,\nu)<\infty} 
\frac{H(\Psi_t\mu,\Psi_t\nu)}{H(\mu,\nu)}
\]
The main result is Theorem 1 in~\cite{atar1997exponential} that if the state process is ergodic with a unique stationary measure $\bmu$ then 
\[
\overline{\gamma} \le 
\frac{1}{t} \E^\bmu\big(\log \tau(\Psi_t)\big),\quad \forall\, t > 0
\]
The behavior of $\tau(\Psi_t)$ is studied under different settings: 
discrete time case~\cite[Section 3]{atar1997exponential}, for diffusions on compact manifold~\cite[Section 4]{atar1997exponential} and finite state space case~\cite{atar1997lyapunov}.
A representative result is as follows:

\subsubsection{Ergodic case on finite state space}

For ergodic Markov chain on finite state space, we follow discussion in~\cite{baxendale2004asymptotic}. Set $\mu = \bmu$ and observe that
\[
H(\Psi_t \bmu,\Psi_t \nu) \le \tau\big(\Psi_{t,\lfloor t \rfloor}\big) \Big(\prod_{n=1}^{\lfloor t \rfloor} \tau(\Psi_{n,n-1})\Big)H(\bmu,\nu) 
\]
Therefore,
\[
\overline{\gamma} \le \limsup_{t\to \infty} \frac{1}{\lfloor t\rfloor} \sum_{n=1}^{\lfloor t \rfloor} \tau(\Psi_{n,n-1}) \le \limsup_{t\to \infty} \frac{1}{\lfloor t\rfloor} \sum_{n=1}^{\lfloor t \rfloor} \big(\tau(\Psi_{n,n-1}) \vee -1 \big)
\]
Since $\bmu$ is the stationary measure, $\{Z_t: n-1\le t < n\}$ is i.i.d.~for each $n\in \mathbb{N}$. Therefore by law of large numbers,
\[
\overline{\gamma} \le \E^\bmu\big(\tau(\Psi_{1,0}) \vee -1 \big),\quad \sP^\bmu\text{-a.s.}
\]
It is shown in~\cite[Lemma 5.2]{baxendale2004asymptotic} that if the state process is ergodic then the right-hand side is strictly negative. Hence the following is obtained:

\begin{proposition}[Theorem 4.1 in~\cite{baxendale2004asymptotic}]
	Suppose $\bS$ is finite with a single communicating class. Then there exists $c>0$ such that for any $\mu,\nu \in \clP(\bS)$,
	\[
	\limsup_{T\to \infty}\frac{1}{T}\log \|\pi_T^\mu-\pi_T^\nu\|_\tv < -c,\quad \sP^\mu\text{-a.s.}
	\]
\end{proposition}

\medskip

The following bounds have been established in literature for the finite state space case:
\begin{itemize}
	\item Square-root type bound~\cite[Theorem 5]{atar1997lyapunov}
	\[
	\overline{\gamma} \le -2 \min_{i\neq j} \sqrt{A(i,j)A(j,i)}
	\]
	\item Bound from minimum over row~\cite{baxendale2004asymptotic}
	\[
	\overline{\gamma} \le -\sum_{i\in\bS}
	\bmu(i) \min_{j\neq i}A(i,j)
	\]
	\item Another notable result is proved when the measurement noise is scaled by a factor $r > 0$. Explicitly,
	\[
	Z_t = \int_0^t h(X_s)\ud s + r W_t
	\]
	Theorem 7 in~\cite{atar1997lyapunov} establishes the following limits:
	\begin{align*}
		\limsup_{r\to 0}r^2 \overline{\gamma} &\le -\half  \sum_{i \in \bS}\bmu(i)\min_{j\neq i}|h(i)-h(j)|^2\\
		\liminf_{r\to 0}r^2 \overline{\gamma} &\le -\half  \sum_{i,j \in \bS}\bmu(i)|h(i)-h(j)|^2
	\end{align*}

\end{itemize}

\medskip

\section{Analysis on KL divergence}\label{sec:literature-clark}

While prior discussion mainly assumed ergodicity of the state process, it is apparent that the observation model should also be taken into account.
The earliest work on the filter stability problem without the ergodic assumption is by Clark, Ocone and Coumarbatch~\cite{clark1999relative}.
In this paper, the authors consider the relative entropy metric to compare $\mu,\nu \in \clP(\bS)$:
\[
\kl(\mu\mid \nu) := \begin{cases}\displaystyle \int_\bS\frac{\ud \mu}{\ud \nu} \log \Big(\frac{\ud \mu}{\ud 
		\nu}\Big) \ud \nu\qquad &\text{if }\mu\ll\nu\\
	\infty &\text{if else}
\end{cases}
\]
The main result of~\cite{clark1999relative}, proved for a general class of HMMs, is as follows:

\medskip

\begin{proposition}[Theorem 2.2 in~\cite{clark1999relative}] \label{prop:clark-main-result}
	Assume $\mu\ll \nu$. Then
	\[
	\kl(\mu\mid\nu) = \E^\mu\big(\kl(\pi_t^\mu\mid\pi_t^\nu)\big) + 
	\kl\big(\sP^\mu|_{\clZ_t}\mid\sP^\nu|_{\clZ_t}\big) + 
	\E^\mu\big(\kl(\mu(\cdot|\clZ_t,X_t)\mid\nu(\cdot|\clZ_t,X_t))\big),\quad 
	\forall\,t > 0
	\]
\end{proposition}

\medskip

Consequently, if $\kl(\mu\mid\nu) < \infty$ then KL divergence of $\pi_t^\mu$ from $\pi_t^\nu$ is uniformly bounded because:
\begin{equation}\label{eq:KL-Lyapunov}
	\E^\mu\big(\kl(\pi_t^\mu\mid\pi_t^\nu)\big) \le \kl(\mu\mid\nu) < \infty
\end{equation}
This inequality also implies that $\big\{\kl(\pi_t^\mu\mid\pi_t^\nu):t\ge 0\big\}$ is a non-negative $\sP^\mu$-super-martingale. This is because upon conditioning on $\clZ_s$ for $s \le t$
and using~\eqref{eq:KL-Lyapunov} now with initializations $\pi_s^\mu$ and $\pi_s^\nu$ at time $s$
\[
\E^\mu\big(\kl(\pi_t^\mu\mid\pi_t^\nu)\mid \clZ_s\big)  \le \kl(\pi_s^\mu\mid\pi_s^\nu),\quad 0\le s\le t,\;\sP^\mu\text{-a.s.} 
\]
Therefore, the relative entropy is a Lyapunov function for the filter stability problem, in the sense that $\E^\mu\big(\kl(\pi_t^\mu\mid\pi_t^\nu)\big)$ is non-increasing~\cite[Section 4.1]{chigansky2009intrinsic}.
However, it has not been possible to show whether $\E^\mu\big(\kl(\pi_T^\mu\mid\pi_T^\nu)\big) \to 0$ as $T\to \infty$.

\medskip

\subsection{White noise observation case}

For the white noise observation model, the second term is explicitly obtained as follows~\cite[Theorem 3.1]{clark1999relative}
\[
\kl\big(\sP^\mu|_{\clZ_t}\mid\sP^\nu|_{\clZ_t}\big) = \half 
\E^\mu\Big(\int_0^t |\pi_s^\mu(h)-\pi_s^\nu(h)|^2 \ud s\Big)
\]
If $\kl(\mu\mid\nu)<\infty$, use Prop.~\ref{prop:clark-main-result} to show that
\begin{equation}\label{eq:finiteness-of-h}
	\half \E^\mu\Big(\int_0^\infty |\pi_t^\mu(h)-\pi_t^\nu(h)|^2\ud t\Big) \le \kl(\mu\mid\nu)<\infty
\end{equation}

\medskip

\begin{remark}\label{rm:observable-function-motivation}
	Equation~\eqref{eq:finiteness-of-h} implies that $|\pi_t^\mu(h)-\pi_t^\nu(h)| \to 0$ $\sP^\mu$-a.s.
	This shows that the filter is \emph{always} stable for the observation function $h(\cdot)$. A generalization of this is described in Chigansky and Lipster~\cite{chigansky2006role}.
	In their paper, it is proved that one-step predictive estimates of the observation process are stable. 
\end{remark}

\section{Intrinsic methods in filter stability}\label{sec:65}

This section surveys filter stability of general HMMs. 
The title of this section is the same as a book chapter by Chigansky et.al.~\cite{chigansky2009intrinsic}. 
The authors explain the reason why they use the word `intrinsic'~\cite{chigansky2009intrinsic}:
\begin{quote}
	``{\it By `intrinsic' we mean methods which directly exploit the fundamental representation of the filter as a conditional expectation through classical probabilistic techniques [...] these methods allow one to establish stability of the filter under weaker conditions compared to other methods, e.g., to go beyond strong mixing signals, to reveal connections between filter stability and classical notions of observability, and to discover links to martingale convergence and information theory.}''
\end{quote}
The two sub-sections describe the results for the ergodic and non-ergodic cases, respectively.

\subsection{Ergodic signal case}

We begin with recalling the Bayes formula (Prop.~\ref{prop:Kallianpur-Striebel})
\[
\E^\mu\big(f(X_T)\mid \clZ_T\big) = \frac{\E^\nu\big(\frac{\ud \sP^\mu}{\ud \sP^\nu}f(X_T)\mid \clZ_T\big)}{\E^\nu\big(\frac{\ud \sP^\mu}{\ud \sP^\nu}\mid \clZ_T\big)}
\]
Since $\frac{\ud \sP^\mu}{\ud \sP^\nu} = \frac{\ud \mu}{\ud \nu}(X_0)$ (Lemma~\ref{lm:change-of-Pmu-Pnu}), the equation is expressed as follows:
\begin{align*}
	\int_\bS f(x)\ud \pi_T^\mu(x) &= \frac{\E^\nu\big(\frac{\ud \mu}{\ud \nu}(X_0)f(X_T)\mid \clZ_T\big)}{\E^\nu\big(\frac{\ud \mu}{\ud \nu}(X_0)\mid \clZ_T\big)}\\
	&= \E^\nu\Big(f(X_T) \frac{\E^\nu\big(\frac{\ud \mu}{\ud \nu}(X_0)\mid \clZ_T \vee \sigma\{X_T\}\big)}{\E^\nu\big(\frac{\ud \mu}{\ud \nu}(X_0)\mid \clZ_T\big)}\Big)\\
	&= \int_\bS f(x) \frac{\E^\nu\big(\frac{\ud \mu}{\ud \nu}(X_0)\mid \clZ_T, X_T = x\big)}{\E^\nu\big(\frac{\ud \mu}{\ud \nu}(X_0)\mid \clZ_T\big)} \ud \pi_T^\nu(x)
\end{align*}
Therefore, the Radon-Nikodym derivative is obtained as follows
\begin{equation}\label{eq:gamma_T-explicit}
	\gamma_T(x) := \frac{\ud \pi_T^\mu}{\ud \pi_T^\nu}(x) =  \frac{\E^\nu\big(\frac{\ud \mu}{\ud \nu}(X_0)\mid \clZ_T, X_T = x\big)}{\E^\nu\big(\frac{\ud \mu}{\ud \nu}(X_0)\mid \clZ_T\big)},\quad \sP^\mu\text{-a.s.}
\end{equation}
Therefore, 
\begin{align*}
	\|\pi_T^\mu-\pi_T^\nu\|_\tv &= \int_\bS \big|\gamma_T(x)-1\big| \ud \pi_T^\nu(x)\\
	&=\frac{\E^\nu\Big(\big|\E^\nu\big(\frac{\ud \mu}{\ud \nu}(X_0)\mid \clZ_T\vee \sigma\{X_T\}\big) - \E^\nu\big(\frac{\ud \mu}{\ud \nu}(X_0)\mid \clZ_T\big)\big| \mid \clZ_T\Big)}{\E^\nu\big(\frac{\ud \mu}{\ud \nu}(X_0)\mid \clZ_T\big)}
\end{align*}
Since $(X,Z)$ is a Markov process, the first term in the numerator becomes 
\[
\E^\nu\big(\frac{\ud \mu}{\ud \nu}(X_0)\mid \clZ_T\vee \sigma\{X_T\}\big) = \E^\nu\big(\frac{\ud \mu}{\ud \nu}(X_0)\mid \clZ_\infty\vee \clF^X_{[T,\infty)}\big) 
\]
where $\clF^X_{[T,\infty)} = \sigma\big(\{X_t-X_T: t \ge T\}\big)$ is the tail sigma algebra of the state process $X$. Hence
\[
\E^\nu\Big(\frac{\ud \sP^\mu}{\ud \sP^\nu}\,\big|\, \clZ_T\Big)\|\pi_T^\mu-\pi_T^\nu\|_\tv = \E^\nu\Big(\Big|\E^\nu\big(\frac{\ud \mu}{\ud \nu}(X_0)\mid \clZ_\infty\vee \clF^X_{[T,\infty)}\big) - \E^\nu\big(\frac{\ud \mu}{\ud \nu}(X_0)\mid \clZ_T\big)\Big|\,\big|\, \clZ_T\Big)
\]
Taking expectation $\E^\nu(\cdot)$ on both sides yields
\[
\E^\mu\big(\|\pi_T^\mu-\pi_T^\nu\|_\tv\big) = \E^\nu\Big(\,\Big|\E^\nu\big(\frac{\ud \mu}{\ud \nu}(X_0)\mid \clZ_\infty\vee \clF^X_{[T,\infty)}\big) - \E^\nu\big(\frac{\ud \mu}{\ud \nu}(X_0)\mid \clZ_T\big)\Big|\,\Big)
\]
Note that $\clZ_\infty\vee \clF^X_{[T,\infty)}$ is a decreasing filtration and $\clZ_T$ is an increasing filtration as $T$ increases. Therefore both terms on the right-hand side converges as $T\to \infty$, 
\[
\lim_{T\to \infty} \E^\mu\big(\|\pi_T^\mu-\pi_T^\nu\|_\tv\big) = \E^\nu\Big(\,\Big|\E^\nu\big(\frac{\ud \mu}{\ud \nu}(X_0)\mid \bigcap_{T\ge 0}\clZ_\infty\vee \clF^X_{[T,\infty)}\big) - \E^\nu\big(\frac{\ud \mu}{\ud \nu}(X_0)\mid \clZ_\infty\big)\Big|\, \Big)
\]
The right-hand side is zero if the following tail sigma field identity:
\begin{equation}\label{eq:kunita-tail-sigmafield}
	\bigcap_{T\ge 0} \clZ_\infty \vee \clF^X_{[T,\infty)} \stackrel{?}{=} \clZ_\infty 
\end{equation}
This identity is referred to as the central problem in the stability analysis of the nonlinear filter~\cite{van2010nonlinear}. The problem generated large consequent attention (see~\cite{budhiraja2003asymptotic} and references therein).

If the state process is ergodic---namely $\bigcap_{T\ge 0} \clF^X_{[T,\infty)}$ is $\sP$-almost surely empty---and therefore the identity was believed to be true because
\[
\bigcap_{T\ge 0} \clZ_\infty \vee \clF^X_{[T,\infty)} \stackrel{?}{=} \clZ_\infty \vee \bigcap_{T\ge 0} \clF^X_{[T,\infty)} = \clZ_\infty
\]
However, one cannot interchange the union and the intersection of sigma fields in general, and therefore the ergodicity does not always imply~\eqref{eq:kunita-tail-sigmafield}. This subtle flaw appears in classic paper of Kunita~\cite{kunita1971asymptotic}, as discussed in 
detail in~\cite{baxendale2004asymptotic}.

The identity is shown to be true later with an extra assumption on non-degeneracy of the observation process~\cite[Assumption III.2]{van2010nonlinear}. Namely, there exists a strictly positive function $g:\bS\times \Re^m\to (0,\infty)$ and a measure $\rho \in \clP(\Re^m)$ such that 
\begin{equation}\label{eq:assumption-ergodic}
	\sP^\mu\big(Z_t\in A \mid X_t\big) = \int_A g(X_t,z)\ud \rho(z),\quad A \in \clB(\Re^m)
\end{equation}
for every $\mu \in \clP(\bS)$. In words, the observation kernel possesses a positive density $g$ with respect to some reference measure. This assumption rules out the noiseless observation model of the counter-example~\ref{ex:counter-example}.
It is noted that addition of arbitrarily small noise to the observation will mean~\eqref{eq:assumption-ergodic} holds and the filter will be ergodic. 
The general result for the ergodic signal case is proved specifically for the discrete time system in~\cite{van2010nonlinear}:

\begin{proposition}[Theorem III.3 in~\cite{van2010nonlinear}]
	Suppose the state process is ergodic and the observation model satisfies~\eqref{eq:assumption-ergodic} for every $\mu \in \clP(\bS)$. Then 
	\[
	\|\pi_T^\mu-\pi_T^\nu\|_\tv \;\longrightarrow\; 0,\quad \text{as }T\to \infty,\;\sP^\lambda\text{-a.s.}
	\]
	for any $\mu,\nu,\lambda \in \clP(\bS)$.
\end{proposition}

\subsection{Stochastic observability and filter stability}\label{sec:VHobs-and-filter-stability}

Analysis of the relative entropy for the white noise observation model (see Eq.~\eqref{eq:finiteness-of-h}) shows that the filter is always stable for the observation function $h$, without any assumption on the state process.
It is appealing to consider, possibly under some assumption also on the state process, a class of ``observable'' functions $\clO$ such that if $\mu \ll \nu$ then
\[
|\pi_t^\mu(f)-\pi_t^\nu(f)|\;\longrightarrow\; 0,\quad \forall\, f\in \clO
\]
Certainly $\clO$ is non-trivial because $\ones \in \clO$ and $h\in\clO$.

This idea is investigated by van Handel~\cite{van2009observability}. Observability and space of observable functions are defined in Section~\ref{ssec:observability-hmm}.
The main result is stated as follows:

\begin{proposition}[Theorem 1 in~\cite{van2009observability}]
	Suppose $\bS$ is compact and $\mu\ll\nu$. Then for any $f \in \clO$,
	\[
	|\pi_T^\mu(f)-\pi_T^\nu(f)|\;\longrightarrow\; 0,\quad \text{as }T\to\infty, \;\sP^\mu\text{-a.s.}
	\]
\end{proposition}

Recall that the model was said to be observable if $\clO = C_b(\bS)$. Therefore observability implies filter stability
. 
For non-compact state spaces, a counter-example is given in~\cite[Example 1.2]{van2009uniform} that the observability is not sufficient to achieve the filter stability. A stronger notion of \emph{uniform observability} introduced for this purpose~\cite{van2009uniform}.

The observability naturally extends to detectability (see Def.~\ref{def:detectability} and Remark~\ref{rm:vh-detectabilty}). The following is proved for discrete time system in~\cite{van2010nonlinear}:

\medskip

\begin{proposition}[Theorem V.2 in~\cite{van2010nonlinear}] \label{prop:vanhandel-detectability}
	Suppose $\bS$ is finite. Then for all $\mu,\nu \in\clP(\bS)$ such that $\mu\ll \nu$,
	\[
	\|\pi_T^\mu - \pi_T^\nu\|_\tv \; \longrightarrow \; 0,\quad \sP^\mu\text{-a.s.}
	\]
	if and only if the HMM is detectable.
\end{proposition}

\medskip

\begin{remark}
	It is noted that while (minimum variance) duality considerations are central to the proof of 
	filter stability for linear Gaussian settings (see Section~\ref{ssec:Kalman-filter-stability}), the nonlinear filter stability proofs mainly rely on probabilistic arguments.  This is noted by van Handel in the 
	introduction of his tutorial paper~\cite{van2010nonlinear}:
	\begin{quote}
		``{\it The proofs of the Kalman filter results are of essentially
			no use here, so we must start from scratch.}''
	\end{quote}
	A notable contribution using duality-based method for filter stability analysis is the subject of van Handel's PhD thesis~\cite{van2006filtering}. In the following section, the approach is reviewed.
\end{remark}

\section{Mitter-Newton duality for filter stability}\label{sec:66}

The discussion in this section is adapted from van Handel's PhD thesis~\cite[Section 4.4]{van2006filtering}.
In chapter 4 of his thesis, the following SDE model is considered:
\begin{align*}
	\ud X_t &= a(X_t) \ud t + \ud B_t,\quad X_0 \sim \nu_0(x)  \\
	\ud Z_t &= H^\tp X_t + \ud W_t
\end{align*}
where $a(x) = A x - \nabla \phi(x)$ and the function $\phi$
satisfies suitable technical conditions, $\nu_0$ is a bounded and positive density (with respect to Lebesgue measure), and $B,W$ are mutually independent B.M.

\medskip

\subsubsection{Mitter-Newton duality}
We refer the reader to Appendix~\ref{apdx:min-energy} where Mitter-Newton duality is fully described.
The dual optimal control problem (Section~\ref{ssec:mitter-problem}) arising in Mitter-Newton duality is as follows:
\begin{subequations} \label{eq:mitter-problem}
	\begin{align}
		\mathop{\text{Min }}_{\pi_0, U} &:\sJ(\pi_0,U\,;z) = \E\Big(\log \frac{\ud \pi_0}{\ud \nu_0}(\tilde{X}_0) - z_T^\tp H^\tp \tilde{X}_T + \int_0^T \ell(\tilde{X}_t,U_t\,;z_t)\ud t \Big) \label{eq:mitter-problem-a}\\
		\text{Subj.} &:\ud \tilde{X}_t = a(\tilde{X}_t)\ud t + U_t\ud t + \ud \tilde{B}_t,\quad X_0\sim \pi_0 \label{eq:mitter-problem-b}
	\end{align}
\end{subequations}
For the particular model, the cost function is given by
\[
\ell(x,u\,;z) = \half |u|^2 + |H^\tp x|^2 + (a(x) + u)^\tp Hz
\]
The optimal control problem is solved by defining the value function 
\[
V(t,x) := \min_{U\in \clU} \E\Big(\int_t^T\ell(\tilde{X}_s,U_s\,;z_s)\ud s - z_T^\tp H\tilde{X}_T\mid \tilde{X}_t = x \Big)
\]
In terms of the value function, the optimal control is given by (see Prop.~\ref{thm:opt-ctrl-sde-hjb})
\[
U_t = -\nabla V(t,\tilde{X}_t) - H z_t 
\]
Therefore the optimal controlled process is
\[
\ud \tilde{X}_t = a(\tilde{X}_t) \ud t - \nabla V(t,X_t)\ud t  - H z_t \ud t +\ud {B}_t,\quad X_0\sim \pi_0
\]
It is shown in Appendix~\ref{apdx:min-energy} (Prop.~\ref{thm:forward-backward}) that if $\pi_0(x)\ud x = \sP(X_0\in \ud x \mid \clZ_T)$ then the probability law of $\{\tilde{X}_t:0\le t \le T\}$ is the same as the probability law of $X$ conditioned on $Z =z$.
Since smoothing and the filtering laws match at the terminal time $T$, 
\[
\pi_T(f) = \E\big(f(\tilde{X}_T)\mid \clZ_T\big)
\]

\medskip

\subsubsection{Filter stability analysis}


Noting that the optimal control law is given by the gradient of a function, it is useful to define
\[
\tilde{V}(t,x) = V(t,x) +\phi(x) - \half x^\tp A x
\]
Using this function, the optimal controlled process is
\[
\ud \tilde{X}_t = \half (A-A^\tp)\tilde{X}_t \ud t - \nabla \tilde{V}(t,\tilde{X}_t) \ud t - H z_t \ud t + \ud \tilde{B}_t
\]
Consider the optimal controlled process from two initial conditions
\begin{align*}
	\ud \tilde{X}_t^x = \half(A-A^\tp)\tilde{X}_t^x \ud t - \nabla \tilde{V}(t,\tilde{X}_t^x) \ud t+\ud
	\tilde{B}_t, \quad X_0^x=x\\
	\ud \tilde{X}_t^y = \half(A-A^\tp)\tilde{X}_t^y \ud t - \nabla \tilde{V}(t,\tilde{X}_t^y) \ud t+\ud
	\tilde{B}_t, \quad X_0^y=y
\end{align*}
Define the error process $e_t = \tilde{X}_t^x - \tilde{X}_t^y$.
Then it solves the ODE
\[
\frac{\ud e_t}{\ud t} = \half(A-A^\tp)(\tilde{X}_t^x - \tilde{X}_t^y)\ud t  - (\nabla \tilde{V}(t,\tilde{X}_t^x)  - \nabla \tilde{V}(t,\tilde{X}_t^y)),
\quad e_0=x-y 
\]
Therefore,
\[
\frac{\ud }{\ud t} |e_t|^2 = -2 \langle (\tilde{X}_t^x - \tilde{X}_t^y),
(\nabla \tilde{V}(t,\tilde{X}_t^x)  - \nabla \tilde{V}(t,\tilde{X}_t^y)) \rangle, \quad |e_0|^2 = |x-y|^2
\]
where we used the fact that $\langle x, (A-A^\tp) x \rangle=0$. In order to conclude that $|e_t|^2\to 0$, we recall a definition of uniform convexity:

\begin{definition}[Definition 4.3.4 in~\cite{van2006filtering}] A function $f(x)$ is called {\em $\kappa$-uniformly convex} if $f(x) - \half  \kappa |x|^2$ is convex.
\end{definition}

\begin{lemma}[Lemma 4.3.5 in~\cite{van2006filtering}] A differentiable function $f(x)$ is $\kappa$-uniformly convex if and only if
	\[
	\langle x-y, \nabla f(x) - \nabla f(y) \rangle \ge \kappa |x-y|^2
	\quad \forall x, y \in \Re^d
	\]
\end{lemma}

Therefore, if we can establish that $\tilde{V}(t,\cdot)$ is $\kappa$-uniformly convex then it follows that
\[
\frac{\ud }{\ud t} |e_t|^2 \leq -2\kappa |e_t|^2\quad \Longrightarrow \quad 
\big|\tilde{X}_T^x - \tilde{X}_T^y\big| \le e^{-\kappa T}|x-y|
\]
For any test function $f$ and for each observation sample path,
\[
|\pi_T^{\delta_x}(f) - \pi_T^{\delta_y}(f)| \le \operatorname{Lip}(f)\,\big|\tilde{X}_T^x - \tilde{X}_T^y\big| \le  \operatorname{Lip}(f) \,e^{-\kappa T} |x-y|
\]
Hence the filter is stable in bounded Lipschitz metric for a class of initial conditions comprising of Dirac-delta measures. It is easily extended to compactly supported measures~\cite[p.~100]{van2006filtering}.

\medskip

It remains to show that the $\tilde{V}$ is $\kappa$-uniformly convex.
Our aim is to write an optimal control problem such that $\tilde{V}(t,x)$ is the value function and $\tilde{U}_t=-\nabla \tilde{V}(t,\tilde{X}_t)$ is the optimal control.

Recall that
\begin{align*}
	\tilde{V}(t,x) &= V(t,x) + \phi(x) - \half x^\tp A x \\
	& = \min_{U\in \clU} \E\Big(\int_t^T\ell(\tilde{X}_s,U_s\,;z_s)\ud s - z_T^\tp H\tilde{X}_T\mid \tilde{X}_t = x \Big) + \phi(x) - \half x^\tp A x
\end{align*}
This is transformed into a standard form through a simple application of Dynkin's formula. The calculation for the same appears in Section~\ref{ssec:justification-vanHandel} where the following is shown:
\begin{equation} \label{eq:modified-value-function}
	\tilde{V}(t,x) = \min_{\tilde{U}\in \clU} \E\Big(\half \int_t^T |\tilde{U}_s|^2 + C(\tilde{X}_s) +G(\tilde{X}_s;z_s) \ud s +R(\tilde{X}_T) - z_T^\tp H^\tp \tilde{X}_T \mid \tilde{X}_t = x \Big) 
\end{equation}
where
\begin{align*}
	C(x)&:= |H^\tp x|^2 + |a(x)|^2 + \divg\big(a(x)\big) - \frac{1}{4}|(A-A^\tp)x|^2\\
	G(x;z) &:= z^\tp H^\tp (A-A^\tp)x - |Hz|^2\\
	R(x)&:= \phi(x) - \half x^\tp A x
\end{align*}
This shows that $\tilde{V}(t,x)$ is the value function for the following optimal control problem:
\begin{align*}
	\mathop{\text{Min }}_{U} &:\tilde{\sf J}(U\,;z) = \E\Big(\half \int_0^T |\tilde{U}_t|^2 + C(\tilde{X}_t) +G(\tilde{X}_t;z_t) \ud s +R(\tilde{X}_T) - z_T^\tp H^\tp \tilde{X}_T \Big)\\
	\text{Subj.} &:\ud \tilde{X}_t = \half(A-A^\tp)\tilde{X}_t \ud t -Hz_t\ud t + \tilde{U}_t \ud t+\ud \tilde{B}_t,\quad X_0\sim \pi_0 
\end{align*}
For an optimal control problem whose constraint is a linear system and the cost functions (both the running cost and the terminal cost) are convex, it is known that the value function is convex.
Specifically, the following result is deduced.

\medskip

\begin{proposition}[Proposition 4.4.1 in~\cite{van2006filtering}] \label{prop:vanHandel-filter-stability}
	Suppose $R(x)$ is $\kappa$-uniformly convex, and $C(x)$ is $\kappa'$-uniformly convex with $\kappa' \ge 2\kappa^2$. Then $\tilde{V}(t,x)$ is $\kappa$-uniformly convex. 
\end{proposition}


\medskip

\begin{remark}
	For this class of models, the first such filter stability results were obtained by Stannat~\cite{stannat2005stability,stannat2006} using PDE-based techniques.
	In his paper, Stannat supposes $A=0$ and defines~\cite[Remark 23.1]{stannat2006}
	\[
	\tilde{C}(x) := |Hx|^2 + \frac{\Delta \varphi}{\varphi}(x)
	\]
	where $\varphi = e^{-\phi}$. It is easily checked that $\tilde{C}(x) = C(x)$ when $A=0$. Then the filter stability for the model is deduced from the convexity of $\tilde{C}$. He remarks that:
	\begin{quote}
		``{\it Note that $\tilde{C}$ consists of two parts: the second part $\frac{\Delta \varphi}{\varphi}$ depends on the signal whereas the first part $|Hx|^2$ depends on our choice $H$ how to observe the signal. Basically, the more precise our observation is, the more convex $|Hx|^2$. Conversely, our criterion provides a priori lower bounds on our choice $H$ to reach a certain exponential rate $\kappa$. Also note that ergodic and nonergodic directions of the signal process can be ``separated'' in the criterion.}''
	\end{quote}
	
	The result in Prop.~\ref{prop:vanHandel-filter-stability} is not as strong because the assumption $\phi(x)-\half x^\tp A x$ is $\kappa$-uniformly necessarily requires the signal process to be stable. 		
	In his thesis, van Handel was able to re-derive the Stannat's result---convexity of observation compensates the instability of the signal---by specifying a suitable Mitter-Newton type optimal control problem for a time-reversed conditional signal.  This is done by fixing the terminal time $T>0$ and considering the time reversed signal $\bar{X}_t = X_{T-t}$ for $0\le t\le T$.  Then the time reversed signal is again a diffusion with
	\[
	\ud \bar{X}_t = - a(\bar{X}_t) \ud t + \nabla \log p_{T-t}(\bar{X}_t) \ud t + \ud \tilde{B}_t
	\] 
	where $p_t(x)$ is the unconditional density of $X_t$~\cite[Section 4.2]{van2006filtering}.  The time reversed controlled process is defined by introducing a control input as follows:
	\[
	\ud \tilde{X}_t = - a(\tilde{X}_t) \ud t + \nabla \log p_{T-t}(\tilde{X}_t)\ud t + u_t \ud t + \ud \tilde{B}_t
	\]
	By employing Mitter-Newton duality for such a controlled process, stronger results are obtained in~\cite{van2006filtering}.  
\end{remark}

\subsection{Justification of Eq.~\eqref{eq:modified-value-function}} \label{ssec:justification-vanHandel}

Dynkin's formula is applied to $\phi(x)$ and $\half x^\tp A x$.
\begin{align*}
	\E(\phi(\tilde{X}_T)\mid \tilde{X}_t=x) &= \phi(x) + \E\Big(\int_t^T (a(\tilde{X}_s)+U_s)^\tp \nabla \phi + \half \Delta \phi  \ud s\mid \tilde{X}_t=x \Big)\\
	\E(\tilde{X}_T^\tp A \tilde{X}_T\mid \tilde{X}_t = x) &= x^\tp A x + \E\Big(\int_t^T (a(\tilde{X}_s)+U_s)^\tp (A+A^\tp)x + \half \tr(A) \ud s \mid \tilde{X}_t = x\Big)
\end{align*}
Therefore,
\begin{align*}
	\tilde{V}(t,x) &= \min_{U\in\clU} \E\Big(\int_t^T \half|U_s|^2 + \half|H^\tp \tilde{X}_s|^2 + (a(\tilde{X}_s)+U_s)^\tp Hz_s-(a(\tilde{X}_s)+U_s)^\tp \nabla \phi(\tilde{X}_s) - \half \Delta \phi(\tilde{X}_s)\\
	&\qquad\qquad\qquad+\half (a(\tilde{X}_s)+U_s)^\tp(A+A^\tp)\tilde{X}_s + \frac{1}{4}\tr(A) \ud s +\phi(\tilde{X}_T) - \half \tilde{X}_T^\tp A \tilde{X}_T- z_T^\tp H\tilde{X}_T  \mid \tilde{X}_t = x \Big)\\
	&=\min_{U\in\clU} \E\Big(\half \int_t^T\tilde{\ell}(\tilde{X}_s,U_s;z_s) \ud s + R(\tilde{X}_T) - z_T^\tp H^\tp \tilde{X}_T\mid \tilde{X}_t = x\Big) 
\end{align*}
where the running cost is expressed by
\[
\tilde{\ell}(x,u;z) := |u|^2 + |H^\tp x|^2 + (a(x)+u)^\tp Hz - (a(x)+u)^\tp \nabla\phi(x) - \half \Delta \phi(x) + \half(a(x)+u)^\tp(A+A^\tp)x + \frac{1}{4}\tr(A)
\]
Use $a(x) = Ax-\nabla \phi$ and set
\[
\tilde{u} = -\nabla \tilde{V}(t,x) = u -\nabla \phi(x) +  \half(A+A^\tp)x + H z
\]
to transform this into
\begin{align*}
	\tilde{\ell}(x,u;z) = &|\tilde{u}|^2 - |\nabla \phi|^2 - \frac{1}{4}|(A+A^\tp)x|^2-|Hz|^2 + 2u^\tp \nabla \phi - u^\tp (A+A^\tp)x -2u^\tp Hz + \nabla\phi^\tp (A+A^\tp)x \\
	&+2 \nabla \phi^\tp Hz - x^\tp(A+A^\tp)Hz+|Hx|^2 + 2x^\tp A^\tp Hz - 2\nabla\phi^\tp Hz + 2u^\tp Hz- 2x^\tp A^\tp \nabla \phi \\
	&  + 2|\nabla\phi|^2 - 2u^\tp \nabla \phi -\Delta \phi+x^\tp A^\tp(A+A^\tp)x -\nabla \phi^\tp (A+A^\tp)x+u^\tp (A+A^\tp)x +\tr(A)\\
	=&|\tilde{u}|^2- \frac{1}{4}|(A+A^\tp)x|^2-|Hz|^2+z^\tp H^\tp (A-A^\tp)x 
	+|Hx|^2 -2x^\tp A^\tp \nabla \phi + |\nabla\phi|^2  -\Delta \phi \\
	&+x^\tp A^\tp(A+A^\tp)x+\tr(A)\\
	=&|\tilde{u}|^2+|Hx|^2 +|a(x)|^2+ \divg(a(x)) -\frac{1}{4}|(A-A^\tp)x|^2-|Hz|^2+z^\tp H^\tp (A-A^\tp)x \\
	=&|\tilde{u}|^2 + C(x)+G(x;z)
\end{align*}
\qed

\newpage


\chapter{Filter stability via duality}\label{ch:filter-stability}

Filter stability is investigated in this and the next chapter through the analysis of the dual optimal control problem.  As noted in Chapter~\ref{ch:filter-stability-literature}, an important first consideration is to choose a metric to compare $\pi_T^\mu$ and
$\pi_T^\nu$.  The most natural metric compatible with
the dual formulation is the $\chisq$-divergence.  This is because
$\chisq$-divergence equals the conditional variance of the RN derivative
$\frac{\ud \pi_T^\mu}{\ud \pi_T^\nu}=:\gamma_T$. (Recall that the expected value of the conditional variance is the optimal value function for the dual optimal control problem.)

From the Kalman filter theory, one expects that filter
stability is related to the stability of the dual optimal
control system.  This is indeed the case as formalized through a
simple and elegant expression
\begin{equation}\label{eq:filter-stab-intro}
	\E^\mu\big(\chisq(\pi_T^\mu\mid\pi_T^\nu)\big) = \cov_0^\nu(\gamma_0,Y_0)
\end{equation}
where $\gamma_0 = \frac{\ud \mu}{\ud \nu}$ and $Y_0$ is the optimal solution with $Y_T=\gamma_T$.  We assume $\clV_0^\nu(\gamma_0)<\infty$. The calculations to obtain~\eqref{eq:filter-stab-intro}
are straightforward and described in Section~\ref{sec:duality-and-filter-stability}.  The significance of~\eqref{eq:filter-stab-intro} is that filter stability will follow if $\clV_0^\nu(Y_0)\to 0$ as $T\to \infty$.  The latter may be interpreted as asymptotic
stability of the dual optimal control system.


Based on~\eqref{eq:filter-stab-intro}, the filter stability program
becomes to obtain necessary and sufficient condition for the model
such that $\clV_0^\nu(Y_0)\to 0$.  The interpretation of conditional variance as the value
function is useful for this purpose.  In this chapter, conditional
Poincar\'e inequality (PI) is introduced as a sufficient condition to
conclude $\clV_0^\nu(Y_0)\to 0$. Although the condition is strong, it
a counterpart of the Poincar\'e inequality which is central to the
subject of stochastic stability of Markov
processes~\cite{bakry2013analysis}. 
Using conditional PI, we are able to derive many prior results where
explicit convergence rate are available.       

The outline of the remainder of this chapter is as follows: Several
definitions of filter stability are reviewed in Section~\ref{sec:filter-stability-problem} based on
f-divergence to compare probability measures.  In Section~\ref{sec:duality-and-filter-stability}, the
formula~\eqref{eq:filter-stab-intro} is derived starting from the dual
optimal control problem.  The formula is used to derive the main
result on filter stability in Section~\ref{sec:bvi}.  In Section~\ref{sec:cPI}, the definition of conditional PI is introduced together with a number of examples where
conditional PI can be used to obtain explicit formula for convergence rates.

\section{Filter stability problem}\label{sec:filter-stability-problem}

We begin by recalling the definition of $f$-divergence to compare two 
probability measures.


\begin{definition}[$f$-divergence] \label{def:f-divergences}
	Suppose $\mu,\nu\in\clP(\bS)$ and $\mu\ll\nu$. Let $\gamma = \frac{\ud 
		\mu}{\ud \nu}$. Then
	\begin{align*}
		\text{(KL divergence)}\qquad& \kl(\mu\mid \nu) := \int_\bS \gamma \log(\gamma)\ud \nu\\
		\text{($\chi^2$ divergence)}\qquad& \chisq(\mu\mid \nu) := \int_\bS (\gamma - 1)^2\ud \nu\\
		\text{(Total variation)}\qquad& \|\mu-\nu\|_\tv = \int_\bS \half |\gamma-1|\ud \nu
	\end{align*}
\end{definition}

It is noted that $\chisq(\mu\mid\nu) = \clV_0^\nu(\gamma)$. The relationship 
between these is given in the following Lemma:

\medskip

\begin{lemma}[Lemma 2.5 and 2.7 
	in~\cite{Tsybakov2009estimation}]\label{lm:f-divergence-inequality}
	For $\mu\ll\nu$, the following inequalities hold:
	\[
	2\|\mu-\nu\|_\tv^2 \le \kl(\mu\mid \nu)\le \chisq(\mu\mid\nu)
	\]
	The first inequality is called the Pinsker's inequality.
\end{lemma}


For $\mu\ll\nu$, we define a $\clZ_T$-measurable function $\gamma_T:\bS\to \Re$ 
by
\begin{equation}\label{eq:gamma_t}
	\gamma_T(x) := \frac{\ud \pi_T^\mu}{\ud \pi_T^\nu}(x),\quad x\in \bS
\end{equation}
The RN derivative is well-defined because $\pi_T^\mu\ll\pi_T^\nu$ (see 
Lemma~\ref{lm:change-of-Pmu-Pnu}).
The following definition of filter stability is based on $f$-divergence:

\medskip

\begin{definition}
	The nonlinear filter is \emph{stable} in the sense of 
	\begin{align*}
		\text{(KL divergence)}\qquad& \E^\mu\big(\kl(\pi_T^\mu\mid \pi_T^\nu)\big) \; \longrightarrow\; 0\\
		\text{($\chi^2$ divergence)}\qquad& \E^\mu\big(\chisq(\pi_T^\mu\mid \pi_T^\nu)\big) \; \longrightarrow\; 0\\
		\text{(Total variation)}\qquad& \E^\mu\big( \|\mu-\nu\|_\tv\big) \; \longrightarrow\; 0
	\end{align*}
	as $T\to \infty$ for every $\mu, \nu\in\clP(\bS)$ such that $\mu\ll \nu$.
	
\end{definition}

\medskip

In the above, $\nu$ has the meaning of (incorrect) prior used in computing the filter. The absolutely continuous measure $\mu$ is the true (possibly unknown) prior.
Because $\mu$ is the correct prior, the filter performance is evaluated with respect to $\sP^\mu$.
Apart from $f$-divergence based definitions, the following definitions of 
filter stability are also of historical interest.

\medskip

\begin{definition}\label{def:FS}
	The nonlinear filter is stable in the sense of
	\begin{align*}
		\text{($L^2$)}\qquad & \E^\mu\big(|\pi_T^\mu(f)-\pi_T^\nu(f)|^2\big) \;\longrightarrow\; 0\\
		\text{(a.s.)}\qquad & |\pi_T^\mu(f) - \pi_T^\nu(f)|\;\longrightarrow\; 0\quad \sP^\mu\text{-a.s.}
	\end{align*}
	as $T\to \infty$, for every $f\in C_b(\bS)$ and $\mu,\nu\in\clP(\bS)$ such that $\mu \ll \nu$.
	
\end{definition}

\medskip

It is shown in this chapter that the dual optimal control formulation most 
directly yields filter stability in the sense of $\chisq$ divergence. This is because of the connection between the $\chisq$-divergence and the conditional variance whereby $\chisq(\pi_T^\mu\mid\pi_T^\nu) = \clV_T^\nu(\gamma_T)$. Based on 
Lemma~\ref{lm:f-divergence-inequality}, this also implies other types of 
stability. The proof of the following theorem is in Section~\ref{ss:pf-prop71}.

\medskip

\begin{proposition}\label{prop:chisq-stability-implication}
	If the filter is stable in the sense of $\chisq$ then it is also stable in the sense of KL divergence, total variation, and $L^2$.
\end{proposition}

\medskip

%

\begin{proposition}[Theorem 2.3 in~\cite{clark1999relative}] \label{prop:kl-is-supermg}
	The process $\big\{\kl(\pi_t^\mu\mid\pi_t^\nu): 0\le t \le t\big\}$ is a $\sP^\mu$-supermartingale. Consequently, if the filter is stable in the sense of KL divergence then it is also stable almost surely.
\end{proposition}

\medskip

In the remainder of this chapter, we assume the following:

\begin{assumption}\label{ass:gamma-0}
	Two initial measures $\mu,\nu\in\clP(\bS)$ satisfy $\mu\ll\nu$ and
	\begin{align*}
		0< \underline{a} &:= \mathop{\operatorname{essinf}}_{x\in\bS} \frac{\ud \mu}{\ud \nu}(x) > 0\\
		\overline{a} &:= \mathop{\operatorname{esssup}}_{x\in\bS} \frac{\ud \mu}{\ud \nu}(x) <\infty		
	\end{align*}
\end{assumption}

Using the formula~\eqref{eq:gamma_T-explicit}, we then also have
\[
\gamma_T(x) \le \frac{\overline{a}}{\underline{a}},\quad \forall\, x\in\bS,\; \sP^\nu\text{-a.s.}
\]
In particular, $\gamma_T \in L^2_{\clZ_T}\big(\Omega;C_b(\bS)\big)$.

\section{Formula for $\chisq$ divergence}\label{sec:duality-and-filter-stability}

In the setting of this thesis, a filter is obtained by solving the dual optimal 
control problem. A user who (incorrectly) believes the prior to be $\nu$ solves 
the optimal control problem under the measure $\sP^\nu$:
%
\begin{align*}
	\mathop{\text{Minimize}}_{U\in\clU}\text{:}\quad\quad\sJ_T^\nu(U) &= \E^\nu\Big(|Y_0(X_0)-\nu(Y_0)|^2 + \int_0^T  (\Gamma Y_t)(X_t) + |U_t + V_t(X_t)|^2 \ud t
	\Big) \\
	\text{Subject to:}\; -\ud Y_t(x) &= \big((\clA Y_t)(x) + h^\tp(x)(U_t+V_t(x))\big)\ud t - V_t^\tp(x)\ud Z_t,\quad Y_T(x) = F(x),\; x\in\bS
\end{align*}
Note the two changes from the correctly initialized problem: (1) the 
expectation is now with respect to $\sP^\nu$, and (2) $\nu(Y_0)$ appears in the 
terminal cost. The optimal control for this problem is denoted $U^\nu$. From 
the duality principle (see Remark~\ref{rm:mean and the variance})
\[
\pi_T^\nu(F) = \nu(Y_0) - \int_0^T \big(U_t^\nu\big)^\tp \ud Z_t,\quad \tsP^\nu\text{-a.s.}
\]
To obtain the formula for the divergence, we consider the 
$\sJ_T^\nu(\cdot)$ problem with $F = \gamma_T$.
Because
\[
\pi_T^\nu(\gamma_T) = \int_\bS \frac{\ud \pi_T^\mu}{\ud \pi_T^\nu}(x) \ud \pi_T^\nu(x) = \int_\bS \ud \pi_T^\mu(x) = 1
\]
we have
\[
1 = \nu(Y_0) - \int_0^T \big(U_t^\nu\big)^\tp \ud Z_t,\quad \tsP^\nu\text{-a.s.}
\]
and therefore by the uniqueness of the It\^o representation,
\begin{align*}
	U_t^\nu &= 0,\quad \tsP^\nu\text{-a.s.}\\
	\nu(Y_0) &= 1
\end{align*}
The optimal trajectory $(Y,V)$ is then the solution of the BSDE
\begin{equation}\label{eq:backward-filter-stability}
	-\ud Y_t(x) = \big((\clA Y_t)(x) + h^\tp(x)V_t(x)\big) - V_t^\tp(x) \ud Z_t,\quad Y_T(x) = \gamma_T(x)
\end{equation}
Now because $\chisq(\pi_T^\mu\mid \pi_T^\nu) = \clV_T^\nu(\gamma_T)$, a simple 
calculation reveals
\[
\chisq(\pi_T^\mu\mid \pi_T^\nu) = \clV_T^\nu(\gamma_T) = \pi_T^\nu(\gamma_T^2) 
- 1 = \pi_T^\mu(\gamma_T) - 1
\]
Therefore, useful insight may be obtained by considering the process 
$\{\pi_t^\mu(Y_t):0\le t\le T\}$. The calculation for the same, based on using  
It\^o-Wentzell formula for $\pi_t^\mu(Y_t)$, is contained in  
Section~\ref{ssec:pf-gamma-martingale}. It is shown that the process is 
$\sP^\mu$-martingale and consequently, 
\[
\E^\mu\big(\pi_T^\mu(\gamma_T)\big) = \mu(Y_0)
\]
It follows that
\begin{align*}\label{eq:chisq-main-result}
	\E^\mu\big(\chisq(\pi_T^\mu\mid\pi_T^\nu)\big) &= \mu(Y_0) - 1 \\ 
	& =\nu\big(\gamma_0 Y_0\big)-\nu(\gamma_0)\nu(Y_0) \\
	&= \cov_0^\nu(\gamma_0,Y_0)
\end{align*}
This proves~\eqref{eq:filter-stab-intro}.
By the Cauchy-Schwarz inequality, 
\begin{equation}\label{eq:CS-bound}
	\E^\mu\big(\chisq(\pi_T^\mu\mid\pi_T^\nu)\big)  \le \sqrt{\clV_0^\nu(\gamma_0)\clV_0^\nu(Y_0)}
\end{equation}
Therefore, for any $\mu\in\clP(\bS)$ with $\clV_0^\nu(\gamma_0) <\infty$, 
the filter stability in $\chisq$ divergence is obtained if it can be shown that
\begin{equation}\label{eq:necessary-variance-decay}
	\clV_0^\nu(Y_0)\;\longrightarrow\; 0\quad \text{as }T\to \infty
\end{equation}
Note that we have not yet used the equation for the conditional variance. This is the 
subject 
of the remainder of this chapter where~\eqref{eq:necessary-variance-decay} is 
shown to hold under certain additional assumptions on the model. Before doing 
so, we make some remarks.

\medskip

\begin{remark}[Forward equation for divergence]
	A reader may wonder whether the formula for $\chisq$-divergence can also be 
	derived more directly without the use of duality. Indeed, because the 
	equations for $\pi^\mu$ and $\pi^\nu$ are known, a direct application of 
	the It\^o formula (see formal calculation in 
	Section~\ref{ssec:pf-gamma-martingale}) is used to 
	derive the following forward equation:
	\begin{equation}\label{eq:forward-equation-chisq}
		\ud \chisq(\pi_t^\mu\mid \pi_t^\nu) = -\pi_t^\nu\big(\Gamma \gamma_t\big)\ud t
		-\cov_t^\mu(\gamma_t,h)\,\cov_t^\nu(\gamma_t,h)
		\ud t  + (\cdots)\ud I_t^\mu
	\end{equation}
	The first term $-\pi_t^\nu(\Gamma \gamma_t)$ is non-positive. It 
	has not been possible to determine the sign of the product term
	$\cov_t^\mu(\gamma_t,h)\,\cov_t^\nu(\gamma_t,h)$.
	It may be possible to express the equation in a more amenable form through 
	a clever choice of integrating factor. However, we have not been successful 
	in this endeavor.
\end{remark}

\medskip

\begin{remark}[Stochastic stability]
	
	A special case of the filter stability is when the observation function 
	$h = c\ones$ (a constant function). In this case, $\pi^\mu$ and $\pi^\nu$ are 
	deterministic processes obtained as solutions of the forward Kolmogorov equation starting from prior $\mu$ and $\nu$, respectively. 
	Equation~\eqref{eq:forward-equation-chisq} is now an ODE
	\[
	\frac{\ud }{\ud t} \,\chisq(\pi_t^\mu\mid\pi_t^\nu) = -\pi_t^\nu\big(\Gamma 
	\gamma_t\big)
	\]
	In the study of Markov processes, a standard  
	assumption is that there exists an invariant measure $\bmu$ and 
	a certain \emph{Poincar\'e inequality} (PI) holds. For this purpose, we 
	define the following: 
	\begin{align*}
		\text{(energy):}\qquad\den^\bmu(f) &:= \bmu\big(\Gamma f)\\
		\text{(variance):}\qquad\dvar^\bmu(f) &:= \bmu\big(|f - \bmu(f)|^2\big) 
	\end{align*}
	The PI relates the two as follows:
	\begin{equation}\label{eq:standard-PI}
		\text{(PI)}\qquad\qquad	\den^\bmu(f) \geq c \;
		\dvar^\bmu(f)\quad \forall\, f \in \clD
	\end{equation}
	where $\clD$ is a suitable set of test functions.
	Setting $\nu = \bmu$, the equation for $\chisq$-divergence becomes 
	\[
	\frac{\ud}{\ud t} \dvar^\bmu(\gamma_t) = -\den^\bmu(\gamma_t)
	\]
	and using the PI,
	\[
	\dvar^\bmu(\gamma_T)\le e^{-cT}\dvar^\bmu(\gamma_0)
	\]
	It is also entirely straightforward to obtain this formula by introducing a dual 
	deterministic process. This calculation is included Section~\ref{ssec:pf-gamma-martingale}, mainly for the purpose of comparing stochastic stability and filter stability. However, because the forward calculations are also easy and standard, the utility for doing so is questionable.
	
	A stability condition weaker than PI is as follows:
	\[
	\den^\bmu(f) = 0\quad \Longrightarrow\quad \dvar^\bmu(f)=0
	\]
	In finite state-space settings, this implies a Poincar\'e constant $c > 0$. 
	In more general settings, under certain additional compactness assumptions, this condition is used to describe the ergodicity of the Markov process~\cite[Remark 4.2.2]{bakry2013analysis}.
	
	%
	%
	%
	
\end{remark}

%

\section{Filter stability using dual formulation}\label{sec:bvi}

The goal is to establish a sufficient conditions such 
that~\eqref{eq:necessary-variance-decay} holds.  For this purpose, the 
interpretation of the conditional variance as the optimal value of the dual 
optimal control problem (see Prop.~\ref{prop:optimal-solution}) is useful. In particular,
\begin{equation}\label{eq:variance-equality}
	\clV_0^\nu(Y_0) + \E^\nu\Big(\int_0^t\pi_s^\nu(\Gamma Y_s) + 
	\pi_s^\nu\big(|U_s^\nu + V_s(\cdot)|^2\big)\ud s\Big) = 
	\E^\nu\big(\clV_t^\nu(Y_t)\big),\quad 0\le t\le T
\end{equation}
The above shows that the deterministic process $\big\{\E^\nu\big(\clV_t^\nu(Y_t)\big):0\le t\le T\big\}$ is non-decreasing as a function of $t$. Under suitable assumptions on the model $(\clA,h)$, we assert that the following  \emph{backward variance inequality} holds:
\begin{equation}\label{eq:variance-inequality-2}
	\clV_0^\nu(Y_0) \le e^{-cT}\E^\nu\big(\clV_T^\nu(Y_T)\big)
\end{equation}
Note that the inequality trivially holds with a constant $c=0$.

The backward variance inequality~\eqref{eq:variance-inequality-2} is important because combined with~\eqref{eq:CS-bound}, it yields the following result on filter stability whose proof appears in Section~\ref{ss:pf-thm71}:

\medskip

\begin{theorem}\label{thm:filter-stability-chisq-RT}
	Suppose~\eqref{eq:variance-inequality-2} holds. Then
	\begin{equation*}
		\underline{a}\,\E^\mu\big(\chisq(\pi_T^\mu\mid\pi_T^\nu)\big) \le 
		e^{-cT}\clV^\nu_0(\gamma_0)
	\end{equation*}
	i.e.~the filter is stable in $\chisq$ divergence.
\end{theorem}

\medskip

\begin{remark}\label{rm:better-RT-formula}
	The proof of Theorem~\ref{thm:filter-stability-chisq-RT} is presented in a slightly more general form where we conclude	
	\begin{equation*}
		R_T\,\E^\mu\big(\chisq(\pi_T^\mu\mid\pi_T^\nu)\big) \le 
		e^{-cT}\clV^\nu_0(\gamma_0)
	\end{equation*}
	where
	\[
	R_T = 
	\frac{\E^\mu\big(\clV^\nu_T(\gamma_T)\big)}{\E^\nu\big(\clV^\nu_T(\gamma_T)\big)}
	\]
	A conservative lower bound $R_T\ge \underline{a}$ is then used.
	An alternative formula for the ratio $R_T$ is obtained by consider the following change of measure
	(see Section~\ref{ssec:pf-clark-result} for the derivation):
	\[
	A_T:=\frac{\ud \sP^\mu|_{\clZ_T} }{\ud \sP^\nu|_{\clZ_T}}= 
	\exp\Big(\int_0^T \big(\pi_t^\mu(h)-\pi_t^\nu(h)\big)\ud I_t^\mu - 
	\half\int_0^T \big|\pi_t^\mu(h)-\pi_t^\nu(h)\big|^2\ud t\Big)
	\]
	Since $\clV_T^\nu(\gamma_T)$ is $\clZ_T$-measurable random variable, the change of measure formula gives $\E^\mu\big(\clV_T^\nu(\gamma_T)\big) = \E^\nu\big(A_T\clV_T^\nu(\gamma_T)\big)$, and therefore
	\[
	R_T 
	=\frac{\E^\nu\big(A_T\clV_T^\nu(\gamma_T)\big)}{\E^\nu\big(\clV_T^\nu(\gamma_T)\big)}
	\]
	is now expressed as a ratio where both the numerator and denominator are expectations  with respect to the same measure $\sP^\nu$.
	
	It is easily seen that $\{A_T:T\geq 0\}$ is a positive $\sP^\nu$-martingale with $\E^\nu(A_T)=\E^\nu(A_0)=1$. By martingale convergence theorem~\cite[Theorem 3.21]{le2016brownian}, there exists a random variable $A_\infty$ such that
	\[
	A_T\;\longrightarrow\; A_\infty \quad \text{as }T\to \infty,\; \sP^\nu\text{-a.s.}
	\]
	It is possible that the resulting formula for the ratio can be manipulated to obtain a better asymptotic bound for the constant $a$. 
\end{remark}

\medskip

Based on the result of Theorem~\ref{thm:filter-stability-chisq-RT}, a quantitative analysis of the filter stability is possible by establishing backward variance 
inequality~\eqref{eq:variance-inequality-2} where the constant $c$ depends upon the model parameters.
Some partial results along this line of inquiry appear as part of the following section.

\section{Conditional Poincar\'e inequality}\label{sec:cPI}

In this section, we provide a sufficient condition for the backward variance inequality~\eqref{eq:variance-inequality-2}. 
Starting from~\eqref{eq:variance-equality}, note that irrespective of $U_t$ and $V_t$, we have
\[
\E^\nu\big(\clV_t^\nu(Y_t)\big) + \int_t^T\E^\nu\big(\pi_s^\nu(\Gamma Y_s)\big)\ud s \le 
\E^\nu\big(\clV_T^\nu(Y_T)\big),\quad 0\le t \le T
\]
Therefore a natural condition to obtain~\eqref{eq:variance-inequality-2} is
\begin{equation}\label{eq:cPI-for-Y}
	\E^\nu\big(\pi_t^\nu(\Gamma Y_t)\big) \ge c\, \E^\nu\big(\clV_t^\nu(Y_t)\big),\quad 0\le t \le T
\end{equation}
Indeed, suppose~\eqref{eq:cPI-for-Y} holds. Then
\[
\E^\nu\big(\clV_t^\nu(Y_t)\big) + c\,\int_t^T\E^\nu\big(\clV_s^\nu(Y_s)\big)\ud s \le 
\E^\nu\big(\clV_T^\nu(Y_T)\big),\quad 0\le t \le T
\]
which~\eqref{eq:variance-inequality-2} follows from an application of Gronwall. This motivates the following definition:

\medskip

\begin{definition}\label{def:cPI}
	Suppose $\nu\in\clP(\bS)$ and $\clZ = \{\clZ_t:t\ge 0\}$ is a filtration. The model satisfies the \emph{conditional Poincar\'e inequality} (cPI) if there exists a constant $c > 0$ such that
	\[
	\E^\nu\big(\pi_t^\nu(\Gamma F)\big) \ge c\, \E^\nu\big(\clV_t^\nu(F)\big)
	\]
	for all test functions $F(x,\omega) = f(x)\ones_A(\omega)$ where $f\in\clD$ and $A \in \clZ_t$.
\end{definition}

\medskip

Clearly, if the system satisfies the conditional PI then~\eqref{eq:cPI-for-Y} and thus the backward variance inequality~\eqref{eq:variance-inequality-2} follows.
A more general result, described in the following proposition, is obtained by considering the martingale for the dual optimal control problem (see Theorem~\ref{thm:martingale}). 
The proof of the following proposition appears in Section~\ref{ss:pf-prop73}.

\medskip

\begin{proposition}\label{prop:bvi-pathwise}
	Suppose $\beta = \{\beta_t \ge 0:t\ge 0\}$ is a $\clZ$-adapted process such that 
	\begin{equation}\label{eq:beta-process}
		\pi_t^\nu \big(\Gamma f\big) \ge \beta_t \clV_t^\nu(f)\qquad \forall\, f\in \clD,\;\sP^\nu\text{-a.s.},\;0\le t\le T
	\end{equation}
	Then the backward inequality of the form
	\begin{equation*}\label{eq:bvi}
		\clV_0^\nu(Y_0) \le  \E^\nu\Big[\exp\Big(-\int_0^T\beta_t \ud 
		t\Big)\clV_T^\nu(Y_T)\Big]
	\end{equation*}
\end{proposition}

Consequently, if 
\[
\frac{1}{T} \int_0^T \beta_t \ud t \; \longrightarrow\; c > 0
\]
then the backward variance inequality~\eqref{eq:variance-inequality-2} for the conditional variance is obtained asymptotically. The example~\ref{ex:min_row} in Section~\ref{sec:examples} considers such a case.

\medskip

The inequality~\eqref{eq:beta-process} is the pathwise version of the conditional PI. Note the pathwise inequality needs to be specified only for deterministic functions. The following proposition provides an alternative description for the conditional PI. The proof appears in Section~\ref{ss:pf-prop74}.

\medskip

\begin{proposition}\label{prop:cpi-equivalent-def}
	The nonlinear model $(\clA,h)$ satisfies cPI with constant $c$ if and only if 
	\[
	\beta_t \ge c > 0,\quad \sP^\nu\text{-a.s.},\quad \forall \, t \ge 0
	\]
\end{proposition}

\medskip



\begin{remark}
	The conditional PI trivially holds if standard PI holds for all $\rho \in \clP(\bS)$ with uniform constant $c>0$, that is,
	\[
	\rho\big(\Gamma f\big) \ge c\, \dvar^\rho(f),\quad \forall\,f \in \clD,\; \rho \in \clP(\bS)
	\]
	This appears to be a very strong requirement, but certain mixing conditions 
	in finite case indeed satisfies uniform PI. 
	Some examples of Markov processes that satisfies uniform PI are presented in 
	the following section.
	
\end{remark}

\subsection{Examples of {Poincar\'e} inequality}\label{sec:examples}

We begin by noting that Def.~\ref{def:cPI} is stated for a general class of filtrations ($\clZ_t$ not necessarily defined according to the model~\eqref{eq:obs-model}). 
One may conjecture that the filter ``inherits'' the standard PI~\eqref{eq:standard-PI} from the underlying Markov process. (Note the standard PI is for $\nu=\bmu$, the invariant measure).

In the general settings, the conditional PI holds for deterministic functions $f\in \clD$. This is because
\begin{align*}
	\E^\bmu\big(\pi_t^\bmu\big(\Gamma f\big)\big) = \bmu\big(\Gamma f\big) \ge  c \, \E^\bmu \big(|f(X_T) - \bmu(f)|^2\big)  \geq c \, \E^\bmu\big(\clV_T^\bmu(f) \big)
\end{align*}
This shows that the PI and also the constant $c$ is inherited on the subset of deterministic functions. 
However, with general types of filtration, it may not hold for random 
functions. This is shown with the aid of the famous counter-example of filtering theory reviewed in Section~\ref{sec:counter-example}:

\medskip

\begin{example}\label{ex:counter-example-revisit}  
	The state-space	$\bS=\{1,2,3,4\}$ and the rate matrix 
	\[
	A = \begin{pmatrix}
		-1 & 1 & 0 & 0\\
		0 & -1 & 1 & 0\\
		0 & 0 & -1 & 1\\
		1 & 0 & 0 & -1
	\end{pmatrix}
	\]
	whose unique invariant measure $\bmu =
	[\frac{1}{4},\frac{1}{4},\frac{1}{4},\frac{1}{4}]$.
	It is readily verified that the standard PI holds with a constant $c= 2$. 
	Consider a sigma-algebra $\clZ_T = \sigma([X_T \in\{1,3\}])$ along with
	a $\clZ_T$-measurable function:
	\[
	F(\cdot) = \begin{cases}
		\begin{pmatrix} 1 & 1 & -1 & -1\end{pmatrix} & \text{if}  \; \; X_T \in\{1,3\}  \\[4pt]
		\begin{pmatrix} -1 & 1 & 1 & -1 \end{pmatrix} & \text{if}  \; \;  X_T
		\in\{2,4\} 
	\end{cases} 
	\]
	Then the conditional distribution
	\[
	\pi_T^{\bmu}(\cdot)
	= \begin{cases}
		\begin{pmatrix} \half  & 0 & \half  & 0\end{pmatrix} & \text{if}  \; \; X_T \in\{1,3\}  \\[4pt]
		\begin{pmatrix} 0 & \half  & 0 &
			\half  \end{pmatrix} & \text{if}  \; \;  X_T
		\in\{2,4\} 
	\end{cases} 
	\]
	The conditional mean $\pi_T^{\bmu}(F)=0$, the conditional
	variance $\pi_T^{\bmu}(F^2)=1$, and therefore
	$\E^\bmu\big(\clV_T^\bmu(F)\big)=1$.  On the other hand, the energy
	$
	\E^\bmu\big(\clE_T^\bmu(F)\big) = 0
	$.  
	Therefore, the conditional PI does not hold for this example.  
\end{example}

The following examples are the cases where the conditional PI holds with certain constants. We use the following formulae for the energy and variance in finite case:
\begin{align*}
	\text{(energy):}\qquad\rho(\Gamma f) &= \sum_{i,j\in \bS} \rho(i)A(i,j)(f(i)-f(j))^2\\
	\text{(variance):}\quad\dvar^\rho(f) &= \half \sum_{i,j\in\bS} \rho(i)\rho(j) (f(i)-f(j))^2
\end{align*}

\begin{example}[uniform PI for 2-state case] \label{ex:2-states} 
	Consider the simplest case with $\bS= \{1,2\}$ and
	$
	A = \begin{bmatrix} -a_1 & a_1 \\ a_2 & - a_2 \end{bmatrix}
	$
	is irreducible.  Then $a_1>0$ and $a_2> 0 $. Observe that
	\begin{align*}
		\rho(\Gamma f) &= \big(a_1 \rho(1) + a_2 \rho(2)\big)(f(1)-f(2))^2\\
		\dvar^\rho(f) &= \rho(1)\rho(2)(f(1)-f(2))^2
	\end{align*}
	Let $\rho = (p,1-p)$ then uniform PI holds
	\[
	c = \min_{p\in(0,1)} \frac{a_1 p + a_2 (1-p)}{p(1-p)}
	\]
	It attains its minimum at $p = \frac{-a_2+\sqrt{a_1a_2}}{a_1-a_2}$, and the minimum value is $c = a_1+a_2 + 2\sqrt{a_1a_2}$. Hence, the uniform PI holds for every irreducible 2-state Markov chain.  
	Note that the best constant for standard PI is
	$2(a_1+a_2)$ which is strictly greater than  $c$ unless $a_1 = a_2$.
	
\end{example}

\medskip

\begin{example}[uniform PI for Doeblin case] \label{ex:exdoeblin} 
	A Markov chain is {\em Doeblin} if there exist a state $j^*\in \bS$
	such that $A(i,j^*)$ is bounded away from 0 for all $i\in \bS \setminus \{j^*\}$. It is related to the strong mixing condition of Markov chain~\cite[Assumption 4.3.24]{Moulines2006inference}. 
	
	In this case,
	\begin{align*}
		\rho(\Gamma f) & = \sum_{i,j \in \bS} \rho(i)A(i,j) (f(i) - f(j))^2 \\
		&\geq  \sum_{j \in \bS} \min_{i:\;i\neq j} A(i,j) \sum_{i\in\bS} \rho(i)
		(f(i) - f(j))^2 \\
		&\geq \Big( \sum_{j \in \mathbb{S}} \min_{i:\;i\neq j} A(i,j) \Big) \dvar^\rho(f)
	\end{align*}
	Therefore, uniform PI holds with 
	\[
	c = \sum_{j} \min_{i\in\bS:\;\; i\neq j} A(i,j)
	\]
	and $c>0$ for Doeblin case.
\end{example}

\medskip

\begin{example}[uniform PI with square-root constant] \label{ex:minsqrt} 
	Because algebraic mean dominates the geometric mean,
	\begin{align*}
		\rho(\Gamma f)&=\sum_{i,j\in \bS} \rho(i)A(i,j)(f(i)-f(j))^2\\
		&=\sum_{i,j\in \bS} \frac{\rho(i)A(i,j)+\rho(j)A(j,i)}{2}(f(i)-f(j))^2\\
		&\ge\sum_{i,j\in \bS} \sqrt{\rho(i)\rho(j)}\sqrt{A(i,j)A(j,i)}(f(i)-f(j))^2\\
		&\ge \big(\min_{i\neq j} \sqrt{A(i,j)\,A(j,i)}\big) \sum_{i,j\in\bS} \rho(i)\rho(j) (f(i)-f(j))^2\\
		&= c\;\dvar^\rho(f) 
	\end{align*}
	where we used the fact that $\sqrt{x} \geq x$ for $0\leq x\leq 1$.
	Therefore, conditional PI holds with 
	\[
	c = \min_{i\neq j} 2\sqrt{A(i,j)\,A(j,i)}
	\]
	provided this is positive. This rate of convergence appears in literature:~\cite[Theorem 6]{atar1997lyapunov},~\cite[Theorem 4.3]{baxendale2004asymptotic} and~\cite[Corollary 2.3.2]{van2006filtering}.  	
\end{example}

\medskip

The following example illustrates the bound in Prop.~\ref{prop:bvi-pathwise}.

\medskip

\begin{example}[Asymptotic cPI] \label{ex:min_row} 
	Set $\nu = \bmu$. It is a straightforward calculation to verify
	\[
	\pi_T^\bmu\big(\Gamma F\big) \geq \Big(\sum_{i\in\bS} \pi_T^{\bmu}(i) \min_{j\in\bS:\;\; i\neq j}
	A(i,j) \Big) \; \clV_T^\bmu(F),\quad \forall\;T\geq 0
	\]
	Set $\beta_t = \sum_{i} \pi_t^{\bmu}(i) \min_{j:i\neq
		j}A(i,j) $.  Using~\cite[Eq.~(5.15)]{baxendale2004asymptotic}, it is
	known that
	\[
	\lim_{T\to \infty} \frac{1}{T}\int_0^T \beta_t \ud t =  \sum_{i\in\bS}
	\bmu(i) \min_{j\in\bS:\;i\neq j}A(i,j) \quad \text{a.s.}
	\]
	The asymptotic convergence rate also appears
	in~\cite[Theorem 4.2]{baxendale2004asymptotic}.   
\end{example}

\section{Proofs of the statements}

\subsection{Proof of Proposition~\ref{prop:chisq-stability-implication}} \label{ss:pf-prop71}

KL divergence and total variation result follows the Lemma~\ref{lm:f-divergence-inequality}.


\medskip

For $L^2$ stability, observe that for any $f \in C_b(\bS)$,
\[
\pi_T^\mu(f) - \pi_T^\nu(f) = \pi_T^\nu(f\gamma_T) - 
\pi_T^\nu(f)\pi_T^\nu(\gamma_T)
\]
Therefore by Cauchy-Schwarz inequality,
\[
|\pi_T^\mu(f)-\pi_T^\nu(f)|^2 \le \clV_T^\nu(\gamma_T)\clV_T^\nu(f) \le 
\frac{\operatorname{osc}(f)}{4} \clV_T^\nu(\gamma_T)
\]
where $\operatorname{osc}(f) = \sup f - \inf f$ denotes the oscillation of 
$f$. Taking $\E^\mu(\cdot)$ on both sides yields the conclusion. \qed

%
%

\subsection{Justification of the computations in Section~\ref{sec:duality-and-filter-stability}}\label{ssec:pf-gamma-martingale}

\subsubsection{$\pi_t^\mu(Y_t)$ is a $\sP^\mu$-martingale}
Apply It\^o-Wentzell theorem for measures~\cite[Theorem 1.1]{krylov2011ito} on $\pi_t^\mu(Y_t)$ where $\pi^\mu$ is the solution of the nonlinear filter~\eqref{eq:nonlinear-filter} and $Y$ a solution to~\eqref{eq:backward-filter-stability}:
\begin{align*}
	\ud \pi_t^\mu(Y_t) 
	&= \pi_t^\mu\big(h Y_t + V_t\big)\ud Z_t - \pi_t^\mu(Y_t)\pi_t^\mu(h) \ud Z_t \\
	&\quad + \pi_t^\mu(Y_t)\pi_t^\mu(h) \pi_t^\mu(h)\ud t -
	\pi_t^\mu\big(h Y_t + V_t\big)\pi_t^\mu(h)\ud t\\
	&= \big(\pi_t^\mu(h Y_t)-\pi_t^\mu(Y_t)\pi_t^\mu(h) + \pi_t^\mu(V_t)\big)(\ud Z_t - \pi_t^\mu(h)\ud t)
\end{align*}
The claim follows because $\ud Z_t - \pi_t^\mu(h)\ud t = \ud I_t^\mu$ is the innovation increment and $I^\mu$ is a $\sP^\mu$-martingale.
\qed

\subsubsection{Justification of Eq.~\eqref{eq:forward-equation-chisq}}

Here we consider the Euclidean case where $\pi_t^\mu$ and $\pi_t^\nu$ have probability densities. In this case,
\[
\chisq\big(\pi_t^\mu\mid\pi_t^\nu\big) = \clV_t^\nu(\gamma_t) = \int \Big(\frac{\pi_t^\mu(x)}{\pi_t^\nu(x)}\Big)^2 \pi_t^\nu(x)\ud x -1 = \int \frac{(\pi_t^\mu(x))^2}{\pi_t^\nu(x)}\ud x-1
\]
The Kushner's equation is 
\[
\ud \pi_t = \clA^\dagger \pi_t \ud t + (h-\pi_t(h))\pi_t \ud I_t
\]
Note that $\ud I_t^\nu = \ud I_t^\mu + \big(\pi_t^\mu(h)-\pi_t^\nu(h)\big)\ud t$.
By It\^o's formula,
\begin{align*}
	\ud \Big(\frac{(\pi_t^\mu)^2}{\pi_t^\nu}\Big) &= \frac{2\pi_t^\mu}{\pi_t^\nu} \big(\clA^\dagger \pi_t^\mu \ud t + (h-\pi_t^\mu(h))\pi_t^\mu \ud I_t^\mu\big)
	- \frac{(\pi_t^\mu)^2}{(\pi_t^\nu)^2}\big(\clA^\dagger \pi_t^\nu \ud t + (h-\pi_t^\nu(h))\pi_t^\nu \ud I_t^\nu\big)\\
	&\quad + \frac{1}{\pi_t^\nu} (h-\pi_t^\mu(h))^2 (\pi_t^\mu)^2 \ud t
	- \frac{2\pi_t^\mu}{(\pi_t^\nu)^2}(h-\pi_t^\mu(h)) \pi_t^\mu(h-\pi_t^\nu(h))\pi_t^\nu \ud t\\
	&\quad + \frac{(\pi_t^\mu)^2}{(\pi_t^\nu)^3} (h-\pi_t^\nu(h))^2 (\pi_t^\nu)^2 \ud t\\
	&=\frac{2\pi_t^\mu}{\pi_t^\nu}\clA^\dagger \pi_t^\mu \ud t -  \frac{(\pi_t^\mu)^2}{(\pi_t^\nu)^2}\clA^\dagger \pi_t^\nu \ud t
	+ \frac{(\pi_t^\mu)^2}{\pi_t^\nu}\Big[2(h-\pi_t^\mu(h))\ud I_t^\mu - (h-\pi_t^\nu(h))\ud I_t^\mu\Big]\\
	&\quad + \frac{(\pi_t^\mu)^2}{\pi_t^\nu}\Big[-(h-\pi_t^\nu(h))(\pi_t^\mu(h)-\pi_t^\nu(h))\\
	&\qquad +(h-\pi_t^\mu(h))^2 - 2(h-\pi_t^\mu(h))(h-\pi_t^\nu(h))+(h-\pi_t^\nu(h))^2\Big]\ud t\\
	&= \Big[\frac{2\pi_t^\mu}{\pi_t^\nu}\clA^\dagger \pi_t^\mu -  \frac{(\pi_t^\mu)^2}{(\pi_t^\nu)^2}\clA^\dagger \pi_t^\nu\Big] \ud t
	+ \frac{(\pi_t^\mu)^2}{\pi_t^\nu}\big(h -2\pi_t^\mu(h) +\pi_t^\nu(h)\big)\ud I_t^\mu\\
	&\quad + \frac{(\pi_t^\mu)^2}{\pi_t^\nu}(\pi_t^\nu(h)-\pi_t^\mu(h))(h-\pi_t^\mu(h))\ud t
\end{align*}
Note that
\[
\int \frac{2\pi_t^\mu}{\pi_t^\nu}\clA^\dagger \pi_t^\mu -  \frac{(\pi_t^\mu)^2}{(\pi_t^\nu)^2}\clA^\dagger \pi_t^\nu \ud x = -\int \big(\clA \gamma_t^2 - 2\gamma_t\clA \gamma_t\big)\pi_t^\nu \ud x = - \pi_t^\nu\big(\Gamma \gamma_t\big)
\]
and the last term
\begin{align*}
	\int \frac{(\pi_t^\mu)^2}{\pi_t^\nu}(\pi_t^\nu(h)-\pi_t^\mu(h))(h-\pi_t^\mu(h))\ud t \ud x &= \pi_t^\nu\big(\gamma_t^2(h-\pi_t^\mu(h))\big)(\pi_t^\nu(h)-\pi_t^\mu(h)) \\
	&=\pi_t^\mu\big(\gamma_t (h-\pi_t^\mu(h))\big)(\pi_t^\nu(h)-\pi_t^\nu(\gamma_t h))\\
	&=-\cov_t^\mu(\gamma_t,h)\,\cov_t^\nu(\gamma_t,h)
\end{align*}
Therefore
\[
\ud \clV_t^\nu(\gamma_t) = - \pi_t^\nu\big(\Gamma \gamma_t\big)\ud t -\cov_t^\mu(\gamma_t,h)\,\cov_t^\nu(\gamma_t,h) \ud t + (\cdots)\ud I_t^\mu
\]
\qed

\subsubsection{Deterministic dual formulation and stochastic stability}

Consider the backward Kolmogorov equation:
\[
-\frac{\partial y_t}{\partial t} = \clA y_t,\quad y_T \text{ is given}
\]
Then a standard application of It\^o rule gives
\[
y_T(X_T) = y_0(X_0) + \int_0^T \ud N_t(y_t) 
\]
where $N_t(g)$ is the martingale associated with the infinitesimal generator $\clA$ defined as~\eqref{eq:martingale-generator}. Therefore the dual representation
\[
\E^\bmu\big(y_T(X_T)\big) = \bmu(y_0)
\]
and the equation for the variance is
\[
\dvar^\bmu(y_0) + \int_0^t \den^\bmu(y_s)\ud s = \dvar^\bmu(y_t),\quad 0\le t \le T
\]
Using the standard PI, it follows from the Gronwall inequality:
\[
\dvar^\bmu(y_0) \le e^{-cT}\dvar^\bmu(y_T)
\]
This is a counterpart of~\eqref{eq:variance-inequality-2}.

\subsection{Proof of Theorem~\ref{thm:filter-stability-chisq-RT}}\label{ss:pf-thm71}

Combining~\eqref{eq:CS-bound} from 
Section~\ref{sec:duality-and-filter-stability} 
with~\eqref{eq:variance-inequality-2} yields
\[
\big(\E^\mu\big(\clV_T^\nu(\gamma_T)\big)\big)^2 \le 
\clV_0^\nu(\gamma_0)\clV_0^\nu(Y_0) \le e^{-cT} 
\clV_0^\nu(\gamma_0)\E^\nu\big(\clV_T^\nu(Y_T)\big)
\]
Therefore
\[
R_T \E^\mu\big(\clV_T^\nu(\gamma_T)\big) \le e^{-cT} \clV_0^\nu(\gamma_0)
\]
where
\[
R_T = \frac{\E^\mu\big(\clV^\nu_T(\gamma_T)\big)}{\E^\nu\big(\clV^\nu_T(\gamma_T)\big)}
\]
Since $R_T$ is the ratio of expectations of the same random variable $\clV_T^\nu(\gamma_T)$ under different measures $\sP^\mu$ and $\sP^\nu$, 
\[
\frac{\E^\mu\big(\clV^\nu_T(\gamma_T)\big)}{\E^\nu\big(\clV^\nu_T(\gamma_T)\big)} \ge \mathop{\operatorname{essinf}}\frac{\ud \sP^\mu}{\ud \sP^\nu} = \mathop{\operatorname{essinf}}\frac{\ud \mu}{\ud \nu} = \underline{a}
\]
\qed

\subsection{Proof of Proposition~\ref{prop:bvi-pathwise}}\label{ss:pf-prop73}
%

Recall that the process
%
\[
M_t = \clV_t^\nu(Y_t) - \int_0^t \pi_s^\nu(\Gamma Y_s) + 
\pi_s^\nu\big(|U_s^\opt + V_s(\cdot)|^2\big)\ud s,\quad 0\le t \le T
\]
is a $\sP^\nu$-martingale.
Since $\beta_t$ satisfies~\eqref{eq:beta-process},
\[
\ud\big(\clV_t^\nu(Y_t)\big) \ge \beta_t \clV_t^\nu(Y_t)\ud t + \ud M_t
\]
Consider an integrating factor $\exp\big(-\int_0^t \beta_s \ud s\big)$ to obtain
\[
\exp\Big(-\int_0^T \beta_t \ud t\Big) \clV_T^\nu(Y_T) - \clV_0^\nu(Y_0) \ge M_T-M_0
\]
and therefore the claim follows by taking expectation.
\qed

\subsection{Proof of Proposition~\ref{prop:cpi-equivalent-def}}\label{ss:pf-prop74}

\noindent$(\Longrightarrow)$ $F(\omega) \in \clD$, so the inequality holds point-wise.

\noindent$(\Longleftarrow)$ Assume conditional PI and consider $F = \ones_A f$ for $A \in\clZ_T$. Observe that $\pi_T^\nu(\Gamma F) = \ones_A \pi_T^\nu(\Gamma f)$ and $\clV_T^\nu(F)= \ones_A\clV_T^\nu(f)$, and therefore
\[
\E^{\nu}\big(\ones_A\pi_T^\nu(\Gamma f)\big) \ge c\,\E^{\nu}\big(\ones_A\clV_T^\nu(f)\big)
\]
Set $A = [\pi_T^{\nu} (\Gamma f) < c\, \clV_T^\nu(f)]$, then $\sP(A) = 0$ or the assumption is violated. \qed

\newpage


\chapter{Stabilizability of the dual control system}\label{ch:filter-stability-2}


The dual control problem plays a key role to prove stability of the Kalman-Bucy filter (Section~\ref{ssec:Kalman-filter-stability}). The stabilizability of the dual system is a necessary and sufficient condition for the stability of the Kalman filter (Remark~\ref{rm:stabilizability-and-kalman}).
In this chapter, the relationship between the nonlinear filter stability and the stabilizability of the dual BSDE control system is investigated.

%

Out study is motivated by~\cite{baxendale2004asymptotic} who
formulated certain ``identifying conditions'' that
are shown to be sufficient for the stability of Wonham filter.  These
conditions are formulated in terms of the model parameters	(transition matrix and the observation function). They showed that these conditions are sufficient to asymptotically detect the correct ergodic class.

The main result of this chapter is to show that the stabilizability of the dual system is both necessary and sufficient to asymptotically detec the correct ergodic class where the state lies in. 

The outline of the chapter is as follows: In Section~\ref{sec:FS}, stabilizability of the dual BSDE is shown to be a necessary condition for filter stability in $L^2$. In the following Section~\ref{sec:82}, the main result is presented.

\section{Filter stability in $L^2$}\label{sec:FS}

We begin by relating the dual optimal contgrol formation to filter stability in $L^2$. (The definition is given in Def.~\ref{def:FS}). Because $\pi_T^\mu(f)$ is the orthogonal projection of $f(X_T)$ onto $L^2_{\clZ_T}(\Omega;\Re)$, by the Pythagoras theorem,
\begin{equation}\label{eq:pythagoras}
	\E^\mu\big(|\pi_T^\mu(f)-\pi_T^\nu(f)|^2\big) = \E^\mu\big(|f(X_T)-\pi_T^\nu(f)|^2\big) - \E^\mu\big(|f(X_T)-\pi_T^\mu(f)|^2\big)
\end{equation}
The second term in the right-hand side is the expectation of the conditional variance, and therefore it is the optimal value of $\sJ_T^\mu(\cdot)$ problem, denoted by $\sJ_T^\mu$. For the first term, recall that $\pi_T^\nu(f)$ can be obtained using $U^\nu$, which is the optimal control for the $\sJ_T^\nu(\cdot)$ problem (see Section~\ref{sec:duality-and-filter-stability}).
\[
\pi_T^\nu(f) = \nu(Y_0) - \int_0^T \big(U_t^\nu\big)^\tp \ud Z_t,\quad \sP^\nu\text{-a.s.}
\]
This is similar to the estimator~\eqref{eq:estimator} except that the constant term is now $\nu(Y_0)$ instead of $\mu(Y_0)$.
By the general form of the duality principle introduced in Remark~\ref{rm:general-form-of-estimator},
\[
\E^\mu\big(|f(X_T)-\pi_T^\nu(f)|^2\big) = \sJ_T^\mu(U^\nu) + |\mu(Y_0) - \nu(Y_0)|^2
\]
Substituting this to~\eqref{eq:pythagoras}, 
\[
\E^\mu\big(|\pi_T^\mu(f)-\pi_T^\nu(f)|^2\big) = \big(\sJ_T^\mu(U^\nu) -\sJ_T^\mu\big)+ |\mu(Y_0) - \nu(Y_0)|^2
\]
Because both the terms on the right-hand side are non-negative, the following proposition is obtained.

\medskip

\begin{proposition}\label{prop:asymptotic-optimality}
	The filter is stable in the sense of $L^2$ if and only if 
	\begin{subequations}
		\begin{align}
			\mu(Y_0) - \nu(Y_0) \;\;
			&\stackrel{(T\to\infty)}{\longrightarrow} \;\; 0 \label{eq:a1_st}\\
			\sJ_T^\mu(U^\nu)-\sJ_T^\mu
			\;\;&\stackrel{(T\to\infty)}{\longrightarrow} \;\; 0 \label{eq:a2_st}
		\end{align}
	\end{subequations}
	for all $f \in C_b(\bS)$ and $\mu,\nu\in\clP(\bS)$ such that $\mu\ll\nu$.
\end{proposition}

\medskip

\begin{remark}\label{rm:L2-stability-remark}
	These conditions are the nonlinear counterparts of the Assumption~\ref{as:R} for the stability of the Kalman filter in Section~\ref{ssec:Kalman-filter-stability}.
	\begin{enumerate}
		\item Equation~\eqref{eq:a1_st} means that the closed-loop
		system is asymptotically stable.  That is, $Y_0\to c\ones$ as $T\to\infty$.  This is also the reason why the stabilizability
		condition is important to the problem of filter stability.  The
		condition plays the same role in linear and nonlinear settings. 
		
		\item Equation~\eqref{eq:a2_st} means that the $U^\nu$ is asymptotically optimal for $\sJ_T^\mu(\cdot)$ problem.
		Since the optimal value $\sJ_T^\mu$ has the interpretation of the conditional variance, its convergence is analogous to the convergence of the solution of the DRE in the Kalman filter.  In linear settings, the latter is deduced by establishing
		an asymptotic optimality of the stationary control law (see Section~\ref{ssec:Kalman-filter-stability}).
	\end{enumerate}
	
\end{remark}

In the enumerated list above, the first point is related to the stabilizability of the dual system (Def.~\ref{def:stabilizability}). In particular, we have the following theorem whose proof appears in Section~\ref{ss:pf-thm81}.

\medskip

\begin{theorem}\label{thm:necessary-condition-filter-stability}
	If the filter is stable in $L^2$ then the dual control system~\eqref{eq:dual-bsde} is stabilizable (see Def.~\ref{def:stabilizability}).
\end{theorem}

\medskip


%

\section{Stabilizability and asymptotic detection of the ergodic class}\label{sec:82}

\subsection{Finite case with multiple ergodic classes}

In this section, we consider the finite state case as in to~\cite{baxendale2004asymptotic}.	
We partition the state space $\bS$ in $M$ \emph{ergodic classes} $\{\bS_k:k=1,2,\ldots,M\}$ such that: 
\begin{enumerate}
	\item $\bS = 
	\bigcup_{k=1}^M \bS_k$ where $\sP([X_t\in \bS_l] \mid [X_0\in \bS_k]) = 0$ for all $t\ge 0$ 
	and $l \neq k$.
	\item By choosing an appropriate coordinate, the rate matrix
	\[
	A = \begin{pmatrix}
		A_1 & 0 & \cdots & 0\\
		0 & A_2 & \cdots & 0\\
		\vdots & \vdots & \ddots & \vdots\\
		0 & 0 & \cdots & A_M
	\end{pmatrix}
	\]
	where $A_k$ is a rate matrix on $\bS_k$ for $k = 1,2,\ldots,M$.
	\item Each $A_k$ admits a unique invariant measure. Equivalently, 0 is a 
	simple eigenvalue of $A_k$.
\end{enumerate} 
The system is \emph{ergodic} if $M = 1$.  Under this setting, the nonlinear filter is decomposed as described in the following lemma:

\medskip

\begin{lemma}\label{lemma:ergodic-class-decomposition}
	Suppose $\mathbb{S} = \cup_{k=1}^M \mathbb{S}_k$ is an ergodic partition.   
	For each such ergodic class with $\sP^\nu([X_0\in\bS_k]) > 0$, define 
	\[
	\nu_k(x) := \left\{ \begin{array}{cc}
		\dfrac{\nu(i)}{\sP^\nu([X_0\in\bS_k])} &
		\text{if}\;\; x\in \bS_k \\ 0 &
		\text{if}\;\; x\notin \bS_k \end{array} \right.
	\]
	Then
	\[
	\pi_T^\nu(f) = \sum_{k=1}^M \pi_T^\nu(\ones_{\bS_k})\pi_T^{\nu_k}(f) 
	\]
\end{lemma}

\medskip

\subsection{Main result}\label{ssec:detectable}

%


We begin by recalling the definition of stabilizability and its characterization described in Section~\ref{sec:finite-observability} for finite state space case.
In these settings, 
\[
S_0 := \{f \in \Re^d \mid \; Af = 0\} 
\]
It was shown in Corollary~\ref{prop:detectability-nullspace} that the BSDE~\eqref{eq:dual-bsde} is stabilizable if and only if $S_0\subset \clC$ where $\clC$ is the controllable subspace.
The following proposition provides another characterization using the notation introduced in this section.

\medskip

\begin{proposition}\label{prop:detectability-implications}
	Consider the BSDE~\eqref{eq:dual-bsde}. Then
	\begin{enumerate}
		\item If $\bS$ has a single ergodic class then BSDE is stabilizable.  
		\item If $\bS=\cup_{k=1}^M \bS_k$ is partitioned into $M$ ergodic
		classes then the BSDE is stabilizable 
		if and only if the indicator functions $\ones_{\mathbb{S}_k} \in \clC$ for
		$k=1,2,\ldots,M$.  
	\end{enumerate}
\end{proposition}

\medskip

\begin{remark}
	A direct corollary of Prop.~\ref{prop:detectability-implications} is 
	that two typical assumptions for filter stability---(1) the state 
	process is ergodic; (2) the HMM is observable---both imply the stabilizability of the dual 
	system. 
\end{remark}

The above decomposition relates the filter stability problem with the convergence of $\pi_T^\nu(\ones_{\bS_k})$. The following theorem states the main result of this section:

\medskip

\begin{theorem}\label{thm:main-result}
	If the dual BSDE is stabilizable, then 
	\[
	\pi_T^{\nu}(\ones_{\bS_k}) \stackrel{(T\to\infty)}{\longrightarrow}
	\ones_{\bS_k}(X_0) \quad \sP^{\mu}\text{-a.s.}
	\] 
	whenever $\mu \ll \nu$. (That is, the filter asymptotically
	detects the correct ergodic class.)
\end{theorem}

%

\medskip

\begin{remark}
	The sufficient condition stated in~\cite[Theorem
	4.4]{baxendale2004asymptotic} stress the importance of the
	``identifying'' property of the filter to identify the correct ergodic
	class~\cite[Lemma 6.3]{baxendale2004asymptotic}.  Subsequently, the
	definition of detectability is introduced
	in~\cite{van2009observability,van2010nonlinear}.  For the Wonham
	filter, the detectability property is shown to be equivalent to
	filter stability~\cite[Theorem 2]{van2009observability}.
\end{remark}

\section{Proofs of the statements}

\subsection{Proof of Theorem~\ref{thm:necessary-condition-filter-stability}}\label{ss:pf-thm81}

Consider the dual system~\eqref{eq:dual-bsde} with terminal condition $Y_T = f$ is deterministic and $U = 0$. Then the problem reduces to a deterministic PDE
\[
-\ud Y_t(x) = -(\clA Y_t)(x),\quad Y_T(x) = f(x),\; x\in \bS
\]
is given by $Y_t = \clS_{T-t} f$ for all $0\le t\le T$. where $\{\clS_t:t\ge 0\}$ is the signal semigroup defined by
\[
(\clS_t f)(x) = \E\big(f(X_t)\mid X_0 = x\big),\quad x\in\bS
\]
Now for any $U\in \clU$, the solution to the BSDE~\eqref{eq:dual-bsde} with $Y_T = f$ is given by using linearity:
\[
Y_0 = \clS_T f + \clL(U,0)
\]

Assume that the dual BSDE is not stabilizable. That is, there exists $\tilde{\mu} \in \clC^\bot$ such that $\tilde{\mu} \notin S_s$. Choose $\mu, \nu \in \clP(\bS)$ and $\alpha >0$ such that $\nu = \mu + \alpha\tilde{\mu}$ and $\mu\ll\nu$.
Since $(\mu-\nu)\notin S_s$, there exists $f \in C_b(\bS)$ and $\epsilon >0$ such that for all $T$,
\[
|\mu(\clS_T f) - \nu(\clS_T f)| > \epsilon
\]
Since $\mu-\nu \in \Rsp(\clL)^\bot$,
\begin{align*}
|\mu(Y_0) -\nu(Y_0)| &= \big|\mu(\clS_T f) - \nu(\clS_T f) + (\mu-\nu)(\clL(U,0))\big|\\
&= |\mu(\clS_T f) - \nu(\clS_T f)| > \epsilon
\end{align*}
Therefore the filter is not stable due to Prop.~\ref{prop:asymptotic-optimality}.
\qed

\subsection{Proof of Lemma~\ref{lemma:ergodic-class-decomposition}}

Clearly $\nu_k\ll\nu$ and 
\[
\frac{\ud \sP^{\nu_k}}{\ud \sP^\nu} (\omega) = \sum_{x\in\bS} \frac{\nu_k(x)}{\nu(x)}\ones_{[X_0=x]}(\omega) = \frac{\ones_{[X_0(\omega)\in\bS_k]}(\omega)}{\sP^\nu([X_0\in\bS_k])}
\]
An application of the Bayes' formula gives
\[
\E^{\nu_k}\big(f(X_T)|\clZ_T\big) = \frac{\E^\nu\big(f(X_T)\ones_{[X_0\in\bS_k]}|\clZ_T\big)}{\E^\nu\big(\ones_{[X_0\in\bS_k]}|\clZ_T\big)}
\]
and therefore
\begin{equation}\label{eq:Bayes_Identity}
\E^\nu\big(f(X_T)\ones_{[X_0\in\bS_k]}|\clZ_T\big) = \pi_T^\nu(\ones_{\bS_k})\pi_T^{\nu_k}(f) 
\end{equation}
where we have used the fact that
\[
\ones_{[X_0(\omega) \in\bS_k]} (\omega) =
\ones_{[X_T(\omega) \in\bS_k]}(\omega) \quad {\sf P}^\nu-a.s.
\]
Note that the identity~\eqref{eq:Bayes_Identity} is true for all
$k=1,2,\ldots,M$.  (If $\sP^\nu([X_0\in\bS_k]) = 0$ then both sides
are zero.)  Summing the identity over the index $k$ yields the conclusion. \qed

\subsection{Technical lemmas to prove the main theorem}

\begin{lemma}[Reachability of deterministic functions]\label{lm:controllable-reachable}
If $f\in \clC$ then there exists a control $U \in \clU$ and $c\in \Re$ such that the solution to the dual BSDE~\eqref{eq:dual-bsde} satisfies $Y_T = f$ and $Y_0 = c\ones$.
\end{lemma}

\begin{proof}	
Suppose $f \in \clC$. Since $\clC$ is $\clA$-invariant (by Prop.~\ref{thm:controllable-subspace}), $\clS_T f \in
\clC$. Therefore, from definition of $\clC$, there is a constant $c\in\Re$ and $U\in\clU$ such that $\clL(U,c) = \clS_T f$.  Now consider a second solution of the BSDE~\eqref{eq:dual-bsde} with $Y_T = c\ones+f$ and zero control input. Since $\clA \ones = 0$, this second solution is $(Y_t,V_t) = (c\ones + \clS_{T-t}f, 0)$ for $t \in [0,T]$. By linearity, we subtract the two solutions to show that with $Y_T = f$
and control $-U$, one obtains $Y_0 = c\ones$. 
\end{proof}

\begin{lemma}\label{lemma:continuity}
Consider a family of measures $\{\nu_n \in \clP(\bS):n=1,2,\ldots\}$ such that $\nu_n \ll \nu$ and $\nu_n \to \nu$ as $n$ increases. Then
\begin{equation}\label{eq:continuity-property}
	|\sJ_T^{\nu_n}(U^\nu) - \sJ_T^\nu| \longrightarrow 0
\end{equation}
where the convergence is uniform in $T$.
\end{lemma}

\begin{proof}
We want to show that $\sJ_T^{\nu_n}(U^\nu) \longrightarrow \sJ_T^\nu(U^\nu)$ where 
\[
\sJ_T^{\nu_n}(U^\nu) = \tE^{\nu_n}\Big(|Y_0(X_0)-\nu_n(Y_0)|^2 + \int_0^T l(Y_t,V_t,U_t^\nu,t)\ud t \Big)
\]
First of all,
\[
\nu_n\big(|Y_0-\nu_n(Y_0)|^2\big) \longrightarrow \nu\big(|Y_0-\nu(Y_0)|^2\big) 
\]
Let $\xi_T := \int_0^T l(Y_t,V_t,U_t^\nu,t) \ud t$ then 
\[
\tE^{\nu}(\xi_T) = \sum_{i\in\bS}\nu(i) \tE^{\delta_i}(\xi_T),\quad \tE^{\nu_n}(\xi_T) = \sum_{i\in\bS}\nu_n(i) \tE^{\delta_i}(\xi_T)
\]
Because $\tE^\nu(\xi_T)\le \sJ_T^\nu <\infty$.
\end{proof}

\medskip

\begin{lemma}\label{lemma:stationary}
Suppose $\bmu$ is an invariant measure of $A$ (i.e., $A^\tp \bmu  =
0$).  Then for each fixed $f\in\Re^d$,
\begin{enumerate}
	\item The sequence $\{\sJ_T^{\bmu}(f):T\geq 0\}$ is bounded,
	non-negative, and non-increasing in $T$.  Therefore,
	$\sJ_T^{\bmu}(f)$ converges as $T\rightarrow \infty$.  Denote the
	limit as $\sJ_\infty^{\bmu}(f)$.
	\item For a given $\mu\in{\cal P}(\mathbb{S})$, denote $\mu_T :=
	e^{A^\tp T} \mu$.  Suppose $\mu_T\to \bmu$ as $T\to\infty$.  Then
	\[
	\limsup_{T\to\infty} \sJ_{T}^\mu(f) \le \sJ_{\infty}^{\bmu}(f)
	\]
\end{enumerate}
\end{lemma}

\begin{proof}

The proof of the lemma requires a technical construction. Consider the time horizon $[0,T_1+T_2]$.  If $X_0\sim \mu$ then
$X_{T_1}\sim e^{A^\tp T_1}\mu=:\mu_{T_1}$.  This is useful to relate
the properties of $\sJ^\mu_{T_1+T_2}(\cdot)$ and
$\sJ^{\mu_{T_1}}_{T_2}(\cdot)$.  For this purpose, consider first the
time horizon $[T_1,T_1+T_2]$.  Over this time horizon, introduce the filtration 
\[
\tilde\clZ_{t-T_1}:=\{Z_t - Z_{T_1}\;:\; T_1 \le  t \le T_1+T_2\}
\]
For a control $\tilde{U} \in L^2_{\tilde{\clZ}}([0,T_2])$, let
$\{(\tilde{Y}_t, \tilde{V}_t): t\in[0,T_2]\}$ denote the solution of the
BSDE~\eqref{eq:dual-optimal-control-b} with $\tilde{Y}_{T_2}=f$.  The control $\tilde{U}$ is extended to the
time-horizon $[0,T_1+T_2]$ as follows:
\begin{equation}\label{eq:zero-U-control}
	U_t=\begin{cases}
		0 & 0 \le t < T_1\\
		\tilde{U}_{t-T_1} & T_1 \le t \le T_1+T_2
	\end{cases}
\end{equation}
The control $U\in{\cal U}$ and yields the following solution of
the BSDE~\eqref{eq:dual-optimal-control-b}:
\[ \label{eq:soln-barU}
(Y_t,V_t)=\begin{cases}
	(e^{A(T_1-t)}\tilde{Y}_{0},\;0) & 0 \le t < T_1\\
	(\tilde{Y}_{t-T_1},\;\tilde{V}_{t-T_1}) & T_1 \le t \le T_1+T_2
\end{cases}
\]
Under this definition, we claim that  
\begin{equation}\label{eq:claim1}
	\sJ^\mu_{T_1+T_2}(U) = \sJ^{\mu_{T_1}}_{T_2}(\tilde{U})
\end{equation}
and then the two results in the lemma are direct consequences of the claim:

\newP{Step 1} Take $\mu=\bmu$
and  $\tilde{U} = U^{\bmu}$.  Then
\[
\sJ_{T_1+T_2}^{\bmu} \le \sJ_{T_1+T_2}^{\bmu}(U)
\overset{\text{Eq.~\eqref{eq:claim1}}}{=\joinrel=} \sJ_{T_2}^{\bmu}(U^{\bmu}) = \sJ_{T_2}^{\bmu}
\] 
where we used the facts that (1) $\mu_{T_1}=e^{A^\tp T_1}\bmu =
\bmu$ because $\bmu$ is the invariant measure; and (2)
$U^{\bmu}$ is the optimal control for the $\sJ_{T_2}^{\bmu}(\cdot)$
problem.  Therefore, $\sJ_{T}^{\bmu}$ is monotone in $T$ 
and converges as $T\rightarrow \infty$.  Denote the
limit as $\sJ_\infty^{\bar{\mu}}$.

\medskip

\newP{Step 2} Let $T=T_1+T_2$.  For $\mu\in{\cal P}(\mathbb{S})$,
with $\tilde{U} = U^{\bmu}$
\begin{equation}\label{eq:cor1_part2}
	\sJ_T^\mu \le \sJ_{T_1+T_2}^{\mu}(U)
	\overset{\text{Eq.~\eqref{eq:claim1}}}{=\joinrel=} \sJ_{T_2}^{\mu_{T_1}}(U^\bmu) 
\end{equation}
We have
\[
|\sJ_{T_2}^{\mu_{T_1}}(U^{\bmu}) - \sJ_\infty^{\bar{\mu}}| \leq
|\sJ_{T_2}^{\mu_{T_1}}(U^{\bmu}) -
\sJ_{T_2}^{\bmu}(U^{\bmu})| + |\sJ_{T_2}^{\bmu}(U^{\bmu}) - \sJ_\infty^{\bar{\mu}}|
\]
The second term on the righthand-side does not depend upon
$T_1$. Because $U^{\bmu}$ is the optimal control input, this term
goes to zero as $T_2\to \infty$:  That is, given $\epsilon>0$, there
exist an $n_2$ such that 
\[
|\sJ_{T_2}^{\bmu}(U^{\bmu}) -
\sJ_\infty^{\bar{\mu}}|\leq \epsilon\quad \forall T_2 \ge n_2
\]
Now fix $T_2 = n_2$ and apply continuity property~\eqref{eq:continuity-property} to the first
term on the righthand-side: There exists $n_1=n_1(n_2)$ such that 
\[
|\sJ_{n_2}^{\mu_{T_1}}(U^{\bmu}) -
\sJ_{n_2}^{\bmu}(U^{\bmu})| \leq \epsilon \quad \forall T_1 \ge n_1
\]
Combine these inequalities concludes for all $T_1 \ge n_1$,
\[
\sJ_{n_2}^{\mu_{T_1}}(U^\bmu) \leq
\sJ_\infty^{\bar{\mu}} + 2\epsilon
\]    
From~\eqref{eq:cor1_part2}, $
\sJ_T^\mu \le \sJ_{n_2}^{\mu_{T-n_2}}(U^\bmu)$. Therefore, for all $T \ge n_1+n_2$, 
\[
\sJ_{T}^\mu \le \sJ_{n_2}^{\mu_{T-n_2}}(U^\bmu) \leq \sJ_\infty^{\bar{\mu}} + 2\epsilon
\]
Since $\epsilon$ is arbitrary, the result follows.  

\medskip

It remains to prove the claim~\eqref{eq:claim1}.
We have
\begin{align*}
	\sJ_{T}^\mu(U) &= \E^{\mu}\Big(|Y_0(X_0)-\mu(Y_0)|^2 +\int_0^{T_1} \Gamma (Y_t)(X_t) \ud
	t\Big) \\
	&\quad+\quad \E^{\mu} \Big(\int_{T_1}^{T_1+T_2} \Gamma (Y_t)(X_t)
	+ |U_t+V_t(X_t)|_R^2\ud t \Big)
\end{align*}
Each of the two terms is simplified separately. 

Consider the control input $U$ defined according
to~\eqref{eq:zero-U-control}.  
Since $Y_{T_1} = \tilde{Y}_{0}$ is a deterministic function and the control is set to be zero, $V_t = 0$ on $0\le t < T_1$ and the BSDE becomes ODE:
\[
-\frac{\ud}{\ud t} Y_t = AY_t,\quad Y_{T_1} = \tilde{Y}_0
\]
A straightforward calculation shows that
\begin{align*}
	&\E^{\mu}\Big(|Y_0(X_0)-\mu(Y_0)|^2 +\int_0^{T_1} \Gamma (Y_t)(X_t) \ud
	t\Big) \\
	&= \E^\mu\Big(|Y_{T_1}(X_{T_1})-\mu_{T_1}(Y_{T_1})|^2\Big)= \E^{\mu_{T_1}}\Big(|\tilde{Y}_{0}(X_0)-\mu_{T_1}(\tilde{Y}_{0})|^2\Big)
\end{align*}

The second term
\begin{align*}
	& \E^{\mu} \Big(\int_{T_1}^{T_1+T_2}  \Gamma (Y_t)(X_t)  +
	|U_t+V_t(X_t)|_R^2\ud t \Big) \\
	& = \E^{\mu} \Big(\int_{T_1}^{T_1+T_2}  \Gamma (\tilde{Y}_{t-T_1})(X_t) +  |\tilde{U}_{t-T_1}+\tilde{V}_{t-T_1}(
	X_t)|_R^2\ud t \Big)\\
	& = \E^{\mu_{T_1}} \Big(\int_{0}^{T_2}  \Gamma(\tilde{Y}_{t})({X}_{t})+  |\tilde{U}_{t}+\tilde{V}_{t}( {X}_{t})|_R^2\ud t \Big)
\end{align*}
Combining the results of the two calculations yields:
\begin{align*}
	\sJ_{T}^\mu(U)
	&= \E^{\mu_{T_1}}\Big(|\tilde{Y}_{0}(X_0)-\mu_{T_1}(\tilde{Y}_{0})|^2\Big)\\
	&\quad\quad+ \E^{\mu_{T_1}}\Big(\int_{0}^{T_2}  \Gamma(\tilde{Y}_{t})(\tilde{X}_{t}) +  |\tilde{U}_{t}+\tilde{V}_{t}( \tilde{X}_{t})|_R^2\ud t \Big)\\
	&= \sJ_{T_2}^{\mu_{T_1}}(\tilde{U}) 
\end{align*}

\end{proof}

\subsection{Proof of Theorem~\ref{thm:main-result}}

The proof is in the following three steps:
\begin{enumerate}
\item In step~1, we show that 
$\pi_T^\nu(\ones_{\bS_k})$ converges ${\sf P}^\nu$-a.s.  
\item In step~2, we show that if 
$\ones_{\bS_k}\in\clC$ then $\sJ_T^\bmu(\ones_{\bS_k}) \to 0$ as $T\to\infty$ where
$\bmu$ is any invariant measure of $A$.  We use part (i) of
Lemma~\ref{lemma:stationary} to prove this result.  
\item In step~3, we combine the conclusions of steps~1 and~2 to prove
the result.  We use part (ii) of
Lemma~\ref{lemma:stationary} to prove this result.  
\end{enumerate}

\newP{Step 1} Consider the Wonham filter~\eqref{eq:nonlinear-filter} with $\pi_0=\nu$.  Since
$A \ones_{\bS_k} = 0$, $\{\pi_T^\nu(\ones_{\bS_k}):T\geq 0\}$ is a
bounded ${\sf
P}^\nu$-martingale and therefore converges $\sP^\nu$-a.s.
(Therefore, the a.s. convergence does not require stabilizability of the
model.)

\newP{Step 2} Suppose $\bmu$ is any invariant measure.  Then $\sJ_T^\bmu$
is monotone (part~(i) of Lemma~\ref{lemma:stationary}).  In the
following, we construct a
sequence of admissible control input $\{U^{(T)}:T=1,2,\ldots\}$ such that 
$\sJ_T^\bmu(U^{(T)};\ones_{\bS_k}) \to 0$ as $T\to\infty$.  Since
$\sJ_T^\bmu(\ones_{\bS_k})$ is the minimum value this implies
$\sJ_T^\bmu(\ones_{\bS_k}) \to 0$ as $T\to\infty$ (for this
particular sub-sequence).  Since $\sJ_T^\bmu$
is monotone, the limit exists and equals this sub-sequential
limit.

Suppose $\ones_{\bS_k} \in \clC$. By Lemma~\ref{lm:controllable-reachable}, there exists an admissible
control $U^{(1)}:=\{U_t^{(1)}:0\leq t \leq 1\}$ and a constant
$c\in\Re$ such that $Y_T^{(1)} =
\ones_{\bS_k}$ and $Y_0^{(1)} = c \ones$.  
Assuming the claim to be true for now, denote the associated
solution of the BSDE~\eqref{eq:dual-bsde} as
$(Y^{(1)},V^{(1)}):=\{(Y_t^{(1)},V_t^{(1)}): 0\leq t \leq 1\}$.  

Since $Z$ is a Brownian motion under the Girsanov change of measure, there exists
a functional $\phi:[0,1]\times C\big([0,1];\Re\big)\to\Re$ such
that 
\[
U_t^{(1)} = \phi(t\,,\,\{Z_s:0\leq s \leq t\}) \quad \sP^{\bmu}-\text{a.s.}
\]

Now consider the following control over the time-horizon $[0,n]$:  For $l=0,1,2,\ldots,n-1$
\[
U_t^{(n)} := \frac{1}{n}\phi\big(t-l\,,\, \{Z_s:l\leq s \leq t\}\big),\quad t\in(l,(l+1)]
\]  
Such a control input is clearly admissible.  
With $Y_n = \ones_{\bS_k}$, one obtains the following solution
$(Y^{(n)},V^{(n)}):=\{(Y_t^{(n)},V_t^{(n)}): 0\leq t \leq n\}$ of the
BSDE~\eqref{eq:dual-bsde}:  
\[
V_t^{(n)} \stackrel{\text{(d)}}{=} \frac{1}{n}V_{t-l}^{(1)},\quad t\in(l,(l+1)]
\]
for $l=0,1,2,\ldots,n-1$, $V_0^{(n)} = V_0^{(1)}$, and
\begin{align*}
Y_t^{(n)} \stackrel{\text{(d)}}{=} \begin{cases} \frac{1}{n}Y_t^{(1)}
	+ \frac{n-1}{n} \ones_{\bS_k} &\mbox{if } t \in ((n-1),n] \\
	\frac{1}{n}Y_t^{(1)} + \frac{n-2}{n}\ones_{\bS_k} + \frac{c}{n}\ones &
	\mbox{if }t \in ((n-2),(n-1)] \\
	\frac{1}{n}Y_t^{(1)} + \frac{n-3}{n}\ones_{\bS_k} + \frac{2c}{n}\ones &
	\mbox{if }t \in ((n-3),(n-2)] \\
	\vdots & \vdots \\
	\frac{1}{n}Y_t^{(1)} +\frac{1}{n}\ones_{\bS_k} + \frac{(n-2)c}{n}\ones & \mbox{if }t \in (1,2] \\
	\frac{1}{n}Y_t^{(1)} +
	\frac{(n-1)c}{n}\ones & \mbox{if }t \in (0,1] 
\end{cases}
\end{align*}
and $Y_0^{(n)} = c \ones$.  

Since $Y_0^{(n)} = c\ones$, the terminal cost 
$|Y_0^{(n)} (X_0)-\bmu(Y_0^{(n)})|^2 = 0$. And since $X_t\sim \bmu$, for $l=0,1,2,\ldots,n-1$:
\begin{align*}
\E^{\bmu}\Big(&\int_{l}^{(l+1)}  \Gamma(Y_t^{(n)})(X_t)  + 
|U_t^{(n)} + V_t^{(n)} (X_t)|_R^2 \ud t \Big) \\
&=\frac{1}{n^2}\E^{\bmu}\Big(\int_{0}^{1} \Gamma(Y_t^{(1)})(X_t) +
|U_t^{(1)}+ V_t^{(1)} (X_t)|_R^2 \ud t \Big)
\end{align*}
Therefore,
\[
\sJ_{n}^{\bmu}(U^{(n)}) =\frac{1}{n}\sJ_1^{\bmu}(U^{(1)})
\]
and thus the optimal value 
\begin{equation}\label{eq:l2-conv-indicator}
\sJ_n^\bmu \le\sJ_{n}^{\bmu}(U^{(n)}) =\frac{1}{n}\sJ_1^{\bmu}(U^{(1)}) \longrightarrow 0 \quad \text{as}\quad n\to\infty
\end{equation}

\newP{Step 3} Suppose $\nu \in\clP(\bS)$ and $\ones_{\bS_k} \in
\clC$.  In this final step, we show that $\sJ_T^\nu \to 0$ and 
\[
\pi_T^{\nu}(\ones_{\bS_k}) \stackrel{(T\to\infty)}{\longrightarrow}
\ones_{\bS_k}(X_0) \quad \sP^{\nu}\text{-a.s.}
\] 

Let $\bmu_l\in\clP(\bS)$ be the invariant measure for the
$l^{\text{th}}$-ergodic class and $a_l := \nu (\ones_{\bS_l})$ for
$l=1,2,\ldots,m$.  Choose the invariant measure as follows:
\[
\bmu = a_1\bmu_1 + a_2\bmu_2 + \ldots + a_m\bmu_m 
\]
From step~2, we know that $\sJ_\infty^\bmu = 0$.  Also, $\nu_T :=
e^{A^\tp T}\nu\to \bmu$ as $T\to\infty$. 
Therefore, using part~(ii) of Lemma~\ref{lemma:stationary}, 
\[
\limsup_{T\to \infty} \sJ_T^\nu \leq \sJ_\infty^\bmu  = 0
\]
which shows that $\sJ_T^\nu \to 0$ as $T\to\infty$.   

Since $\bS_k$ is an ergodic class,
\[
\sJ_T^\nu = \E^\nu\big(|\ones_{\bS_k}(X_T)  -
\pi_T^\nu(\ones_{\bS_k})|^2\big) = \E^\nu\big(|\ones_{\bS_k}(X_0)  - \pi_T^\nu(\ones_{\bS_k})|^2\big)
\]
By Fatou's lemma,
\[
\E^\nu\big(\liminf_{T\to\infty}|\pi_T^\nu(\ones_{\bS_k})-\ones_{\bS_k}(X_0)|^2 \big) \le \lim_{T\to \infty}\sJ_T^\nu = 0
\]
In step 1, we showed that $\pi_T^\nu(\ones_{\bS_k})$ has an
a.s. limit.  So, $\liminf$ is replaced as
\[
\lim_{T\to \infty} |\pi_T^\nu(\ones_{\bS_k})-\ones_{[X_0\in\bS_k]}|^2 = 0\quad \sP^\nu\text{-a.s.}
\]
and therefore also $\sP^\mu$-a.s. whenever $\mu\ll \nu$.
\qed

\newpage


\chapter{Future works}\label{ch:future}


The three main contributions of this thesis are as follows: 
\begin{enumerate}
	\item Dual controllability characterization of stochastic observability.
	\item Dual minimum variance optimal control formulation of the stochastic filtering problem.
	\item Filter stability analysis using the dual optimal control formulation.
\end{enumerate}
We conclude the thesis by listing some of the remaining questions and possible extensions for each of the three main contributions.

\section{Observability gramian and model reduction}

In linear systems theory, controllability and observability are important properties to obtain minimal state-space realization~\cite{kalman1965irreducible}.
In a nutshell, an input-state-output system is controllable and observable if and only if the dimension of the state space is minimal.
The idea also extends to nonlinear systems~\cite[Section IV]{krener1985nonlinear}.

Model reduction problem considers a reduced order model that has an approximately the same impulse response as the original system~\cite{moore1981principal}.
Controllability and observability gramians are used for balanced model reduction of linear systems~\cite[Chapter 4]{dullerud2013course}.
In literature, the idea is extended to more general setting, e.g.,~\cite{vaidya2007observability} discussed finite dimensional approximation of observability gramian for discrete time nonlinear systems and~\cite{kawamura2020nonlinear} used Fleming-Mitter duality to define stochastic observability.

In Section~\ref{ssec:gramian}, the controllability gramian $\sW$ of the dual control system is presented. Because of the dual relationship, one may refer to it as the observability gramian of the HMM.
It is an important question to consider model reduction based on analysis of the gramian.

\section{Sub-optimal solutions}

The duality principle provides the equality between the cost 
functional~\eqref{eq:dual-optimal-control-a} and the error variance. The 
optimal solution yields the conditional expectation if the control $U$ is 
minimized over the entire $\clU$. 
Therefore, if the control is chosen from a restricted set, then we can obtain 
a sub-optimal solution.

It is a subject of the future research to design an efficient approximation of 
the nonlinear filter by considering sub-optimal solution of the dual optimal control 
problem.
The simplest choice of restricted set is the space of deterministic control input $U \in L^2\big([0,T];\Re^m\big)$.
Another choice is finite order truncation of Wiener chaos expansion. Such method already utilized to numerically simulate BSDEs~\cite{briand2014simulation} and obtain approximation of stochastic PDEs~\cite{luo2006wiener}.
It has also been used for filter approximation~\cite{ocone1986application}.

\subsection{Kalman filter for Markov chains}

Consider the deterministic control input $U \in L^2\big([0,T];\Re^m\big)$. In this case, $Y$ is a deterministic function of time and $V=0$.

Consider the finite state-space case. the objective function~\eqref{eq:dual-optimal-control-a} is simplified as
\begin{equation*}\label{eq:DLQ-finite}
	\sJ(U) = Y_0^\tp \Sigma_0 Y_0 + \int_0^T  |U_t|^2 +
	Y_t^\tp \E ( Q(X_t) ) Y_t \ud t
\end{equation*}
where $\Sigma_0 := \E\big((X_0-\pi_0)(X_0-\pi_0)^\tp\big)$ and $Q$ is 
defined as~\eqref{eq:Q-matrix}.
The resulting problem is a deterministic LQ problem whose optimal solution
$\{U_t:0\le t \le T\}$ will (in general) yield a sub-optimal estimate $S_T$
using~\eqref{eq:estimator}. The derivation is identical to the procedure in 
Section~\ref{ssec:Kalman-filter}.

The optimal solution to the deterministic LQ problem is obtained by:
\begin{align*}
	\frac{\ud}{\ud t}\bSig_t 
	&= \bSig_t A + A^\tp \bSig_t+ \E(Q(X_t)) -\bSig_t H H^\tp 
	\bSig_t,\quad \bSig_0 =\Sigma_0\\
	U_t &= -H^\tp \bSig_t Y_t
\end{align*} 
Upon substituting the optimal control
solution into the estimator~\eqref{eq:estimator}
\[
S_T = Y_0^\tp\mu + \int_0^T Y_t^\tp \bSig_t H \ud Z_t
\]
where, given the state-feedback form of the optimal control, $Y$ solves
\[
\frac{\ud Y_t}{\ud t} = (-A + HH^\tp \bSig_t) Y_t,\quad Y_T = f
\]
Let $\Phi_{t,T}$ denote the state transition matrix and express the
solution as $Y_t= \Phi_{t,T} f$.  Thus,
\[
S_T = f^\tp \underbrace{(\Phi_{0,T}^\tp \pi_0 + \int_0^T 
	\Phi_{t,T}^\tp\bSig_t H \ud Z_t)}_{=:\bar{X}_T}
\]
Noting that time $T$ is arbitrary, upon differentiating with respect
to $T$, one obtains the Kalman filter
\begin{align}
	\ud \barX_t &= (A^\tp-\bSig_tHH^\tp )\barX_t \ud t + \bSig_tH\ud Z_t 
	\nonumber\\
	&= A^\tp \barX_t \ud t - \bSig_tH (\ud Z_t - H^\tp \barX_t\ud t) 
	\label{eq:Kalman-filter-MC}
\end{align}
where we have replaced $T$ by $t$. 
This is the Kalman filter algorithm for Markov chains. Such sub-optimal 
filters for Markov chains have been applied 
in~\cite{lipkrirub84,chebusmey17a}.

\medskip

\begin{remark}
	For the Euclidean case, the optimal solution over the deterministic 
	admissible control is considered in~\cite{jan2021master}.
	Although a recursive formulation for the sub-optimal filter does not appear, the solution can be used to provide approximation of $\pi_t(f)$ for a given function $f\in \clD$.
\end{remark}

\section{Remaining questions in filter stability}

Although we provided number of examples in Section~\ref{sec:examples}, it is an open question under which condition on the HMM implies conditional PI.
As noted in Example~\ref{ex:counter-example-revisit}, the standard PI of the state process alone does not implies conditional PI.
It is possible with additional assumptions on observation process such as~\eqref{eq:assumption-ergodic}---which is satisfied in our white noise observation model---standard PI implies conditional PI.

\medskip

The conditional Poincar\'e inequality only considers the energy term $\pi_t^\nu(\Gamma Y_t)$ from the Lagrangian.  In fact, a weaker condition for Prop.~\ref{prop:bvi-pathwise} is
\[
\pi_t^\nu(\Gamma Y_t) + |\cov_t^\nu(h,Y_t)|^2+\clV_t^\nu(V_t) \ge \alpha_t\clV_t^\nu(Y_t),\quad \sP^\nu\text{-a.s.},\; 0\le t\le T
\]
Then a version of backward variance inequality is obtained by:
\[
\clV_0^\nu(Y_0) \le \E^\nu\Big[\exp\Big(-\int_0^T \alpha_t \ud t\Big)\clV_T^\nu(Y_T)\Big]
\]
It is conjectured that if the dual BSDE is stabilizable, along with suitable technical conditions, there exists a constant $c>0$ that
\[
\frac{1}{T}\int_0^T \alpha_t \ud t \;\longrightarrow\; c,\quad \sP^\nu\text{-a.s.}
\]
The conditional PI is a special case when the state process is sufficiently ergodic.

\newpage

	
\appendix

\chapter{Backward stochastic differential equations}\label{apdx:bsde}

\section{Backward Stochastic Differential Equation}

Backward stochastic differential equation (BSDE) is first introduced and analyzed by Bismut~\cite{bismut1978introductory}, as the adjoint equation for linear-quadratic stochastic control. Later, Pardoux and Peng~\cite{pardoux1990adapted} proved the existence and uniqueness of the solution for general Lipschitz cases. BSDE is also closely related to the finance problem~\cite{el1997backward} and certain types of PDE~\cite{pardoux2014stochastic}. 
For function spaces, backward stochastic partial differential equation (BSPDE) is studied in~\cite{ma1999linear} for elliptic operators.

\subsection{Problem definition} Let $\{\Omega, \clF, \sP\}$ be a probability space and $Z=\{Z_t\in\Re^m:0\le t\le T\}$ be a $m$-dimensional standard Brownian motion. Define canonical filtration $\clZ_t = \sigma\{Z_s:0\le s \le t\}$. The BSDE on Euclidean pace $\Re^d$ seeks an $\clZ$-adapted square-integrable processes on $\Re^d\times \Re^{d\times m}$: $\{(Y_t,V_t):0\le t \le T\}$ such that for all $t\in[0,T]$,
\begin{equation}\label{eq:bsde_int}
	Y_t + \int_t^T f(Y_s, V_s, s)\ud s + \int_t^T V_s\ud Z_s = \xi
\end{equation}
It may be written in a differential form:
\begin{equation}\label{eq:bsde_diff}
	\ud Y_t = f(Y_s,V_t,t) \ud t + V_t\ud W_t, \quad Y_T = \xi
\end{equation}
Technical assumptions are introduced to obtain existence and uniqueness of the solution: We assume
\begin{enumerate}
	\item[(i)] For each $(y,v)\in\Re^d\times \Re^{d\times m}$, $f(y,v,\cdot) \in L^2_\clZ\big(\Omega\times [0,T];\Re^d\big)$.
	\item[(ii)] $f$ is Lipschitz with both $y$ and $v$ for almost every $t$ almost surely.
	\item[(iii)] $\xi \in L^2_{\clZ_T}(\Omega;\Re^d)$ is a $\clZ_T$-measurable random vector.
\end{enumerate}
Under these assumption, the following theorem is proved in~\cite{pardoux1990adapted}.

\medskip

\begin{theorem}[Proposition 2.2 in~\cite{pardoux1990adapted}] There exists a unique pair $\{(Y_t,V_t):0\le t \le T\} \in L^2_{\clZ}\big(\Omega\times[0,T];\Re^d\times\Re^{d\times m}\big)$ that satisfy~\eqref{eq:bsde_int} for all $t\in[0,T]$.
\end{theorem}

\medskip

\section{Linear BSDE and its explicit solution} \label{apdx:linear-bsde}

If $f$ is linear function of its arguments, the solution to BSDE can be obtained in an explicit form. We follow~\cite[Theorem 7.2.2]{yong1999stochastic}. Consider linear BSDE
\begin{equation}\label{eq:bsde_linear}
	\ud Y_t = \Big(A_t Y_t  + \sum_{j=1}^m B_t^j V_t^j + C_t \Big)\ud t + V_t \ud Z_t
\end{equation}
where $A_t$, $B_t^j$ for $j=1,\ldots,m$ and $C_t$ are $\clZ$-adapted processes with proper dimension, and $V_t^j$ is the $j^{\text{th}}$ column of $V_t$. We define a $\Re^{n\times n}$-valued stochastic process $\Psi_t$ according to
\begin{equation}
	\ud \Psi_t = -\Psi_t A_t\ud t - \sum_{j=1}^m \Psi_t B_t^j \ud Z_t^j,\quad \Psi_0 = I
\end{equation}
where $Z_t^j$ is the $j^{\text{th}}$ component of $Z_t$ and $I$ is the $n$-dimensional identity matrix. By applying It\^{o}'s rule on $\Psi_t Y_t$, we have
\begin{equation}\label{eq:lbsde_martingale}
	\ud \big(\Psi_t Y_t\big) = \Psi_t C_t\ud t + \sum_{j=1}^m \Psi_t(V_t^j-B_t^jY_t) \ud Z_t^j
\end{equation}
Therefore it is concluded that
\[
\Psi_t Y_t - \int_0^t \Psi_s C_s \ud s =: \Theta_t
\]
is a $\tsP$-martingale. Therefore, rearranging $\E(\Theta_T|\clZ_t) = \Theta_t$ yields the solution
\begin{equation}\label{eq:lbsde_soln_Y}
	Y_t = \Psi_t^{-1}\Big(\E(\Theta_T|\clF_t) + \int_0^t \Psi_sC_s \ud s\Big)
\end{equation}
Also, by martingale representation theorem, there exists a unique $\clF_t$-adapted process $\Lambda_t^j$ for $j=1,\ldots,m$ such that
\[
\E(\Theta_T|\clF_t) = \E\Theta_T + \sum_{j=1}^m\int_0^t \Lambda_t^j \ud Z_s^j
\]
Combining this with~\eqref{eq:lbsde_martingale} yields the solution
\begin{equation}\label{eq:lbsde_soln_Z}
	V_t^j = \Psi_t^{-1}\Lambda_t^j + B_t^jY_t
\end{equation}
Since both~\eqref{eq:lbsde_soln_Y} and~\eqref{eq:lbsde_soln_Z} have $\Psi_t^{-1}$, we need the existence of the inverse. Indeed, $\Psi_t^{-1}$ is well defined by the following linear SDE:
\[
\ud \Psi_t^{-1} = \Big(A_t \Psi_t^{-1} + \sum_{j=1}^m B_t^jB_t^j\Psi_t^{-1} \Big)\ud t + \sum_{j=1}^m B_t^j \Psi_t^{-1} \ud Z_t^j,\quad \Psi_0^{-1} = 0
\]

\section{Optimal control on BSDE}

Since BSDE is a well-defined dynamical system, it is naturally asked to consider a control problem on it~\cite{peng1993backward}. 

\subsubsection{Optimal control problem} In this note, we consider the following optimal control problem, where the dynamics constraint is given by a BSDE:
\begin{align}
	&\mathop{\mathrm{Minimize}}_{U\in{\cal U}}: \sJ(U):= \E\Big(h(Y_0) + \int_0^T l(Y_\tau, V_\tau, U_\tau) \ud \tau \Big) \label{eq:cost}\\
	&\text{Subject to}: \ud Y_t = f(Y_t,V_t,U_t) \ud t + V_t \ud Z_t,\quad Y_T = \xi \label{eq:dynamics}
\end{align}
where $\xi \in L_{\clZ_T}^2(\Omega; \Re^d)$. The set of admissible control is $\clU = L^2_\clZ(\Omega\times [0,T]\,;D)$ where $D \subset \Re^p$ is a convex set. Functions $h:\Re^d\to \Re$, $l:\Re^d\times \Re^{d\times m} \times \Re^p \to \Re$ and $f:\Re^d\times \Re^{d\times m} \times \Re^p \to \Re^n$ are assumed to have~\eqref{eq:dynamics} to admit the unique solution pair $(Y,V) \in L_\clZ^2\big(\Omega\times[0,T];\Re^d\times \Re^{d\times m}\big)$. 

\subsection{Maximum principle}\label{sec:bsde-maximum-principle}

In this section, we follow the proof in~\cite{peng1993backward}.
It is assumed that $f, h, \sigma, l$ are continuously differentiable with bounded derivatives.

\subsubsection{Variational processes} Suppose $U$ is the optimal control, and consider a perturbed control $U_t^\varepsilon = U_t + \varepsilon V_t\in{\cal U}$. Given fixed $V$, we define variational processes by:
\begin{align*}
	\ud \eta_t &= \big( f_y(Y_t, V_t, U_t) \eta_t+ f_z(Y_t, V_t, U_t)\zeta_t  + f_u(Y_t, V_t, U_t)V_t\big)\ud t + \zeta_t \ud Z_t\\
	\eta_T &= 0
\end{align*}
Then one can show that the first order variation of $\sJ(U)$ with respect to $\tilde{U}$ at $U$ is given by
\begin{align*}
	\delta_{\tilde{U}} \sJ(U) &:= \lim_{\varepsilon\downarrow 0} \frac{1}{\varepsilon}\Big(J(U+\varepsilon \tilde{U}) - \sJ(U)\Big) \\
	&=\E\Big[\int_0^T \big(\langle l_y(Y_t,V_t,U_t),\eta_t\rangle+\langle l_v(Y_t,V_t,U_t),\zeta_t\rangle + l_u(Y_t,V_t,U_t)\tilde{U}_t\big)\ud t + \langle h_y(Y_0), \eta_0\rangle\Big]
\end{align*}
is well defined since ${\cal U}$ is convex, and the optimality of $U$ implies $\delta_{\tilde{U}} \sJ(U) \geq 0$. In order to transform this into more explicit operator on $\tilde{U}$, we define the following adjoint equations:
\begin{align*}
	-\ud P_t &= \big( f_y^\dagger(Y_t, V_t, U_t) P_t - l_y(Y_t, V_t, U_t)\big)\ud t + \big( f_v^\dagger( Y_t, V_t, U_t)P_t - l_v(Y_t, V_t, U_t)\big) \ud Z_t \nonumber \\
	P_0 &= h_y(Y_0)
\end{align*}
where the asterisk denotes formal adjoint. The existence of the solution $\{P_t\}_{t\in[0,T]}$ is guaranteed from the assumption. Apply It\^{o}'s rule on $\langle\eta_t,P_t\rangle$. After some rearranging and cancellation, we have
\[
\langle h_y(Y_0),\eta_0 \rangle= \int_0^T \big(-\langle P_t,f_u\tilde{U}_t\rangle -\langle \eta_t,l_y\rangle-\langle\zeta_t,l_v\rangle\big) \ud t + \int_0^T \big(\;\cdots\;\big) \ud Z_t
\]
This goes back to the optimality condition and it arrives:
\begin{equation}\label{eq:optimality_condition}
	\delta_{\tilde{U}} \sJ(U) = \E\Big(\int_0^T -\big(\langle P_t,f_u\rangle -l_u\big)\tilde{U}_t \ud t\Big) \geq 0
\end{equation}

\subsubsection{Hamiltonian and Maximum principle} The Hamiltonian is thus defined by:
\begin{equation}
	\clH(y,z,u,p) = \langle f(y,z,u),p\rangle-l(y,z,u)
\end{equation}
Then the adjoint equation for $P$ is simplified to
\begin{equation}\label{eq:adjoint}
	-\ud P_t = \clH_y(Y_t,V_t,U_t,P_t) \ud t + \clH_v(Y_t,V_t,U_t,P_t) \ud Z_t,\quad P_0 = h_y(Y_0)
\end{equation}
and the optimality of the first variation~\eqref{eq:optimality_condition} is
\[
\E\Big[\int_0^T \clH_u(Y_t,V_t,U_t,P_t)\tilde{U}_t \ud t\Big] \leq 0
\]
and this implies the following statement of maximum principle.

\begin{theorem}[Maximum principle on BSDE, Theorem 4.4 in~\cite{peng1993backward}] Consider an optimal control problem on BSDE defined by~\eqref{eq:dynamics} and \eqref{eq:cost}. For the optimal control $U$ and corresponding trajectory $\{Y_t,V_t\}_{t\in[0,T]}$, there exist adjoint process $\{P_t\}_{t\in[0,T]}$ according to~\eqref{eq:adjoint} and
	\begin{equation}
		\big\langle \clH_u(Y_t,V_t,U_t,P_t), u-U_t\big\rangle \leq 0,\quad \forall u \in D,\; \text{a.e. } t,\;\text{a.s.}
	\end{equation}
\end{theorem}

\subsection{Dynamic programming}

While stochastic maximum principle is mainly discussed in literature, dynamic programming approach has been less noticed. The results below are original study by the author from~\cite{kim2021dynamicprogramming}.

To formulate the dynamic programming principle, consider a partial problem up to time $t \le T$ from  $\zeta\in L^2_{\clZ_t}(\Omega;\Re^d)$ defined by:
\[
\sJ_t(U;\zeta) = \E\Big(h(Y_0) + \int_0^t l(Y_s^{\zeta,t},V_s^{\zeta,t},U_s,s) \ud s \Big)
\]
where $\big\{\big(Y_s^{\zeta,t}, V_s^{\zeta,t}\big) : \tau \in [0,t]\big\}$ is the solution to the BSDE:
\[
\ud Y_s^{\zeta,t} = f\big(Y_s^{\zeta,t},V_s^{\zeta,t},U_s,\tau\big)\ud s + V_s^{\zeta,t} \ud Z_s,\quad Y_t^{\zeta,t} = \zeta
\]
\begin{definition}\label{def:value-function}
	Consider the optimal control problem~\eqref{eq:cost}-\eqref{eq:dynamics}. The \emph{value function} is a sequence of functions $\clV:=\{\clV_t:0\le t\le T\}$ where $\clV_t:L^2_{\clZ_t}(\Omega\times \Re^d) \to \Re$ is defined by:
	\begin{equation}\label{eq:value-function}
		\clV_t(\zeta) = \inf_{U\in{\cal U}} \sJ_t(U;\zeta)
	\end{equation}
\end{definition}

\medskip

Analogously with the forward-in-time stochastic DP principle, the following theorem is proposed:

\medskip

\begin{theorem}\label{thm:DP-principle}
	Let $\clV$ be the value function of the optimal control problem~\eqref{eq:cost}-\eqref{eq:dynamics}. Then it satisfies the following:
	\begin{enumerate}
		\item[(i)] $\clV_0(\cdot) = h$
		\item[(ii)] For any $0\le s < t \le T$ and any $\zeta\in L^2_{\clZ_t}(\Omega;\Re^d)$,
		\begin{equation}\label{eq:DP-principle}
			\clV_t(\zeta) = \inf_{U\in{\cal U}} \E\Big(\clV_s(Y_s^{\zeta,t}) + \int_s^t l(Y_s^{\zeta,t},V_s^{\zeta,t},U_s,s)\ud s\Big)
		\end{equation}
	\end{enumerate}
\end{theorem}

\begin{proof}
	We start from the definition of the value function:
	\begin{align*}
		\clV_t(\zeta) &= \inf_{U\in\clU}  \E\Big(h(Y_0^{\zeta,t}) + \int_0^t l(Y_\tau^{\zeta,t},V_\tau^{\zeta,t},U_\tau,\tau) \ud \tau \Big)\\
		&= \inf_{U\in{\cal U}}  \E\Big(h(Y_0^{\zeta,t}) + \int_0^s l(Y_\tau^{\zeta,t},V_\tau^{\zeta,t},U_\tau,\tau) \ud \tau + \int_\tau ^t l(Y_\tau^{\zeta,t},V_\tau^{\zeta,t},U_\tau,\tau) \ud \tau \Big)
	\end{align*}
	The claim is that the first two terms are precisely $\clV_s(Y_s^{\zeta,t})$. Observe that
	\[
	Y_s^{\zeta,t} = \zeta - \int_s^t f(Y_\tau,V_\tau,U_\tau,\tau) \ud \tau - \int_s^t V_\tau \ud Z_\tau
	\] 
	Note that the right-hand side depends only on $U_\tau: \tau\in[s,t]$. Meanwhile,
	\[
	Y_u^{\zeta,t} = Y_s^{\zeta,t} - \int_u^t f(Y_\tau,V_\tau,U_\tau,\tau) \ud \tau - \int_u^s V_\tau \ud Z_\tau,\quad u\le s
	\]
	depends only on $U_\tau:\tau\in[0,s]$ given $Y_s^{\zeta,t}$, and therefore \[
	\big(Y_u^{\zeta,t},V_u^{\zeta,t}\big) = \Big(Y_u^{Y_s^{\zeta,t},s},V_u^{Y_s^{\zeta,t},s}\Big)
	\]
\end{proof}

\medskip

For stochastic optimal control problems, an appealing formulation for the value function is to construct a martingale associated with it. The martingale version of DP principle is as follows.

\medskip

\begin{proposition}\label{thm:martingale-DP}
	Let $\clV$ be the value function of the optimal control problem~\eqref{eq:cost}-\eqref{eq:dynamics}. Then
	\begin{enumerate}
		\item $\clV_0(\cdot) = h$
		
		\item Define $M^U = \{M_t^U:0\leq t\leq T\}$ for any admissible control $U\in{\cal U}$ by:
		\begin{equation}\label{eq:martingale-def}
			M_t^U = \clV_t(Y_t) - \int_0^t l(Y_s,V_s,U_s,\tau) \ud \tau
		\end{equation}
		where $(Y,V)$ is the solution to~\eqref{eq:dynamics}. $M^U$ is a super-martingale for any admissible control $U$; and it is a margingale if and only if $U$ is the optimal solution.
	\end{enumerate}
\end{proposition}

\begin{proof}
	We begin with recalling~\eqref{eq:DP-principle}:
	\[
	\clV_t(\zeta) \leq \E\Big(\clV_s(Y_s^{\zeta,t}) + \int_s^t l(Y_\tau^{\zeta,t},V_\tau^{\zeta,t},U_\tau,\tau)\ud \tau \Big)
	\]
	Note that both sides are map a random variable $\zeta$ to a scalar. For $\zeta = Y_t$, 
	\[
	(Y_s,V_s) = \big(Y_s^{\zeta,t},V_s^{\zeta,t}\big)
	\]
	and therefore
	\[
	\E\big(\clV_t(Y_t)\mid \clZ_s \big) \le \E\Big(\clV_s(Y_s) + \int_s^t l(Y_\tau,V_\tau,U_\tau,\tau)\ud \tau \mid \clZ_s \Big)
	\]
	Upon subtracting $\E\big(\int_0^t l(Y_\tau,V_\tau,U_\tau,\tau)\ud \tau \mid \clZ_s \big)$ on both sides, we have
	\begin{align*}
		\E\Big(\clV_t(Y_t) - \int_0^t &(Y_\tau,V_\tau,U_\tau,\tau)\ud \tau \mid \clZ_s \Big) \\
		&\leq \E\Big(\clV_s(Y_s) - \int_0^s l(Y_\tau,V_\tau,U_\tau,\tau)\ud \tau  \mid \clZ_s \Big)
	\end{align*}
	Since the right-hand side is $\clZ_s$-measurable, we may drop conditional expectation, and hence
	$$
	\E\big(M_t^U\mid \clZ_s\big)\leq M_s^U
	$$
	Therefore, $M^U$ is a super-martingale. The inequality becomes equality upon choosing the optimal control.
\end{proof}

\medskip

\subsubsection{Optimal control obtained via the martingale DP principle}

The converse of the DP principle is often called verification theorem.

\medskip

\begin{theorem}\label{thm:converse-DP}
	Suppose there exists $\clV$ and $U^\opt \in {\cal U}$ such that:
	\begin{enumerate}
		\item $\clV_0(\cdot) = h(\cdot)$.
		
		\item The process $M^U$ defined by~\eqref{eq:martingale-def} is a super-martingale for each admissible control $U\in \clU$, and a martingale for $U=U^\opt$.
	\end{enumerate}
	Then $U = U^\opt$ is an optimal control with cost $\E\big(\clV_T(\xi)\big)$.
\end{theorem}

\begin{proof}
	Since $M^U$ is a super-martingale, 
	\[
	\E M_T^U \leq M_0^U = \clV(Y_0,0) = h(Y_0)
	\]
	Take expectation on the right-hand side and expand the left-hand side as
	\[
	\E\Big(\clV(\xi,T) - \int_0^T l(Y_t,V_t,U_t)\ud t\Big) \leq \E\big(h(Y_0)\big)
	\]
	Therefore we have
	\[
	\E\big(\clV(\xi,T)\big) \leq \sJ(U),\quad \forall U
	\]
	where equality holds for $U=U^\opt$.
\end{proof}

\medskip

\begin{remark}
	From the definition of super-martingale, the second condition is equivalent to write for any $0\le s\le t\le T$,
	\begin{equation}\label{eq:Martingale-DP-conclusion}
		\clV_s(Y_s) \geq \E\Big(\clV_t(Y_t) - \int_s^t l(Y_s,V_s,U_s,\tau) \ud \tau \mid \clZ_s \Big)
	\end{equation}
	For forward-in-time Markovian stochastic control problems, the counterpart of~\eqref{eq:Martingale-DP-conclusion} is exactly the dynamic programming principle and $\clV$ is the value function.
	However for BSDE problems, conditioning on $\clZ_s$ is not the
	same as fixing on $Y_s$. Therefore, the conditions in
	Theorem~\ref{thm:converse-DP} do not yield an interpretation of $\clV_t$ as the value function at time $t$ in this case but only concludes that $U^\opt$ is optimal.
\end{remark}

\newpage


\chapter{Minimum energy duality and optimal smoothing}\label{apdx:min-energy}

\def\normfactor{C}

In this chapter, we provide a self-contained
exposition of the equations of nonlinear smoothing as well as connections and
interpretations to some of the more recent developments in
mean-field-type optimal control theory.
This chapter is entirely based on~\cite{kim2020smoothing}.

\section{Preliminaries and Background}
\label{sec:prelim}

\subsection{The smoothing problem}

Consider a pair of continuous-time stochastic processes $(X,Z)$.  The
state $X=\{X_t:t\in[0,T]\}$ is a Markov process taking values in the
state space $\bS$. The observation process $Z =
\{Z_t:t\in[0,T]\}$ is defined according to the model:
\begin{equation}\label{eq:obs_model}
	Z_t = \int_0^t h(X_s) \ud s + W_t
\end{equation}
where $h:\bS\rightarrow \Re$ is the observation function
and $W=\{W_t: t\ge 0\}$ is a standard Wiener process.

The smoothing problem is to compute the posterior distribution ${\sf
	P}(X_t\in \;\cdot\;|\clZ_T)$ for arbitrary $t\in [0,T]$, where
$\clZ_T := \sigma(Z_s : 0\le s \le T)$ is the sigma-field generated by
the observation up to the terminal time $T$.

\subsection{Solution of the smoothing problem}

The smoothing problem requires a model of the Markov process $X$.  In
applications involving nonlinear smoothing, a common model is
the It\^o-diffusion in Euclidean settings:


\subsubsection{Euclidean state space} The state space $\bS=\Re^d$. The
state process $X$ is modeled as an It\^o diffusion:
\begin{equation*}
	\ud X_t = a (X_t) \ud t + \sigma(X_t) \ud B_t, \quad X_0\sim \nu_0
\end{equation*}
where $a\in C^1(\Re^d; \Re^d)$, $\sigma\in C^2(\Re^d; \Re^{d\times
	p})$ and $B=\{B_t:t\ge 0\}$ is a standard Wiener process.  The initial
distribution of $X_0$ is denoted as $\nu_0(x) \ud x$ where $\nu_0(x)$
is the probability density with respect to the Lebesgue measure.
For~\eqref{eq:obs_model}, the observation function $h\in C^2(\Re^d;\Re)$.  
It is assumed that $X_0, B, W$ are mutually independent.

The infinitesimal generator of $X$, denoted as ${\cal A}$,
acts on $C^2$ functions in its domain according to 
\[
({\cal A} f)(x):= a^\tp(x) \nabla f(x) + \half \tr\big(\sigma\sigma^\tp(x)(D^2f)(x)\big)
\]
The adjoint operator is denoted by ${\cal A}^\dagger$.  It acts on
$C^2$ functions in its domain according to 
\[
({\cal A}^\dagger f)(x) = -\divg(af)(x) + \half \sum_{i,j=1}^d \frac{\partial^2}{\partial x_i \partial x_j}\big([\sigma\sigma^\tp]_{ij} f\big)(x) 
\]


The solution of the smoothing problem is described by a
forward-backward system of stochastic partial differential equations
(SPDE) (see~\cite[Thm. 3.8]{pardoux1981non}):
\begin{subequations}\label{eq:forward-backward-Zakai}
	\begin{flalign}
		&\text{(forward)}:&\ud p_t (x) &= ({\cal A}^\dagger p_t)(x) \ud t +
		h(x) p_t(x) \ud Z_t,\quad p_0(x) = \nu_0(x), \quad \forall x\in\Re^d & \label{eq:Zakai-a}\\
		&\text{(backward)}:&-\ud q_t(x) &= {(\cal A} q_t)(x) \ud t + h(x)
		q_t(x) \overleftarrow{\ud Z}_t,\quad q_T (x) \equiv 1 &\label{eq:Zakai-b}
	\end{flalign}
\end{subequations}
where $\overleftarrow{\ud Z}_t$ denotes a backward It\^{o} integral (see \cite[Remark 3.3]{pardoux1981non}). The smoothed distribution is then obtained as follows:
\[
{\sf P}(X_t \in \ud x\; |\clZ_T) \propto p_t(x)q_t(x) \ud x.
\]
Each of~\eqref{eq:forward-backward-Zakai} is referred to as the Zakai equation of nonlinear
filtering.  

\subsection{Path-wise representation of the Zakai equations}\label{ssec:path-wise-Zakai}

There is a representation of the forward-backward SPDEs where the only
appearance of randomness is in the coefficients.  This is referred to
as the pathwise (or robust) form of the
filter~\cite[Sec. VI.11]{rogers2000diffusions}.

Using It\^o's formula for $\log p_t$,
\begin{align*}
	\ud (\log p_t) (x)&=\frac{1}{p_t(x)} ({\cal A}^\dagger p_t)(x) \ud t + h(x)\ud Z_t - \half h^2(x) \ud t
\end{align*}
Therefore, upon defining $\mu_t(x) := \log p_t(x) - h(x)Z_t$, the forward
Zakai equation~\eqref{eq:Zakai-a} is transformed into a parabolic partial differential
equation (pde):
\begin{equation}
	\frac{\partial \mu_t}{\partial t} (x) =
	e^{-(\mu_t(x)+Z_th(x))}\big({\cal A}^\dagger
	e^{(\mu_t(\cdot) +Z_th(\cdot) )}\big)(x) - \half h^2(x),\quad \mu_0(x) = \log \nu_0(x),\quad \forall x\in\Re^d \label{eq:pw-Zakai-a}
\end{equation}
Similarly, upon defining $\lambda_t(x) = \log q_t(x) + h(x) Z_t$, the
backward Zakai equation~\eqref{eq:Zakai-b} is transformed into the
parabolic pde:
\begin{equation}
	-\frac{\partial \lambda_t}{\partial t} (x) = e^{-(\lambda_t(x) - Z_th(x))}\big({\cal A}e^{\lambda_t(\cdot) - Z_th(\cdot)}\big)(x) - \half h^2(x), \quad \lambda_T(x) = Z_Th(x),\quad \forall x\in\Re^d \label{eq:pw-Zakai-b}
\end{equation}
The pde~\eqref{eq:pw-Zakai-a}-\eqref{eq:pw-Zakai-b} are referred to as
pathwise equations of nonlinear smoothing.  

\subsection{The finite state-space case}

Apart from It\^{o}-diffusion, another common model is a Markov chain in
finite state-space settings: 

\subsubsection{Finite state space} Let the state-space be $\bS =
\{e_1,e_2,\ldots,e_d\}$, the canonical basis in $\Re^d$.  
For~\eqref{eq:obs_model}, the linear observation model is chosen
without loss of generality: for any function $h:\bS\to\Re$, we
have $h(x)=\tilde{h}^\tp x$ where $\tilde{h}\in\Re^{d}$ is defined by
$\tilde{h}_i = h(e_i)$.  
Thus, the function space on $\bS$ is
identified with $\Re^d$.  With a slight abuse of notation, we will
drop the tilde and simply write $h(x)=h^\tp x$. 

The state process $X$ is a continuous-time Markov chain evolving in $\bS$. 
The initial distribution for
$X_0$ is denoted as $\nu_0$.  It is an element of the probability simplex
in $\Re^d$. 
The generator of the chain is denoted as $A$.  It is a
$d\times d$ row-stochastic matrix.  It acts on a function $f\in\Re^d$
through right multiplication: $f\mapsto Af$.  The adjoint operator is the
matrix transpose $A^\tp$.  
It is assumed that $X$ and $W$ are mutually independent.

The solution of the smoothing problem for the finite state-space
settings is entirely analogous: Simply replace the generator ${\cal
	A}$ in~\eqref{eq:forward-backward-Zakai} by the matrix $A$, and the probability density by the probability mass function.  The Zakai pde is now the
Zakai sde.  The formula for the pathwise representation are also
entirely analogous:
\begin{align}
	\Big[\frac{\ud \mu_t}{\ud t}\Big]_i &= [e^{-(\mu_t+Z_th)}]_i[A^\tp e^{\mu_t+Z_th}]_i - \half [h^2]_i \label{eq:pw-Zakai-finite-a}\\
	-\Big[\frac{\ud \lambda_t}{\ud t}\Big]_i &= [e^{-(\lambda_t-Z_th)}]_i[Ae^{\lambda_t-Z_th}]_i - \half [h^2]_i \label{eq:pw-Zakai-finite-b}
\end{align}
with boundary condition $[\mu_0]_i = \log [\nu_0]_i$ and $[\lambda_0]_i = Z_T[h]_i$, for $i=1, \ldots,d$.

\section{Optimal Control Problem}\label{sec:control-problem}

\subsection{Variational formulation}

For the smoothing problem, an optimal control formulation is derived in the following two steps:

\subsubsection{Step 1} A control-modified version of the Markov process $X$ is introduced.  The controlled process is denoted as $\tilde{X} :=
\{\tilde{X}_t:0\le t\le T\}$.  
The control problem is to pick (i) the initial distribution
$\pi_0\in{\cal P}(\bS)$ and (ii) the state transition, such that
the distribution of $\tilde{X}$ equals the conditional distribution.
For this purpose, an optimization problem is formulated in the next step.

\subsubsection{Step 2} The optimization problem is formulated on the space of
probability laws. 
Let $\sP_{X}$ denote the law for $X$, $\sQ$ denote the law for
$\tilde{X}$, and $\sP_{X\mid z}$ denote the law for $X$ given an observation path
$z=\{z_t:0\le t\le T\}$.  Assuming
these are equivalent, the objective function is the relative entropy between $\sQ$ and $\sP_{X\mid z}$:
\begin{equation*}\label{eq:rel-entropy-cost}
	\min_{\sQ} \quad \E_{\sQ}\Big(\log \frac{\ud \sQ}{\ud \sP_{X}}\Big) - \E_{\sQ}\Big(\log\frac{\ud \sP_{X\mid z}}{\ud \sP_{X}}\Big)
\end{equation*}

Upon using the Kallianpur-Striebel formula
(see~\cite[Lemma 1.1.5 and Prop. 1.4.2]{van2006filtering}), the
optimization problem is equivalently expressed as follows:
\begin{equation}\label{eq:cost-equiv-form}
	\min_{\sQ} \quad \kl(\sQ\mid \sP_{X}) +\E\Big(\int_0^T z_t \ud h(\tilde{X}_t) + \half|h(\tilde{X}_t)|^2 \ud t- z_Th(\tilde{X}_T)\Big)
\end{equation}
The first of these terms depends upon the details of the model used to
parametrize the controlled Markov process $\tilde{X}$.  For the two
types of Markov processes, this is discussed in the following sections.

\begin{remark}
	The Schr\"{o}dinger bridge problem is a closely related problem of
	recent research interest where one picks $\sQ$ to minimize $\kl(\sQ\mid \sP_{X})$
	subject to the constraints on marginals at time $t=0$ and $T$; cf.,~\cite{chen2016relation} where connections to
	stochastic optimal control theory are also described.  Applications of
	such models to the filtering and smoothing problems is discussed
	in~\cite{reich2019data}.  There are two differences between the
	Schr\"{o}dinger bridge problem and the smoothing problem considered
	here:
	\begin{enumerate}
		\item The objective function for the smoothing problem also includes
		an additional integral term in~\eqref{eq:cost-equiv-form} to account for
		conditioning due to observations $z$ made over time $t\in[0,T]$;
		\item The constraints on the marginals at time $t=0$ and $t=T$ are not
		present in the smoothing problem.
		Rather, one is allowed to pick the initial distribution $\pi_0$ for
		the controlled process and there is no constraint present on the 
		distribution at the terminal time $t=T$.  
	\end{enumerate}
\end{remark}

\subsection{Optimal control: Euclidean state-space}
\label{sec:euclidean_ss}

The modified process $\tilde{X}$ evolves on
the state space $\Re^d$.  It is modeled as a controlled
It\^{o}-diffusion
\[
\ud \tilde{X}_t = a(\tilde{X}_t) \ud t + \sigma(\tilde{X}_t)\big(u_t(\tilde{X}_t)\ud t + \ud \tilde{B}_t\big),\quad \tilde{X}_0\sim \pi_0
\]
where $\tilde{B} = \{\tilde{B}_t:0\leq t\leq T\}$ is a copy of the process noise $B$. The controlled process is parametrized by:
\begin{enumerate}
	\item The initial density $\pi_0(x)$.
	\item The control function $u\in C^1([0,T]\times\Re^d;\Re^p)$.  The
	function of two arguments is denoted as $u_t(x)$. 
\end{enumerate}
The parameter $\pi_0$ and the function $u$ are chosen as a solution of an optimal control problem.

For a given function $v\in C^1(\Re^d;\Re^p)$, the generator of the controlled
Markov process is denoted by $\tilde{\cal A}(v)$.  It acts on a $C^2$
function $f$ in its domain according to
\begin{align*}
	(\tilde{\cal A}(v)f)(x)
	&= ({\cal A}f)(x) + (\sigma v)^\tp(x)\nabla f(x)
\end{align*}
The adjoint operator is denoted by ${\cal A}^\dagger(v)$.  It acts on
$C^2$ functions in its domain according to 
\begin{align*}
	(\tilde{\cal A}^\dagger(v) f)(x)
	&=({\cal A}^\dagger f)(x) - \divg(\sigma v f)(x)
\end{align*}

For a density $\rho$ and a function $g$, define $\langle \rho, g
\rangle := \int_{\Re^d} g(x) \rho(x) \ud x$. With this notation, define the controlled
Lagrangian ${\cal L}:C^2(\Re^d;\Re^{+})\times C^1(\Re^d;\Re^p) \times \Re \to \Re$ as follows:
\[
{\cal L}(\rho,v\,;y) := \half \langle \rho,  |\,v\;|^2 + h^2\rangle + y\,\langle \rho,\tilde{\cal A}(v)h \rangle
\]
The justification of this form of the
Lagrangian starting from the relative entropy cost appears in
Appendix~\ref{apdx:lagrangian-sde}.


For a given fixed observation path $z = \{z_t:0\leq t\leq T\}$, the
optimal control problem is as follows:
\begin{subequations}\label{eq:opt-cont-sde}
	\begin{align}
		&\mathop{\text{Min}}_{\pi_0, u} : \sJ(\pi_0, u\,;z) =\kl(\pi_0\mid \nu_0) - z_T \langle \pi_T,h\rangle + \int_0^T {\cal
			L}(\pi_t,u_t;z_t) \ud t \label{eq:cost-sde}\\
		&\text{Subj.} :\frac{\partial\pi_t}{\partial t}(x) = (\tilde{\cal A}^\dagger(u_t)\pi_t)(x)
	\end{align}
\end{subequations}

\medskip

\begin{remark}
	This optimal control problem is a mean-field-type problem on account
	of the presence of the entropy term $\kl(\pi_0\mid \nu_0)$ in the
	objective function.  The Lagrangian is in a standard
	stochastic control form and the problem can be solved as a stochastic
	control problem as well~\cite{mitter2003}.  In this paper, the
	mean-field-type optimal control formulation is stressed as a
	straightforward way to
	derive the equations of the nonlinear smoothing.  
\end{remark}

The solution to this problem is given in the following proposition, whose proof appears in the Appendix~\ref{apdx:opt-ctrl-sde}. 
\begin{proposition}\label{thm:opt-ctrl-sde}
	Consider the optimal control problem~\eqref{eq:opt-cont-sde}.
	For this problem, the Hamilton's equations are as follows:
	\begin{subequations}\label{eq:hamiltons-eqn-sde}
		\begin{flalign}
			&\text{(forward)}&\frac{\partial \pi_t}{\partial t}(x) &= (\tilde{\cal
				A}^\dagger(u_t) \pi_t)(x) \label{eq:hamiltons-eqn-sde-a}&\\
			&\text{(backward)}&-\frac{\partial\lambda_t}{\partial t} (x) &= e^{-(\lambda_t(x) - z_th(x))}({\cal A}e^{\lambda_t(\cdot) - z_th(\cdot)})(x) - \half h^2(x) \label{eq:hamiltons-eqn-sde-b}&\\
			&\text{(boundary)}& \lambda_T(x) &= z_Th(x) \nonumber &
		\end{flalign}
	\end{subequations}
	The optimal choice of the other boundary condition is as follows:
	\[
	\pi_0(x) = \frac{1}{\normfactor} \nu_0(x) e^{\lambda_0(x)}
	\]
	where $\normfactor = \int_{\Re^d}
	\nu_0(x)e^{\lambda_0(x)}\ud x$ is the normalization factor.
	The optimal control is as follows:
	\[
	u_t(x) = \sigma^\tp(x)\, \nabla (\lambda_t-z_th)(x)
	\]
\end{proposition}


\subsection{Optimal control: finite state-space}
\label{sec:finite_ss}

The modified process $\tilde{X}$ is a
Markov chain that also evolves in $\bS =
\{e_1,e_2,\ldots,e_d\}$.  The control problem is parametrized by the
following:
\begin{enumerate}
	\item The initial distribution denoted as $\pi_0\in\Re^d$.  
	\item The state transition matrix denoted as $\tilde{A}(v)$ where
	$v\in (\Re^+)^{d\times d}$ is the control input.  After~\cite[Sec. 2.1.1.]{van2006filtering}, it is
	defined as follows:
	\[
	[\tilde{A}(v)]_{ij} = \begin{cases}
		[A]_{ij}[v]_{ij}\quad & i\neq j\\
		-\sum_{j\neq i} [\tilde{A}(v)]_{ij} & i=j
	\end{cases}
	\]
	and we set $[v]_{ij} = 1$ if $i=j$ or if $[A]_{ij} = 0$.
\end{enumerate}

To set up the optimal control problem, define a function
$C:(\Re^+)^{d\times d}\to\Re^d$ as follows
\[
[C(v)]_i = \sum_{j=1}^d [A]_{ij}[v]_{ij}(\log [v]_{ij}-1),\quad i = 1,\ldots,d 
\]
The Lagrangian for the optimal control problem is as follows:
\[
{\cal L}(\rho,v;y) := \rho^\tp(C(v)+\half h^2) + y \;\rho^\tp(\tilde{A}(v)h)
\]
The justification of this form of the
Lagrangian starting from the relative entropy cost appears in
Appendix~\ref{apdx:lagrangian-finite}.


For given observation path $z=\{z_t:0\le t\le T\}$, the optimal
control problem is as follows:
\begin{subequations}\label{eq:opt-cont-finite}
	\begin{align}
		\mathop{\text{Min }}_{\pi_0, u} &:\sJ(\pi_0, u\,;z) 
		= \kl(\pi_0\mid \nu_0) - z_T \pi_T^\tp h + \int_0^T {\cal L}(\pi_t,u_t;z_t) \ud t \label{eq:cost-finite}\\
		\text{Subj.} &:\frac{\ud \pi_t}{\ud t} = \tilde{A}^\tp(u_t) \pi_t \label{eq:dyn_finite}
	\end{align}
\end{subequations}

The solution to this problem is given in the following proposition, whose proof appears in the Appendix.

\begin{proposition}\label{thm:opt-ctrl-finite}
	Consider the optimal control problem~\eqref{eq:opt-cont-finite}. For this problem, the Hamilton's equations are as follows:
	\begin{subequations}\label{eq:hamiltons-eqn-finite}
		\begin{flalign}
			&\text{(forward)}&\frac{\ud \pi_t}{\ud t} &= \tilde{A}^\tp(u_t) \pi_t \label{eq:hamiltons-eqn-finite-a}&\\
			&\text{(backward)}&-\frac{\ud \lambda_t}{\ud t} &=
			\dv(e^{-(\lambda_t-z_th)}) \; A\, e^{\lambda_t-z_th} - \half h^2 \label{eq:hamiltons-eqn-finite-b}&\\
			&\text{(boundary)}& \lambda_T &= z_Th. \nonumber &
		\end{flalign}
	\end{subequations}
	The optimal boundary condition for $\pi_0$ is given by:
	\[
	[\pi_0]_i = \frac{1}{\normfactor} [\nu_0]_i[e^{\lambda_0}]_i,\quad i=1,\ldots,d
	\]
	where $\normfactor = \nu_0^\tp e^{\lambda_0}$. The optimal control is
	\[
	[u_t]_{ij} = e^{([\lambda_t - z_th]_j-[\lambda_t - z_th]_i)}
	\]

\end{proposition}

\medskip

\subsection{Derivation of the smoothing equations}


The pathwise equations of nonlinear filtering are
obtained through a coordinate transformation.  The proof for the following proposition is contained in the Appendix \ref{apdx:forward-backward}.  

\begin{proposition}\label{thm:forward-backward}
	Suppose $(\pi_t(x),\lambda_t(x))$ is the solution to the Hamilton's equation \eqref{eq:hamiltons-eqn-sde}.  Consider the following transformation:
	\[
	\mu_t(x) = \log(\pi_t(x)) -\lambda_t(x) + \log(\normfactor)
	\]   
	The pair $(\mu_t(x), \lambda_t(x))$ satisfy path-wise smoothing equations~\eqref{eq:pw-Zakai-a}-\eqref{eq:pw-Zakai-b}.
	Also,
	\[
	{\sf P}(X_t\in\ud x\;|\clZ_T) = {\pi}_t(x)\ud x \quad\forall t\in[0,T]
	\]
	For the finite state-space case~\eqref{eq:hamiltons-eqn-finite}, the analogous formulae are as follows:
	\[
	[\mu_t]_i = \log( [\pi_t]_i) -[\lambda_t]_i +  \log(\normfactor)
	\]
	and 
	\[
	{\sf P}(X_t = e_i\;|\clZ_T) = [{\pi}_t]_i \quad\forall t\in[0,T]
	\]
	for $i = 1,\ldots,d$.
\end{proposition}



\subsection{Relationship to the log transformation}\label{ssec:mitter-problem}

In this paper, we have stressed the density control viewpoint.
Alternatively, one can express the problem as a stochastic control
problem for the $\tilde{X}$ process.  For this purpose, define the
cost function $l:\Re^d \times \Re^p \times \Re \rightarrow \Re $ as follows:
\[
\ell(x,v\,;y) := \half |v|^2 + \half h^2(x) + y(\tilde{\cal A}(v)h)(x)
\]
The stochastic optimal control problem for the Euclidean case then is
as follows:
\begin{subequations}\label{eq:opt-cont-sde-hjb}
	\begin{align}
		\mathop{\text{Min }}_{\pi_0, U} &:\sJ(\pi_0,U\,;z) = \E\Big(\log \frac{\ud \pi_0}{\ud \nu_0}(\tilde{X}_0) - z_T h(\tilde{X}_T) + \int_0^T \ell(\tilde{X}_t,U_t\,;z_t)\ud t\Big)\\
		\text{Subj.} &:\ud \tilde{X}_t = a(\tilde{X}_t)\ud t + \sigma(\tilde{X}_t)(U_t\ud t + \ud \tilde{B}_t),\quad X_0\sim \pi_0
	\end{align}
\end{subequations}
Its solution is given in the following proposition whose proof appears
in the Appendix~\ref{apdx:proof-hjb}.  

\begin{proposition}\label{thm:opt-ctrl-sde-hjb}
	Consider the optimal control problem~\eqref{eq:opt-cont-sde-hjb}.
	For this problem, the HJB equation for the value function $V$ is
	as follows:
	\begin{align*}
		-\frac{\partial V_t}{\partial t}(x) &= \big({\cal A}(V_t+z_th)\big)(x) + \half h^2(x) -\half|\sigma^\tp\nabla (V_t+z_th)(x)|^2\\
		V_T(x) &= - z_Th(x)
	\end{align*}
	The optimal control is of the state feedback form as follows:
	\[
	U_t = u_t(\tilde{X}_t) 
	\]
	where $u_t(x) = -\sigma^\tp \nabla(V_t + z_th)(x)$.  
\end{proposition}

\medskip

The HJB equation thus is exactly the Hamilton's
equation~\eqref{eq:hamiltons-eqn-sde-b} and 
\[
V_t(x) = -\lambda_t(x) ,\quad \forall x\in\Re^d,\ \forall\, t\in[0,T]
\]
Noting $\lambda_t(x) = \log q_t(x) + h(x)z_t$, the HJB equation for the
value function $V_t(x)$ is related to the backward Zakai equation for
$q_t(x)$ through the log transformation
(cf.~\cite[Eqn. 1.4]{fleming1982optimal}):
\[
V_t(x) = -\log \big(q_t(x)e^{z_th(x)}\big)
\]

\subsection{Linear Gaussian case}\label{apss:linear-Gaussian}

The linear-Gaussian case is a special case in the Euclidean setting
with the following assumptions on the model:
\begin{enumerate}
	\item The drift is linear in $x$.  That is, 
	\[ 
	a(x)=A^\tp x \;\; \text{and}\;\; h(x) = H^\tp x
	\] 
	where $A\in\Re^{d\times d}$ and $H\in\Re^{d}$.
	\item The coefficient of the process noise 
	\[
	\sigma(x) = \sigma
	\] 
	is a constant matrix.  We denote $Q:=\sigma\sigma^\tp\in\Re^{d\times
		d}$.
	\item The prior $\nu_0$ is a Gaussian distribution with mean
	$\bar{m}_0\in\Re^{d}$ and variance $\Sigma_0\succ 0$.
\end{enumerate}

For this problem, we make the following restriction: The control input
$u_t(x)$ is restricted to be constant over $\Re^d$.  That is, the
control input is allowed to depend only upon time.  With such a
restriction, the controlled state evolves according to the sde:
\[
\ud \tilde{X}_t = A^\tp \tilde{X}_t \ud t + \sigma u_t \ud t + \sigma \ud \tilde{B}_t,\quad \tilde{X}_0\sim {\cal N}(m_0,V_0).
\]  
With a Gaussian prior, the distribution $\pi_t$ is also Gaussian whose
mean $m_t$ and variance $V_t$ evolve as follow:
\begin{align*}
	\frac{\ud m_t}{\ud t} &= A^\tp m_t + \sigma u_t\\
	\frac{\ud V_t}{\ud t} &= A^\tp V_t + V_t A + \sigma\sigma^\tp
\end{align*}
Since the variance is not affected by control, the only constraint for
the optimal control problem is due to the equation for the mean.  

It is an easy calculation to see that for the linear model,
\[
(\tilde{\cal A}(v) h)(x) = H^\tp (A^\tp x + \sigma v)
\]
Therefore, the Lagrangian becomes 
\begin{align*}
	{\cal L}(\rho,v;y) & = |v|^2 + |H^\tp m|^2 + \tr(HH^\tp V) + y H^\tp (A^\tp m + \sigma v)
\end{align*}
provided that $\rho \sim {\cal N}(m,V)$. 

For Gaussian distributions $\pi_0 = {\cal N}(m_0,V_0)$ and $\nu_0={\cal N}(\bar{m}_0,\Sigma_0)$, the divergence is given by the well known formula
\[
\kl(\pi_0\mid \nu_0) = \half\log\frac{|V_0|}{|\Sigma_0|} - \frac{d}{2} + \half \tr(V_0\Sigma_0^{-1}) + \half(m_0-\bar{m}_0)^\tp \Sigma_0^{-1}(m_0-\bar{m}_0)
\]
and the term due to the terminal condition is easily evaluated as
\[
\langle \pi_T, h\rangle = H^\tp m_T
\]
Because the control input does not affect the variance process,
we retain only the terms with mean and the control and express the
optimal control problem as follows: 
\begin{subequations}\label{eq:opt-cont-linear}
	\begin{align}
		&\mathop{\text{Minimize}}_{m_0, u}: \sJ(m_0,u\,;z) =  \half(m_0-\bar{m}_0)^\tp {\Sigma}_0^{-1}(m_0-\bar{m}_0) \label{eq:opt-cont-linear-a} \\
		&\quad\quad\quad+ \int_0^T\half |u_t|^2 + \half|H^\tp m_t|^2 + z_t^\tp H^\tp \dot{m}_t \ud t-z_T^\tp H^\tp m_T \nonumber \\
		&\text{Subject to} : \frac{\ud m_t}{\ud t} = A^\tp m_t + \sigma u_t \label{eq:opt-cont-linear-b}
	\end{align}
\end{subequations}
By a formal integration by parts, 
\begin{align*}
	\sJ(m_0,u\,;z) &= \half(m_0-\bar{m}_0)^\tp \bar{\Sigma}_0^{-1}(m_0-\bar{m}_0) \\
	&+ \int_0^T \half |u_t|^2 + \half|\dot{z} - H^\tp m_t|^2\ud t - \int_0^T\half|\dot{z}_t|^2\ud t
\end{align*}
This form appears in the construction of the minimum energy
estimator~\cite[Ch. 7.3]{bensoussan2018estimation}.

\section{Proofs of the statements}

\subsection{Derivation of Lagrangian: Euclidean case}\label{apdx:lagrangian-sde}

By Girsanov's theorem, the Radon-Nikodym derivative is obtained
(see~\cite[Eqn. 35]{reich2019data}) as follows:
\[
\frac{\ud \sQ}{\ud \sP_{X}}(\tilde{X}) = \frac{\ud \pi_0}{\ud \nu_0}(\tilde{X}_0) \; \exp\Big(\int_0^T \half|u_t(\tilde{X}_t)|^2 \ud t + u_t(\tilde{X}_t) \ud \tilde{B}_t \Big)
\]
Thus, we obtain the relative entropy formula:
\begin{align*}
	\kl(\sQ\mid \sP) &= \E\Big(\log\dfrac{\ud \pi_0}{\ud \nu_0}(\tilde{X}_0) + \int_0^T \half|u_t(\tilde{X}_t)|^2 \ud t + u_t(\tilde{X}_t) \ud \tilde{B}_t \Big)\\
	&=\kl(\pi_0\mid \nu_0) + \int_0^T \half\langle \pi_t,|u_t|^2\rangle \ud t
\end{align*}

\subsection{Derivation of Lagrangian: finite state-space  case}\label{apdx:lagrangian-finite}

The derivation of the Lagrangian is entirely analogous to the Euclidean case except the R-N derivative is given according to~\cite[Prop. 2.1.1]{van2006filtering}:
\begin{align*}
	\frac{\ud \sQ}{\ud \sP_{X}}(\tilde{X}) &= \frac{\ud \pi_0}{\ud \nu_0}(\tilde{X}_0)\exp\Big(-\sum_{i,j} \int_0^T [A]_{ij}[u_t]_{ij}1_{\tilde{X}_t = e_i} \Big)\\
	&\quad\quad\quad \prod_{0<t\leq T} \sum_{i\neq j}[u_{t-}]_{ij} 1_{\tilde{X}_{t-} = e_i}1_{\tilde{X}_{t} = e_j}
\end{align*}
Upon taking log and expectation of both sides,  
we arrive at the relative entropy formula:
\begin{align*}
	\kl(\sQ\mid \sP_{X}) &= \E\Big(\log\dfrac{\ud \pi_0}{\ud \nu_0}(\tilde{X}_0) + \int_0^T -\sum_{i,j}[A]_{ij}[u]_{ij} 1_{\tilde{X}_t = e_i}\Big)\\
	& \quad + \E\Big(\sum_{0<t\leq T} \sum_{i\neq j} \log [u_{t-}]_{ij} 1_{\tilde{X}_{t-} = e_i}1_{\tilde{X}_{t} = e_j}\Big)\\
	&=\kl(\pi_0\mid \nu_0) + \int_0^T \pi_t^\tp C(u_t) \ud t
\end{align*}

\subsection{Proof of Proposition~\ref{thm:opt-ctrl-sde}}
\label{apdx:opt-ctrl-sde}

The standard approach is to incorporate the constraint into the objective function by introducing the Lagrange multiplier $\lambda = \{\lambda_t:0\leq t\leq T\}$ as follows:
\begin{align*}
	\tilde{J}(u,\lambda\,;\pi_0,z)	&= \kl(\pi_0 \mid  \nu_0) + \int_0^T \half \langle \pi_t, |u_t|^2 + h^2\rangle + z_t\langle \pi_t, \tilde{\cal A}(u_t)h\rangle \ud t\\
	&\quad +\int_0^T \langle \lambda_t,  \frac{\partial\pi_t}{\partial t} - \tilde{\cal A}^\dagger(u_t) \pi_t \rangle   \ud t - z_T\langle \pi_T,h\rangle
\end{align*}
Upon using integration by parts and the definition of the adjoint operator, after some manipulation involving completion of squares, we arrive at
\begin{align*}
	\tilde{\sf J}(u,&\lambda\,;\pi_0,z)=\kl(\pi_0 \mid  \nu_0) + \int_0^T \half \langle \pi_t, |u_t - \sigma^\tp\nabla(\lambda_t - z_th)|^2\rangle \ud t\\
	&-\int_0^T \langle \pi_t, \frac{\partial}{\partial t}\lambda_t + {\cal A}(\lambda_t - z_th)-\half h^2+\half|\sigma^\tp\nabla(\lambda_t - z_th)|^2\rangle \ud t\\
	&+ \langle \pi_T,\lambda_T-z_Th\rangle - \langle \pi_0,\lambda_0\rangle
\end{align*}
Therefore, it is natural to pick $\lambda$ to satisfy the following partial differential equation:
\begin{align}
	-\frac{\partial\lambda_t}{\partial t}(x) &= \big({\cal A} (\lambda_t(\cdot) - z_th(\cdot))\big) - \half h^2(x)+\half \big|\sigma^\tp\nabla (\lambda_t - z_th)(x)\big|^2 \label{eq:lambda-basic-form}\\
	&= e^{-(\lambda_t(x) - z_th(x))}({\cal A}e^{\lambda_t(\cdot) - z_th(\cdot)})(x) - \half h^2(x) \nonumber
\end{align}
with the boundary condition $\lambda_T(x) = z_Th(x)$. With this choice, the objective function becomes
\begin{align*}
	\tilde{\sf J}(u\,;\lambda,\pi_0,z) &= \kl(\pi_0 \mid  \nu_0) - \langle \pi_0, \lambda_0\rangle \\
	&+ \int_0^T \half\pi_t\big( \big|u_t -  \sigma^\tp \nabla(\lambda_t-z_th)\big|^2\big) \ud t 
\end{align*}
which suggest the optimal choice of control is:
\begin{equation*}
	u_t(x) = \sigma^\tp(x) \nabla(\lambda_t-z_th)(x)
\end{equation*}
With this choice, the objective function becomes 
\begin{align*}
	\kl(\pi_0\mid \nu_0) - \langle \pi_0, \lambda_0\rangle &= \int_{\bS}\pi_0(x) \log\frac{\pi_0(x)}{\nu_0(x)}- \lambda_0(x)\pi_0(x)\ud x\\
	&=\int_{\bS} \pi_0(x)\log\frac{\pi_0(x)}{\nu_0\exp(\lambda_0(x))}\ud x
\end{align*}
which is minimized by choosing 
\[
\pi_0(x) = \frac{1}{\normfactor} \nu_0(x)\exp(\lambda_0(x))
\]
where $C$ is the normalization constant.  
\qed

\subsection{Proof of Proposition~\ref{thm:opt-ctrl-finite}}
\label{apdx:opt-ctrl-finite}

The proof for the finite state-space case is entirely analogous to the proof for the Euclidean case. The Lagrange multiplier $\lambda = \{\lambda_t\in\Re^d: 0\leq t\leq T\}$ is introduced to transform the optimization problem into an unconstrained problem:
\begin{align*}
	\tilde{\sf J}(u,\lambda\,;\pi_0,z) &= \kl(\pi_0\mid \nu_0)+ \int_0^T\pi_t^\tp \big(C(u_t)+\half h^2 + z_t \tilde{A}(u_t)h \big)\ud t \\
	&\quad + \int_0^T\lambda_t^\tp\big(\frac{\ud \pi_t}{\ud t} - \tilde{A}^\tp(u_t) \pi_t\big)\ud t - z_T h^\tp \pi_T
\end{align*}
Upon using integral by parts,
\begin{align*}
	\tilde{\sf J}(u,\lambda\,;\pi_0,z)&= \kl(\pi_0\mid \nu_0) + \int_0^T \pi_t^\tp \big(C(u_t) -\tilde{A}(u_t)(\lambda_t - z_th)\big) \ud t\\
	&\quad+\int_0^T\pi_t^\tp (-\dot{\lambda}_t+ \half h^2)\ud t +\pi_T^\tp(\lambda_T-z_Th) - \pi_0^\tp \lambda_0
\end{align*}
The first integrand is
\begin{align*}
	[C(u_t) -\tilde{A}(u_t)&(\lambda_t - Z_th)]_i = \sum_{j\neq i} A_{ij}\big([u]_{ij}(\log[u_t]_{ij}-1) -[u_t]_{ij}([\lambda_t - Z_th]_j-[\lambda_t - Z_th]_i)\big)- A_{ii}
\end{align*}
The minimizer is obtained, element by element, as
\[
[u_t]_{ij} = e^{([\lambda_t - z_th]_j-[\lambda_t - z_th]_i)}
\]
and the corresponding minimum value is obtained by:
\[
[C(u_t^*) -\tilde{A}_t(\lambda_t - Z_th)]_i = -[Ae^{\lambda_t-z_th}]_i [e^{-(\lambda_t-z_th)}]_i
\]
Therefore with the minimum choice of $u_t$ above,
\begin{align*}
	\tilde{\sf J}(u,\lambda\,;\pi_0,z)&= \kl(\pi_0\mid \nu_0) + \int_0^T \pi_t^\tp \big(-(Ae^{\lambda_t-z_th}) \cdot  e^{-(\lambda_t-z_th)}\big) \ud t\\
	&\quad+\int_0^T\pi_t^\tp (-\dot{\lambda}_t+ \half h^2)\ud t +\pi_T^\tp(\lambda_T-z_Th) - \pi_0^\tp \lambda_0
\end{align*}
Upon choosing $\lambda$ according to:
\[
-[\dot{\lambda}_t]_i = [Ae^{\lambda_t-z_th}]_i [e^{-(\lambda_t-z_th)}]_i - \half h_i^2,\quad \lambda_T = z_Th
\]
The objective function simplifies to 
\[
\kl(\pi_0\mid \nu_0) - \pi_0^\tp \lambda_0 = \sum_{i=1}^d[\pi_0]_i \log\frac{[\pi_0]_i}{[\nu_0]_ie^{[\lambda_0]_i}}
\]
where the minimum value is obtained by choosing
\[
[\pi_0]_i =\frac{1}{\normfactor} [\nu_0]_ie^{[\lambda_0]_i}
\]
where $C$ is the normalization constant.  \qed

\subsection{Proof of Proposition~\ref{thm:forward-backward}}
\label{apdx:forward-backward}

\subsubsection{Euclidean case} Equation~\eqref{eq:hamiltons-eqn-sde-b} is
identical to the backward path-wise 
equation~\eqref{eq:pw-Zakai-b}.  So, we need to only derive the
equation for $\mu_t$.  
Using the regular form of the product formula,
\begin{align*}
	\frac{\partial \mu_t}{\partial t} &= \frac{1}{\pi_t}\frac{\partial \pi_t}{\partial t} -\frac{\partial \lambda_t}{\partial t} \\
	&=\frac{1}{\pi_t}(\tilde{\cal A}^\dagger(u_t)\pi_t) + e^{-(\lambda_t - z_th)}({\cal A}e^{\lambda_t(\cdot) - z_th(\cdot)}) - \half h^2
\end{align*}
With optimal control $u_t = \sigma^\tp\nabla(\lambda_t-z_th)$,
\begin{equation*}
	(\tilde{\cal A}^\dagger(u_t)\pi_t)
	=({\cal A}^\dagger\pi_t)-\divg\big(\sigma\sigma^\tp \nabla
	\pi_t\big) +\pi_t \divg\big(\sigma\sigma^\tp \nabla(\mu_t+z_th)\big) +(\nabla \pi_t)^\tp(\sigma\sigma^\tp \nabla(\mu_t+z_th))
\end{equation*}
and
\begin{equation*}
	e^{-(\lambda_t - z_th)}({\cal A}e^{\lambda_t(\cdot) - z_th(\cdot)})
	=\frac{1}{\pi_t}({\cal A}\pi_t) - \half|\sigma^\tp \nabla \log \pi_t|^2 - ({\cal A}(\mu_t + z_th)) + \half \big|\sigma^\tp\nabla \log(\pi_t) - \sigma^\tp \nabla(\mu_t +z_th)\big|^2
\end{equation*}
Therefore,
\begin{align*}
	\frac{\partial \mu_t}{\partial t} 
	&=\frac{1}{\pi_t}\big(({\cal A}^\dagger\pi_t) + ({\cal A}\pi_t)-\divg(\sigma\sigma^\tp \nabla \pi_t)\big) -({\cal A}(\mu_t + z_th)) + \divg\big(\sigma\sigma^\tp \nabla(\mu_t+z_th)\big) + \half\big|\sigma^\tp \nabla(\mu_t+z_th)\big|^2 - \half h^2\\
	&=e^{-(\mu_t(x)+z_th(x))}\big({\cal A}^\dagger
	e^{(\mu_t(\cdot) +z_th(\cdot) )}\big)(x) - \half h^2(x)
\end{align*}
with the boundary condition $\mu_0 = \log \nu_0$.  

\subsubsection{Finite state-space case}
Equation~\eqref{eq:hamiltons-eqn-finite-b} is identical to the
backward path-wise equation~\eqref{eq:pw-Zakai-finite-b}.
To derive the equation for $\mu_t$, use the product formula
\begin{align*}
	\Big[\frac{\ud \mu_t}{\ud t}\Big]_i &= \frac{1}{[\pi_t]_i}\Big[\frac{\ud \pi_t}{\ud t}\Big]_i - \Big[\frac{\ud \lambda_t}{\ud t}\Big]_i\\
	&=\frac{1}{[\pi_t]_i}\big[\tilde{A}^\tp(u_t)\pi_t\big]_i +  [e^{-(\lambda_t-z_th)}]_i[Ae^{\lambda_t+z_th}]_i - \half [h^2]_i
\end{align*}
The first term is:
\begin{align*}
	\big[\tilde{A}^\tp(u_t)\pi_t\big]_i 
	&= \sum_{j=1}^d \Big([A]_{ji}[u_t]_{ji}[\pi_t]_j-[A]_{ij}[u_t]_{ij}[\pi_t]_i\Big)
\end{align*}
and the second term is:
\begin{align*}
	[e^{-(\lambda_t-z_th)}]_i[Ae^{\lambda_t+z_th}]_i &= \frac{1}{[\pi_t]_i}[e^{\mu_t+z_th}]_i\sum_{j=1}^d [A]_{ij} [\pi_t]_j [e^{-(\mu_t +z_th)}]_j
\end{align*}
The formula for the optimal control gives
\begin{align*}
	[u_t]_{ij} 
	&=\frac{[\pi_t]_j}{[\pi_t]_i}[e^{-(\mu_t + z_th)}]_j[e^{\mu_t + z_th}]_i
\end{align*}
Combining these expressions,
\begin{align*}
	\Big[\frac{\ud \mu_t}{\ud t}\Big]_i &= \sum_{j=1}^d[A]_{ji}[e^{-(\mu_t + z_th)}]_i[e^{\mu_t + z_th}]_j - \half [h^2]_i\\
	&=[e^{-(\mu_t + z_th)}]_i [A^\tp e^{\mu_t + z_th}]_i - \half [h^2]_i
\end{align*}
which is precisely the path-wise form of the equation~\eqref{eq:pw-Zakai-finite-a}. 
At time $t=0$, $\mu_0 = \log(\normfactor[\pi_0]_i) - [\lambda_0]_i = \log[\nu_0]_i
$.

\subsubsection{Smoothing distribution}
Since $(\lambda_t, \mu_t)$ is the solution to the path-wise form of
the Zakai equations, the optimal trajectory
\[
\pi_t = \frac{1}{\normfactor}e^{\mu_t+\lambda_t}
\]
represents the smoothing distribution.
\qed

\subsection{Proof of Proposition~\ref{thm:opt-ctrl-sde-hjb}}\label{apdx:proof-hjb}

The dynamic programming equation for the optimal control problem is given by
(see~\cite[Ch. 11.2]{bensoussan2018estimation}):
\begin{equation}\label{eq:HJB-V}
	\min_{u\in\Re^p} \Big\{\frac{\partial V_t}{\partial t}(x) + (\tilde{\cal A}(u) V_t)(x) + \ell(x,u\,;z_t)\Big\} = 0
\end{equation}
Therefore,
\begin{align*}
	-\frac{\partial V_t}{\partial t}(x) &= ({\cal A}V_t)(x) + \half h^2(x) + z_t({\cal A}h)(x) + \min_{u}\Big\{\half |u|^2 + u^\tp\big(\sigma^\tp \nabla V_t(x) + z_t\sigma^\tp \nabla h(x)\big)\Big\}
\end{align*}
Upon using the completion-of-square trick, the minimum is attained by a feedback form:
\[
u^* = -\sigma^\tp \nabla (V_t +z_th)(x)
\]
The resulting HJB equation is given by
\begin{align*}
	-\frac{\partial V_t}{\partial t}(x) &= \big({\cal A}(V_t + z_th)\big)(x) + \half h^2(x) -\half|\sigma^\tp\nabla (V_t+z_th)|^2
\end{align*}
with boundary condition $V_T(x) = - z_Th(x)$. 
Compare the HJB equation with the equation~\eqref{eq:lambda-basic-form} for $\lambda$, and it follows
\[
V_t(x) = -\lambda_t(x)
\]
\qed

\newpage
\backmatter

\printbibliography[heading=bibintoc,title={References}]

\end{document}